\documentclass[a4paper, reqno, 14pt]{amsart}

\usepackage[usenames,dvipsnames]{color}

\usepackage{amsthm,amsfonts,amssymb,amsmath,amsxtra}
\usepackage[all]{xy}
\SelectTips{cm}{}
\usepackage[colorlinks=citecolor=Black, linkcolor=Red, urlcolor=Blue]{hyperref}

   \usepackage{verbatim}

\usepackage[margin=1.25in]{geometry}

\usepackage{mathrsfs}

\RequirePackage{xspace}
\RequirePackage{etoolbox}
\RequirePackage{varwidth}
\RequirePackage{enumitem}
\RequirePackage{tensor}
\RequirePackage{mathtools}
\RequirePackage{longtable}
\RequirePackage{multirow}

\newcommand{\mar}[1]{\marginpar{\tiny #1}} 

\newcommand{\BA}{\ensuremath{\mathbb {A}}\xspace}

\newcommand{\BG}{\ensuremath{\mathbb {G}}\xspace}

\newcommand{\BQ}{\ensuremath{\mathbb {Q}}\xspace}
\newcommand{\BR}{\ensuremath{\mathbb {R}}\xspace}
\newcommand{\BS}{\ensuremath{\mathbb {S}}\xspace}

\newcommand{\CA}{\ensuremath{\mathcal {A}}\xspace}

\newcommand{\CS}{\ensuremath{\mathcal {S}}\xspace}

\newcommand{\ab}{{\mathrm{ab}}}

\DeclareMathOperator{\Aut}{Aut}

\DeclareMathOperator{\End}{End}

\DeclareMathOperator{\Gal}{Gal}
\newcommand{\GL}{\mathrm{GL}}

\DeclareMathOperator{\Hom}{Hom}

\newcommand{\id}{\ensuremath{\mathrm{id}}\xspace}

\DeclareMathOperator{\Lie}{Lie}

\DeclareMathOperator{\Nm}{Nm}

\DeclareMathOperator{\ord}{ord}

\DeclareMathOperator{\rank}{rank}

\newcommand{\PGL}{{\mathrm{PGL}}}

\DeclareMathOperator{\Res}{Res}

\DeclareMathOperator{\Spec}{Spec}
\DeclareMathOperator{\Spf}{Spf}

\newcommand{\Sh}{{\mathrm{Sh}}}

\newtheorem{theorem}{Theorem}
\newtheorem{proposition}[theorem]{Proposition}
\newtheorem{lemma}[theorem]{Lemma}

\newtheorem{corollary}[theorem]{Corollary}

\theoremstyle{definition}
\newtheorem{definition}[theorem]{Definition}

\newtheorem{remark}[theorem]{Remark}

\newtheorem{variant}[theorem]{Variant}

\numberwithin{equation}{section}
\numberwithin{theorem}{section}

\setitemize[0]{leftmargin=*,itemsep=\the\smallskipamount}
\setenumerate[0]{leftmargin=*,itemsep=\the\smallskipamount}

\renewcommand{\to}{%
   \ifbool{@display}{\longrightarrow}{\longrightarrow}%
   }
\let\shortmapsto\mapsto
\renewcommand{\mapsto}{%
   \ifbool{@display}{\longmapsto}{\shortmapsto}%
   }
\newlength{\olen}
\newlength{\ulen}
\newlength{\xlen}
\newcommand{\xra}[2][]{%
   \ifbool{@display}%
      {\settowidth{\olen}{$\overset{#2}{\longrightarrow}$}%
       \settowidth{\ulen}{$\underset{#1}{\longrightarrow}$}%
       \settowidth{\xlen}{$\xrightarrow[#1]{#2}$}%
       \ifdimgreater{\olen}{\xlen}%
          {\underset{#1}{\overset{#2}{\longrightarrow}}}%
          {\ifdimgreater{\ulen}{\xlen}%
             {\underset{#1}{\overset{#2}{\longrightarrow}}}
             {\xrightarrow[#1]{#2}}}}%
      {\xrightarrow[#1]{#2}}
   }
\makeatother
\newcommand{\xyra}[2][]{%
   \settowidth{\xlen}{$\xrightarrow[#1]{#2}$}%
   \ifbool{@display}%
      {\settowidth{\olen}{$\overset{#2}{\longrightarrow}$}%
       \settowidth{\ulen}{$\underset{#1}{\longrightarrow}$}%
       \ifdimgreater{\olen}{\xlen}%
          {\mathrel{\xymatrix@M=.12ex@C=3.2ex{\ar[r]^-{#2}_-{#1} &}}}%
          {\ifdimgreater{\ulen}{\xlen}%
             {\mathrel{\xymatrix@M=.12ex@C=3.2ex{\ar[r]^-{#2}_-{#1} &}}}
             {\mathrel{\xymatrix@M=.12ex@C=\the\xlen{\ar[r]^-{#2}_-{#1} &}}}}}%
      {\mathrel{\xymatrix@M=.12ex@C=\the\xlen{\ar[r]^-{#2}_-{#1} &}}}%
   }
\makeatletter
\newcommand{\xla}[2][]{%
   \ifbool{@display}%
      {\settowidth{\olen}{$\overset{#2}{\longleftarrow}$}%
       \settowidth{\ulen}{$\underset{#1}{\longleftarrow}$}%
       \settowidth{\xlen}{$\xleftarrow[#1]{#2}$}%
       \ifdimgreater{\olen}{\xlen}%
          {\underset{#1}{\overset{#2}{\longleftarrow}}}%
          {\ifdimgreater{\ulen}{\xlen}%
             {\underset{#1}{\overset{#2}{\longleftarrow}}}
             {\xleftarrow[#1]{#2}}}}%
      {\xleftarrow[#1]{#2}}
   }
\newcommand{\isoarrow}{%
   \ifbool{@display}{\overset{\sim}{\longrightarrow}}{\xrightarrow\sim}%
   }
   
\DeclareMathOperator{\Trace}{Tr}

\DeclareMathOperator{\height}{height}

\begin{document}


\title{On the $p$-adic uniformization of  quaternionic Shimura curves }
\author{Jean-Fran\c cois Boutot}
\address{12 av. Parmentier, 75011 Paris, France.}
\email{jf.boutot@outlook.fr}
\author{Thomas Zink}
\address{Fakult\"at f\"ur Mathematik,
Universit\"at Bielefeld,
Postfach 100131,
33501 Bielefeld, Germany}
\email{zink@math.uni-bielefeld.de}

\date{\today}
\maketitle

\tableofcontents

\section{Introduction}

Let $D$ be a quaternion division algebra over a totally real number field $F$.
We assume that $D$ splits at exactly one infinite place
$\chi_0: F \rightarrow \mathbb{R}$. We fix an isomorphism
$D \otimes_{F, \chi_0} \mathbb{R} \cong {\rm M}_2(\mathbb{R})$. Let $h_D$ be the
homomorphism  
\begin{equation}\label{defhD}
    \mathbb{C}^{\times}  \rightarrow  (D \otimes_{F, \chi_0} \mathbb{R})^{\times}  
    \subset (D \otimes \mathbb{R})^{\times} , \quad 
z = a + b \mathbf{i}  \mapsto  \left(
  \begin{array}{rr} 
    a & -b\\
    b & a
  \end{array} \right)   .
\end{equation} 

We regard $D^\times$ as an algebraic group over $\mathbb{Q}$. Then
$(D^\times, h_D)$ is a Shimura datum  of a Shimura curve $\Sh(D^{\times}, h_D)$.
We prove
the $p$-adic uniformization of these curves under certain conditions, discovered by Cherednik \cite{Ch}. Let us describe our main result.

We fix a prime number $p$ and choose a diagram  of field embeddings,
\begin{equation}\label{LDiagramm} 
  \mathbb{C} \leftarrow \bar{\mathbb{Q}} \rightarrow \bar{\mathbb{Q}}_p.
\end{equation}
This gives a bijection 
\begin{displaymath}
   \Hom_{\mathbb{Q}{\rm -Alg}}(F, \mathbb{C}) = 
  \Hom_{\mathbb{Q}{\rm -Alg}}(F, \bar{\mathbb{Q}}_p) .
  \end{displaymath}
Let $\mathfrak{p}_0, \ldots, \mathfrak{p}_s$ be the prime ideals of $F$ over
$p$. We assume that $\mathfrak{p}_0$ is induced by
$\chi_0: F \rightarrow \bar{\mathbb{Q}}_p$ and that
$D_{\mathfrak{p}_0} = D \otimes_{F} F_{\mathfrak{p}_0}$ is a division algebra.
Let $O_{D_{\mathfrak{p}_0}}$
be the maximal order of $D_{\mathfrak{p}_0}$. We choose an open and compact
subgroup $\mathbf{K} \subset (D \otimes_{\mathbb{Q}} \mathbb{A}_f)^{\times}$
as follows. We set $\mathbf{K}_{\mathfrak{p}_0} = O_{D_{\mathfrak{p}_0}}^{\times}$. For
$i = 1, \ldots, s$ we choose arbitrarily open and compact subgroups
$\mathbf{K}_{\mathfrak{p}_i} \subset D_{\mathfrak{p}_i}^{\times}$. We set
\begin{equation}\label{defKp}
  \mathbf{K}_p = \prod_{i=0}^{s} \mathbf{K}_{\mathfrak{p}_i} \subset
  (D \otimes_{\mathbb{Q}} \mathbb{Q}_p)^{\times}. 
  \end{equation}
We also choose a sufficiently small open compact subgroup
$\mathbf{K}^p \subset (D \otimes_{\mathbb{Q}} \mathbb{A}^p_f)^{\times}$ and set
\begin{displaymath}
\mathbf{K} = \mathbf{K}_p \mathbf{K}^p.
\end{displaymath}
The Shimura field $E(D^{\times}, h_{D})$ is $\chi_0(F)$. The diagram
(\ref{LDiagramm}) induces a $p$-adic place $\nu$ of the Shimura field, and 
$\chi_0$ gives an identification  $E(D^{\times}, h_{D})_{\nu} \cong F_{\mathfrak{p}_0}$.
As an abbreviation we write $E_{\nu} = E(D^{\times}, h_{D})_{\nu}$. We denote by
$\breve{E}_{\nu}$ the completion of the maximal unramified extension of
$E_{\nu}$. 

We will prove (see Corollary \ref{remstab}) that the  curve $\Sh_{\mathbf{K}}(D^{\times}, h_D)$  has stable reduction over
$\Spec O_{E_{\nu}}$, i.e., $\Sh_{\mathbf{K}}(D^{\times}, h_D)$ extends to a  stable curve  $\widetilde{\Sh}_{\mathbf{K}}(D^{\times}, h_D)$ over  $O_{E_{\nu}}$, in the sense of Deligne-Mumford \cite{DM}. By \cite[Lem. 1.12]{DM}, this extension is unique up to unique isomorphism.  The action of 
$(D \otimes \mathbb{A}_f)^{\times}$ on the tower $\Sh_{\mathbf{K}}(D^{\times}, h_D)$ for
varying $\mathbf{K}$ as above extends to the stable model. 

Let $\check{D}$ be a \emph{Cherednik twist} of $D$.  This is a quaternion
algebra over $F$ such that
\begin{equation}\label{CherednikUnif1e}
  \check{D} \otimes_{F} \mathbb{A}_{F,f}^{\mathfrak{p}_0} \cong
  D \otimes_{F} \mathbb{A}_{F,f}^{\mathfrak{p}_0} 
  \end{equation} 
and such that
$\check{D} \otimes_{F} F_{\mathfrak{p}_0} \cong \rm{M}_2 (F_{\mathfrak{p}_0})$  
and such that $\check{D}$ is non-split at all infinite places of $F$.
For a more canonical definition of $\check{D}$ see (\ref{Cheredniktwist}).

Let $\hat{\Omega}_{F_{\mathfrak{p}_0}}^2$ be the \emph{integral model of the Drinfeld
  halfplane} for the local field $F_{\mathfrak{p}_0}$, cf. \cite{Dr}. 
It is a $p$-adic formal scheme over $\Spf O_{F_{\mathfrak{p}_0}}$ with an action
of the group
$\check{D}_{\mathfrak{p}_0}^{\times} = (\check{D} \otimes_{F} F_{\mathfrak{p}_0})^{\times} \cong \GL_2(F_{\mathfrak{p}_0})$, cf. \cite{Dr}, (\ref{PGL-O1e}). 
This action factors through an action of $\PGL_2(F_{\mathfrak{p}_0})$.  
We consider on 
\begin{equation}\label{CherednikUnif2e} 
  (\hat{\Omega}_{F_{\mathfrak{p}_0}}^2 \times_{\Spf O_{F_{\mathfrak{p}_0}},\chi_0}
  \Spf O_{\breve{E}_\nu})\times (D \otimes_{F} F_{\mathfrak{p}_0})^{\times}/
  \mathbf{K}_{\mathfrak{p}_0} =
  (\hat{\Omega}_{F_{\mathfrak{p}_0}}^2
  \times_{\Spf O_{F_{\mathfrak{p}_0}},\chi_0} \Spf O_{\breve{E}_\nu}) \times \mathbb{Z}, 
\end{equation}
the action of $\check{D}_{\mathfrak{p}_0}$ which is on the first factor on the
right hand side obtained from the action introduced above and which acts on
$\mathbb{Z}$ by translation with
$\ord_{F_{\mathfrak{p}_0}} \det_{\check{D}_{\mathfrak{p}_0}/F_{\mathfrak{p}_0}}$, cf. Proposition
\ref{RZ7p}. 
We formulate our main result as follows. 
\begin{theorem}\label{MainIntro}
Let 
$\widetilde{\Sh}_{\mathbf{K}}(D^{\times}, h_{D})^\wedge_{ \, /\Spf O_{\breve{E}_{\nu}}}$  be the
completion of the scheme
$\widetilde{\Sh}_{\mathbf{K}}(D^{\times},h_D)\times_{\Spec O_{E_{\nu}}}\Spec O_{\breve{E}_{\nu}}$
along the special fiber. Then there is an isomorphism of formal schemes
  \begin{equation}\label{CherednikUnif3e}  
    \check{D}^{\times} \backslash \big((\hat{\Omega}_{F_{\mathfrak{p}_0}}^2 
    \times_{\Spf O_{F_{\mathfrak{p}_0}},\chi_0} \Spf O_{\breve{E}_\nu})\times
    (D \otimes \mathbb{A}_f)^{\times}/\mathbf{K})\big) 
\overset{\sim}{\longrightarrow}
    \widetilde{\Sh}_{\mathbf{K}}(D^{\times}, h_{D})^\wedge_{\, / \Spf O_{\breve{E}_{\nu}}} .
  \end{equation}
  The action of $\check{D}^{\times}$ is given by (\ref{CherednikUnif1e}) and
  (\ref{CherednikUnif2e}). 
  For varying $\mathbf{K}$ this uniformization isomorphism is compatible with 
  the action of Hecke correspondences in $(D\otimes \mathbb{A}_f)^{\times}$ on both sides.  
 
  Let $\Pi \in D_{\mathfrak{p}_0}$ be a prime element in this division algebra
  over $F_{\mathfrak{p}_0}$. We denote also by $\Pi$ the image by the canonical
  embedding $D_{\mathfrak{p}_0}^\times \subset (D \otimes \mathbb{A}_f)^{\times}$. 
  Let $\tau \in \Gal(\breve{E}_{\nu}/ E_{\nu})$ be the Frobenius automorphism and 
  $\tau_c = \Spf\tau^{-1}: \Spf O_{\breve{E}_{\nu}}\rightarrow\Spf O_{\breve{E}_{\nu}}$.  
  The natural Weil descent datum with respect to  
  $O_{\breve{E}_\nu}/O_{E_{\nu}}$ on the right hand side of (\ref{CherednikUnif3e})
  induces on the
  left hand side the Weil descent datum given by the following diagram
  \begin{displaymath}
\xymatrix{
    \check{D}^{\times} \backslash ((\hat{\Omega}_{F_{\mathfrak{p}_0}}^2 
    \times_{\Spf O_{F_{\mathfrak{p}_0}},\chi_0} \Spf O_{\breve{E}_\nu})\times
    D^{\times}(\mathbb{A}_f)/\mathbf{K})
  \ar[d]_{ \id \times \tau_c \times \Pi^{-1}} \ar[r] & 
 \widetilde{\Sh}_{\mathbf{K}}(D^{\times}, h_{D})^\wedge_{\,/ \Spf O_{\breve{E}_{\nu}}}
  \ar[d]^{\id \times \tau_c}\\
   \check{D}^{\times} \backslash ((\hat{\Omega}_{F_{\mathfrak{p}_0}}^2 
    \times_{\Spf O_{F_{\mathfrak{p}_0}},\chi_0} \Spf O_{\breve{E}_\nu})\times
    D^{\times}(\mathbb{A}_f)/\mathbf{K})
  \ar[r] & 
 \widetilde{\Sh}_{\mathbf{K}}(D^{\times}, h_{D})^\wedge_{\,/ \Spf O_{\breve{E}_{\nu}}}. 
     }
    \end{displaymath}
\end{theorem}
The left hand side of (\ref{CherednikUnif3e}) can be written in more concrete
terms as follows.  We write
$\mathbf{K} = \mathbf{K}_{\mathfrak{p}_0} \mathbf{K}^{\mathfrak{p}_0}$ where
$\mathbf{K}^{\mathfrak{p}_0} \subset (\check{D}\otimes_{F}\mathbb{A}_{F,f}^{\mathfrak{p}_0})^{\times} = (D \otimes_{F} \mathbb{A}_{F,f}^{\mathfrak{p}_0})^{\times}$. For 
$g \in (D \otimes_{F} \mathbb{A}_{F,f}^{\mathfrak{p}_0})^{\times}$, let
$$\Gamma_g = \{d \in \check{D}^\times \cap g\mathbf{K}^{\mathfrak{p}_0}g^{-1} \mid \ord_{F_{\mathfrak{p}_0}} \det d = 0\} .$$ 
Let $\bar{\Gamma}_g$ be the image of $\Gamma_g$ by the natural map
$\check{D}^{\times} \rightarrow  \check{D}^{\times}_{\mathfrak{p}_0}\rightarrow \PGL_2(F_{\mathfrak{p}_0})$. 
Then $\bar{\Gamma}_g$ is a discrete cocompact subgroup of
$\PGL_2(F_{\mathfrak{p}_0})$, comp.  the proof of Proposition
\ref{uniform4l}. It acts properly discontinuously on  the formal scheme $\hat{\Omega}_{F_{\mathfrak{p}_0}}^2 \times_{\Spf O_{F_{\mathfrak{p}_0}},\chi_0} \Spf O_{\breve{E}_\nu}$, and the quotients 
$\mathfrak{X}_{\Gamma_g} := \bar{\Gamma}_g \backslash (\hat{\Omega}_{F_{\mathfrak{p}_0}}^2 \times_{\Spf O_{F_{\mathfrak{p}_0}},\chi_0} \Spf O_{\breve{E}_\nu})$ are exactly the connected components of the formal scheme on the LHS of  (\ref{CherednikUnif3e}), for varying $g$. By \cite{Mum}, $\mathfrak{X}_{\Gamma_g}$ 
is algebraizable, i.e. it is the formal scheme associated to a proper scheme
$\mathfrak{X}_{\Gamma_g}^{\mathrm{alg}}$ over $O_{\breve{E}_\nu}$. The general fibers
of these schemes  for varying $g$ give back  the
connected components of $\Sh(D^{\times}, h_D)_{\breve{E}_{\nu}}$.

We prove Theorem \ref{MainIntro}  using the method which Drinfeld \cite{Dr}
used in the case
$F = \mathbb{Q}$. The case $F \neq \mathbb{Q}$ becomes more difficult
because in this case the  Shimura curve is not described by a PEL-moduli 
problem. In fact, the Shimura curve is then a Shimura variety of abelian type which is not of Hodge type. Also, the weight homomorphism $w\colon \BG_m\to D^\times_\BR$ is not defined over $\BQ$. The existence of a canonical model is proved by the method of Shimura and Deligne \cite[\S 6]{D-TS}, by embedding this Shimura variety into one of PEL type   (\emph{m\'ethode des mod\`eles \'etranges}).  We use here a variant of this method to construct  integral models  over $O_{\breve{E}_{\nu}}$ of the Shimura curve. A similar approach was used by  Carayol \cite{C}.  More precisely, we show that  the Shimura curve
$\Sh(D^{\times}, h_D)$ can be embedded as an open and closed subscheme in a Shimura variety 
  which is an unramified twist of a PEL-moduli scheme which has a natural integral model. This PEL-moduli scheme can be so chosen that it 
has a $p$-adic uniformization by \cite[\S 6]{RZ}.  In this way, we obtain the isomorphism  \eqref{CherednikUnif3e}. Finally we must determine the
descent datum to obtain the result over $E_{\nu}$. Let us explain our strategy
in more detail. 

Let $K/F$ be a CM-field and assume that each $\mathfrak{p}_i$ is split in
$K$, i.e. $\mathfrak{p}_i O_K = \mathfrak{q}_i \bar{\mathfrak{q}_i}$.
By (\ref{LDiagramm}) we write 
\begin{equation}\label{CherednikUnif4e}
  \begin{array}{l} 
    \Phi := \Hom_{\mathbb{Q}{\rm -Alg}}(K, \mathbb{C}) = 
    \Hom_{\mathbb{Q}{\rm -Alg}}(K, \bar{\mathbb{Q}}_p) = \\[2mm] 
  \big(\coprod_{i=0}^s \Hom_{\mathbb{Q}{\rm -Alg}}(K_{\mathfrak{q}_i}, \bar{\mathbb{Q}}_p)\big)
  \; \coprod \; 
  \big(\coprod_{i=0}^s \Hom_{\mathbb{Q}{\rm -Alg}}(K_{\bar{\mathfrak{q}}_i}, \bar{\mathbb{Q}}_p)\big). 
    \end{array}
\end{equation}
We denote by $\varphi_0 \in \Phi$ the extension of $\chi_0$ which on the right
hand side of (\ref{CherednikUnif4e}) lies in the first summand. We define a
function $r: \Phi \rightarrow \{0,1,2\}$ as follows. We set
$r_{\varphi_0} = r_{\bar{\varphi}_0} = 1$. If the restriction of $\varphi \in \Phi$
to $F$ is not $\chi_0$ we set $r_{\varphi} = 0$ if $\varphi$ is in the first $s$
summands on the right hand side and $r_{\varphi} = 2$ if $\varphi$ is in the last
$s$ summands. If $\chi\neq\chi_0$ the extension $\varphi$ of $\chi$ such that
$r_{\varphi} = 2$ defines an isomorphism
$K \otimes_{F, \chi} \mathbb{R} \cong \mathbb{C}$. We define the group
homomorphism
\begin{displaymath}
  \begin{array}{rcr}
  h_K: \mathbb{C}^{\times} & \rightarrow & (K \otimes \mathbb{R})^{\times} \cong
  \prod_{\chi} (K \otimes_{F, \chi} \mathbb{R})^{\times} =
  (K \otimes_{F, \chi_0}\mathbb{R})^{\times} \times \prod_{\chi \neq \chi_0}
  \mathbb{C}^{\times}.
  \\[2mm]
  1 & \mapsto & (1,z,z,\ldots,z) \qquad\\
    \end{array}
\end{displaymath}
Let $B = D^{{\rm opp}} \otimes_{F} K$. We denote by $d \mapsto d^{\iota}$ the main  
involution of $D$ and by $a \mapsto \bar{a}$  the conjugation of $K/F$.
We denote by $b \mapsto b'$ the involution of the second kind on $B/K$
which is defined by $d \otimes a \mapsto d^{\iota} \otimes \bar{a}$.  
Let $V = B$ considered as a $B$-left module. Multiplication from the right
defines a ring homomorphism
\begin{displaymath}
D \otimes_F K \rightarrow \End_B V.   
\end{displaymath}
In particular the group $D^\times \times K^\times$ acts on $V$. 
By (\ref{defhD}) we obtain a ring homomorphism 
\begin{displaymath}
  \mathbb{C} \rightarrow D \otimes_{F, \chi_0} \mathbb{R} \rightarrow
  (D \otimes_F K) \otimes_{F, \chi_0} \mathbb{R}
  \end{displaymath}
and by the isomorphisms $K \otimes_{F, \chi} \mathbb{R} \cong \mathbb{C}$
chosen above for $\chi \neq \chi_0$ we obtain ring homomorphisms 
\begin{displaymath}
  \mathbb{C} \rightarrow K \otimes_{F, \chi} \mathbb{R} \rightarrow
  (D \otimes_F K) \otimes_{F, \chi} \mathbb{R}. 
\end{displaymath}
Taking the product of these ring homomorphisms over all 
$\chi: F \rightarrow \mathbb{R}$ we obtain a ring homomorphism
\begin{displaymath}
\mathbb{C} \rightarrow (D \otimes_F K) \otimes_{\mathbb{Q}} \mathbb{R}  
\end{displaymath}
and therefore a complex structure on the real vector space
$V \otimes \mathbb{R}$. 
Alternatively, this complex structure is given by the group homomorphism 
\begin{displaymath}
  h = h_D \times h_K: \mathbb{S} \rightarrow
  \prod_{\chi\in \Hom_{\mathbb{Q}{\rm -Alg}}(F,\mathbb{C})}  \big((D\otimes_{F,\chi} \mathbb{R})^{\times}
  \times (K \otimes_{F, \chi} \mathbb{R})^{\times}\big) ,
\end{displaymath}
where the group on the right hand side acts on $V \otimes \mathbb{R}$ by
the action of $D^\times \times K^\times$ on $V$. 

We consider
$\mathbb{Q}$-bilinear forms $\psi: V \times V \rightarrow \mathbb{Q}$
such that
\begin{displaymath}
\psi(x b , y) = \psi(x , y b'),  \quad \text{for}\; x,y \in V, \; b \in B. 
\end{displaymath}
By \cite{D-TS} one can choose $\psi$ in such a way that the complex structure
$h$ satisfies the Riemann period relations. We consider 
$G^{\bullet} = \{b \in B^{{\rm opp}} \mid b'b \in F^{\times}\}$ as an algebraic
group over $\mathbb{Q}$. The right multiplication by elements $d \otimes 1$ and
$1 \otimes a$ define elements of $G^{\bullet}$. This gives a homomorphism of algebraic groups,
\begin{equation}\label{DmalK-G.1e}
 D^{\times} \times K^{\times} \overset{\kappa}{\rightarrow}
 G^{\bullet} .
  \end{equation}
Then $(G^{\bullet}, h)$ is the Shimura datum for a Shimura variety of PEL-type. 
By (\ref{DmalK-G.1e}) we have an embedding $D^\times \rightarrow G^{\bullet}$. The decomposition
$B\otimes\mathbb{Q}_p=\prod_{i=0}^s (B_{\mathfrak{q}_i}\times B_{\bar{\mathfrak{q}}_i})$
induces a similiar decomposition of $V \otimes \mathbb{Q}_p$.
 We choose maximal orders $O_{D_{\mathfrak{p}_i}} \subset D_{\mathfrak{p}_i}$ and hence
 maximal orders $O_{B_{\mathfrak{q}_i}} \subset B_{\mathfrak{q}_i}$. 
 We assume in the definition \eqref{defKp} that
 $\mathbf{K}_{\mathfrak{p}_i} \subset O^{\times}_{D_{\mathfrak{p}_i}}$.
 There is a natural isomorphism
$D_{\mathfrak{p}_i}^{\times} \cong (B^{{\rm opp}}_{\mathfrak{q}_i})^{\times}$. The image
$\mathbf{K}_{\mathfrak{p}_i}$ by this isomorphism will be denoted by
$\mathbf{K}^{\bullet}_{\mathfrak{q}_i}$. From these last groups we define a subgroup
$\mathbf{K}^{\bullet}_p \subset G^{\bullet}(\mathbb{Q}_p)$, cf. (\ref{BZKpPkt1e})
with $\mathbf{M}^{\bullet}_{\mathfrak{p}_i} = O_{F_{\mathfrak{p}_i}}^{\times}$. This
subgroup satisfies
$\mathbf{K}_p = D^\times(\mathbb{Q}_p) \cap \mathbf{K}^{\bullet}_p$. Moreover 
we choose $\mathbf{K}^{\bullet,p}$ such that
$\mathbf{K}^p = (D\otimes\mathbb{A}_f^p)^\times \cap \mathbf{K}^{\bullet,p}$.  

The form $\psi$ induces an involution $\star$ of the second kind on $B$,
 \begin{displaymath}
\psi(bx, y) = \psi(b^{\star}x, y), \quad x,y \in V, \; b \in B. 
 \end{displaymath}
 We denote by $O_{B_{\bar{\mathfrak{q}}_i}} \subset B_{\bar{\mathfrak{q}}_i}$ the image
  of $O_{B_{\mathfrak{q}_i}}$ by $\star$. Let $O_{B, (p)}$ be the set of elements of $B$
  whose images in $B_{\mathfrak{q}_i}$ and $B_{\bar{\mathfrak{q}}_i}$ lie in the
  chosen maximal orders. 
 We obtain the lattice   $\Lambda_{\mathfrak{q}_i} = O_{B_{\mathfrak{q}_i}} \subset V_{\mathfrak{q}_i}$. 

 Let $U_p(F) \subset F^{\times}$ be the subgroup of elements which are units in
 each $F_{\mathfrak{p}_i}$. We
 define the following functor on the category of $O_{E_{\nu}}$-schemes $S$. The upper index $t$ is referring to the fact that $\tilde{\mathcal{A}}^{\bullet t}_{\mathbf{K}^{\bullet}}$, when restricted to the category of $E_\nu$-schemes, is a twisted form of another functor ${\mathcal{A}}^{\bullet }_{\mathbf{K}^{\bullet}}$.

 \begin{definition}
Let $S$ be an $O_{E_{\nu}}$-scheme. 
A point of $\tilde{\mathcal{A}}^{\bullet t}_{\mathbf{K}^{\bullet}}(S)$ consists of the  
following data: 
\begin{enumerate} 
\item[(a)] An abelian scheme $A$ over $S$ up to isogeny prime to $p$ with an
  action 
  $\iota: O_{B,(p)} \rightarrow \End A \otimes_{\mathbb{Z}} \mathbb{Z}_{(p)}$.

\item[(b)] 
  An $U_p(F)$-homogeneous polarization $\bar{\lambda}$ of $A$ which is
  principal in $p$.

\item[(c)]
  A class $\bar{\eta}^p$   modulo $\mathbf{K}^{\bullet, p}$ of
  $B \otimes \mathbb{A}^p_f$-module isomorphisms
  \begin{displaymath}
    \eta^p: V \otimes \mathbb{A}^p_f \isoarrow \mathrm{V}^p_f(A) ,
  \end{displaymath}
  such that 
  \begin{displaymath}
\psi(\xi^{(p)}(\lambda) v_1, v_2) = E^{\lambda}(\eta^p(v_1), \eta^p(v_2))
  \end{displaymath}
  for some function $\xi^{(p)}(\lambda) \in (F \otimes \mathbb{A}^p_f)^{\times}(1)$
  on $\bar{\lambda}$.  
\item[(e)]
 A class $\bar{\eta}_{\mathfrak{q}_i}$  modulo $\mathbf{K}^{\bullet}_{\mathfrak{q}_i}$
  of $O_{B_{\mathfrak{q}_i}}$-module isomorphisms  for each $i = 1, \ldots s$, 
  \begin{displaymath}
    \eta_{\mathfrak{q}_i}: \Lambda_{\mathfrak{q}_i} \isoarrow T_{\mathfrak{q}_i}(A) .
  \end{displaymath} 
  \end{enumerate}
We require that the following Kottwitz condition ${\rm (KC)}$ holds,
  \begin{equation}
     {\rm char}(T, \iota(b) \mid \Lie A) = \prod_{\varphi: K \rightarrow \bar{\mathbb{Q}}}
    \varphi(\Nm^o_{B/K} (T -b))^{r_{\varphi}} . 
    \end{equation}

\end{definition}
 The general fiber over $E_{\nu}$ of this functor is a Galois form of 
 $\Sh_{\mathbf{K}^{\bullet}}(G^{\bullet}, h)_{E_{\nu}}$ but this is irrelevant for
 this Introduction.     
 We prove that the \'etale sheafification of 
 $\tilde{\mathcal{A}}^{\bullet t}_{\mathbf{K}^{\bullet}}$ is representable, cf. 
 Proposition \ref{BZ8p}. We also show that (cf. (\ref{Hecke4e})) 
 \begin{equation}\label{Intro1e}
     (K \otimes \mathbb{Q}_p)^{\times} = \prod_{i=0}^{s} K_{\mathfrak{q}_i}^{\times} \times
   \prod_{i=0}^{s} K_{\bar{\mathfrak{q}}_i}^{\times} 
 \end{equation}  
 acts by Hecke operators on the functor
 $\tilde{\mathcal{A}}^{\bullet t}_{\mathbf{K}^{\bullet}}$.
 
  We consider the homomorphism
 \begin{equation}\label{def:bulletshim}
   h_D^{\bullet} = h_D \times 1: \mathbb{C}^{\times} \rightarrow
   (D \otimes \mathbb{R})^{\ast} \times (K \otimes \mathbb{R})^{\ast}
   \rightarrow G^{\bullet}_{\mathbb{R}}. 
   \end{equation} 
The Shimura variety $\Sh_{\mathbf{K}^{\bullet}}(G^{\bullet}, h^{\bullet}_{D})$ 
is defined over $E_{\nu}$. It is a Galois form of
$\Sh_{\mathbf{K}^{\bullet}}(G^{\bullet}, h)$.  

We find a model
$\widetilde{\Sh}_{\mathbf{K}^{\bullet}}(G^{\bullet}, h^{\bullet}_{D})$ over $O_{E_{\nu}}$
of this Shimura variety and a commutative diagram
 \begin{equation}\label{Intro2e}
  \begin{aligned}\xymatrix{
    \tilde{\mathcal{A}}^{\bullet t}_{\mathbf{K}^{\bullet}} \times_{\Spec O_{E_{\nu}}}
    \Spec O_{E^{nr}_{\nu}}   
  \ar[d]_{\dot{z} \times \tau_c} \ar[r]
  & \widetilde{\Sh}_{\mathbf{K}^{\bullet}}(G^{\bullet}, h^{\bullet}_{D}) 
  \times_{\Spec O_{E_{\nu}}} \Spec O_{E^{nr}_{\nu}} \ar[d]^{\id \times \tau_c}\\   
  \tilde{\mathcal{A}}^{\bullet t}_{\mathbf{K}^{\bullet}} \times_{\Spec O_{E_{\nu}}}
  \Spec O_{E^{nr}_{\nu}} \ar[r] &
  \widetilde{\Sh}_{\mathbf{K}^{\bullet}}(G^{\bullet}, h^{\bullet}_{D}) 
  \times_{\Spec O_{E_{\nu}}} \Spec O_{E^{nr}_{\nu}} ,\\   
   } 
   \end{aligned}
 \end{equation}
 cf. Proposition \ref{Sh_D1p} and (\ref{tildeSh_D1e}). 
 Here $E_{\nu}^{nr}$ is the maximal unramified extension of $E_{\nu}$, and 
 $\tau \in \Gal(E_{\nu}^{nr}/E_{\nu})$ is the Frobenius automorphism, and 
 $\tau_c = \Spec \tau^{-1}$. We denote by
 $\pi_{\mathfrak{p}_0} \in F_{\mathfrak{p}_0} = K_{\bar{\mathfrak{q}}_0}$ a prime
 element and by $f_{\nu}$ the inertia index of $E_{\nu}/\mathbb{Q}_p$. 
 The element 
 \begin{displaymath}
 \dot{z} =  (1, \ldots, 1) \times
   (\pi_{\mathfrak{p}_0}^{-1} p^{f_{\nu}}, p^{f_{\nu}}, \ldots p^{f_{\nu}})
   \end{displaymath}
 from the right hand side of (\ref{Intro1e}) acts as an Hecke operator.
 
 The horizontal arrow in the diagram (\ref{Intro2e}) is the \'etale
 sheafification. It follows from \cite{RZ} that the \'etale sheafification of
   $\tilde{\mathcal{A}}^{\bullet t}_{\mathbf{K}^{\bullet}}$  has a $p$-adic uniformization
   by the formal scheme $\hat{\Omega}^2_{E_\nu}$, cf. Theorem \ref{4epeg1t}. 
 This gives a uniformization of the model 
 $\widetilde{\Sh}_{\mathbf{K}^{\bullet}}(G^{\bullet}, h^{\bullet}_{D})$. The embedding of Shimura data $(D^{\times}, h_{D}) \subset
   (G^{\bullet}, h^{\bullet}_{D})$ and the fact that $\mathbf{K}\subset \mathbf{K}^\bullet$ define a morphism of Shimura varieties $ \Sh_{\mathbf{K}}(D^{\times}, h_{D})_{E_{\nu}} \to
   \Sh_{\mathbf{K}^{\bullet}}(G^{\bullet}, h^{\bullet}_{D})_{E_{\nu}} $.
By a theorem of Chevalley, for  suitable
$\mathbf{K}^{\bullet,p} \in G^{\bullet}(\mathbb{A}^p_f)$ of the type considered
above, this morphism induces  an open and closed embedding,
 \begin{equation}
   \Sh_{\mathbf{K}}(D^{\times}, h_{D})_{E_{\nu}} \subset
   \Sh_{\mathbf{K}^{\bullet}}(G^{\bullet}, h^{\bullet}_{D})_{E_{\nu}} .
 \end{equation}
  The closure of the left hand side in
 $\widetilde{\Sh}_{\mathbf{K}^{\bullet}}(G^{\bullet}, h^{\bullet}_{D})$ gives the stable 
  model $\widetilde{\Sh}_{\mathbf{K}}(D^{\times}, h_D)$ whose formal scheme
 inherits a uniformization by $\hat{\Omega}^2_{E_{\nu}}$, proving the main theorem. 
 
So far, we have only mentioned the Shimura pairs   $(G^\bullet, h)$ and
$(G^\bullet, h^\bullet_D)$. However, in the body of the paper, also Shimura pairs
$(G, h)$ and $(G, h\delta)$ play an important role. Here $G\subset G^\bullet$ is the subgroup where the similitude factor lies in $\BQ$, and $\delta$ is a central character of $G$. The  Shimura variety for  $(G, h)$ is  of PEL-type and has the key property that it is a fine moduli scheme for a moduli problem $\CA_{\mathbf K}$, for small enough level ${\mathbf K}$. Similarly, the  Shimura variety for  $(G, h\delta)$ is the unramified twist of the fine moduli scheme for a moduli problem  $\CA^t_{\mathbf K}$, which, furthermore,  has a natural extension $\tilde\CA^t_{\mathbf K}$ over $\Spec O_{E_\nu}$. The fine moduli scheme  for $\tilde\CA^t_{\mathbf K}$ is then used to show that  the horizontal arrow in the diagram (\ref{Intro2e}) is the \'etale sheafification. 
 
   The lay-out of the paper is as follows. In \S \ref{s:LA} we explain the linear algebra behind the formation of the Shimura pairs $(G^\bullet, h)$ and $(G, h)$. In \S \ref{s:AG}, we explain the Shimura varieties for $(G, h)$ and $(G, h\delta)$ and the corresponding moduli problems $\CA_{\mathbf K}$ and $\CA^t_{\mathbf K}$ and the integral extension $\tilde\CA^t_{\mathbf K}$ of the latter. In \S \ref{s:AGbullet} we explain the Shimura varieties for $(G^\bullet, h)$ and $(G^\bullet, h_D^\bullet)$ and the corresponding moduli problems $\CA^\bullet_{\mathbf K}$ and $\CA^{\bullet, t}_{\mathbf K}$  and the integral extension $\tilde\CA^t_{\mathbf K}$. Furthermore, we establish a relation between  the integral extensions $\tilde\CA^t_{\mathbf K}$ and the integral extension $\tilde\CA^{\bullet, t}_{\mathbf K}$ and use this to show that the horizontal arrow in the diagram (\ref{Intro2e}) is the \'etale sheafification. In \S \ref{s:RZ} we explain the Rapoport-Zink spaces relevant to these moduli problems. In \S \ref{s:uniform} we prove the $p$-adic uniformization of the integral models of the Shimura varieties for the pairs $(G^\bullet, h^\bullet_D)$ and $(D^\times, h_D)$. The last two sections are really appendices.  In \S \ref{s:desc}, we clarify our conventions about Galois descent, and in \S \ref{s:shimvar} we make precise our sign conventions for Shimura varieties.
   We formulate a result of Kisin \cite{KisinJAMS} on embeddings of Shimura
   varieties in the form needed here.

 The present paper is an improved version of parts of the preprint \cite{BZ}.
 The strategy here is the same but some serious gaps in the arguments are
 repaired. However, not all results of \cite{BZ} are covered.   

 We thank M. Rapoport for his many useful suggestions which helped to improve
 our work.
 
 \section{The Shimura data}\label{s:LA}
 In this section, we introduce the linear algebra which leads to the
 definition of the Shimura pairs $(G^\bullet, h^\bullet_D)$ and $(G^\bullet, h)$
 and $(G, h)$.
 
 Let $K/F$ be a CM-field.
Let $a \mapsto \bar{a}$ be the conjugation of $K/F$. We
consider a quaternion algebra $D$ over $F$. Let $d \mapsto d^{\iota}$ be the
main involution of $D$. We set $B = D^{{\rm opp}} \otimes_F K$. We extend the map
$d \mapsto d^{\iota}$ $K$-linearily to $B$. Then we obtain the main involution
$b \mapsto b^{\iota}$ of $B/K$. The conjugation acts via the second factor
on $B = D^{{\rm opp}} \otimes_F K$. We set $b' = \bar{b}^{\iota}$. We consider the
sesquilinear form  
\begin{displaymath}
  \varkappa_0: B \times B \rightarrow K, \quad
  \varkappa_0(b_1, b_2) = {\rm Tr}^{o}_{B/K} b_2 b'_1. 
\end{displaymath}
It is $K$-linear in the second variable and antilinear in the first and
it is hermitian
\begin{displaymath}
  \varkappa_0(b_1, b_2) = \overline{\varkappa_0(b_2, b_1)}.
  \end{displaymath}
Moreover we obtain 
\begin{equation}\label{BZ1e} 
  \varkappa_0(x b, y) = \varkappa_0(x, y b'), \quad 
  \varkappa_0(bx, y) = \varkappa_0(x, b'y),\quad x, y, b \in B. 
\end{equation}
We set
\begin{equation}\label{Gpunkt2e} 
G^{\bullet} = \{b \in B^{{\rm opp}} \; | \; b'b \in F^{\times} \},   
\end{equation}
and consider it as an algebraic group over $\mathbb{Q}$. We write
$\tilde{G}^{\bullet}$ if we consider it as an algebraic group over $F$, i.e.
$\Res_{F/\mathbb{Q}} \tilde{G}^{\bullet} = G^{\bullet}$. 

We will write $V = B$ considered as a $B$-left module. The right multiplication
by an element of $B$ gives an isomorphism $\End_B V = B^{{\rm opp}} = D \otimes_F K$.
Therefore we can write 
\begin{equation}\label{Gpunkt1e} 
  G^{\bullet} = \{g \in GL_B (V) \mid \varkappa_0(g v_1, g v_2) =
  \mu(g) \varkappa_0 (v_1, v_2), \mu(g) \in F^{\times}   \}  . 
  \end{equation}

\begin{lemma}\label{BZ1l}  
  There is an exact sequence of algebraic groups over $\BQ$,
  \begin{displaymath}
    0 \rightarrow F^{\times} \rightarrow D^{\times} \times K^{\times}
    \overset{\kappa}{\rightarrow}
    G^{\bullet} \rightarrow 0.   
  \end{displaymath}
  The map $\kappa$ maps $(d, k)$ to $d \otimes k$. \qed
  \end{lemma} 
We set $\Phi = \Hom_{\mathbb{Q}{\rm -Alg}}(K, \mathbb{C})$. We assume that there is
a unique embedding $\chi_0: F \rightarrow \mathbb{R}$ such that the quaternion
algebra $D \otimes_{F, \chi_0} \mathbb{R}$ splits. We consider a generalized
CM-type of rank $2$ in the sense of \cite{KRnew}, comp. \cite{KRZ},
\begin{equation}\label{BZgCM1e} 
r: \Phi \rightarrow \mathbb{Z}_{\geq 0}, 
  \end{equation}
such that $r_{\varphi_0} = r_{\bar{\varphi}_0} = 1$ for the extensions
$\varphi_0, \bar{\varphi}_0: K \rightarrow \mathbb{C}$ of $\chi_0$ 
and such that $r_{\varphi} = 0, 2$ for all other $\varphi \in \Phi$. 

We will define a complex structure on the $\mathbb{R}$-vector space
$V \otimes \mathbb{R}=B \otimes \mathbb{R}$. For this we consider the decomposition
\begin{displaymath}
  B \otimes \mathbb{R} = \bigoplus_{\chi: F \rightarrow \mathbb{R}} B \otimes_{F, \chi}
  \mathbb{R} = \bigoplus_{\chi}\big( (D^{{\rm opp}} \otimes_{F, \chi} \mathbb{R})
  \otimes_{\mathbb{R}} (K \otimes_{F, \chi} \mathbb{R})\big). 
\end{displaymath}
We define the complex structure on each summand on the right hand side.
Let  $\chi \neq \chi_0$ and let
$\varphi : K \rightarrow \mathbb{C}$ be the
extension of $\chi$ such that $r_{\varphi} = 2$. Then $\varphi$ defines an
isomorphism 
$K \otimes_{F, \chi} \mathbb{R} \cong \mathbb{C}$. This induces a complex
structure on the summand belonging to $\chi$ via the second factor of the
tensor product.  For $\chi_0$ we have
$D \otimes_{F, \chi_0} \mathbb{R} \cong {\rm M}_2(\mathbb{R})$. We endow the
$\mathbb{R}$-vector space $D^{{\rm opp}} \otimes_{F, \chi_0} \mathbb{R}$ with the
complex structure $J_{\chi_0}$ given by right multiplication by 
\begin{displaymath}
  J_{\chi_0} =
\left(
  \begin{array}{rr} 
    0 & -1\\
    1 & 0
  \end{array}
  \right).  
\end{displaymath}
This induces a complex structure on 
$(D^{{\rm opp}}\otimes_{F,\chi_0}\mathbb{R})\otimes_{\mathbb{R}}(K\otimes_{F,\chi_0}\mathbb{R})$
via the first factor. Together we obtain a complex structure $J$ on
$B \otimes_{\mathbb{Q}} \mathbb{R}$ such that 
\begin{displaymath}
  \Trace_{\mathbb{C}} (k | (B \otimes_{\mathbb{Q}} \mathbb{R}, J)) =
  \sum_{\varphi \in \Phi} 2 r_{\varphi} \varphi(k), \quad k\in K. 
  \end{displaymath}
This complex structure on $V \otimes_{\mathbb{Q}} \mathbb{R}$ commutes with the
$B \otimes_{\mathbb{Q}} \mathbb{R}$-module structure and defines therefore a
homomorphism $\mathbb{C} \rightarrow B^{{\rm opp}} \otimes_{\mathbb{Q}} \mathbb{R}$.
This homomorphism induces a homomorphism of groups 
\begin{equation}\label{BZh1e}
  h: \mathbb{S} \rightarrow \prod_{\chi \in \Hom_{\mathbb{Q}{\rm -Alg}}(F, \mathbb{C})}
  \big((D \otimes_{F, \chi} \mathbb{R})^{\times}
  \times (K \otimes_{F, \chi} \mathbb{R})^{\times}\big). 
\end{equation}
Let $z \in \mathbb{C}^\times = \mathbb{S}(\mathbb{R})$. Then the $\chi_0$-component
$h_{\chi_0}(z)$ is
\begin{displaymath}
\left(
  \begin{array}{rr} 
    a & -b\\
    b & a
  \end{array}
  \right) \times 1, \quad z = a + b{\bf i} ,  
\end{displaymath}
and for $\chi \neq \chi_0$ the component $h_{\chi}(z)$ is
$1 \times 1\otimes z \in (D \otimes_{F,\chi}\mathbb{R})^{\times}\times (K \otimes_{K, \varphi} \mathbb{C})^{\times}$.
Here $\varphi \in \Phi$ is the extension of $\chi$ with $r_{\varphi} = 2$.
We have used the natural isomorphism
$K \otimes_{F, \chi} \mathbb{R} = K \otimes_{K, \varphi} \mathbb{C}$.
We can write $h({\bf i}) = J$. The composite with the projection to
$G^{\bullet}_{\mathbb{R}}$ given by Lemma \ref{BZ1l} is also denoted by $h$,
\begin{equation}\label{hforbul}
h\colon \BS\to G^{\bullet}_{\mathbb{R}} .
\end{equation}  
\begin{lemma}
  There exist elements $\gamma \in B$ such that 
  $\mathfrak{h}(x,y) = \varkappa_0(\gamma x, y J)$ is hermitian and positive
  definite on $B \otimes \mathbb{R}$. More precisely, this means that for
  each $\varphi$ the form
  $$\mathfrak{h}_{\varphi}: B \otimes_{K, \varphi} \mathbb{C}
  \times B \otimes_{K, \varphi} \mathbb{C} \rightarrow  K \otimes_{K, \varphi}
  \mathbb{C}$$
  is hermitian and positive definite. 
\end{lemma}
This follows as in \cite{D-TS}. Note that alternatively we can say that
$\Trace_{K/F} \mathfrak{h}$ is symmetric and positive definite on
$B \otimes \mathbb{R}$. Let
\begin{equation}\label{defG}
G = \{b \in B^{{\rm opp}} \; | \; b'b \in \mathbb{Q}^{\times} \} \subset G^{\bullet}.   
\end{equation}
Since $h(z)'h(z) = \bar{z} z \in \mathbb{R}^{\times}$ for $z \in \mathbb{C}^{\times}=\CS(\BR)$,
the morphism $h$ factors through $G_{\mathbb{R}}$. We define
\begin{displaymath}
\varkappa: V \times V \rightarrow K ,  \quad \quad\varkappa(x, y) = \varkappa_0(\gamma x, y), \quad x, y \in V = B.
  \end{displaymath}
 The first 
equation of (\ref{BZ1e}) continues to hold for $\varkappa$. We have
$\mathfrak{h}(x, y) = \varkappa(x, y J)$. We note that $\varkappa$ is an
antihermitian form: 
\begin{displaymath}
  \overline{\varkappa(y, x)} = - \overline{\varkappa(y, x J^{2})} =
  - \overline{\mathfrak{h}(y, x J)} = - \mathfrak{h}(x J, y) =
  - \varkappa(x J, y J) =  - \varkappa(x, y). 
\end{displaymath}
It is easily seen that  $\gamma' = -\gamma$ is equivalent with the property
that $\varkappa$ is antihermitian or that $\mathfrak{h}$ is hermitian. Equivalently
one could use the alternating
form 
\begin{equation}\label{BZpsi1e}
\psi : V \times V \rightarrow \mathbb{Q}, \quad\psi(x, y) = \Trace_{K/\mathbb{Q}} \varkappa (x, y), \quad x,y \in B ,
  \end{equation} 
which satisfies
\begin{displaymath}
\psi(k x, y) = \psi (x, \bar{k} y), \quad k \in K. 
\end{displaymath}
Then $\psi (x, y J)$ is symmetric and positive definite. We define an involution $b \mapsto b^{\star}$ on $B$ by
\begin{equation}\label{defstar}
\varkappa(b x, y) = \varkappa(x , b^{\star} y). 
\end{equation}
Because the same equation holds for $\mathfrak{h}$, the involution
$b \mapsto b^{\star}$ is positive. From the definition we obtain
$b = \gamma^{-1} b' \gamma$. We obtain
\begin{equation}\label{BZpsi3e}
  \psi(b x, y) = \psi(x , b^{\star} y), \quad  \psi(x b , y) = \psi(x , y b'),
  \quad  x,y,b \in B. 
  \end{equation}
We can also write 
\begin{equation}\label{BZh3e}
  G = \{g \in \End_B(V) \; | \; \psi(g x, g y) = \mu(g) \psi(x, y), \;
  \text{for} \; \mu(g) \in \mathbb{Q}^{\times}  \}. 
\end{equation}
We also obtain $G$ if we replace on the right hand side $\psi$ by $\varkappa$.  

The action of $ g = (d, k) \in D^{\times} \times K^{\times}$ on $B$ is by
definition
\begin{displaymath}
(d,k) (u \otimes a) = ud \otimes ak, \quad u \otimes a \in D^{{\rm opp}} \otimes_{F} K 
= B. 
\end{displaymath}
The product $ud$ is taken in $D^{{\rm opp}}$.

The homomorphism induced by (\ref{BZh1e}) 
\begin{equation}\label{BZh2e} 
h: \mathbb{S} \rightarrow G_{\mathbb{R}} 
  \end{equation}
gives a Shimura datum in the sense of \cite{D-TS}, except that we denote by
$h$ what is $h^{-1}$ in Deligne's normalization. The Hodge structure on $V$ is therefore
in this paper of type $(1,0), (0,1)$.

We fix a prime number $p$ and we choose a diagram
\begin{equation}\label{BZ2e}
\mathbb{C} \leftarrow \bar{\mathbb{Q}} \rightarrow \bar{\mathbb{Q}}_p.
\end{equation} 
By this diagram we obtain $\Phi =  \Hom_{\mathbb{Q}{\rm -Alg}}(K, \bar{\mathbb{Q}}_p)$. 
We assume that all prime ideals of $O_F$ containing $pO_F$ are split in $K/F$. 
We denote these prime ideals of $O_F$ by
\begin{equation}\label{BZsplit1e}
\mathfrak{p}_0, \ldots, \mathfrak{p}_s.
\end{equation}
Let $\mathfrak{q}_i, \bar{\mathfrak{q}}_i$ the two prime ideals of $O_K$ over
$\mathfrak{p}_i$. We obtain
\begin{displaymath}
\mathfrak{p}_i O_K = \mathfrak{q}_i \bar{\mathfrak{q}}_i. 
  \end{displaymath}
We obtain a decomposition 
\begin{equation}\label{p-zerlegt}  
  \begin{array}{l} 
    \Hom_{\mathbb{Q}{\rm -Alg}}(K, \bar{\mathbb{Q}}_p) = \\[2mm] 
    \quad (\coprod_{i=0}^{s}
    \Hom_{\mathbb{Q}_p{\rm -Alg}}(K_{\mathfrak{q}_i}, \bar{\mathbb{Q}}_p))
    \; \amalg\;
    (\coprod_{i=0}^{s}
    \Hom_{\mathbb{Q}_p{\rm -Alg}}(K_{\bar{\mathfrak{q}}_i},\bar{\mathbb{Q}}_p)). 
    \end{array}
\end{equation}

We denote the components of this disjoint sum by
$\Phi_{\mathfrak{q}_i}$, resp. $\Phi_{\bar{\mathfrak{q}}_i}$. We assume that
$\varphi_0 \in \Phi_{\mathfrak{q}_0}$ and $ \bar{\varphi}_0 \in \Phi_{\bar{\mathfrak{q}}_0}$.
For all other $\varphi$ we require that
\begin{equation}\label{BZEnu1e} 
  \begin{array}{ll}
    r_{\varphi} = 0 & \text{if} \; \varphi \in \Phi_{\mathfrak{q}_i} \;
    \text{for some} \; i = 0, \ldots, s,
    \; \text{and} \, \varphi \neq \varphi_0 \\
    r_{\varphi} = 2 & \text{if} \; \varphi \in \Phi_{\bar{\mathfrak{q}}_i} \; 
    \text{for some} \; i = 0, \ldots, s,
    \; \text{and} \, \varphi \neq \bar{\varphi}_0 .  
    \end{array}
  \end{equation}
Let $E = E(G,h)$ be the reflex field, i.e.,
\begin{equation}\label{BZE1e}
  \Gal(\bar{\mathbb{Q}}/E) = \{\sigma \in \Gal(\bar{\mathbb{Q}}/\mathbb{Q})
  \; | \; r_{\sigma\varphi} = r_{\varphi}, \; \text{for all}\; \varphi \in \Phi\}. 
\end{equation}
The embedding $E \rightarrow \bar{\mathbb{Q}}_p$ in the sense of diagram
(\ref{BZ2e}) defines a place $E_{\nu} \subset \bar{\mathbb{Q}}_p$. We call this
the\emph{ local Shimura field}. If $\varphi \neq \varphi_0, \bar{\varphi}_0$, the
number $r_{\varphi}$ depends only on the place $\mathfrak{q}_i$ of $K$ which is
induced by $\varphi: K \rightarrow \bar{\mathbb{Q}}_p$. We conclude that
\begin{equation}\label{BZ3e} 
  \Gal(\bar{\mathbb{Q}}_p/E_{\nu}) = \{\sigma \in
  \Gal(\bar{\mathbb{Q}}_p/\mathbb{Q}_p) \; | \;
  r_{\sigma \varphi_0} = r_{\varphi_0} \}. 
\end{equation}
 The condition (\ref{BZ3e})
on $\sigma$ signifies that $\sigma$ fixes the embedding
$F_{\mathfrak{p}_0} \rightarrow \bar{\mathbb{Q}}_p$ induced by $\varphi_0$.
We obtain that
\begin{equation}\label{BZ4e}
  E_{\nu} = \varphi_0(F_{\mathfrak{p}_0}). 
\end{equation}
We remark that $E_\nu$ coincides with the localization of the Shimura field
$\chi_0(F)$ of the Shimura curve we have chosen, cf. the beginning of this
section. 

\bigskip

\section{The moduli problem for $\Sh(G, h)$ and a reduction modulo $p$}\label{s:AG}

We consider the alternating $\mathbb{Q}$-bilinear form $\psi$ on the
$B$-module $V$ (\ref{BZpsi1e}). It satisfies 
\begin{displaymath}
\psi(b v_1, v_2) = \psi(v_1, b^{\star} v_2) \quad v_1, v_2 \in V.  
\end{displaymath}
We state the moduli problem associated to the $B$-module $V$ and the alternating form $\psi$, cf. \cite[4.10]{D-TS}. Recall 
$(G, h)$ from (\ref{BZh3e}), (\ref{BZh2e}).

Let $\mathbf{K} \subset G(\mathbb{A}_f)$ be an open compact subgroup. 
The Shimura variety $\Sh(G,h)_{\mathbf{K}}$ is the coarse moduli scheme of the
following functor $\mathcal{A}_{\mathbf{K}}$ on the category of schemes over
$E=E(G, h)$. If $\mathbf{K}$ satisfies the condition (\ref{BZneat1e})
below, the functor $\mathcal{A}_{\mathbf{K}}$ is representable. 
\begin{definition}\label{BZAK1d} 
Let $S$ be a scheme over $E$. 
A point of $\mathcal{A}_{\mathbf{K}}(S)$ is given by the following data:
\begin{enumerate} 
\item[(a)] An abelian scheme $A$ over $S$ up to isogeny with an action
  $\iota: B \rightarrow \End^o A$.
\item[(b)] A $\mathbb{Q}$-homogeneous polarization $\bar{\lambda}$ of $A$
  which induces on $B$ the involution $b \mapsto b^{\star}$.
\item[(c)] A class $\bar{\eta}\; \text{modulo}\; \mathbf{K}$ of
  $B \otimes \mathbb{A}_f$-module isomorphisms
  \begin{displaymath}
\eta: V \otimes \mathbb{A}_f \isoarrow V_f(A) 
  \end{displaymath}
  such that for each $\lambda \in \bar{\lambda}$ there is locally
  for the Zariski topology on $S$ a constant
  $\xi \in \mathbb{A}_f^{\times}(1)$ with
  \begin{equation}\label{BZsimilis1e} 
\xi \psi(v_1,, v_1) = E^{\lambda}(\eta(v_1), \eta(v_2)). 
  \end{equation}
   \end{enumerate}
 We require that the following condition $ {\rm (KC)}$ holds,
  \begin{equation}
  {\rm char}(T, \iota(b) \mid \Lie A) = \prod_{\varphi: K \rightarrow \bar{\mathbb{Q}}}
    \varphi(\Nm^o_{B/K} (T -b))^{r_{\varphi}} . 
    \end{equation}
 
  \end{definition} 
A more precise formulation of the datum $(c)$ is as follows. We
assume that $S$ is connected and we choose a geometric point $\bar{s}$ of
$S$. Then we may regard $V_f(A)$ resp. $V \otimes \mathbb{A}_f$ as continuous
representation of the fundamental group $\pi_1(\bar{s}, S)$. We denote by
$\mathbb{A}_f(1)$ the group $\mathbb{A}_f$ endowed with the action of
$\pi_1(\bar{s}, S)$ via the cyclotomic character
\begin{displaymath}
  \varsigma: \pi_1(\bar{s}, S) \rightarrow \hat{\mathbb{Z}}^{\times} \subset
  \mathbb{A}_f^{\times}. 
  \end{displaymath}

Then $\bar{\eta}$ is determined by a
$B \otimes \mathbb{A}_f$-linear symplectic similitude $\eta$, i.e. 
(\ref{BZsimilis1e}) holds with $\xi \in \mathbb{A}_f$. The class $\bar{\eta}$
must be invariant by the action of $\pi_1(\bar{s}, S)$, i.e., for each
$\gamma \in \pi_1(\bar{s}, S)$ there is $k(\gamma) \in \mathbf{K}$ such that
\begin{equation}\label{BZeta1e}
\gamma \eta(v) = \eta (k(\gamma) v), \quad v \in V \otimes \mathbb{A}_f. 
\end{equation}
Since the polarization $\lambda$ is defined over $S$ the form $E^{\lambda}$
satisfies
\begin{displaymath}
  E^{\lambda}(\gamma \eta(v_1), \gamma \eta(v_2)) = \gamma (E^{\lambda}(v_1, v_2)) =
  \varsigma(\gamma) E^{\lambda}(v_1, v_2) .
\end{displaymath}
 When we apply the symplectic similitude $\eta$, this translates into
\begin{equation}\label{BZeta2e}
\varsigma(\gamma) = \mu(k(\gamma)), 
  \end{equation}
where  $\mu$ is the multiplicator (\ref{BZh3e}). The datum $\bar{\eta}$
of $(c)$ for a connected scheme $S$ is now equivalently a class
modulo $\mathbf{K}$ of symplectic similitudes $\eta$ of
$B \otimes \mathbb{A}_f$-modules
$V \otimes \mathbb{A}_f \isoarrow V_f(A_{\bar{s}})$ such that (\ref{BZeta1e})
and (\ref{BZeta2e}) hold.

An alternative way to describe the functor $\mathcal{A}_{\mathbf{K}}$ is as
follows, cf. \cite[4.12]{D-TS}. We fix a $\mathbb{Z}$-lattice $\Gamma \subset V$
such that $\psi(\Gamma \times \Gamma) \subset \mathbb{Z}$. Let $m > 0$ an
integer and assume that $\mathbf{K} = \mathbf{K}_m$ is the subgroup of all
$g \in G(\mathbb{A}_f)$, such that $g \hat\Gamma = \hat\Gamma$ and $g \equiv \id_{\hat\Gamma}$
modulo $m\hat\Gamma$. Let $O_B \subset B$ the order of all elements $b$ such that
$b \Gamma \subset \Gamma$. Then for a connected scheme $S$ over $E$ a point of
$\mathcal{A}_{\mathbf{K}_m}(S)$ consists of 
\begin{enumerate} 
\item[(a)] An abelian scheme $A_0$ over $S$ with an action
  $\iota: O_B \rightarrow \End A_0$.
\item[(b)] A polarization $\lambda$ of $A_0$ 
  which induces on $B$ the involution $b \mapsto b^{\star}$.
\item[(c)] An isomorphism of \'etale sheaves on $S$ 
  \begin{displaymath}
\eta_m: \Gamma/m \Gamma \rightarrow A_0[m] 
    \end{displaymath}
such that $\eta_m$ lifts to an isomorphism of $O_B$-modules 
  \begin{displaymath}
\eta: \Gamma \otimes \hat{\mathbb{Z}} \rightarrow \hat{T}(A_{0, {\bar{s}}})  
  \end{displaymath}
and such that there is $\xi \in \hat{\mathbb{Z}}^{\times}(1)$ with 
  \begin{displaymath}
    \xi \psi(v_1,, v_1) = E^{\lambda}(\eta(v_1), \eta(v_2)), \quad
    v_1, v_2 \in \Gamma. 
    \end{displaymath}
     \end{enumerate}
We require that the following condition ${\rm (KC)}$ holds,
  \begin{equation}
    {\rm char} (T, \iota(b) \mid \Lie A_0) = \prod_{\varphi: K \rightarrow \bar{\mathbb{Q}}}
    \varphi(\Nm^o_{B/K} (T -b))^{r_{\varphi}} , \quad b \in O_B.  
    \end{equation}

If we start with a point $(A, \iota, \bar{\lambda}, \bar{\eta})$ of the first
description of $\mathcal{A}_{\mathbf{K}}(S)$, we construct a point of the second
description as follows. We choose $\eta \in \bar{\eta}$. Then there is
an abelian variety $A_0 \in A$ such that
\begin{displaymath}
  \eta: \Gamma \otimes \hat{\mathbb{Z}} \overset{\sim}{\longrightarrow}
  \hat{T}(A_0). 
  \end{displaymath}
Then $A_0$ is independent of the choice of $\eta$. There exists a unique
$\lambda \in \bar{\lambda}$ such that the equation (\ref{BZsimilis1e})
holds with $\xi \in \hat{\mathbb{Z}}^{\times}(1)$. Modulo $m$ we obtain an
isomorphism
\begin{displaymath}
  \eta_m: \Gamma/ m\Gamma \overset{\sim}{\longrightarrow}
  \hat{T}(A_0) / m\hat{T}(A_0) \cong A_0[m], 
\end{displaymath}
which is also independent of the choice of $\eta$.
Conversely, to produce from a point of the second decription a point of the first
description is even more obvious and we omit it.  

It follows from these considerations that the functor $\mathcal{A}_{\mathbf{K}}$
is representable if $\mathbf{K}$ satisfies the following condition: 
\begin{equation}\label{BZneat1e} 
 \begin{aligned}
  \text{\emph{There is a $\mathbb{Z}$-lattice $\Gamma \subset V$ and an integer $m \geq 3$
  such that}}\\
 \mathbf{K}\subset  \{g\in G(\mathbb{A}_f)\mid g (\Gamma\otimes \hat{\mathbb{Z}})\subset \Gamma\otimes \hat{\mathbb{Z}}, \, g\equiv\id \mod m(\Gamma\otimes \hat{\mathbb{Z}})\}.
   \end{aligned}
    \end{equation}
As above (\ref{BZsplit1e}) we will assume that all prime ideals of $O_F$
over $p$ are split in $K/F$. Let $K \rightarrow F_{\mathfrak{p}_i}$ be the embedding
over $F$ which induces the prime ideal $\mathfrak{q}_i$ of $K$. It we compose 
the embedding with the conjugation on $K$, the induced prime ideal is
$\bar{\mathfrak{q}}_i$. We write
  \begin{displaymath}
  \begin{aligned} 
    K \otimes_{F} F_{\mathfrak{p}_i} & \overset{\sim}{\longrightarrow} &
    F_{\mathfrak{p}_i} \times F_{\mathfrak{p}_i} =
    K_{\mathfrak{q}_i} \times K_{\bar{\mathfrak{q}}_i}\\
  x \otimes f & \longmapsto & (xf, \bar{x} f)\\
    \end{aligned}
  \end{displaymath}
We will from now on always assume that the
function $r_{\varphi}$ satifies (\ref{BZEnu1e}). We consider the moduli problem
$\mathcal{A}_{\mathbf{K}}$ over the local reflex field $E_{\nu}$ (\ref{BZ3e}). We will extend it to a
moduli problem over $O_{E_{\nu}}$. For this, we need to impose some restrictions on $\mathbf{K}$.

 We set 
\begin{displaymath}
  V_{\mathfrak{q}_i} = V \otimes_K K_{\mathfrak{q}_i}, \quad 
  V_{\bar{\mathfrak{q}}_i} = V \otimes_K K_{\bar{\mathfrak{q}}_i}, \quad 
  V_{\mathfrak{p}_i} = V \otimes_F F_{\mathfrak{p}_i} = V_{\mathfrak{q}_i} \oplus
  V_{\bar{\mathfrak{q}}_i}. 
  \end{displaymath}
We use the decompositions 
\begin{equation}\label{p-zerlegt3e}
  B \otimes \mathbb{Q}_p = \prod_{i=0}^s 
  (B_{\mathfrak{q}_i} \times B_{\bar{\mathfrak{q}}_i}), \quad 
  V \otimes \mathbb{Q}_p = \bigoplus_{i=0}^s V_{\mathfrak{p}_i} = \bigoplus_{i=0}^s
  (V_{\mathfrak{q}_i} \oplus V_{\bar{\mathfrak{q}}_i}), 
\end{equation} 
cf. (\ref{p-zerlegt}). All  $V_{\mathfrak{p}_i} $ in the last decomposition are orthogonal with respect to
$\psi_p: V\otimes\mathbb{Q}_p\times V\otimes\mathbb{Q}_p\rightarrow\mathbb{Q}_p$.

An element $g \in G(\mathbb{Q}_p)$ has the form
$g = (\ldots, g_{\mathfrak{q}_i}, g_{\bar{\mathfrak{q}}_i}, \ldots)$, where
$g_{\mathfrak{q}_i} \in \End_{B_{\mathfrak{q}_i}} V_{\mathfrak{q}_i}$ and
$g_{\bar{\mathfrak{q}}_i} \in \End_{B_{\bar{\mathfrak{q}}_i}} V_{\bar{\mathfrak{q}}_i}$.
We define $g'_{\mathfrak{q}_i} \in \End_{B_{\bar{\mathfrak{q}}_i}} V_{\bar{\mathfrak{q}}_i}$
by
\begin{equation}\label{BZGQ_p4e} 
  \psi_p(g_{\mathfrak{q}_i} v, w) = \psi_p(v, g'_{\mathfrak{q}_i} w), \quad v \in
  V_{\mathfrak{q}_i}, \; w \in V_{\bar{\mathfrak{q}}_i}.
\end{equation}
We see that $g \in G(\mathbb{Q}_p)$ if and only if 
\begin{equation}\label{BZGQ_p1e}
g'_{\mathfrak{q}_i} g_{\bar{\mathfrak{q}}_i} \in \mathbb{Q}_p^{\times}
  \end{equation}
and  if this value is independent of $i$. 
We set
\begin{equation}\label{BZGQ_p5e}
G_{\mathfrak{q}_i} = \Aut_{B_{\mathfrak{q}_i}} \! V_{\mathfrak{q}_i}
  \end{equation}
By (\ref{BZGQ_p1e}) we obtain a canonical isomorphism
\begin{equation}\label{BZGQ_p2e}
  G(\mathbb{Q}_p) \cong
  G_{\mathfrak{q}_0} \times \ldots \times G_{\mathfrak{q}_s} \times \mathbb{Q}_p^{\times}.
\end{equation}
The multiplier homomorphism $\mu: G(\mathbb{Q}_p) \rightarrow \mathbb{Q}_p^{\times}$
corresponds on the right hand side to the projection on the factor
$\mathbb{Q}_p^{\times}$. 

We are only interested in the case where 
\begin{equation}\label{Dsplit-in0} 
  D_{\mathfrak{p}_0}^{{\rm opp}} \cong B_{\mathfrak{q}_0} \cong B_{\bar{\mathfrak{q}}_0} \quad
  \text{\emph{is a quaternion division algebra over}} \; F_{\mathfrak{p}_0} .
  \end{equation}

For 
each prime $\mathfrak{q}_i$, $i = 0, 1, \ldots, s$ we choose a maximal order
$O_{B_{\mathfrak{q}_i}} \subset B_{\mathfrak{q}_i}$. The image of $O_{B_{\mathfrak{q}_i}}$
by the involution $\star : B_{\mathfrak{q}_i} \rightarrow B_{\bar{\mathfrak{q}}_i}$
will be denoted by $O_{B_{\bar{\mathfrak{q}}_i}}$. 

We set
$\Lambda_{\mathfrak{q}_i} = O_{B_{\mathfrak{q}_i}} \subset V_{\mathfrak{q}_i}$.
Moreover we set
\begin{displaymath}
  \Lambda_{\bar{\mathfrak{q}}_i} = \{u \in V_{\bar{\mathfrak{q}}_i} \; | \; 
  \psi_{\mathfrak{p}_i}(x, u) \in \mathbb{Z}_p, \; \text{for all}\;
  x \in \Lambda_{\mathfrak{q}_i} \}, 
\end{displaymath}
where $\psi_{\mathfrak{p}_i}$ is the restriction of $\psi_p$ to $V_{\mathfrak{p}_i}$. 
Then $\Lambda_{\bar{\mathfrak{q}}_i}$ is an $O_{B_{\bar{\mathfrak{q}}_i}}$-module and the
pairings 
\begin{equation}\label{BZLambda1e}
  \psi_{\mathfrak{p}_i}: \Lambda_{\mathfrak{q}_i} \times \Lambda_{\bar{\mathfrak{q}}_i}
  \rightarrow \mathbb{Z}_p 
  \end{equation}
are perfect. We write
$\Lambda_{\mathfrak{p}_i} =\Lambda_{\mathfrak{q}_i}\oplus\Lambda_{\bar{\mathfrak{q}}_i}$
and $\Lambda_p = \oplus_{i=0}^{s} \Lambda_{\mathfrak{p}_i}$. 

We choose an open subgroup $\mathbf{M} \subset \mathbb{Z}_p^{\times}$. 
We set $\mathbf{K}_{\mathfrak{q}_0} = \Aut_{O_{B_{\mathfrak{q}_0}}} \Lambda_{\mathfrak{q}_0}$. 
For $i > 0$ we choose arbitrarily open and compact subgroups  
$\mathbf{K}_{\mathfrak{q}_i} \subset \Aut_{O_{B_{\mathfrak{q}_i}}} \Lambda_{\mathfrak{q}_i}$. We set 
\begin{equation}
  G_{\mathfrak{p}_i} = \{g \in \Aut_{B_{\mathfrak{p}_i}} \! V_{\mathfrak{p}_i} \; | \;
  \psi_{\mathfrak{p}_i}(g v, g w) = \mu_{\mathfrak{p}_i}(g)\psi_{\mathfrak{p}_i}(v, w) \;
  \text{for} \; \mu_{\mathfrak{p}_i}(g) \in \mathbb{Q}_p^{\times} \}. 
\end{equation}
We define
$\mathbf{K}_{\mathfrak{p}_i} \subset G_{\mathfrak{p}_i}$ 
as the group of all pairs $g=(c_1, c_2)$ of automorphisms
\begin{displaymath}
  c_1 \in \mathbf{K}_{\mathfrak{q}_i}, \quad c_2 \in
  \Aut_{B_{\bar{\mathfrak{q}}_i}} V_{\bar{\mathfrak{q}}_i}
\end{displaymath}
such that for some $m \in \mathbf{M}$  
\begin{displaymath}
  \psi(c_1 v, c_2 w) = m \psi(v, w), \quad \text{for all} \;
  v \in V_{\mathfrak{q}_i}, w \in V_{\bar{\mathfrak{q}}_i}. 
\end{displaymath}
Since $c_1 (\Lambda_{\mathfrak{q}_i}) \subset \Lambda_{\mathfrak{q}_i}$ it follows
from (\ref{BZGQ_p4e}) that
$c'_1 (\Lambda_{\bar{\mathfrak{q}}_i}) \subset \Lambda_{\bar{\mathfrak{q}}_i}$.
Since $c'_1c_2 = m$, this implies that 
$c_2 \in \Aut_{O_{B_{\bar{\mathfrak{q}}_i}}} \Lambda_{\bar{\mathfrak{q}}_i}$.

We obtain an isomorphism 
\begin{displaymath}
  \begin{array}{ccc}
    \mathbf{K}_{\mathfrak{p}_i} & \cong &\mathbf{K}_{\mathfrak{q}_i}\times \mathbf{M}\\
    (c_1,c_2) & \mapsto & c_1 \times m .
    \end{array}
\end{displaymath}
We define the subgroup $\mathbf{K}_p \subset G(\mathbb{Q}_p)$  as
\begin{equation}\label{BZKp1e}
  \begin{array}{rl}
  \mathbf{K}_p = & \{g = (g_{\mathfrak{p}_i})\in\prod_i \mathbf{K}_{\mathfrak{p}_i} \;
  | \; \mu (g_{\mathfrak{p}_0}) =\ldots= \mu(g_{\mathfrak{p}_s})
  \in \mathbf{M} \}\\[2mm]  
  \cong & \mathbf{K}_{\mathfrak{q}_0} \times \ldots \times \mathbf{K}_{\mathfrak{q}_s}
  \times \mathbf{M}.
  \end{array}
\end{equation}
The last equation follows from (\ref{BZGQ_p2e}). 
We choose an arbitrary open compact subgroup
$\mathbf{K}^p \subset G(\mathbb{A}^p_f)$ and define
\begin{equation}\label{BZ7e} 
\mathbf{K} = \mathbf{K}_p \mathbf{K}^p \subset G(\mathbb{A}_f). 
  \end{equation} 
This concludes the description of the class of open compact subgroups $\mathbf{K}$ for which we will extend $\mathcal{A}_{\mathbf{K}}$ to a  moduli problem  over $\Spec O_{E_\nu}$. For these $\mathbf{K}$ we may reformulate the Definition \ref{BZAK1d} of
the functor $\mathcal{A}_{\mathbf{K}}$. The datum $\bar{\eta}$ is then the
product of two classes $\bar{\eta}^p$ modulo $\mathbf{K}^p$, resp. 
$\bar{\eta}_p$ modulo $\mathbf{K}_p$, of isomorphisms 
\begin{displaymath}
  \eta^p: V \otimes \mathbb{A}^p_f \isoarrow V^p_f(A), \quad \text{resp.} 
  \quad  \eta_p: V \otimes \mathbb{Q}_p \isoarrow V_p(A),   
\end{displaymath}
which respect the bilinear forms on both sides up to a constant in
$(\mathbb{A}_f^p)^{\times}(1)$, resp. $\mathbb{Q}_p^{\times}(1)$. In particular there
is for each $\lambda \in \bar{\lambda}$ locally on $S$ a constant
$\xi_p(\lambda) \in \mathbb{Q}_p^{\times}(1)$ such that for the Riemann form
$E^{\lambda}$ 
\begin{equation}\label{BZ21e}
  E^{\lambda}(\eta_p(v), \eta_p(w)) = \xi_p(\lambda) \psi(v,w), \quad v,w \in
  V \otimes_{\mathbb{Q}} \mathbb{Q}_p. 
\end{equation}
If we change $\eta_p$ in its class by an element $g \in \mathbf{K}_p$ we find
\begin{displaymath}
  E^{\lambda}(\eta_p(gv), \eta_p(gw)) = \xi_p(\lambda) \psi(gv,gw) =
  \xi_p(\lambda) \mu(g) \psi(v,w).   
  \end{displaymath}
Since $\mu(g) \in \mathbf{M}$, the class of
$\xi_p(\lambda) \in \mathbb{Q}_p^{\times}(1)/\mathbf{M}$ is well-defined by the
class $\bar{\eta}_p$. If we change $\lambda$ into $u\lambda$ for
$u \in \mathbb{Q}^{\times}$, we obtain
\begin{displaymath}
\xi_p(u \lambda) = u \xi_p(\lambda). 
\end{displaymath}
By (\ref{p-zerlegt3e}) $\eta_p$ decomposes into isomorphisms
\begin{displaymath}
  \eta_{\mathfrak{q}_i}: V \otimes_K K_{\mathfrak{q}_i} \isoarrow V_{\mathfrak{q}_i}(A),
  \quad  \eta_{\bar{\mathfrak{q}}_i}: V \otimes_K K_{\bar{\mathfrak{q}}_i}
  \isoarrow V_{\bar{\mathfrak{q}}_i}(A), \quad \text{for} \; i= 0, \ldots, s.
\end{displaymath}
The equation (\ref{BZ21e}) becomes equivalent to the equations for  $i = 0, \ldots, s$,
\begin{equation}
  E^{\lambda}(\eta_{\mathfrak{q}_i}(v_i), \eta_{\bar{\mathfrak{q}}_i}(w_i) =
  \xi_p(\lambda) \psi(v_i, w_i), \quad v_i\in V\otimes_{K} K_{\mathfrak{q}_i},\, w_i\in V\otimes_{K} K_{\bar{\mathfrak{q}}_i} .
\end{equation}
From these equations it is clear that the set of data
$\eta_{\mathfrak{q}_i}, \eta_{\bar{\mathfrak{q}}_i}$ is determined by
$\eta_{\mathfrak{q}_i}, \xi_p(\lambda)$.

We obtain the following reformulation of Definition \ref{BZAK1d}. 
\begin{definition}\label{BZAK1altd} (alternative of Definition \ref{BZAK1d} for
  $\mathcal{A}_{\mathbf{K}}$)  
Let $\mathbf{K}=\mathbf{K}_p\mathbf{K}^p \subset G(\mathbb{A}_f)$, where $\mathbf{K}_p$ is defined as in \eqref{BZKp1e}.
  Then we can replace the datum $(c)$ of Definition \ref{BZAK1d} 
by the following data
  \begin{enumerate}
  \item[($c^p$)] A class $\bar{\eta}^p$ modulo $\mathbf{K}^p$ of
    $B \otimes \mathbb{A}^p_f$-module isomorphisms 
  \begin{displaymath}
    \eta^p: V \otimes \mathbb{A}^p_f \isoarrow \mathrm{V}^p_f(A) 
  \end{displaymath}
  such that for each $\lambda \in \bar{\lambda}$ there is a constant
  $\xi^{(p)}(\lambda) \in \mathbb{A}^p_f(1)$
  with
  \begin{displaymath}
\xi^{(p)}(\lambda) \psi(v_1, v_2) = E^{\lambda}(\eta^p(v_1), \eta^p(v_2)). 
  \end{displaymath}
\item[($c_p$)] For each $i = 0, \ldots s$ a class $\bar{\eta}_{\mathfrak{q}_i}$
  modulo $\mathbf{K}^{\bullet}_{\mathfrak{q}_i}$ of $B_{\mathfrak{q}_i}$-module
  isomorphisms 
  \begin{displaymath}
    \eta_{\mathfrak{q}_i}: V \otimes_K K_{\mathfrak{q}_i} \isoarrow
    V_{\mathfrak{q}_i}(A). 
  \end{displaymath}
\item[($\xi_p$)] A function
  $$\xi_p: \bar{\lambda} \rightarrow \mathbb{Q}_p^{\times}(1)/\mathbf{M}$$
   such that
  $\xi_p(u \lambda) = u \xi_p(\lambda)$ for each $u \in \mathbb{Q}^{\times}$. 
  \end{enumerate}
  \end{definition}
Let $g \in G(\mathbb{Q}_p) \subset G(\mathbb{A}_f)$ and consider the Hecke
operator
\begin{equation}\label{Heckeoperator1e}
  g: \mathcal{A}_\mathbf{\mathbf{K}} \rightarrow
  \mathcal{A}_{g^{-1}\mathbf{K}g}  
\end{equation}
which maps a point $(A, \iota, \bar{\lambda}, \bar{\eta})$ to 
$(A, \iota, \bar{\lambda}, \overline{\eta g})$. This makes sense, since with
$\mathbf{K}$ also $g^{-1} \mathbf{K}g$ satisfies the conditions imposed in
\eqref{BZ7e}. According to the isomorphism (\ref{BZGQ_p2e}) we represent $g$
in the form
$(\ldots, g_{\mathfrak{q}_i}, g_{\bar{\mathfrak{q}}_i}, \ldots, \mu(g))$. 
In terms of the Definition \ref{BZAK1altd}, the point is represented in the form
$(A, \iota, \bar{\lambda}, \bar{\eta}^p, (\bar{\eta}_{\mathfrak{q}_i})_i, \xi_p)$. 
The Hecke operator (\ref{Heckeoperator1e}) maps it to the point
$(A,\iota,\bar{\lambda},\bar{\eta}^p,(\overline{\eta_{\mathfrak{q}_i}g_{\mathfrak{q}_i}})_i,\mu(g)\xi_p)$.

Let $u \in \mathbb{Q}_p^{\times}$. The action which maps
$(A, \iota, \bar{\lambda}, \bar{\eta}^p, (\bar{\eta}_{\mathfrak{q}_i})_i, \xi_p)$
to
$(A,\iota,\bar{\lambda},\bar{\eta}^p, (\bar{\eta}_{\mathfrak{q}_i})_i, u\xi_p)$
is denoted by 
\begin{equation}\label{xioperator1e}
u_{\mid \xi}: \mathcal{A}_\mathbf{\mathbf{K}} \rightarrow
  \mathcal{A}_{\mathbf{K}} .
  \end{equation}
The element
\begin{displaymath}
  s_u := (1, \ldots 1) \times (u, \ldots, u)\in (K \otimes \mathbb{Q}_p)^{\times}
  = (\prod_{i=0}^{s} K^{\times}_{\mathfrak{q}_i}) \times
  (\prod_{i=0}^{s} K^{\times}_{\bar{\mathfrak{q}}_i}) 
  \end{displaymath}
lies in $G(\mathbb{Q}_p)$. The action of the Hecke operator $s_u$ coincides
with the action of $u_{\mid \xi}$. The notation $u_{\mid \xi}$ is a reminder that the action of $s_u$  only changes the last entry in a tuple $(A, \iota, \bar{\lambda}, \bar{\eta}^p, (\bar{\eta}_{\mathfrak{q}_i})_i, \xi_p)$.  

We set
\begin{equation}\label{BZG1e}
O_{B,p} = \prod_{i=0}^s
  (O_{B_{\mathfrak{q}_i}} \times O_{B_{\bar{\mathfrak{q}}_i}}) \subset B \otimes \mathbb{Q}_p. 
  \end{equation}
Let $O_{B,(p)} \subset B$ the subring of elements which lie in $O_{B,p}$. This subring is invariant under the involution $\star$, cf. \eqref{defstar}. 
We will consider abelian schemes $A$ with an action
\begin{displaymath}
  O_{K,(p)} = O_K \otimes_{\mathbb{Z}} \mathbb{Z}_{(p)} \rightarrow
\End A \otimes_{\mathbb{Z}} \mathbb{Z}_{(p)}. 
\end{displaymath}
Then $O_K \otimes_{\mathbb{Z}} \mathbb{Z}_p$ acts on the $p$-divisible group $X$
of $A$. Therefore we obtain a decomposition
\begin{displaymath}
  X = (\prod_{i=0}^{s} X_{\mathfrak{q}_i}) \times
  (\prod_{i=0}^{s} X_{\bar{\mathfrak{q}}_i}). 
  \end{displaymath}
 We will write $X_{\mathfrak{p}_i} = X_{\mathfrak{q}_i} \times X_{\bar{\mathfrak{q}}_i}$.
This continues to make sense if $A$ is an abelian scheme up to isogeny of order
prime to $p$.
We set
\begin{displaymath}
  U_p(\mathbb{Q}) = \{d \in \mathbb{Q}^{\times} \; | \; \ord_p d = 0 \} =
  \mathbb{Z}_{(p)}^{\times}. 
  \end{displaymath}
\begin{definition}\label{BZsAK2d}  
Let $\mathbf{K}=\mathbf{K}_p\mathbf{K}^p \subset G(\mathbb{A}_f)$, where $\mathbf{K}_p$ is defined as in \eqref{BZKp1e}. 
We define a variant $\mathcal{A}^{ bis}_{\mathbf{K}}$ of the functor
$\mathcal{A}_{\mathbf{K}}$. A point of $\mathcal{A}^{bis}_{\mathbf{K}}$ with
values in an $E$-scheme $S$ consists of:  
\begin{enumerate} 
\item[(a)] An abelian scheme $A$ over $S$ up to isogeny of degree prime to
  $p$ with an action
  $\iota: O_{B,(p)} \rightarrow \End A \otimes_{\mathbb{Z}} \mathbb{Z}_{(p)}$, 

\item[(b)] A $U_p(\mathbb{Q})$-homogeneous polarization $\bar{\lambda}$ of $A$
  which induces on $B$ the involution $b \mapsto b^{\star}$ and which is
  principal in $p$, i.e. of degree prime to $p$. 
 
\item[(c)] A class $\bar{\eta}^p$  modulo $\mathbf{K}^p$ of
  $B \otimes \mathbb{A}^p_f$-module isomorphisms 
  \begin{displaymath}
    \eta^p: V \otimes \mathbb{A}^p_f \isoarrow \mathrm{V}^p_f(A) .
   \end{displaymath}
  such that for each $\lambda \in \bar{\lambda}$ there is a constant
  $\xi^{(p)}(\lambda) \in \mathbb{A}^p_f(1)$
  with
  \begin{displaymath}
\xi^{(p)}(\lambda) \psi(v_1, v_2) = E^{\lambda}(\eta^p(v_1), \eta^p(v_2)). 
  \end{displaymath}
\item[(d)]
  For each polarization $\lambda \in \bar{\lambda}$ a section 
  $\xi_p(\lambda) \in \mathbb{Z}_p^{\times}(1)/\mathbf{M}$ such that
  $\xi_p(u \lambda) = u \xi_p(\lambda)$ for each $u \in U_p(\mathbb{Q})$. 
\item[(e)]
  A class $\bar{\eta}_{\mathfrak{q}_j}$   modulo $\mathbf{K}_{\mathfrak{q}_i}$ of
  $O_{B_{\mathfrak{q}_j}}$-module isomorphisms for each $j = 1, \ldots s$,
  \begin{displaymath}
    \eta_{\mathfrak{q}_j}: \Lambda_{\mathfrak{q}_j} \isoarrow T_{\mathfrak{q}_j}(A) .
   \end{displaymath} 
   \end{enumerate}
We require that the  condition ${\rm (KC)}$ holds
  \begin{equation}
    {\rm char}(T, \iota(b) \mid \Lie A) = \prod_{\varphi: K \rightarrow \bar{\mathbb{Q}}}
    \varphi(\Nm^o_{B/K} (T -b))^{r_{\varphi}} . 
  \end{equation}
 
\end{definition}
Here and in the sequel, the index $i$ will run through $0,\ldots, s$ and the index $j$  through $1,\cdots, s$. 
\begin{proposition}\label{BZ1p}
  The functors $\mathcal{A}_{\mathbf{K}}$ and $\mathcal{A}^{bis}_{\mathbf{K}}$ 
  on the category of $E$-schemes are isomorphic.
  \end{proposition}
\begin{proof}
  We begin with a point $(A, \iota, \bar{\lambda}, \bar{\eta})$ of
  $\mathcal{A}_{\mathbf{K}}(S)$ and construct a point of
  $\mathcal{A}^{bis}_{\mathbf{K}}(S)$.
  We choose an element $\eta \in \bar{\eta}$ and consider the component 
  $\eta_p: V \otimes \mathbb{Z}_p \isoarrow \mathrm{V}_p(A)$. By the choice
  of $\mathbf{K}_{p}$, the image of $\Lambda_p$ by this morphism is independent
  of the choice of $\eta_p$. Therefore we find an abelian variety $A_0 \in A$
  up to isogeny prime to $p$ such that $\eta_p$ induces an isomorphism 
  \begin{displaymath}
\eta_p: \Lambda_{p} \isoarrow T_p(A_0).  
    \end{displaymath}
  We choose a polarization $\lambda_0 \in \bar{\lambda}$. Then we obtain 
  \begin{displaymath}
    E_p^{\lambda_0}(\eta_p(v_1), \eta_p(v_2)) = \xi \psi_p(v_1, v_2), \quad
    v_1, v_2 \in \Lambda_p, \; \xi \in \mathbb{Q}_p^{\times}(1). 
  \end{displaymath}
  After multiplying $\lambda_0$ by a power of $p$ we may assume that 
  $\xi \in \mathbb{Z}_p^{\times}(1)$. This determines an 
  $U_p(\mathbb{Q})$-homogeneous polarization $\bar{\lambda}_0$ on $A_0$.
  We remark that the class of $\xi$ in $\mathbb{Z}_p^{\times}(1)/\mathbf{M}$
  is independent of the choice of $\eta$. 
  Finally $\eta_p$ induces $O_{B_{\mathfrak{q}_i}}$-module isomorphisms
$\eta_{\mathfrak{q}_i}:\Lambda_{\mathfrak{q}_i}\isoarrow T_{\mathfrak{q}_i}(A_0)$.  
  We have obtained a point
  $(A_0, \iota, \bar{\lambda}, \bar{\eta}^p,(\eta_{\mathfrak{q_j}})_j,  \xi_p)$ of
  $\mathcal{A}^{bis}_{\mathbf{K}}$. 

  Conversely assume that
  $(A_0, \iota, \bar{\lambda}, \bar{\eta}^p,(\eta_{\mathfrak{q_j}})_j,  \xi_p) \in
  \mathcal{A}^{bis}_{\mathbf{K}}$ 
is given. Let $\lambda \in \bar{\lambda}_0$. The Riemann form 
  $E^{\lambda}$ on $T_{\mathfrak{p}_0}(A_0)$ is by
assumption perfect. Therefore we find for any given
$\xi' \in \mathbb{Z}_p^{\times}(1)$ an isomorphism of $O_{B_{\mathfrak{p}_0}}$-modules 
  \begin{displaymath}
\eta_{\mathfrak{p}_0}: \Lambda_{\mathfrak{p}_0} \isoarrow T_{\mathfrak{p}_0}(A_0) 
  \end{displaymath}
such that 
  \begin{displaymath}
    \xi' \psi(v_1, v_2) =
    E^{\lambda}(\eta_{\mathfrak{p}_0}(v_1), \eta_{\mathfrak{p}_0}(v_2)), \quad
    \; v_1, v_2 \in \Lambda_{\mathfrak{p}_0}. 
  \end{displaymath}
  We choose $\eta_{\mathfrak{p}_0}$ such that $\xi' = \xi_p(\lambda)$. 
  For $i > 0$ the isomorphism $\eta_{\mathfrak{q}_i}$ induces by duality
  $\Hom( ?, \mathbb{Z}_p)$ an isomorphism
  $T_{\bar{\mathfrak{q}}_i}(A_0) \isoarrow \Lambda_{\bar{\mathfrak{q}}_i}(1)$.
  We multiply the inverse map with $\xi_p(\lambda)$ and obtain
  \begin{displaymath}
    \eta_{\bar{\mathfrak{q}}_i}: \Lambda_{\bar{\mathfrak{q}}_i} \isoarrow
    T_{\bar{\mathfrak{q}}_i}(A_0). 
    \end{displaymath}
 We obtain an isomorphism 
  \begin{displaymath}
    \eta_{\mathfrak{p}_i} = \eta_{\mathfrak{q}_i} \oplus \eta_{\bar{\mathfrak{q}}_i}:
    \Lambda_{\mathfrak{p}_i} \isoarrow T_{\mathfrak{p}_i} (A_0)
    \end{displaymath}
  which respects the bilinear forms on both sides up to the factor
  $\xi_p(\lambda)$. Then $\eta_p = \oplus^{s}_{i=0} \eta_{\mathfrak{p}_i}$ defines
  an isomorphism $\eta_p: \Lambda_p \rightarrow T_p(A_0)$ such that
  \begin{displaymath}
    \xi_p(\lambda) \psi(v_1, v_2) = E^{\lambda}(\eta_p(v_1), \eta_p(v_2)), \quad
    v_1, v_2 \in \Lambda_p.
    \end{displaymath}
 Therefore we have constructed a point of $\mathcal{A}_{\mathbf{K}}(S)$. The two procedures are inverses of each other, proving the proposition.
\end{proof}

From now on we will always assume that $\mathbf{K}$ satisfies the assumptions
made in Definition \ref{BZsAK2d}. If we write $\mathcal{A}_{\mathbf{K}}$,
we understand that this functor is given in the form of Definition
\ref{BZsAK2d}. To extend 
the functor $\mathcal{A}_{\mathbf{K}}(S)$ to an arbitrary $O_{E_{\nu}}$-scheme $S$, the main obstacle
is the datum $(d)$, since we do not have the
$(\mathbb{Z}/p\mathbb{Z})^{\times}$-\'etale torsor of primitive $p$th roots of unity.
Therefore we define a new functor $\tilde{\mathcal{A}}^t_{\mathbf{K}}$ by
replacing the data $\xi_p(\lambda) \in \mathbb{Z}_p^{\times}(1)/\mathbf{M}$ by
sections in the constant sheaf
$\xi_p(\lambda) \in \mathbb{Z}_p^{\times}/\mathbf{M}$. Here the upper index $t$ in  $\tilde{\mathcal{A}}^t_{\mathbf{K}}$ indicates that this functor, when restricted to the category of $E_\nu$-schemes, is a twisted version of $\mathcal {A}_{\mathbf{K}}$. 
\begin{definition}\label{BZsAtK2d}
Let $\mathbf{K}=\mathbf{K}_p\mathbf{K}^p \subset G(\mathbb{A}_f)$, where $\mathbf{K}_p$ is defined as in \eqref{BZKp1e}. We define a functor $\tilde{\mathcal{A}}^t_{\mathbf{K}}$ on the category of
$O_{E_{\nu}}$-schemes $S$. 
An $S$-valued point consists of the data $(a), (b), (c), (e)$ as in Definition
\ref{BZsAK2d}. But we replace $(d)$ by the following datum
\begin{enumerate}
\item[$(d^t)$]
   A section 
  $\xi_p(\lambda) \in \mathbb{Z}_p^{\times}/\mathbf{M}$ for each polarization $\lambda \in \bar{\lambda}$ such that
  $$\xi_p(u \lambda) = u \xi_p(\lambda), \quad u \in U_p(\mathbb{Q}) .
  $$
  \end{enumerate}
We continue to impose the condition ${\rm (KC)}$. 
\end{definition}
 The data (a)--(c) continue to make sense over a 
$O_{E_{\nu}}$-scheme $S$. Since an isogeny of
degree prime to $p$ induces an isomorphism on  tangent spaces, the condition (KC) also makes
sense. Since $r_{\varphi} = 0$ for each
$\varphi: K_{\mathfrak{q}_j} \rightarrow \bar{\mathbb{Q}}_p$ for $j= 1, \ldots, s$,
the $p$-divisible groups $X_{\mathfrak{q}_j}$ are \'etale. We note that
$X_{\mathfrak{q}_j}$ is a $p$-divisible group of height
$4[K_{\mathfrak{q}_j}: \mathbb{Q}_p]$ and this implies that $T_{\mathfrak{q}_j}(A)$
is a free $O_{B_{\mathfrak{q}_j}}$-module of rank $1$. Therefore the datum 
$(e)$ also continues to make sense. 

The functor $\tilde{\mathcal{A}}^t_{\mathbf{K}}$ is representable if the group
$\mathbf{K}$ satisfies the condition (\ref{BZneat1e}) for some integer
$m \geq 3$ which is prime to $p$. A standard argument shows that
$\tilde{\mathcal{A}}^t_{\mathbf{K}}$ is proper over $\Spec O_{E_{\nu}}$, cf.
\cite[Prop. 4.1]{Dr}. If we have another open and compact
subgroup $\tilde{\mathbf{K}} \subset \mathbf{K}$ as in Definition
\ref{BZsAtK2d},  we obtain an \'etale covering
\begin{displaymath}
  \widetilde{\mathcal{A}}^t_{\tilde{\mathbf{K}}} \rightarrow
  \widetilde{\mathcal{A}}^t_{\mathbf{K}}. 
\end{displaymath} 

The general fibre $\tilde{\mathcal{A}}^t_{\mathbf{K}, E_{\nu}}$  of
$\tilde{\mathcal{A}}^t_{\mathbf{K}}$ is a Galois twist of
$\mathcal{A}_{\mathbf{K},E_{\nu}}=\mathcal{A}_{\mathbf{K}}\times_{\Spec E}\Spec E_{\nu}$. Let us explain this. 
The definition of 
$\tilde{\mathcal{A}}^{t}_{\mathbf{K}}(S)$ makes sense for any $E$-scheme $S$.
We denote this functor on the category of $E$-schemes by 
$\mathcal{A}^{t}_{\mathbf{K}}$. (The fact, used above,  that the $p$-divisible groups $X_{\mathfrak{q}_j}$
are \'etale is automatic  in characteristic $0$, even for $X_{\mathfrak{q}_0}$.) We may represent a
point of $\mathcal{A}^{t}_{\mathbf{K}}(S)$ in the same way as in
Definition \ref{BZAK1altd} by
\begin{equation}\label{A^t-points-alt}
(A, \iota, \bar{\lambda}, \bar{\eta}^p, (\bar{\eta}_{\mathfrak{q}_i})_i, \xi_p) ,
\end{equation}
except that now $\xi_p$ is a function
$\xi_p: \bar{\lambda} \rightarrow \mathbb{Q}_p^{\times}/\mathbf{M}$.   
There is the 
canonical isomorphism
\begin{equation}\label{Gtwist4e}
  \mathcal{A}^{t}_{\mathbf{K}} \times_{\Spec E} \Spec E_{\nu} \cong
  \tilde{\mathcal{A}}^{t}_{\mathbf{K}} \times_{\Spec O_{E_{\nu}}} \Spec E_{\nu}. 
  \end{equation}

We will identify 
$\mathcal{A}^{t}_{\mathbf{K}}$ with a Shimura variety of the form $\Sh(G, h \cdot c)$ for some
$c: \mathbb{S} \rightarrow G_{\mathbb{R}}$ which factors through the center
of $G$. We consider the cyclotomic character  
\begin{displaymath}
\varsigma_{p^{\infty}}: \Gal(\bar{\BQ}/E) \rightarrow \mathbb{Z}_p^{\times}/\mathbf{M}.   
\end{displaymath}
Let $L \subset \bar{\BQ}$ be the subfield fixed by the kernel of this
homomorphism. Let $\zeta_{p^{\infty}} \in \mathbb{C}$ be a compatible system of
primitive
$p^n$-th roots of unity. We obtain a natural isomorphism
\begin{equation}\label{Gtwist1e}
  \begin{array}{ccc}
    \mathcal{A}^{t}_{\mathbf{K}} \times_{\Spec E} \Spec L & \isoarrow & 
    \mathcal{A}_{\mathbf{K}} \times_{\Spec E} \Spec L,\\
    \xi_p & \mapsto & \zeta_{p^{\infty}} \xi_p
    \end{array}
\end{equation}
i.e., in  Definition \ref{BZsAK2d}
we have to change only $(d)$ to pass from one functor to the other. 
\begin{proposition}\label{Gtwist1p}
  Let $\tau \in \Gal(L/E)$ be an automorphism. 
  The morphism (\ref{Gtwist1e}) fits into a commutative diagram
  \begin{displaymath}
    \xymatrix{
      \mathcal{A}^{t}_{\mathbf{K}} \times_{\Spec E} \Spec L \ar[d]_{\id \times \tau_c}
      \ar[r] & \mathcal{A}_{\mathbf{K}} \times_{\Spec E} \Spec L
      \ar[d]^{\varsigma_{p^{\infty}}(\tau^{-1})_{\mid \xi} \times \tau_c}\\
    \mathcal{A}^{t}_{\mathbf{K}} \times_{\Spec E} \Spec L \ar[r] & 
    \mathcal{A}_{\mathbf{K}} \times_{\Spec E} \Spec L .\\
    }
  \end{displaymath}
  {\rm (see (\ref{xioperator1e}) for the notation $\varsigma_{p^{\infty}}(\tau^{-1})_{\mid \xi}$).} 
\end{proposition} 
\begin{proof}
  Let $\pi\colon S \rightarrow \Spec L$ be an $L$-scheme. Denote by $S_{[\tau_c]}$  the
  $L$-scheme obtained when the structure morphism is changed to
  $\tau_c \circ \pi$. Our assertion says that the following diagram is
  commutative
\begin{equation}\label{Gtwist2e}
\begin{aligned}
    \xymatrix{
      (\mathcal{A}^{t}_{\mathbf{K}} \times_{\Spec E} \Spec L)(S)
      \ar[d]_{{ can}} \ar[r] & (\mathcal{A}_{\mathbf{K}} \times_{\Spec E} \Spec L)(S) 
      \ar[d]^{\varsigma_{p^{\infty}}(\tau^{-1})_{\mid \xi}\circ { can}}\\
    (\mathcal{A}^{t}_{\mathbf{K}} \times_{\Spec E} \Spec L)(S_{[\tau_c]}) \ar[r] & 
    (\mathcal{A}_{\mathbf{K}} \times_{\Spec E} \Spec L)(S_{[\tau_c]}) .\\
   }
   \end{aligned}
\end{equation}
The morphism ${ can}$ on the left hand side is defined by the identification
$\mathcal{A}^{t}_{\mathbf{K}}(S) = \mathcal{A}^{t}_{\mathbf{K}}(S_{[\tau_c]})$ which
exists because the functor is defined over $E$. In the same way there is
${can}$ on the right hand side. To show the commutativity of (\ref{Gtwist2e}),
we start with a point
$(A, \iota, \bar{\lambda}, \bar{\eta}^p, (\bar{\eta}_{\mathfrak{q}_i})_i, \xi_p)$ 
of the upper left corner. The image $\Theta$ of this point by the left $can$
in $\mathcal{A}^{t}_{\mathbf{K}}(S_{[\tau_c]})$ is given by the same tuple.
By the upper horizontal map the point is mapped to
\begin{equation}\label{Gtwist3e}
  (A,\iota,\bar{\lambda}, \bar{\eta}^p, (\bar{\eta}_{\mathfrak{q}_i})_i,
  \pi^*(\zeta_{p^{\infty}})\xi_p)
  \end{equation}
where $\pi^*: L \rightarrow \Gamma(S,\mathcal{O}_S)$ is the comorphism of the structure map.
The image of  (\ref{Gtwist3e}) by the right $can$-morphism is represented
by the same tuple but we write the last item as
\begin{displaymath}
  \pi^*( \tau^{-1} (\tau (\zeta_{p^{\infty}}))) \xi_p = \varsigma_{p^{\infty}}(\tau)
  \pi^* (\tau^{-1} (\zeta_{p^{\infty}})) \xi_p. 
\end{displaymath}
On the other hand the image of $\Theta$ by the lower horizontal bijection 
is the tuple 
$$(A, \iota, \bar{\lambda}, \bar{\eta}^p, (\bar{\eta}_{\mathfrak{q}_i})_i, \pi^* (\tau^{-1} (\zeta_{p^{\infty}}))  \xi_p).$$ 
This shows the commutativity of (\ref{Gtwist2e}). 
\end{proof}

Let $\Xi \subset \Phi$ be the CM-type of $K$ defined by
\begin{equation*}
\Xi = \coprod_{i=0}^{s}
\Hom_{\mathbb{Q}_p{\rm -Alg}}(K_{\bar{\mathfrak{q}}_i},\bar{\mathbb{Q}}_p).
\end{equation*}
As defined by diagram (\ref{BZ2e}), we can write
\begin{equation}\label{Xi2e}
\Xi = \{ \bar{\varphi}_0 \} \cup \{\varphi \in \Phi \; | \; r_{\varphi} =2 \} .
  \end{equation}
We obtain an isomorphism 
\begin{equation}\label{CMXi1e}
  K \otimes_{\mathbb{Q}} \mathbb{R}  \isoarrow  \prod\nolimits_{\Xi} \mathbb{C}, \quad 
  a \otimes \rho  \mapsto  (\varphi(a) \rho)_{\varphi \in \Xi} .
\end{equation}
This puts a complex structure on the left hand side and hence defines a homomorphism
$\mathbb{C}^{\times} \rightarrow (K \otimes_{\mathbb{Q}} \mathbb{R})^{\times}$ which
we view as a morphism of algebraic groups
$\delta: \mathbb{S} \rightarrow (\Res_{K/\mathbb{Q}} \mathbb{G}_{m,K})_\BR$.
We consider the algebraic torus over $\mathbb{Q}$ given by  
\begin{equation*}
  T(\mathbb{Q}) = \{t \in K^{\times} \; | \; t\bar{t} \in \mathbb{Q}^{\times} \}
  \subset G(\mathbb{Q}). 
\end{equation*}
This is a central subtorus of $G$. 
Clearly $\delta$ factors through 
\begin{equation}\label{delta1e}
\delta: \mathbb{S} \rightarrow T_{\mathbb{R}} \subset G_{\mathbb{R}}. 
  \end{equation}
Let $\Sh(G, h \delta^{-1})$ be the canonical model over $E(G, h \delta^{-1})$ of the Shimura variety attached to the Shimura datum $(G, h \delta^{-1})$. Note that $E(G, h \delta^{-1}) \subset E_{\nu}$.
Indeed, if we restrict
$\delta_{\mathbb{C}}: \mathbb{G}_{m, \mathbb{C}} \times \mathbb{G}_{m, \mathbb{C}} \rightarrow (K \otimes_{\mathbb{Q}} \mathbb{C})^{\times}$
to the first factor, we obtain
\begin{equation}\label{zentralerTwist1e}
  \mu_{\delta}: \mathbb{G}_{m, \mathbb{C}} \rightarrow
  (K \otimes_{\mathbb{Q}} \mathbb{C})^{\times} \cong \prod_{\Phi} \mathbb{C}^{\times}, \quad z\mapsto (1, \ldots 1, z, \ldots, z) .
\end{equation}
Here $z$ appears
exactly at the places $\Xi \subset \Phi$. By our choice of the diagram
(\ref{BZ2e}) $\mu_{\delta}$ is defined over $\mathbb{Q}_p$ since it comes
from the canonical embedding
\begin{equation*}
  \mathbb{Q}_p^{\times} \rightarrow \prod_{i=0}^{s} K_{\bar{\mathfrak{q}}_i} \subset
  (K \otimes_{\mathbb{Q}} \mathbb{Q}_p)^{\times}. 
\end{equation*}
Therefore $E(G, \delta^{-1}) \subset \mathbb{Q}_p$, which shows our assertion.
\begin{proposition}\label{Gtwist3p}
  Let 
  $\mathbf{K} \subset G(\mathbb{A}_f)$ be as in Definition \ref{BZsAtK2d}. 
  We denote by $\Sh_{\mathbf{K}}(G, h \delta^{-1})_{E_{\nu}}$ the scheme obtained
  by base change via $E(G, h\delta^{-1}) \subset E_{\nu}$ from the canonical
  model. We set 
  $\mathcal{A}^{t}_{\mathbf{K}, E_{\nu}} = \tilde{\mathcal{A}}^{t}_{\mathbf{K}} \times_{\Spec O_{E_{\nu}}} \Spec E_{\nu}={\mathcal{A}}^{t}_{\mathbf{K}} \times_{\Spec {E}} \Spec E_{\nu}$
  (cf. (\ref{Gtwist4e})). There exists an isomorphism over the maximal unramified extension
  $E^{nr}_{\nu}$
  \begin{equation}
    \Sh_{\mathbf{K}}(G, h \delta^{-1})_{E_{\nu}} \times_{\Spec E_{\nu}} \Spec E^{nr}_{\nu} 
    \isoarrow
    \mathcal{A}^{t}_{\mathbf{K}, E_{\nu}} \times_{\Spec E_{\nu}} \Spec E^{nr}_{\nu}, 
  \end{equation}
  which for varying $\mathbf{K}$ is compatible with the Hecke operators in
  $G(\mathbb{A}_f^p)$.
  
  Let $\tau \in \Gal(E_{\nu}^{nr}/E_{\nu})$ be the Frobenius automorphism.
  Let $f_{\nu}$ be the inertia index of $E_{\nu}/\mathbb{Q}_p$. Then the
  following diagram is commutative,
  \begin{displaymath}
   \xymatrix{
  \Sh_{\mathbf{K}}(G, h \delta^{-1})_{E_{\nu}}\times_{\Spec E_{\nu}}\Spec E^{nr}_{\nu}  
  \ar[d]_{\id \times \tau_c} \ar[r]
  & \mathcal{A}^{t}_{\mathbf{K}, E_{\nu}} \times_{\Spec E_{\nu}}\Spec E^{nr}_{\nu}
  \ar[d]^{{p^{f_{\nu}}}_{\mid \xi} \times \tau_c}\\
  \Sh_{\mathbf{K}}(G, h \delta^{-1})_{E_{\nu}}\times_{\Spec E_{\nu}}\Spec E^{nr}_{\nu}  
  \ar[r] & 
  \mathcal{A}^{t}_{\mathbf{K}, E_{\nu}}\times_{\Spec E_{\nu}}\Spec E^{nr}_{\nu} .
   }
    \end{displaymath}
  {\rm (See (\ref{xioperator1e}) for the notation ${p^{f_{\nu}}}_{\mid \xi}$).}  
  \end{proposition}  
\begin{proof}
  Let $e \in E_{\nu}^{\times}$ and let $\sigma \in \Gal(E^{ab}_{\nu}/E_{\nu})$ be the
  automorphism that corresponds to it by local class field theory. We set
  $\varsigma_{p^{\infty}}(e) = \varsigma_{p^{\infty}}(\sigma)$. Then we obtain from
  local class field theory 
  \begin{equation}\label{formLFT}
\varsigma_{p^{\infty}}(e) = (\Nm_{E_{\nu}/\mathbb{Q}_p} e)^{-1} p^{f_{\nu} \ord e} ,
  \end{equation}
  where $\ord=\ord_\nu\colon E_{\nu}^{\times} \rightarrow \mathbb{Z}$ maps a uniformizer
  to $1$. Indeed, this formula makes sense for an arbitrary $p$-adic local
  field $E_{\nu}$. In the case $E_{\nu} = \mathbb{Q}_p$ (\ref{formLFT}) follows
  from \cite{CF} Chapt. VI, Thm. 3.2. In the general case the action of
  $\sigma$ on $\mu_{p^{\infty}}$ depends only on the restriction of $\sigma$ to
  $\mathbb{Q}^{\ab}$. But this restriction corresponds to
  $\Nm_{E_{\nu}/\mathbb{Q}_p} e$ by the reciprocity law of the local field
  $\mathbb{Q}_p$ by the last diagram of \cite{CF} Chapt. VI, \S 2.4. This
  shows (\ref{formLFT}).

  It is enough to show that the two squares in the following diagram are
  commutative. 
  \begin{displaymath}
\xymatrix{
  \Sh_{\mathbf{K}}(G, h \delta^{-1}) \times \Spec E^{ab}_{\nu}  
  \ar[d]_{\id \times \sigma_c} \ar[r] &
  \mathcal{A}_{\mathbf{K}} \times \Spec E^{ab}_{\nu}
  \ar[d]_{(\varsigma_{p^{\infty}}(e^{-1}) p^{f_{\nu} \ord e})_{\mid \xi} \times \sigma_c} 
  & \mathcal{A}^{t}_{\mathbf{K}, E_{\nu}} \times \Spec E^{ab}_{\nu}
  \ar[d]^{(p^{f_{\nu} \ord e})_{\mid \xi} \times \sigma_c} \ar[l]  \\
  \Sh_{\mathbf{K}}(G, h \delta^{-1}) \times \Spec E^{ab}_{\nu}  
  \ar[r] & \mathcal{A}_{\mathbf{K}} \times \Spec E^{ab}_{\nu}
  & \mathcal{A}^{t}_{\mathbf{K}, E_{\nu}} \times \Spec E^{ab}_{\nu} \ar[l]
   }
  \end{displaymath}
  The commutativity of the diagram on the right hand side  follows from Proposition
  \ref{Gtwist1p}. Since $\mathcal{A}_{\mathbf{K}} \cong \Sh_{\mathbf{K}}(G,h)$, the
  square on the left hand side becomes commutative if we replace the vertical
  arrow in the middle by
   \begin{equation*}  
    r_{\nu}^{\rm cft}(T, \delta^{-1})(\sigma) \times \sigma_c\colon 
    \Sh_{\mathbf{K}}(G,h)_{E_{\nu}} \times_{\Spec E_{\nu}} \Spec E_{\nu}^{ab} \rightarrow 
    \Sh_{\mathbf{K}}(G,h)_{E_{\nu}} \times_{\Spec E_{\nu}} \Spec E_{\nu}^{ab}.
    \end{equation*}
  This is a consequence of Corollary \ref{zentralerTwist1c}. It remains to
  compute the class field theory version $r_{\nu}^{\rm cft}(T, \delta^{-1})$ of the reciprocity law.

  The morphism $\mu_{\delta^{-1}}$ is the inverse of (\ref{zentralerTwist1e}).
  Therefore we find the local reciprocity law
  \begin{equation}\label{rec-delta1e}
    \begin{array}{rcc}
      r_{\nu}(T, \delta^{-1}): E_{\nu}^{\times} & \rightarrow & \hspace{-3cm}
      \prod_{i=0}^s 
    K^\times_{\mathfrak{q}_i} \times \prod_{i=0}^s K^\times_{\bar{\mathfrak{q}}_i}\\[2mm] 
    e & \mapsto  &  (1, \ldots, 1)  \times 
    (\Nm_{E_{\nu}/\mathbb{Q}_p} e, \ldots, \Nm_{E_{\nu}/\mathbb{Q}_p} e). 
      \end{array}    
  \end{equation}
  We write
  $\Nm_{E_{\nu}/\mathbb{Q}_p} e = \varsigma_{p^{\infty}}(e^{-1}) p^{f \ord e}$, cf. \eqref{formLFT}.
  Under the isomorphism (\ref{BZGQ_p2e}), the element on the right hand side of
  (\ref{rec-delta1e}) corresponds to
  $(1,\ldots, 1,\varsigma_{p^{\infty}}(e^{-1}) p^{f \ord e})\in G_{\mathfrak{q}_0} \times \ldots \times G_{\mathfrak{q}_s} \times \mathbb{Q}_p^{\times}$. By the description
  of the Hecke operators (\ref{Heckeoperator1e}) we obtain the proposition. 
\end{proof}
We will next show how the action of $p^{f_{\nu}}_{\mid \xi}$ on
$\mathcal{A}^{t}_{\mathbf{K}, E_{\nu}}$ extends naturally to the model
$\widetilde{\mathcal{A}}^t_{\mathbf{K}}$ over $O_{E_{\nu}}$. It is enough
to define $p_{\mid \xi}$. Let 
\begin{equation}\label{wAt1e}  
  (A, \iota, \bar{\lambda}, \bar{\eta}^p, (\bar{\eta}_{\mathfrak{q}_j})_j, \xi_p) \in
  \widetilde{\mathcal{A}}^t_{\mathbf{K}}(S)
  \end{equation}
be a point as in Definition \ref{BZsAtK2d}. Let
$\bar{\lambda}_{\mathbb{Q}}$ be the $\mathbb{Q}$-homogeneous polarization
which contains $\bar{\lambda}$. We extend $\xi_p$ to a map
$\xi_p: \bar{\lambda}_{\mathbb{Q}} \rightarrow \mathbb{Q}_p^{\times}/\mathbf{M}$ such
that $\xi_p(u\lambda) = u \xi_p(\lambda)$ for $u \in \mathbb{Q}^{\times}$. 
Let $X = \prod_{i=0}^{s} (X_{\mathfrak{q}_i} \times X_{\bar{\mathfrak{q}}_i})$ be the
$p$-divisible group of $A$. We consider the isogeny 
\begin{equation}\label{Gtwist6e}
a: \prod_{i=0}^{s} (X_{\mathfrak{q}_i} \times X_{\bar{\mathfrak{q}}_i}) \rightarrow  
\prod_{i=0}^{s} (X_{\mathfrak{q}_i} \times X_{\bar{\mathfrak{q}}_i}) 
\end{equation}
which is the identity on the factors $X_{\mathfrak{q}_i}$ and multiplication
by $p$ on the factors $X_{\bar{\mathfrak{q}}_i}$. Let us fix a polarization
$\lambda \in \bar{\lambda}$. Since $\lambda$ is principal in $p$ it is given
on $X$ by an isomorphism of $X_{\bar{\mathfrak{q}}_i}$ with the dual of
$X_{\mathfrak{q}_i}$ for each $i = 0,\ldots,s$. The inverse
image of $\lambda$ on $X$ by (\ref{Gtwist6e}) is $p \lambda$. 
There is an isogeny of abelian varieties with an $O_{B, (p)}$-action
\begin{displaymath}
\alpha: A' \rightarrow A,  
\end{displaymath}
of order a power of $p$ such that the induced homomorphism of $p$-divisible
groups is isomorphic to (\ref{Gtwist6e}). A polarization $\lambda$ of
$A$  induces a polarization
$\lambda' = \alpha^*(\lambda) := \hat{\alpha} \lambda \alpha$ on $A'$.
This defines a
$\mathbb{Q}$-homogeneous polarization $\bar{\lambda'}_{\mathbb{Q}}$ of $A'$ and
a bijection $\bar{\lambda}_{\mathbb{Q}} \rightarrow \bar{\lambda'}_{\mathbb{Q}}$.
By this bijection $\xi_p$ induces
\begin{displaymath}
\xi'_p: \bar{\lambda'}_{\mathbb{Q}} \rightarrow \mathbb{Q}_p^{\times}/\mathbf{M}.  
\end{displaymath}
We obtain $\xi'_p(\lambda') = \xi_p(\lambda) \in \mathbb{Z}_p^{\times}$. 
But the polarization $\lambda_1 = (1/p) \lambda'$ is principal in $p$, as
we see by looking at the $p$-divisible groups (\ref{Gtwist6e}). Then
\begin{equation}\label{wAt2e}
(A', \iota', \bar{\lambda_1}, \bar{\eta}'^p, (\bar{\eta}'_{\mathfrak{q}_j})_j, p\xi'_p)
  \end{equation}
is a point of $\tilde{\mathcal{A}}^t_{\mathbf{K}}(S)$. Here $\eta'^p$, resp. ${\eta'_{\mathfrak{q}_j}}$,
denotes the composite
\begin{displaymath}
  \eta'^p: V \otimes \mathbb{A}^p_f \overset{\eta^p}{\longrightarrow}
  \mathrm{V}^p_f(A) \overset{\alpha^{-1}}{\longrightarrow} \mathrm{V}^p_f(A') ,
  \end{displaymath}
resp. 
\begin{displaymath}
  \eta'_{\mathfrak{q}_j}: \Lambda_{\mathfrak{q}_j}
  \overset{\eta_{\mathfrak{q}_j}}{\longrightarrow} T_{\mathfrak{q}_j}(A) 
  \overset{\alpha^{-1}}{\longrightarrow} T_{\mathfrak{q}_j}(A').
\end{displaymath}
Note that the last arrow  is an isomorphism
by definition of $\alpha$. 
The map which assigns (\ref{wAt2e}) to (\ref{wAt1e}) defines the desired
extension of the operator $p_{\mid \xi}$ to $\tilde{\mathcal{A}}^t_{\mathbf{K}}$.

\begin{corollary}
  The Shimura variety $\Sh_{\mathbf{K}}(G, h \delta^{-1})_{E_{\nu}}$ has  
  a unique model $\widetilde{\Sh}_{\mathbf{K}}(G, h \delta^{-1})$ over
  $O_{E_{\nu}}$ with the following properties.
There is an isomorphism
 \begin{equation}
   \widetilde{\Sh}_{\mathbf{K}}(G, h \delta^{-1}) \times_{\Spec O_{E_{\nu}}}
   \Spec O_{E^{nr}_{\nu}} 
    \isoarrow
     \widetilde{\mathcal{A}}^t_{\mathbf{K}} \times_{\Spec O_{E_{\nu}}}
   \Spec O_{E^{nr}_{\nu}} 
  \end{equation}
 which is compatible with the Hecke operators $G(\mathbb{A}_f^p)$. Let $\tau \in \Gal(E_{\nu}^{nr}/E_{\nu})$ be the Frobenius automorphism.
  Let $f_{\nu}$ be the inertia index of $E_{\nu}/\mathbb{Q}_p$. Then the
  following diagram is commutative. 
  \begin{displaymath}
   \xymatrix{
  \widetilde{\Sh}_{\mathbf{K}}(G, h \delta^{-1}) \times_{\Spec O_{E_{\nu}}}
   \Spec O_{E^{nr}_{\nu}} \ar[d]_{\id \times \tau_c} \ar[r]
  & \widetilde{\mathcal{A}}^t_{\mathbf{K}} \times_{\Spec O_{E_{\nu}}}
   \Spec O_{E^{nr}_{\nu}} \ar[d]^{(p^{f_{\nu}})_{\mid \xi} \times \tau_c}\\
  \widetilde{\Sh}_{\mathbf{K}}(G, h \delta^{-1}) \times_{\Spec O_{E_{\nu}}}
  \Spec O_{E^{nr}_{\nu}} \ar[r] & \widetilde{\mathcal{A}}^t_{\mathbf{K}}
  \times_{\Spec O_{E_{\nu}}} \Spec O_{E^{nr}_{\nu}} 
   }
    \end{displaymath}\qed
  \end{corollary}
 The homomorphism $h \delta^{-1}$ may be written as follows. It factors as 
  \begin{displaymath}
  \mathbb{S} \longrightarrow 
   \prod_{\chi \in \Hom_{\mathbb{Q}{\rm -Alg}}(F, \mathbb{C})}
  (D \otimes_{F, \chi} \mathbb{R})^{\times}
  \times (K \otimes_{F, \chi} \mathbb{R})^{\times}
   \overset{\kappa}{\longrightarrow} G^{\bullet}(\mathbb{R}),  
    \end{displaymath}
The image of $z = a + b{\mathbf i} \in \mathbb{C} = \mathbb{S}(\mathbb{R})$ by the first
arrow is given by
\begin{displaymath}
((\left(
  \begin{array}{rr} 
    a & -b\\
    b & a
  \end{array}
  \right)  \times z^{-1}), 1, \ldots ,1), 
  \end{displaymath}
where the first entry in this vector is at the place $\chi_0$.

\section{The moduli problem for $\Sh(G^{\bullet}, h)$ and a reduction modulo $p$}\label{s:AGbullet}

We recall the groups $G$ of \eqref{BZh3e} and $G^\bullet$ of (\ref{Gpunkt1e}). We consider the map of Shimura data $(G, h) \rightarrow (G^{\bullet} , h)$, cf.
\eqref{BZh2e} and (\ref{hforbul}). The Shimura fields coincide, i.e. $E = E(G,h) = E(G^{\bullet},h)$.
We consider a pair of open compact subgroups
$\mathbf{K} \subset G(\mathbb{A}_f)$ and
$\mathbf{K}^{\bullet} \subset G^{\bullet}(\mathbb{A}_f)$ such that
$\mathbf{K} \subset \mathbf{K}^{\bullet}$ and such that the induced map
$\Sh(G,h)_{\mathbf{K}} \rightarrow \Sh(G^{\bullet}, h)_{\mathbf{K}^{\bullet}}$ is a closed
immersion, cf. \cite[Prop. 1.15.]{D-TS}.

The Shimura variety $\Sh(G^{\bullet}, h)_{\mathbf{K}^{\bullet}}$ is the coarse moduli scheme associated to
the following functor. 
\begin{definition}\label{BZApkt3d}
  Let $\mathbf{K}^{\bullet} \subset G^{\bullet}(\mathbb{A}_f)$ be an open  compact
  subgroup. We define a functor $\mathcal{A}_{\mathbf{K}^{\bullet}}^{\bullet}$ on the
  category of $E$-schemes. A point of
$\mathcal{A}^{\bullet}_{\mathbf{K}^{\bullet}}(S)$ is given by the following data:
\begin{enumerate} 
\item[(a)] An abelian scheme $A$ over $S$ up to isogeny with an action
  $\iota: B \rightarrow \End^o A$.
\item[(b)] A $F$-homogeneous polarization $\bar{\lambda}$ of $A$
  which induces on $B$ the involution $b \mapsto b^{\star}$.
\item[(c)] A class $\bar{\eta}\; \text{modulo}\; \mathbf{K}^{\bullet}$ of
  $B \otimes \mathbb{A}_f$-module isomorphisms
  \begin{displaymath}
\eta: V \otimes \mathbb{A}_f \isoarrow \mathrm{V}_f(A) 
  \end{displaymath}
  such that for each $\lambda \in \bar{\lambda}$ there is
  locally for the Zariski-topology on $S$ a constant
  $\xi \in (F \otimes_{\mathbb{Q}} \mathbb{A}_f)(1)$ with
  \begin{displaymath}
 \psi(\xi v_1, v_2) = E^{\lambda}(\eta(v_1), \eta(v_2)). 
  \end{displaymath}
  \end{enumerate}
We require that the following condition ${\rm (KC)}$ holds
  \begin{displaymath}
    {\rm char} (T, \iota(b) \mid \Lie A) = \prod_{\varphi: K \rightarrow \bar{\mathbb{Q}}}
    \varphi(\Nm^o_{B/K} (T -b))^{r_{\varphi}} . 
    \end{displaymath}

\end{definition}

We will reformulate this moduli problem in a way that makes sense over $O_{E_{\nu}}$.
As in the previous section,  we need additional requirements for the group
$\mathbf{K}^{\bullet} \subset G^{\bullet}(\mathbb{A}_f)$, similar to those discussed 
before (\ref{BZ7e}). We assume that
$\mathbf{K}^{\bullet} = \mathbf{K}^{\bullet}_p \mathbf{K}^{\bullet, p}$ where
$\mathbf{K}^{\bullet}_p \subset G(\mathbb{Q}_p)$. The decomposition
$V \otimes \mathbb{Q}_p = \oplus_{i=0}^{s} V_{\mathfrak{p}_i}$ is
orthogonal with respect to $\psi$. 
We obtain that
\begin{displaymath}
G^{\bullet}(\mathbb{Q}_p) = \prod_{i=0}^{s} G^{\bullet}_{\mathfrak{p}_i}, 
  \end{displaymath}
where 
\begin{equation}\label{BZGpi1e}
  G^{\bullet}_{\mathfrak{p}_i} =  \{g \in
  \Aut_{B_{\mathfrak{p}_i}} V_{\mathfrak{p}_i} \; | \;
  \psi_{\mathfrak{p}_i}(g v, g w) =\psi_{\mathfrak{p}_i}(\mu_{\mathfrak{p}_i}(g) v, w),
  \;   \mu_{\mathfrak{p}_i}(g) \in F^{\times}_{\mathfrak{p}_i}\}. 
\end{equation}
This defines the homomorphism
$\mu_{\mathfrak{p}_i}: G^{\bullet}_{\mathfrak{p}_i} \rightarrow F^{\times}_{\mathfrak{p}_i}$. 
According to the decomposition
$V \otimes_{F} F_{\mathfrak{p}_i} = V_{\mathfrak{q}_i} \oplus V_{\bar{\mathfrak{q}}_i}$   
we can write $g = (g_{\mathfrak{q}_i}, g_{\bar{\mathfrak{q}}_i})$. Then we obtain
\begin{equation}\label{BZGpi2e} 
  G^{\bullet}_{\mathfrak{p}_i} = \{g = (g_{\mathfrak{q}_i}, g_{\bar{\mathfrak{q}}_i}), \; 
  g_{\mathfrak{q}_i} \in \Aut_{B_{\mathfrak{q}_i}} \! V_{\mathfrak{q}_i}, \, 
  g_{\bar{\mathfrak{q}}_i} \in \Aut_{B_{\bar{\mathfrak{q}}_i}} \! V_{\bar{\mathfrak{q}}_i}  \;
  | \;
  g'_{\mathfrak{q}_i} g_{\bar{\mathfrak{q}}_i} \in F^{\times}_{\mathfrak{p}_i} \} ,
\end{equation}
cf. (\ref{BZGQ_p4e}). We note that
$\mu_{\mathfrak{p}_i}(g) = g'_{\mathfrak{q}_i} g_{\bar{\mathfrak{q}}_i}$. By this
equation we obtain an isomorphism of groups
\begin{equation}\label{BZGdot1e}
  \begin{array}{cccr} 
    G^{\bullet}_{\mathfrak{p}_i} & \cong & G^{\bullet}_{\mathfrak{q}_i} \times
    F^{\times}_{\mathfrak{p}_i}, & \quad \; G^{\bullet}_{\mathfrak{q}_i}
    = \Aut_{B_{\mathfrak{q}_i}} \! V_{\mathfrak{q}_i} ,\\
    g & \mapsto & g_{\mathfrak{q}_i} \times \mu_{\mathfrak{p}_i}(g) & 
    \end{array}
\end{equation}
cf. (\ref{BZGQ_p5e}). Altogether we obtain an isomorphism
\begin{equation}\label{BZGdot2e}
  G^{\bullet}(\mathbb{Q}_p) = \prod_{i=0}^{s} G^{\bullet}_{\mathfrak{q}_i} \times
  \prod_{i=0}^{s} F^{\times}_{\mathfrak{p}_i}.  
\end{equation}
We use the notations $\Lambda_{\mathfrak{q}_i}$, $\Lambda_{\bar{\mathfrak{q}}_i}$,
$\Lambda_{\mathfrak{p}_i}$, $O_{B_{\mathfrak{q}_i}}$, $O_{B_{\bar{\mathfrak{q}}_i}}$, and
$O_{B_{\mathfrak{p}_i}} = O_{B_{\mathfrak{q}_i}} \oplus O_{B_{\bar{\mathfrak{q}}_i}}$
as before \eqref{BZLambda1e}. 
For each prime $\mathfrak{p}_i$, $i= 0, \ldots, s$, we choose an open subgroup
$\mathbf{M}^{\bullet}_{\mathfrak{p}_i} \subset O_{F_{\mathfrak{p}_i}}^{\times}$.
We set 
\begin{equation}\label{BZ20e}
  \mathbf{M}^{\bullet} = \prod_{i=0}^{s} \mathbf{M}^{\bullet}_{\mathfrak{p}_i} \subset
  \prod_{i=0}^{s} O_{F_{\mathfrak{p}_i}}^{\times} = (O_F \otimes \mathbb{Z}_p)^{\times}. 
\end{equation}

As in the previous section (see right after \eqref{BZLambda1e}), we set 
$\mathbf{K}_{\mathfrak{q}_0} = \Aut_{O_{B_{\mathfrak{q}_0}}} \Lambda_{\mathfrak{q}_0}$, and  
 choose for $j > 0$ arbitrarily open compact subgroups  
$\mathbf{K}_{\mathfrak{q}_j}\subset\Aut_{O_{B_{\mathfrak{q}_j}}}\Lambda_{\mathfrak{q}_j}$.
We define, for $i=0,\ldots,s$, 
\begin{displaymath}
  \mathbf{K}_{\mathfrak{p}_i} = \mathbf{K}^{\bullet}_{\mathfrak{q}_i} \times
  \mathbf{M}^{\bullet}_{\mathfrak{p}_i} \subset G^\bullet_{\mathfrak{q}_i} \times
  F^{\times}_{\mathfrak{p}_i} \cong G_{\mathfrak{p}_i} ,
\end{displaymath}
cf. (\ref{BZGdot1e}). We obtain 
\begin{displaymath}
  \begin{array}{ll}
  \mathbf{K}^{\bullet}_{\mathfrak{p}_i} = & 
  \{g = (g_{\mathfrak{q}_i}, g_{\bar{\mathfrak{q}}_i}), \; 
  g_{\mathfrak{q}_i} \in \mathbf{K}^{\bullet}_{\mathfrak{q}_i} ,\; 
  g_{\bar{\mathfrak{q}}_i} \in \Aut_{O_{B_{\bar{\mathfrak{q}}_i}}} \! \Lambda_{\bar{\mathfrak{q}}_i} 
  \; | \; \\ & 
  \psi_{\mathfrak{p}_i}(g_{\mathfrak{q}_i}v_1, g_{\bar{\mathfrak{q}}_i} v_2) =
  \psi_{\mathfrak{p}_i}(mv_1, v_2), \; v_1 \in \Lambda_{\mathfrak{q}_i}, \,
  v_2\in \Lambda_{\bar{\mathfrak{q}}_i}, \, m\in \mathbf{M}^{\bullet}_{\mathfrak{p}_i} \},  
    \end{array}
\end{displaymath}
and in particular 
\begin{equation*}
  \mathbf{K}^{\bullet}_{\mathfrak{p}_0} = \{g \in
\Aut_{O_{B_{\mathfrak{p}_0}}} \Lambda_{\mathfrak{p}_0} \; | \;
\psi_{\mathfrak{p}_0}(gv_1, gv_2) =  \psi_{\mathfrak{p}_0}(\mu(g)v_1, v_2), \;
\mu(g) \in \mathbf{M}^\bullet_{\mathfrak{p}_0} \}. 
\end{equation*}
Finally we set
\begin{equation}\label{BZKpPkt1e}  
  \mathbf{K}^{\bullet}_p = \prod_{i=0}^{s} \mathbf{K}^{\bullet}_{\mathfrak{p}_i} =
  \prod_{i=0}^{s} (\mathbf{K}_{\mathfrak{q}_i} \times
  \mathbf{M}^{\bullet}_{\mathfrak{p}_i})  .
  \end{equation}
We choose $\mathbf{K}^{\bullet, p}$ arbitrarily and set
\begin{equation}\label{BZ14e}
\mathbf{K}^{\bullet} = \mathbf{K}^{\bullet}_p \mathbf{K}^{\bullet, p}.
  \end{equation}
With these assumptions on $\mathbf{K}^{\bullet}$ we can rewrite the definition
of the functor $\mathcal{A}^{\bullet}_{\mathbf{K}^{\bullet}}$ as follows.
\begin{definition}\label{BZApkt3altd}  (alternative of Definition \ref{BZApkt3d})
Let $\mathbf{K}^{\bullet}=\mathbf{K}^{\bullet}_p\mathbf{K}^{\bullet,p}\subset G^\bullet(\BA_f)$, with $\mathbf{K}^{\bullet}_p$ as in \eqref{BZKpPkt1e}. Let $S$ be an  
$E$-scheme. A point of $\mathcal{A}^{\bullet}_{\mathbf{K}^{\bullet}}(S)$ consists of the
following data: 
\begin{enumerate} 
\item[(a)] An abelian scheme $A$ over $S$ up to isogeny with an
  action $\iota: B \rightarrow \End^{o} A$.  

\item[(b)] A $F$-homogeneous polarization $\bar{\lambda}$ of $A$ which induces
  on $B$ the involution $b \mapsto b^{\star}$. 
  
\item[(c)] A class $\bar{\eta}^p$   modulo $\mathbf{K}^{\bullet, p}$
of $B \otimes \mathbb{A}^p_f$-module
  isomorphisms
  \begin{displaymath}
    \eta^p: V \otimes \mathbb{A}^p_f \isoarrow \mathrm{V}^p_f(A) ,
    \end{displaymath}
  such that for each $\lambda \in \bar{\lambda}$ there is locally for the
  Zariski topology a constant 
  $\xi^{(p)}(\lambda) \in (F \otimes \mathbb{A}^p_f)(1)$
  with
  \begin{displaymath}
\psi(\xi^{(p)}(\lambda) v_1, v_2) = E^{\lambda}(\eta^p(v_1), \eta^p(v_2)). 
  \end{displaymath}
\item[(d)]
  For each polarization $\lambda \in \bar{\lambda}$ and for each prime
  $\mathfrak{p}|p$ of $O_F$ a section
  $\xi_{\mathfrak{p}}(\lambda) \in F^{\times}_{\mathfrak{p}}(1)/\mathbf{M}_{\mathfrak{p}}^{\bullet}$
  such that $\xi_{\mathfrak{p}}(\lambda u) = u \xi_{\mathfrak{p}}(\lambda)$ for
  each $u \in F^{\times}$. 
\item[(e)]
  For each $i = 0, 1, \ldots, s$ a class $\bar{\eta}_{\mathfrak{q}_i}$   modulo $\mathbf{K}_{\mathfrak{q}_i}^{\bullet}$
  of $B_{\mathfrak{q}_i}$-module isomorphisms 
  \begin{displaymath}
    \eta_{\mathfrak{q}_i}: V_{\mathfrak{q}_i} \isoarrow \mathrm{V}_{\mathfrak{q}_i}(A) .
  \end{displaymath} 
  \end{enumerate}
We require that the following condition ${\rm (KC)}$ holds,
  \begin{displaymath}
    {\rm char}(T, \iota(b) \mid \Lie A) = \prod_{\varphi: K \rightarrow \bar{\mathbb{Q}}}
    \varphi(\Nm^o_{B/K} (T -b))^{r_{\varphi}} . 
    \end{displaymath}

  \end{definition}
We write a point of this functor in the form
\begin{equation}\label{BZGdot4e}  
  (A, \iota, \bar{\lambda}, \bar{\eta}^p, (\bar{\eta}_{\mathfrak{q}_i})_i,
  (\xi_{\mathfrak{p}})_{\mathfrak{p}}) 
\end{equation}
or alternatively
$(A,\iota,\bar{\lambda},\bar{\eta}^p,(\bar{\eta}_{\mathfrak{q}_i})_i,(\xi_{\mathfrak{p}_i})_i)$, $i = 0, \ldots, s$. 

To make the relationship of the last two Definitions \ref{BZApkt3d} and
\ref{BZApkt3altd} explicit, we consider an $S$-valued point
$(A, \iota, \bar{\lambda}, \bar{\eta})$ of Definition \ref{BZApkt3d}.
We fix $\eta \in \bar{\eta}$. Then $\eta_p$ is an isomorphism
\begin{displaymath}
  \eta_p = \oplus_{i=0}^{s} \eta_{\mathfrak{p}_i}: \bigoplus_{i=0}^{s} V_{\mathfrak{p}_i} 
  \isoarrow \bigoplus_{i=0}^{s} \mathrm{V}_{\mathfrak{p}_i} (A).   
\end{displaymath}
The component
$\eta_{\mathfrak{p}_i} = \eta_{\mathfrak{q}_i} \oplus \eta_{\bar{\mathfrak{q}}_i}$
satisfies an equation
\begin{displaymath}
  \psi(\xi_{\mathfrak{p}_i}(\lambda) v, w) =
  E^{\lambda}(\eta_{\mathfrak{q}_i}(v), \eta_{\bar{\mathfrak{q}}_i}(w)), \quad
  v \in V_{\mathfrak{q}_i}, \; w \in V_{\bar{\mathfrak{q}}_i}. 
  \end{displaymath}
We see that the data $\eta_{\mathfrak{q}_i}$ and $\eta_{\bar{\mathfrak{q}}_i}$
for $i =0, \ldots, s$ determine the data $\eta_{\mathfrak{q}_i}$ and
$\xi_{\mathfrak{p}_i}(\lambda)$ and vice versa. Therefore we obtain the
$S$-valued point
$(A,\iota,\bar{\lambda},\bar{\eta}^p,(\bar{\eta}_{\mathfrak{q}_i})_i,(\xi_{\mathfrak{p}_i})_i)$
of (\ref{BZApkt3altd}).

Let $g \in G^{\bullet}(\mathbb{Q}_p)$. According to (\ref{BZGdot2e}) we write
it in the form $(\ldots, g_{\mathfrak{q}_i}, \ldots, a_{\mathfrak{p}_i}, \ldots)$,
with $a_{\mathfrak{p}_i} = \mu_{\mathfrak{p}_i}(g)$. Then the Hecke operator
$g: \mathcal{A}^{\bullet}_{\mathbf{K}^{\bullet}} \rightarrow \mathcal{A}^{\bullet}_{g^{-1}\mathbf{K}^{\bullet}g},$ 
 sends (\ref{BZGdot4e}) to the point
\begin{equation}\label{BZGdot5e}   
  (A,\iota,\bar{\lambda},\bar{\eta}^p, 
  (\overline{{\eta}_{\mathfrak{q}_i} g_{\mathfrak{q}_i}})_i,
  (a_{\mathfrak{p}_i}\xi_{\mathfrak{p}_i})_i),
  \end{equation}
comp. (\ref{Heckeoperator1e}).

It is convenient for us to introduce another action of 
$(F \otimes \mathbb{Q}_p)^{\times}$, 
\begin{equation}\label{xi-action1e} 
  a_{\mid \xi}: \mathcal{A}^{\bullet}_{\mathbf{K}^{\bullet}} \rightarrow
  \mathcal{A}^{\bullet}_{\mathbf{K}^{\bullet}}, \quad
  a \in (F \otimes \mathbb{Q}_p)^{\times}. 
\end{equation}
We write 
$a = (\ldots, a_{\mathfrak{p}}, \ldots) \in (F \otimes \mathbb{Q}_p)^{\times} = \prod_{\mathfrak{p}|p} F_{\mathfrak{p}}^{\times}$.
Then $a_{\mid \xi}$  maps a point (\ref{BZGdot4e}) to
\begin{displaymath}
 (A, \iota, \bar{\lambda}, \bar{\eta}^p, (\bar{\eta}_{\mathfrak{q}_i})_i,
  (a_{\mathfrak{p}}\xi_{\mathfrak{p}})_{\mathfrak{p}}) .
  \end{displaymath} 
For a fixed $\mathfrak{p}$ let $a_{\mathfrak{p}} \in F_{\mathfrak{p}}^{\ast}$. The we
define ${a_{\mathfrak{p}}}_{\mid \xi_{\mathfrak{p}}} := a_{\mid \xi}$ where $a$ is the
element with component $a_{\mathfrak{p}}$ at $\mathfrak{p}$ and with
$a_{\mathfrak{p}'} = 1$ for $\mathfrak{p}' \neq \mathfrak{p}$. 

The action of $a_{\mid \xi}$ of (\ref{xi-action1e}) coincides with the
action of the Hecke operator
$g = (\ldots, g_{\mathfrak{q}_i}, g_{\bar{\mathfrak{q}}_i}. \ldots) \in G^{\bullet}(\mathbb{Q}_p)$,
where $g_{\mathfrak{q}_i} = 1$ and
$g_{\bar{\mathfrak{q}}_i}=a_{\mathfrak{p}_i}\in F^\times_{\mathfrak{p}_i}\cong K^\times_{\bar{\mathfrak{q}}_i}$
for $i = 0, \ldots, s$.  

We will denote by $U_p(F) \subset F^{\times}$ the set of all $a \in F^{\times}$
  which are units in all fields $F_{\mathfrak{p}}$ with $\mathfrak{p} | p$.
\begin{definition}\label{BZApkt4d}
Let $\mathbf{K}^{\bullet}=\mathbf{K}^{\bullet}_p\mathbf{K}^{\bullet,p}\subset G^\bullet(\BA_f)$, with $\mathbf{K}^{\bullet}_p$ as in \eqref{BZKpPkt1e}.  We define a functor$\mathcal{A}^{\bullet bis}_{\mathbf{K}^{\bullet}}$ on the category of
$E$-schemes $S$. A point of $\mathcal{A}^{\bullet bis}_{\mathbf{K}^{\bullet}}(S)$ consists of the
following data: 
\begin{enumerate} 
\item[(a)] An abelian scheme $A$ over $S$ up to isogeny prime to $p$ with an
  action
  $\iota: O_{B,(p)} \rightarrow \End A \otimes_{\mathbb{Z}} \mathbb{Z}_{(p)}$.

\item[(b)] 
  An $U_p(F)$-homogeneous polarization $\bar{\lambda}$ of $A$ which is
  principal in $p$ and which induces on $B$ the involution $b\mapsto b^\star$.

\item[(c)]
  A class $\bar{\eta}^p$    modulo $\mathbf{K}^{\bullet,p}$
of
  $B \otimes \mathbb{A}^p_f$-module isomorphisms
  \begin{displaymath}
    \eta^p: V \otimes \mathbb{A}^p_f \isoarrow \mathrm{V}^p_f(A) , \;
   \end{displaymath}
  such that for each $\lambda \in \bar{\lambda}$ there is locally for the
  Zariski topology a constant 
  $\xi^{(p)}(\lambda) \in (F \otimes \mathbb{A}^p_f)^{\times}(1)$
  with
  \begin{displaymath}
\psi(\xi^{(p)}(\lambda) v_1, v_2) = E^{\lambda}(\eta^p(v_1), \eta^p(v_2)). 
  \end{displaymath}
\item[(d)]
  For each polarization $\lambda \in \bar{\lambda}$ and for each prime
  $\mathfrak{p}|p$ of $O_F$ a section
  $\xi_{\mathfrak{p}}(\lambda) \in O^{\times}_{F_{\mathfrak{p}}}(1)/\mathbf{M}_{\mathfrak{p}}^{\bullet}$
   such that
  $\xi_{\mathfrak{p}}(\lambda u) = u \xi_{\mathfrak{p}}(\lambda)$ for
  each $u \in U_p(F)$. 
\item[(e)]
  For each $j= 1, \ldots s$, a class $\bar{\eta}_{\mathfrak{q}_j}$   modulo $\mathbf{K}^{\bullet}_{\mathfrak{q}_j}$
  of $O_{B_{\mathfrak{q}_j}}$-module isomorphisms 
  \begin{displaymath}
    \eta_{\mathfrak{q}_j}: \Lambda_{\mathfrak{q}_j} \isoarrow T_{\mathfrak{q}_j}(A) .\;
    \end{displaymath} 
  \end{enumerate}
 We require that the following condition ${\rm (KC)}$ holds,
  \begin{displaymath}
    {\rm char}(T, \iota(b) \mid \Lie A) = \prod_{\varphi: K \rightarrow \bar{\mathbb{Q}}}
    \varphi(\Nm^o_{B/K} (T -b))^{r_{\varphi}} . 
    \end{displaymath}
\end{definition}
\begin{variant}\label{varia}
We will also use a modified version of this Definition. We obtain a
functor isomorphic to $\mathcal{A}^{\bullet bis}_{\mathbf{K}^{\bullet}}$ if we modify
the items $(b)$ and $(d)$ of Definition \ref{BZApkt4d} as follows.

\begin{enumerate}
\item[($b^\prime$)] An $F$-homogeneous polarization $\bar{\lambda}$ on $A$ which
  induces on $B$ the involution $\star$ from \eqref{defstar}. 

\item[($d^\prime$)] For each polarization $\lambda \in \bar{\lambda}$ and for
  each prime $\mathfrak{p}|p$ of $O_F$ a section
  $\xi_{\mathfrak{p}}(\lambda) \in F^{\times}_{\mathfrak{p}}(1)/\mathbf{M}_{\mathfrak{p}}^{\bullet}$ such that
  $\xi_{\mathfrak{p}}(u \lambda) = u \xi_{\mathfrak{p}}(\lambda)$ for
each $u \in F^{\times}$ and such that $\lambda$ is principal in $\mathfrak{p}$
iff $\xi_{\mathfrak{p}}(\lambda) \in O_{F^{\times}_{\mathfrak{p}}}^{\times}(1)$.  
\end{enumerate}
\end{variant}
\begin{proposition}\label{BZ11p} 
Let $\mathbf{K}^{\bullet}=\mathbf{K}^{\bullet}_p\mathbf{K}^{\bullet,p}\subset G^\bullet(\BA_f)$, with $\mathbf{K}^{\bullet}_p$ as in \eqref{BZKpPkt1e}.  The 
functors $\mathcal{A}^{\bullet}_{\mathbf{K}^{\bullet}}$ and
$\mathcal{A}^{\bullet bis}_{\mathbf{K}^{\bullet}}$ on the category of
$E$-schemes are canonically isomorphic. 
\end{proposition}
\begin{proof}
  The proof is an obvious modification of the proof of Proposition \ref{BZ1p}.
  But for later use we indicate the point of
  $\mathcal{A}^{\bullet}_{\mathbf{K}^{\bullet}}(S)$ which corresponds to a point
$(A_0,\bar{\lambda}_0, \bar{\eta}^p, (\bar{\eta}_{\mathfrak{q}_j})_j, (\xi_{\mathfrak{p_i}})_i)$ 
  of $\mathcal{A}^{\bullet bis}_{\mathbf{K}^{\bullet}}(S)$
  (recall from Definition \ref{BZsAK2d} that the index $i$ runs from
  $0,\ldots, s$ and the index $j$ from $1,\ldots, s$). Since we work over $E$,
  the Tate module $T_{\mathfrak{q}_0}(A)$ makes sense. By
  our choice of $\mathbf{K}^{\bullet}_{\mathfrak{q}_0}$ there is a unique class of 
  $O_{B_{\mathfrak{q}_0}}$-module isomorphisms
  $\eta_{\mathfrak{q}_0}: \Lambda_{\mathfrak{q}_0} \rightarrow T_{\mathfrak{q}_0}(A)$
  modulo $\mathbf{K}^{\bullet}_{\mathfrak{q}_0}$. Therefore we may replace $j$ in
  datum $(e)$ by  $i = 0, \ldots s$ in Definition \ref{BZApkt4d},
    without changing anything.
  
  Let $A$ be the isogeny
  class of $A_0$ and  choose $\lambda_0 \in \bar{\lambda}_0$ and
  $\eta^{p} \in \bar{\eta}^p$ to construct a point of
  $\mathcal{A}^{\bullet}_{\mathbf{K}^{\bullet}}(S)$.
For $i = 0, \ldots, s$, the isomorphisms
  $\eta_{\mathfrak{q}_i}: \Lambda_{\mathfrak{q}_i} \rightarrow T_{\mathfrak{q}_i}(A_0)$
  induce by duality (using $\psi$ and $E^{\lambda_0}$) an isomorphism
  $T_{\bar{\mathfrak{q}}_i}(A_0) \rightarrow \Lambda_{\bar{\mathfrak{q}}_i}(1)$.
  If we multiply the inverse of this map by $\xi_{\mathfrak{p}_i}(\lambda_0)$
  we obtain an isomorphism
  \begin{displaymath}
    \eta_{\bar{\mathfrak{q}}_i}: 
    \Lambda_{\bar{\mathfrak{q}}_i} \isoarrow T_{\bar{\mathfrak{q}}_i}(A_0) 
    \end{displaymath}
  which satisfies
  \begin{equation}\label{BZ17S1e}
E^{\lambda_0}(\eta_{\mathfrak{q}_i}(v), \eta_{\bar{\mathfrak{q}}_i}(w))  
= \psi(\xi_{\mathfrak{p}_i}(\lambda_0) v, w), \quad v \in \Lambda_{\mathfrak{q}_i},
\; w \in \Lambda_{\bar{\mathfrak{q}}_i}. 
  \end{equation}
  We set $\eta_{\mathfrak{p}_i} = \eta_{\mathfrak{q}_i} \oplus \eta_{\bar{\mathfrak{q}}_i}$
  and $\eta_p = \oplus_{i=0}^{s} \eta_{\mathfrak{p}_i}$. We denote by $\bar{\eta}_p$
  the class modulo $\mathbf{K}_p$ of $\eta_p$, and  by $\lambda$ the
  $\mathbb{Q}$-homogeneous polarization which contains $\lambda_0$. Then
  $(A, \iota, \bar{\lambda}, \bar{\eta})$ is the corresponding point of 
  $\mathcal{A}^{\bullet}_{\mathbf{K}^{\bullet}}(S)$.
\end{proof}

We reformulate the action of the Hecke operator (\ref{BZGdot5e}) in terms
of Variant \ref{varia} of Definition \ref{BZApkt4d}. This will not be used
until section \ref{s:uniform}.  We consider an element 
$g \in G^{\bullet}(\mathbb{Q}_p) \subset G^{\bullet}(\mathbb{A}_f)$. We write
$g = (\ldots g_{\mathfrak{q}_i}, \bar{g}_{\mathfrak{q}_i},  \ldots)$ as before, where
$i = 0, \ldots, s$. 
We consider an open  compact subgroup
$\mathbf{K}^{\bullet}=\mathbf{K}^{\bullet}_p\mathbf{K}^{\bullet,p}\subset G^\bullet(\BA_f)$, with $\mathbf{K}^{\bullet}_p$
as in \eqref{BZKpPkt1e},  such that
$k g\Lambda_{\mathfrak{p}_i} =  g\Lambda_{\mathfrak{p}_i}$, for
$k \in \mathbf{K}^{\bullet}_{\mathfrak{p}_i}$.  

Let
\begin{equation}\label{Heckevar1e}
(A_0, \iota_0, \bar{\lambda}, \bar{\eta}^{p},(\bar{\eta}_{\mathfrak{q}_j})_j,
  (\xi_{\mathfrak{p_i}})_i) \in \mathcal{A}^{\bullet bis}_{\mathbf{K}^{\bullet}}(S) 
  \end{equation}
be a point of Variant \ref{varia}. We recall that there is a unique isomorphism
of $O_{B_{\mathfrak{q}_0}}$-modules 
$\eta_{\mathfrak{q}_0}: \Lambda_{\mathfrak{q}_0} \rightarrow T_{\mathfrak{q}_0}(A_0)$ 
modulo $\mathbf{K}^{\bullet}_{\mathfrak{p}_0}$. This defines the unique class
$\bar{\eta}_{\mathfrak{q}_0}$. First we use this class to describe the Hecke
operator $g$. 

A point 
\begin{equation}\label{RZ15e} 
  (A_1, \iota_1, \bar{\lambda}_1, \bar{\theta}^{p},(\bar{\theta}_{\mathfrak{q}_j})_j,
  (\xi'_{\mathfrak{p_i}})_i) \in \mathcal{A}^{\bullet bis}_{g^{-1}\mathbf{K}^{\bullet}g}(S)
\end{equation}
is the image of (\ref{Heckevar1e}) by the Hecke operator $g$ if the following
conditions are fulfilled. There exists a quasi-isogeny
\begin{equation}\label{Heckevar2e} 
\alpha: (A_1, \iota_1) \rightarrow (A_0, \iota_0)  
\end{equation}
such that 
\begin{equation}\label{RZ16e}
  \begin{array}{l} 
\alpha^{*}( \bar{\lambda}) = \bar{\lambda}_1, \quad 
\alpha^{*} (\bar{\eta}^p) = \bar{\theta}^p,\\[2mm]  
  \xi'_{\mathfrak{p_i}}(\alpha^{*}(\lambda)) =
  \mu_{\mathfrak{p_i}}(g) \xi_{\mathfrak{p_i}}(\lambda), \quad \text{for} \;
  \lambda \in \bar{\lambda}.
  \end{array}
\end{equation}
Moreover we require that the data $\bar{\theta}_{\mathfrak{q}_i}$ and
$\bar{\eta}_{\mathfrak{q}_i}$ for $i = 0, \ldots, s$ are connected by the
following diagrams 
\begin{equation}\label{Heckevar3e}
  \xymatrix{
    T_{\mathfrak{q}_i}(A_1) \otimes \mathbb{Q} \ar[r]^{\alpha} &
    T_{\mathfrak{q}_i}(A_0) \otimes \mathbb{Q}\\
     \Lambda_{\mathfrak{q}_i} \otimes \mathbb{Q} \ar[u]_{\theta_{\mathfrak{q}_i}} 
     \ar[r]_{g_{\mathfrak{q}_i}} &  \Lambda_{\mathfrak{q}_i} \otimes \mathbb{Q} , 
     \ar[u]_{\eta_{\mathfrak{q}_i}} \\ 
  }  
 \end{equation} 
where $\theta_{\mathfrak{q}_i} \in \bar{\theta}_{\mathfrak{q}_i}$, 
$\eta_{\mathfrak{q}_i} \in \bar{\eta}_{\mathfrak{q}_i}$. The diagrams are 
required to be commutative after replacing $\eta_{\mathfrak{q}_i}$ by
$\eta_{\mathfrak{q}_i} k$ for some $k \in \mathbf{K}^{\bullet}_{\mathfrak{q}_i}$.

We may reformulate the condition (\ref{Heckevar3e}) for $i=0$ without mentioning
the classes $\bar{\eta}_{\mathfrak{q}_0}$ and $\bar{\theta}_{\mathfrak{q}_0}$.
Recall that
$G^{\bullet}_{\mathfrak{q}_0}\cong (B^{\rm opp}_{\mathfrak{q}_0})^{\times}\cong D^{\times}_{\mathfrak{p}_0}$,
cf. (\ref{BZGdot1e}), (\ref{Dsplit-in0}).  
Let $m = \ord g_{\mathfrak{q}_0}$ be the valuation in the division algebra
$D_{\mathfrak{p}_0}$. Let $\Pi\in O_{D_{\mathfrak{q}_0}}$ be a prime element which 
we regard also as an  element of $O_{D^{opp}_{\mathfrak{q}_0}}$. We may define $m$ by
$g_{\mathfrak{q}_0} (\Lambda_{\mathfrak{q}_0}) = \Pi^{m} \Lambda_{\mathfrak{q}_0}$.
Let $X_0$ be the $p$-divisible group of $A_0$ and  $X_1$ the $p$-divisible
group of $A_1$. The condition (\ref{Heckevar3e}) for $i=0$ is equivalent
with the condition that the quasi-isogeny of $p$-divisible groups  
\begin{equation}
\Pi^{-m} \alpha: (X_1)_{\mathfrak{q}_0} \rightarrow (X_0)_{\mathfrak{q}_0} 
  \end{equation}
is an isomorphism. 

We note that (\ref{RZ16e}) means implicitly that a polarization
$\lambda_1 \in \bar{\lambda}_1$ is $\mathfrak{p}_i$-principal if it differs
from $\mu_{\mathfrak{p}_i}(g)^{-1} \alpha^{*} (\lambda)$ by a unit in
$O_{F_{\mathfrak{p_j}}}^{\times}$, for a polarization
$\lambda \in \bar{\lambda}$ of $A_0$ which is principal in $p$. 

That (\ref{RZ15e}) indeed describes the image by the Hecke operator given by
$g \in G^{\bullet}(\mathbb{Q}_p)$ follows immediately from (\ref{BZGdot5e}) 
if we pass from $A_0$ to its isogeny class as in Definition \ref{BZApkt3altd}.
Our conditions for $A_1$ only ensure that we obtain a point of Variant 
\ref{varia}. This description of the Hecke operators will allow us to extend
them in section \ref{s:uniform}  to a model of the functor Variant \ref{varia}
over $O_{E_{\nu}}$.  

   Using Proposition \ref{BZ11p}, we identify the functors $\mathcal{A}^{\bullet}_{\mathbf{K}^{\bullet}}$ and
  $\mathcal{A}^{\bullet bis}_{\mathbf{K}^{\bullet}}$
  and  use the  notation $\mathcal{A}^{\bullet}_{\mathbf{K}^{\bullet}}$ for this functor. 
Finally we can define a functor on the category of $O_{E_{\nu}}$-schemes.
\begin{definition}\label{BZsApkt4d} 
Let $\mathbf{K}^{\bullet}=\mathbf{K}^{\bullet}_p\mathbf{K}^{\bullet,p}\subset G^\bullet(\BA_f)$, with $\mathbf{K}^{\bullet}_p$ as in \eqref{BZKpPkt1e}.  We define a functor $\tilde{\mathcal{A}}^{\bullet t}_{\mathbf{K}^{\bullet}}$ on the
category of $O_{E_{\nu}}$-schemes $S$. 
An $S$-valued point consists of the data $(a), (b), (c), (e)$ as in Definition
\ref{BZApkt4d}. But we replace $(d)$ by the following datum,
\begin{enumerate}
\item[($d^t$)]
  For each polarization $\lambda \in \bar{\lambda}$ and for each prime 
  $\mathfrak{p}|p$ of $O_F$ a section
  $\xi_{\mathfrak{p}}(\lambda) \in O^{\times}_{F_{\mathfrak{p}}}/\mathbf{M}_{\mathfrak{p}}^{\bullet}$
   such that
  $\xi_{\mathfrak{p}}(u \lambda) = u \xi_{\mathfrak{p}}(\lambda)$ for
  each $u \in U_p(F)$.
  \end{enumerate}
\end{definition}

  Let $\mathcal{A}^{\bullet t}_{\mathbf{K}^{\bullet}}$ be the functor on the category of
  $E$-schemes $S$ which is obtained by changing in Definition \ref{BZApkt3altd}
  the item $(d)$ into
  \begin{enumerate}
\item[($d^t$)]
  For each polarization $\lambda \in \bar{\lambda}$ and for each prime
  $\mathfrak{p}|p$ of $O_F$ a section
  $\xi_{\mathfrak{p}}(\lambda)\in F^{\times}_{\mathfrak{p}}/\mathbf{M}_{\mathfrak{p}}^{\bullet}$
  such that $\xi_{\mathfrak{p}}(u \lambda) = u \xi_{\mathfrak{p}}(\lambda)$ for
  each $u \in F^{\times}$.
  \end{enumerate}
  By changing $(d)$ in Definition \ref{BZApkt4d} in the same way (i.e., replacing $O^{\times}_{F_{\mathfrak{p}}}(1)/\mathbf{M}_{\mathfrak{p}}^{\bullet}$ by $O^{\times}_{F_{\mathfrak{p}}}/\mathbf{M}_{\mathfrak{p}}^{\bullet}$), we obtain
  another description of $\mathcal{A}^{\bullet t}_{\mathbf{K}^{\bullet}}$.
We call this the $t$-version of Definition \ref{BZApkt4d}.
 
  By the proof of Proposition \ref{BZ11p}, we have a canonical isomorphism 
  \begin{equation}\label{BZGdot6e} 
    \mathcal{A}^{\bullet t}_{\mathbf{K}^{\bullet}} \times_{\Spec E} \Spec E_{\nu} \cong
    \tilde{\mathcal{A}}^{\bullet t}_{\mathbf{K}^{\bullet}} \times_{\Spec O_{E_{\nu}}}
    \Spec E_{\nu}. 
    \end{equation}
\begin{remark}\label{remHecke}
For $g \in G^{\bullet}(\mathbb{A}_f)$ we have the Hecke operators
$g: \mathcal{A}^{\bullet t}_{\mathbf{K}^{\bullet}} \rightarrow \mathcal{A}^{\bullet t}_{g^{-1}\mathbf{K}^{\bullet} g}$ . For $g\in G^{\bullet}(\mathbb{A}^p_f)$ these extend
obviously to
$g: \tilde{\mathcal{A}}^{\bullet t}_{\mathbf{K}^{\bullet}} \rightarrow \tilde{\mathcal{A}}^{\bullet t}_{g^{-1}\mathbf{K}^{\bullet} g}$. 
Let $g \in G^{\bullet}(\mathbb{Q}_p)$, we have defined the Hecke operator by
(\ref{BZGdot5e}).
  In section \ref{s:uniform} we will extend this Hecke operator to
  $g: \tilde{\mathcal{A}}^{\bullet t}_{\mathbf{K}^{\bullet}} \rightarrow \tilde{\mathcal{A}}^{\bullet t}_{g^{-1}\mathbf{K}^{\bullet} g}$,
  whenever both $\mathbf{K}^{\bullet}_p$  and $g^{-1}\mathbf{K}^{\bullet}_pg$ are as specified in \eqref{BZKpPkt1e}, so that  both source and target make sense.
  This will be done by the remark after Proposition \ref{BZ11p}.
For the time being, we only need the extensions of Hecke operators defined by elements in  $(K\otimes{\BQ_p})^\times\subset G^\bullet(\BQ_p)$. For these Hecke operators we give an ad hoc definition, cf. \eqref{Hecke2e}.
  \end{remark}
  
Let $\zeta_{p^{\infty}} \in \bar{\mathbb{Q}}$ be a compatible system of
primitive $p^n$-th roots of unity. It defines over $\bar{\BQ}$ an isomorphism
of \'etale sheaves for each $\mathfrak p$,
  \begin{equation}\label{BZ24e} 
   \kappa_{\mathfrak{p}}\colon O_{F_{\mathfrak{p}}}^{\times}/\mathbf{M}_{\mathfrak{p}}^{\bullet} 
    \isoarrow  O_{F_{\mathfrak{p}}}^{\times}(1)/\mathbf{M}_{\mathfrak{p}}^{\bullet}, \quad\quad
   \xi_{\mathfrak{p}}  \mapsto  \zeta_{p^{\infty}}\xi_{\mathfrak{p}}.
  \end{equation} 
It is defined over a finite abelian extension $L/E$ which we choose
independently of $\mathfrak{p}$. The isomorphism (\ref{BZ24e}) defines an
isomorphism of functors
\begin{equation}\label{BZ33e}
  \mathcal{A}^{\bullet t}_{\mathbf{K}^{\bullet}}\times_{\Spec E} \Spec L
  \isoarrow\mathcal{A}^{\bullet}_{\mathbf{K}^{\bullet}}\times_{\Spec E} \Spec L.
\end{equation}
Here, for an $L$-scheme $S$ a point
  \begin{equation}\label{BZ23e}  
    (A, \iota, \bar{\lambda}, \bar{\eta}^p, (\bar{\eta}_{\mathfrak{q}_j})_j,
   ( \xi_{\mathfrak{p_i}})_i) \in \mathcal{A}^{\bullet t}_{\mathbf{K}^{\bullet}}(S) 
    \end{equation}
is mapped by the  isomorphism (\ref{BZ33e}) to 
$(A, \iota, \bar{\lambda}, \bar{\eta}^p, (\bar{\eta}_{\mathfrak{q}_j})_j, (\zeta_{p^{\infty}} \xi_{\mathfrak{p_i}})_i)$.
The isomorphism (\ref{BZ33e}) is compatible with the Hecke operators
$G^{\bullet}(\mathbb{A}_f)$ acting on both sides. 

\begin{proposition}\label{BZ3c} 
  Let $\tau \in \Gal(L/E)$ be an automorphism. 
  The isomorphism (\ref{BZ33e}) fits into a commutative diagram
  \begin{displaymath}
    \xymatrix{
      \mathcal{A}^{\bullet t}_{\mathbf{K}^{\bullet}} \times_{\Spec E} \Spec L
      \ar[d]_{\id \times \tau_c}
      \ar[r] & \mathcal{A}^{\bullet}_{\mathbf{K}^{\bullet}} \times_{\Spec E} \Spec L 
      \ar[d]^{{\varsigma_{p^{\infty}}(\tau^{-1})}_{\mid \xi} \times \tau_c}\\
    \mathcal{A}^{\bullet t}_{\mathbf{K}^{\bullet}} \times_{\Spec E} \Spec L \ar[r] & 
    \mathcal{A}^{\bullet}_{\mathbf{K}^{\bullet}} \times_{\Spec E} \Spec L .\\
    }
    \end{displaymath}
  Here we take the composite of the cyclotomic character by the inclusion  
  \begin{displaymath}  
  \varsigma_{p^{\infty}}: \Gal(\bar{E}/E) \rightarrow \mathbb{Z}_p^{\times} \subset
  (O_F \otimes \mathbb{Z}_p)^{\times}. 
  \end{displaymath}
 {\rm See (\ref{xi-action1e}) for the definition of the automorphism $a_{\mid \xi}$ of  $\mathcal{A}^{\bullet}_{\mathbf{K}^{\bullet}}$.
  A similiar definition  applies to
  $\mathcal{A}^{\bullet t}_{\mathbf{K}^{\bullet}}$. }
\end{proposition}
\noindent 
The proof coincides with that of Proposition \ref{Gtwist1p}. As a consequence
we have an analogue of Proposition \ref{Gtwist3p}, i.e. there is a morphism
of functors
\begin{equation}\label{compshim}
  \mathcal{A}^{\bullet t}_{\mathbf{K}^{\bullet}, E_{\nu}} \times_{\Spec E_{\nu}} 
  \Spec E^{nr}_{\nu}   
\rightarrow 
\Sh_{\mathbf{K}^{\bullet}}(G, h \delta^{-1})_{E_{\nu}} \times_{\Spec E_{\nu}} 
\Spec E^{nr}_{\nu} .
\end{equation}
The descent data relative to $E^{nr}_{\nu}/E_{\nu}$ on both sides are compatible
up to the factor $p^{f_{\nu}}_{\mid \xi}$ which can be expressed by a diagram
similiar to that of Proposition \ref{Gtwist1p}. In contrast to  Proposition
\ref{Gtwist1p}, the morphism \eqref{compshim} is no longer an isomorphism
since we are dealing with a coarse moduli scheme. 

We will next show that the action of the group 
$(K \otimes_{\mathbb{Q}} \mathbb{Q}_p)^{\times} \subset G^{\bullet}(\mathbb{Q}_p)$ on
$\mathcal{A}^{\bullet t}_{\mathbf{K}^{\bullet}}$ by Hecke operators extends naturally to
an action on the $O_{E_{\nu}}$-scheme
$\tilde{\mathcal{A}}^{\bullet t}_{\mathbf{K}^{\bullet}}$. We write an element of that
group as 
\begin{displaymath}
  z = (\ldots, a_i, b_i, \ldots) \in (K \otimes_{\mathbb{Q}} \mathbb{Q}_p)^{\times}
  \cong \prod_{i=0}^{s} (K_{\mathfrak{q}_i}^{\times} \times K_{\bar{\mathfrak{q}}_i}^{\times})
  \cong \prod_{i=0}^{s} (F_{\mathfrak{p}_i}^{\times} \times F_{\mathfrak{p}_i}^{\times}),  
  \end{displaymath}
where $a_i, b_i \in F_{\mathfrak{p}_i}^{\times}$, for $i = 0, 1, \ldots, s$. We note
that $\mu_{\mathfrak{p}_i}(z) = a_i b_i$.

We consider a point $(A, \iota, \bar{\lambda}, \bar{\eta})$ of
$\mathcal{A}_{\mathbf{K}^{\bullet}}^{\bullet}(S)$ as in Definition \ref{BZApkt3d}.
We write $\eta = \eta^p \eta_p$ and
\begin{displaymath}
  \eta_p = \oplus_{i=0}^s (\eta_{\mathfrak{q}_i} \oplus \eta_{\bar{\mathfrak{q}}_i}),
  \quad
  \eta_{\mathfrak{q}_i}: V_{\mathfrak{q}_i} \rightarrow \mathrm{V}_{\mathfrak{q}_i}(A),
  \; 
  \eta_{\bar{\mathfrak{q}}_i}: V_{\bar{\mathfrak{q}}_i} \rightarrow
  \mathrm{V}_{\bar{\mathfrak{q}}_i}(A). 
  \end{displaymath}
The Hecke operator
$z: \mathcal{A}_{\mathbf{K}^{\bullet}}^{\bullet} \rightarrow \mathcal{A}_{\mathbf{K}^{\bullet}}^{\bullet}$ 
on $S$-valued points is given by
\begin{equation}\label{Hecke2e} 
  (A, \iota, \bar{\lambda}, \bar{\eta}^p, (\bar{\eta}_{\mathfrak{q}_i})_i,
  (\bar{\eta}_{\bar{\mathfrak{q}}_i})_i) \longmapsto (A, \iota, \bar{\lambda},
  \bar{\eta}^p, \; (\bar{\eta}_{\mathfrak{q}_i}\cdot a_i)_i, \;
  (\bar{\eta}_{\bar{\mathfrak{q}}_i} \cdot b_i)_i). 
\end{equation}
Let $\mathbf{x} \in K^{\times}$. We write its image in
$(K\otimes_{\mathbb{Q}_p} \mathbb{Q}_p)^\times$ as

\begin{displaymath}
  (\ldots, x_i, y_i, \ldots) \in (K \otimes_{\mathbb{Q}} \mathbb{Q}_p)^{\times}
  \cong \prod_{i=0}^{s} (K_{\mathfrak{q}_i}^{\times} \times K_{\bar{\mathfrak{q}}_i}^{\times})
  \cong \prod_{i=0}^{s} (F_{\mathfrak{p}_i}^{\times} \times F_{\mathfrak{p}_i}^{\times}),  
\end{displaymath}
where $x_i, y_i \in F_{\mathfrak{p}_i}^{\times}$. We note that $x_i y_i$ is the
image of $\Nm_{K/F} \mathbf{x}$ in $F_{\mathfrak{p}_i}^{\times}$.

We consider the quasi-isogeny $\mathbf{x}: A \rightarrow A$ induced by multiplication by $\mathbf{x}$.
The inverse image of the data
$(A, \iota, \bar{\lambda}, \bar{\eta}^p, \; (\bar{\eta}_{\mathfrak{q}_i}\cdot a_i)_i, \;  (\bar{\eta}_{\bar{\mathfrak{q}}_i} \cdot b_i)_i)$ 
by this quasi-isogeny  is
\begin{equation}\label{Hecke3e}
  (A, \iota, \bar{\lambda}, \bar{\eta}^p \cdot \mathbf{x}^{-1}, \;
  (\bar{\eta}_{\mathfrak{q}_i}\cdot a_i x_i^{-1})_i,
  \;  (\bar{\eta}_{\bar{\mathfrak{q}}_i} \cdot b_i y_i^{-1})_i) .
\end{equation}
Therefore this is just another way to write the image under the Hecke operator
(\ref{Hecke2e}). 

We rewrite (\ref{Hecke2e}) in terms of the alternative
Definition \ref{BZApkt3altd}. 
In terms of this definition, the left hand side of (\ref{Hecke2e}) corresponds to 
$(A, \iota, \bar{\lambda}, \bar{\eta}^p, (\bar{\eta}_{\mathfrak{q}_i})_i,
  (\xi_{\mathfrak{p}_i})_i)$ 
and (\ref{Hecke3e}) corresponds to
$$(A, \iota, \bar{\lambda}, \bar{\eta}^p \mathbf{x}^{-1}, (\eta_{\mathfrak{q}_i} a_i x_i^{-1})_i, \; (a_i b_i (\Nm_{K/F} \mathbf{x}^{-1}) \xi_{\mathfrak{p}_i})_i) .
$$ 

Summarizing, the Hecke operator
$z: \mathcal{A}_{\mathbf{K}^{\bullet}}^{\bullet} \rightarrow \mathcal{A}_{\mathbf{K}^{\bullet}}^{\bullet}$ 
becomes in terms of Definition \ref{BZApkt3altd} the map 
\begin{equation}\label{Hecke4e}
(A, \iota, \bar{\lambda}, \bar{\eta}^p, (\bar{\eta}_{\mathfrak{q}_i})_i, 
(\xi_{\mathfrak{p}_i})_i) \longmapsto
(A, \iota, \bar{\lambda}, \bar{\eta}^p \mathbf{x}^{-1},(\bar{\eta}_{\mathfrak{q}_i}
a_i x_i^{-1})_i, \; (a_i b_i (\Nm_{K/F} \mathbf{x}^{-1}) \xi_{\mathfrak{p}_i})_i) .
\end{equation}
In the same way $z$ acts on $\mathcal{A}^{\bullet t}_{\mathbf{K}^{\bullet}}$.
In fact, we are only interested in the latter functor. 
Let us choose $\mathbf{x} \in K^{\times}$, such that
$a_i x_i^{-1}$ and $b_i y_i^{-1}$ are units in $O_{F_{\mathfrak{p}_i}}^{\times}$ for
$i = 0, \ldots, s$. 
Assume we have chosen the left hand side of (\ref{Hecke4e})
in the form of the t-version of Definition \ref{BZApkt4d}. Note that we
have added to the data of this definition the unique class of
$O_{B_{\mathfrak{q}_0}}$-module isomorphisms
$\bar{\eta}_{\mathfrak{q}_0}: \Lambda_{\mathfrak{q}_0} \isoarrow T_{\mathfrak{q}_0}(A)$ 
modulo $\mathbf{K}^{\bullet}_{\mathfrak{q}_0}$. We then see that the right hand
side of (\ref{Hecke4e}) is also a point in the sense of Definition
\ref{BZApkt4d}. For $i=0$ we have the isomorphism 
$\bar{\eta}_{\mathfrak{q}_0} a_0 x_0^{-1}: \Lambda_{\mathfrak{q}_0} \isoarrow T_{\mathfrak{q}_0}(A)$
as required. Hence we may forget about $i=0$ and obtain a definition of the
Hecke operator in terms of the  t-version of Definition \ref{BZApkt4d}, 
\begin{displaymath}
  z: \mathcal{A}_{\mathbf{K}^{\bullet}}^{\bullet t} \rightarrow
  \mathcal{A}_{\mathbf{K}^{\bullet}}^{\bullet t}  .
  \end{displaymath}
This definition of $z$ makes sense for the functor
$\tilde{\mathcal{A}}^{\bullet t}_{\mathbf{K}^{\bullet}}$. We define 
\begin{equation}\label{BZGdot7e} 
  \tilde{z}: \tilde{\mathcal{A}}^{\bullet t}_{\mathbf{K}^{\bullet}} \rightarrow
    \tilde{\mathcal{A}}^{\bullet t}_{\mathbf{K}^{\bullet}}
\end{equation}
as follows. Let
$(A, \iota, \bar{\lambda}, \bar{\eta}^p, (\bar{\eta}_{\mathfrak{q}_j})_j, ( \xi_{\mathfrak{p_i}})_i)$
be a point of $\tilde{\mathcal{A}}^{\bullet t}_{\mathbf{K}^{\bullet}}$
  with values in an $O_{E_{\nu}}$-scheme $S$. We define the image by morphism 
  (\ref{BZGdot7e}) as
\begin{displaymath}
  (A, \iota, \bar{\lambda}, \bar{\eta}^p \mathbf{x}^{-1},
  (\bar{\eta}_{\mathfrak{q}_j} a_j x_j^{-1})_j,
  (a_i b_i (\Nm_{K/F} \mathbf{x}^{-1})\xi_{\mathfrak{p_i}})_i).
  \end{displaymath}
It is clear that this is an extension of $z$ with respect to the isomorphism
(\ref{BZGdot6e}).

Recall from \eqref{defhD}  
  \begin{equation}\label{h_D1e}
    h_D\colon \mathbb{S}  \rightarrow
    (D \otimes \mathbb{R})^{\times }  \cong  
    \GL_2(\mathbb{R}) \times
    \prod_{\chi \neq \chi_0} (D \otimes_{F, \chi} \mathbb{R})^\times ,\quad 
     z = a + b\mathbf{i} \mapsto  
    \left(
    \begin{array}{rr}
      a & -b\\
      b & a
      \end{array}
    \right) \times (1, \ldots, 1).
  \end{equation}
 Moreover, we consider the composite 
\begin{equation}\label{h_D2e}
  \begin{array}{rcccl} 
    h^{\bullet}_D:  &\mathbb{S} & \rightarrow & \quad (D \otimes \mathbb{R})^{\times}
    \times (K \otimes \mathbb{R})^{\times} & \rightarrow  G^{\bullet}_{\mathbb{R}} ,\\ 
& z & \mapsto & h_D(z)\; \times \; 1 &
    \end{array}
  \end{equation}
  cf. Lemma \ref{BZ1l}.
  
  Recall from \eqref{def:bulletshim} the Shimura datum $(G^\bullet, h_D^\bullet)$. The next proposition relates the Shimura varieties ${\rm Sh}(G^\bullet, h)$ and ${\rm Sh}(G^\bullet, h_D^\bullet)$.  
\begin{proposition}\label{Sh_D1p} 
  Let $\mathbf{K}^{\bullet}=\mathbf{K}^{\bullet}_p\mathbf{K}^{\bullet,p}\subset G^\bullet(\BA_f)$, with $\mathbf{K}^{\bullet}_p$ as in \eqref{BZKpPkt1e}, where
  $\mathbf{M}_{\mathfrak{p}_0} = O_{F_{\mathfrak{p}_0}}^{\times}$. Denote by
  $f_{\nu}$ the inertia index  of $E_{\nu}/\mathbb{Q}_p$. 
  Let $\pi_{\mathfrak{p}_0}$ be an arbitrary prime element of $F_{\mathfrak{p}_0}$.
  We consider the element  
  \begin{displaymath}
    \dot{z} = (\pi_{\mathfrak{p}_0}^{-1} p^{f_{\nu}}, p^{f_{\nu}}, \ldots, p^{f_{\nu}})
    \in (F \otimes \mathbb{Q}_p)^{\times} = \prod_{i=0}^{s} F_{\mathfrak{p}_i}^{\times}.  
    \end{displaymath}
  Let $\tau \in \Gal(E^{nr}_{\nu}/E_{\nu})$ be the Frobenius automorphism.
  Then there is a morphism of functors
  \begin{equation}\label{BZGdot10e}
    \mathcal{A}^{\bullet t}_{\mathbf{K}^{\bullet}, E_{\nu}} \times_{\Spec E_{\nu}}
    \Spec E^{nr}_{\nu} \rightarrow
    \Sh_{\mathbf{K}^{\bullet}}(G^{\bullet}, h^{\bullet}_{D})_{E_{\nu}} \times_{\Spec E_{\nu}}
    \Spec E^{nr}_{\nu}, 
  \end{equation}
  such that the following diagram is commutative
  \begin{equation*}
  \xymatrix{
  \mathcal{A}^{\bullet t}_{\mathbf{K}^{\bullet},E_{\nu}}\times_{\Spec E_{\nu}}\Spec E^{nr}_{\nu}  
  \ar[d]_{\dot{z}_{\mid \xi} \times \tau_c} \ar[r]
  & \Sh_{\mathbf{K}^{\bullet}}(G^{\bullet}, h^{\bullet}_{D})_{E_{\nu}}
  \times_{\Spec E_{\nu}}\Spec E^{nr}_{\nu} \ar[d]^{\id \times \tau_c}\\
  \mathcal{A}^{\bullet t}_{\mathbf{K}^{\bullet}, E_{\nu}} \times_{\Spec E_{\nu}}
  \Spec E^{nr}_{\nu}  \ar[r] & \Sh_{\mathbf{K}^{\bullet}}(G^{\bullet}, h^{\bullet}_{D})_{E_{\nu}}
  \times_{\Spec E_{\nu}} \Spec E^{nr}_{\nu}, 
   }
  \end{equation*} 
  Here, the right hand side of (\ref{BZGdot10e}) is the coarse moduli scheme of
  the functor on the left hand side. {\rm See (\ref{xi-action1e}) for the definition of $\dot{z}_{\mid \xi}$.} 
  \end{proposition}
  We will show in Proposition \ref{BZ7p} that $\Sh_{\mathbf{K}^{\bullet}}(G^{\bullet}, h^{\bullet}_{D})_{E_{\nu}} \times_{\Spec E_{\nu}}
    \Spec E^{nr}_{\nu}$ 
   is in fact the \'etale sheafification of $  \mathcal{A}^{\bullet t}_{\mathbf{K}^{\bullet}, E_{\nu}} \times_{\Spec E_{\nu}}
    \Spec E^{nr}_{\nu}$.  
\begin{proof}
  Recall the morphism to the coarse moduli space
  \begin{equation}\label{BZGdot11e}
    \mathcal{A}^{\bullet}_{\mathbf{K}^{\bullet},E_{\nu}} \rightarrow
    \Sh_{\mathbf{K}^{\bullet}}(G^{\bullet}, h)_{E_{\nu}}. 
    \end{equation}
  Let $T^{\bullet} = (K \otimes \mathbb{Q})^{\times} \subset G^{\bullet}$ be the central
  torus. We consider $\delta: \mathbb{S} \rightarrow T^{\bullet}_{\mathbb{R}}$
  cf. (\ref{delta1e}). The local reciprocity law
  $r_{\nu}(T^{\bullet},\delta^{-1}): E_{\nu}^{\times}\rightarrow T^{\bullet}(\mathbb{Q}_p)$
  is the composite of the local reciprocity law     $r_{\nu}(T, \delta^{-1})$ for $T$ (given by \eqref{rec-delta1e}) with the inclusion
  $T(\mathbb{Q}_p) \subset T^{\bullet}(\mathbb{Q}_p)$. Let $e \in E_{\nu}^{\times}$
  and let $\sigma \in \Gal(E^{ab}_{\nu}/E_{\nu})$ be the automorphism which
  corresponds to it by local class field theory. 
  If we twist the morphism (\ref{BZGdot11e}) by $r_{\nu}(T^{\bullet},\delta^{-1})$
  we obtain as in the proof of Proposition \ref{Gtwist3p} a commutative diagram
  \begin{equation}\label{Gtwist5e} 
  \begin{aligned}\xymatrix{
  \mathcal{A}^{\bullet t}_{\mathbf{K}^{\bullet},E_{\nu}}\times_{\Spec E_{\nu}}\Spec E^{ab}_{\nu}  
  \ar[d]_{(p^{f_{\nu} \ord e})_{\mid \xi} \times \sigma_c} \ar[r]
  & \Sh_{\mathbf{K}^{\bullet}}(G^{\bullet}, h \delta^{-1})_{E_{\nu}}
  \times_{\Spec E_{\nu}}\Spec E^{ab}_{\nu} \ar[d]^{\id \times \sigma_c}\\
  \mathcal{A}^{\bullet t}_{\mathbf{K}^{\bullet}, E_{\nu}} \times_{\Spec E_{\nu}}
  \Spec E^{ab}_{\nu}  \ar[r] & \Sh_{\mathbf{K}^{\bullet}}(G^{\bullet}, h\delta^{-1})_{E_{\nu}}
  \times_{\Spec E_{\nu}} \Spec E^{ab}_{\nu}. 
   }
   \end{aligned}
    \end{equation}
  We consider the homomorphism
  $\gamma: \mathbb{S} \rightarrow (K \otimes_{\mathbb{Q}} \mathbb{R})^{\times}$
  which in terms of the isomorphism (\ref{CMXi1e}) is defined by
   \begin{equation}
    \gamma:  \mathbb{S}  \rightarrow  \mathbb{C}^{\times} \times
    \prod_{\varphi, r_{\varphi}=2} \mathbb{C}^{\times} , \quad
      z  \mapsto   (z,  1, \ldots, 1) ,
  \end{equation}
  i.e., on the right hand side we have $z$ at the factor which corresponds to
  $\bar{\varphi}_0$. We find that 
  \begin{equation}
  h^{\bullet}_{D} = h \delta^{-1} \gamma .
  \end{equation}
  Therefore we must twist the horizontal line of (\ref{Gtwist5e}) by the 
  local reciprocity law of $\gamma$. The one-parameter
  group $\mu_{\gamma}$ associated to the Shimura datum $\gamma$ is
  \begin{displaymath}
    \mu_{\gamma}: \mathbb{C}^{\times} \rightarrow 
    \prod_{\Phi} \mathbb{C}^{\times}, \quad z \mapsto (1,\ldots, 1,z,1,\ldots, 1) ,
  \end{displaymath}
  where $z$ is exactly at the place $\bar{\varphi}_0$. Since we are interested
  in the local reciprocity law we replace $\mathbb{C}$ by $\bar{\mathbb{Q}}_p$
  cf. (\ref{local-rec1e}). The field of definition $E_{\nu}$ of $\mu_{\gamma}$
  is the image of
  $\bar{\varphi}_0: K_{\bar{\mathfrak{q}}_0} \rightarrow \bar{\mathbb{Q}}_p$.
  There is a canonical isomorphism of $K_{\bar{\mathfrak{q}}_0}$-algebras  
  \begin{displaymath}
    K_{\bar{\mathfrak{q}}_0} \otimes_{\mathbb{Q}_p} E_{\nu} \cong
    K_{\bar{\mathfrak{q}}_0} \times C_{\bar{\mathfrak{q}}_0}, 
  \end{displaymath} 
  where the first factor corresponds to the compositum $K_{\bar{\mathfrak{q}}_0}$ of
  the fields $K_{\bar{\mathfrak{q}}_0}$ and $E_{\nu}$, given by $\id_{K_{_{\bar{\mathfrak{q}}_0}}}$
  and $\bar{\varphi}_0^{-1}$.
  
We consider the homomorphism 
\begin{equation}\label{Gtwist7e}
    E_{\nu}^{\times}  \rightarrow   
    (K_{_{\bar{\mathfrak{q}}_0}} \otimes_{\mathbb{Q}_p} E_{\nu})^{\times}\cong K_{\bar{\mathfrak{q}}_0}^{\times} \times
     C_{\bar{\mathfrak{q}}_0}^{\times}, \quad 
    e  \mapsto  \bar{\varphi}_0^{-1} (e) \times 1.
  \end{equation}
 The one-parameter group  $\mu_{\gamma}$ over $E_{\nu}$ is the homomorphism
  \begin{displaymath}
    E_{\nu}^{\times} \rightarrow (K \otimes_{\mathbb{Q}} E_{\nu})^{\times} \cong 
    \prod_i (K_{_{\mathfrak{q}_i}} \otimes_{\mathbb{Q}_p} E_{\nu})^{\times} \times
    \prod_{i} (K_{_{\bar{\mathfrak{q}}_i}} \otimes_{\mathbb{Q}_p} E_{\nu})^{\times} , 
  \end{displaymath}
 which is given by (\ref{Gtwist7e}) on the factor
  $(K_{_{\bar{\mathfrak{q}}_0}} \otimes_{\mathbb{Q}_p} E_{\nu})^{\times}$ and is trivial
  on all other factors.  The map
  \begin{displaymath}
\Nm_{E_{\nu}/\mathbb{Q}_p} =
\Nm_{K_{_{\bar{\mathfrak{q}}_0}}\otimes_{\mathbb{Q}_p} E_{\nu}/K_{\bar{\mathfrak{q}}_0}}: 
K_{\bar{\mathfrak{q}}_0} \otimes_{\mathbb{Q}_p} E_{\nu}\rightarrow K_{_{\bar{\mathfrak{q}}_0}} 
  \end{displaymath}
  becomes in terms of (\ref{Gtwist7e})
  \begin{displaymath}
    (a,c) \in K_{\bar{\mathfrak{q}}_0} \times C_{\bar{\mathfrak{q}}_0} \mapsto
    a \Nm_{C_{\bar{\mathfrak{q}}_0}/K_{\bar{\mathfrak{q}}_0}}.  
    \end{displaymath}
  We conclude that the local reciprocity law associated to $\gamma$ 
  \begin{equation}
    r(T^{\bullet}, \gamma): E_{\nu}^{\times} \rightarrow \prod_i K_{\mathfrak{q}_i}^{\times}
    \times \prod_i K_{\bar{\mathfrak{q}}_i}^{\times}
  \end{equation}
   maps $e$ to the element with component
  $\bar{\varphi}_{\bar{\mathfrak{q}}_0}^{-1}(e^{-1})$ at the
  factor $K_{\bar{\mathfrak{q}}_0}$ and with trivial component at all other
  factors.

  By Corollary \ref{zentralerTwist1c} and our remarks about the Hecke operators
  (\ref{BZGdot5e}) we obtain  from (\ref{Gtwist5e}) a commutative diagram
  \begin{displaymath}
  \xymatrix{
  \mathcal{A}^{\bullet t}_{\mathbf{K}^{\bullet},E_{\nu}}\times_{\Spec E_{\nu}}\Spec E^{ab}_{\nu}  
  \ar[d]_{(\varphi^{-1}_{\mathfrak{p}_0}(e^{-1})_{\mid \xi_{\mathfrak{p}_0}} (p^{f_{\nu} \ord e})_{\mid \xi} \times \sigma_c} \ar[r]
  & \Sh_{\mathbf{K}^{\bullet}}(G^{\bullet}, h \delta^{-1} \gamma)_{E_{\nu}}
  \times_{\Spec E_{\nu}}\Spec E^{ab}_{\nu} \ar[d]^{\id \times \sigma_c}\\
  \mathcal{A}^{\bullet t}_{\mathbf{K}^{\bullet}, E_{\nu}} \times_{\Spec E_{\nu}}
  \Spec E^{ab}_{\nu}  \ar[r] &
  \Sh_{\mathbf{K}^{\bullet}}(G^{\bullet}, h\delta^{-1} \gamma)_{E_{\nu}}
  \times_{\Spec E_{\nu}} \Spec E^{ab}_{\nu}. 
   }
    \end{displaymath}
  The $\xi_{\mathfrak{p}_0}$-part of the datum ($d^t$) in
  Definition \ref{BZsApkt4d} is a function with values in
  $F_{\mathfrak{p}_0}^{\times}/O_{F_{\mathfrak{p}_0}}^{^{\times}}$.  
  Therefore $(\varphi^{-1}_{\mathfrak{p}_0}(e^{-1})_{\mid \xi_{\mathfrak{p}_0}}$ acts on
  this datum exactly like $(\pi_{\mathfrak{p}_0}^{- \ord e})_{\mid \xi_{\mathfrak{p}_0}}$.
  This shows that for $e \in O_{E_{\nu}}^{\times}$ the vertical arrow on the left 
  hand side in the above diagram is equal to $\id \times \sigma_c$. Therefore
  the horizontal arrow in this diagram is defined over $E_{\nu}^{nr}$. 
  The proposition follows.
  \end{proof}
We use  Proposition \ref{Sh_D1p} to define a model of
$\Sh_{\mathbf{K}^{\bullet}}(G^{\bullet}, h^{\bullet}_{D})_{E_{\nu}}$ over $O_{E_{\nu}}$.
The group 
$T^{\bullet}(\mathbb{A}_f)/(\mathbf{K}^{\bullet}\cap T^{\bullet}(\mathbb{A}_f))$ acts through a
finite quotient. Therefore the Hecke operator associated to $\dot{z}$ has finite order.
It follows that the field $E^{nr}_{\nu}$ in Proposition \ref{Sh_D1p} can be
replaced by a finite unramified extension $L/E_{\nu}$. We have extended
$\dot{z}$ to an automorphism of the functor 
$\tilde{\mathcal{A}}^{\bullet t}_{\mathbf{K}^{\bullet}}$ over $O_{E_{\nu}}$. 
\begin{definition}\label{tildeSh_D1d} 
    Let $\mathbf{K}^{\bullet}=\mathbf{K}^{\bullet}_p\mathbf{K}^{\bullet,p}\subset G^\bullet(\BA_f)$, with $\mathbf{K}^{\bullet}_p$ as in \eqref{BZKpPkt1e}, where
  $\mathbf{M}_{\mathfrak{p}_0} = O_{F_{\mathfrak{p}_0}}^{\times}$. We define 
  $\widetilde{\Sh}_{\mathbf{K}^{\bullet}}(G^{\bullet}, h^{\bullet}_{D})$ to be the
  $O_{E_{\nu}}$-scheme given by the descent datum $\dot{z} \times \tau_c$
  on $\tilde{\mathsf{A}}^{\bullet t}_{\mathbf{K}^{\bullet}} \times_{\Spec O_{E_{\nu}}} \Spec O_L$, where   $\tilde{\mathsf{A}}^{\bullet t}_{\mathbf{K}^{\bullet}}$ is the coarse moduli scheme of $\tilde{\mathcal{A}}^{\bullet t}_{\mathbf{K}^{\bullet}}$.
  \end{definition}
The diagram of Proposition \ref{Sh_D1p} becomes
\begin{equation}\label{tildeSh_D1e} 
  \begin{aligned}
  \xymatrix{
    \tilde{\mathcal{A}}^{\bullet t}_{\mathbf{K}^{\bullet}} \times_{\Spec O_{E_{\nu}}}
    \Spec O_{E^{nr}_{\nu}}   
  \ar[d]_{\dot{z}_{\mid \xi} \times \tau_c} \ar[r]
  & \widetilde{\Sh}_{\mathbf{K}^{\bullet}}(G^{\bullet}, h^{\bullet}_{D}) 
  \times_{\Spec O_{E_{\nu}}} \Spec O_{E^{nr}_{\nu}} \ar[d]^{\id \times \tau_c}\\   
  \tilde{\mathcal{A}}^{\bullet t}_{\mathbf{K}^{\bullet}} \times_{\Spec O_{E_{\nu}}}
  \Spec O_{E^{nr}_{\nu}} \ar[r] &
  \widetilde{\Sh}_{\mathbf{K}^{\bullet}}(G^{\bullet}, h^{\bullet}_{D}) 
  \times_{\Spec O_{E_{\nu}}} \Spec O_{E^{nr}_{\nu}} .\\   
   } 
   \end{aligned}
  \end{equation}

\begin{remark} Let us drop the assumption that
$\mathbf{M}_{\mathfrak{p}_0} = O^\times_{F_{\mathfrak{p}_0}}$. We can write the diagram 
at the end of the proof of Proposition \ref{Sh_D1p} in the form 
 \begin{displaymath}
  \xymatrix{
  \mathcal{A}^{\bullet t}_{\mathbf{K}^{\bullet},E_{\nu}}\times_{\Spec E_{\nu}}\Spec E^{ab}_{\nu}  
  \ar[d]_{(\pi_{\mathfrak{p}_0}^{-\ord e})_{\mid \xi_{\mathfrak{p}_0}} (p^{f_{\nu} \ord e})_{\mid \xi}}^{\times \sigma_c} \ar[r]
  & \Sh_{\mathbf{K}^{\bullet}}(G^{\bullet}, h^{\bullet}_D)_{E_{\nu}}
    \times_{\Spec E_{\nu}}\Spec E^{ab}_{\nu}
    \ar[d]_{(\pi_{\mathfrak{p}_0}^{-\ord e})_{\mid \xi_{\mathfrak{p}_0}} (\varphi_{\mathfrak{p}_0}^{-1}(e))_{\mid \xi_{\mathfrak{p}_0}}}^{\times \sigma_c}\\
  \mathcal{A}^{\bullet t}_{\mathbf{K}^{\bullet}, E_{\nu}} \times_{\Spec E_{\nu}}
  \Spec E^{ab}_{\nu}  \ar[r] & 
  \Sh_{\mathbf{K}^{\bullet}}(G^{\bullet}, h^{\bullet}_D)_{E_{\nu}}
  \times_{\Spec E_{\nu}} \Spec E^{ab}_{\nu}. 
   }
    \end{displaymath}
As before $e$ corresponds to $\sigma$ by local class field theory. 
 We define the Galois twist 
 $\Sh_{\mathbf{K}^{\bullet}}(G^{\bullet}, h^{\bullet}_D)_{E_{\nu}}(\pi_{\mathfrak{p}_0})$ of
  $\Sh_{\mathbf{K}^{\bullet}}(G^{\bullet}, h^{\bullet}_D)_{E_{\nu}}$ by the commutative
 diagram
  \begin{displaymath}
  \xymatrix{
  \Sh_{\mathbf{K}^{\bullet}}(G^{\bullet}, h^{\bullet}_D)_{E_{\nu}}
    \times_{\Spec E_{\nu}}\Spec E^{ab}_{\nu}
    \ar[d]^{(\pi_{\mathfrak{p}_0}^{-\ord e})_{\mid \xi_{\mathfrak{p}_0}} (\varphi_{\mathfrak{p}_0}^{-1}(e))_{\mid \xi_{\mathfrak{p}_0}} \times \sigma_c} \ar[r] & 
    \Sh_{\mathbf{K}^{\bullet}}(G^{\bullet}, h^{\bullet}_D)_{E_{\nu}}(\pi_{\mathfrak{p}_0})
    \times_{\Spec E_{\nu}}\Spec E^{ab}_{\nu}
    \ar[d]^{id \times \sigma_c}\\
 \Sh_{\mathbf{K}^{\bullet}}(G^{\bullet}, h^{\bullet}_D)_{E_{\nu}} \times_{\Spec E_{\nu}}
  \Spec E^{ab}_{\nu}  \ar[r] & 
  \Sh_{\mathbf{K}^{\bullet}}(G^{\bullet}, h^{\bullet}_D)_{E_{\nu}}(\pi_{\mathfrak{p}_0})
  \times_{\Spec E_{\nu}} \Spec E^{ab}_{\nu}. 
   }
    \end{displaymath}
  Then we obtain  
  a commutative diagram 
\begin{displaymath}
  \xymatrix{
  \mathcal{A}^{\bullet t}_{\mathbf{K}^{\bullet},E_{\nu}}\times_{\Spec E_{\nu}}\Spec E^{nr}_{\nu}  
  \ar[d]^{(\pi_{\mathfrak{p}_0}^{-\ord e})_{\mid \xi_{\mathfrak{p}_0}} (p^{f_{\nu} \ord e})_{\mid \xi} \times \sigma_c} \ar[r]
  & \Sh_{\mathbf{K}^{\bullet}}(G^{\bullet}, h^{\bullet}_D)_{E_{\nu}}(\pi_{\mathfrak{p}_0})
    \times_{\Spec E_{\nu}}\Spec E^{nr}_{\nu}
    \ar[d]^{(\id \times \sigma_c}\\
  \mathcal{A}^{\bullet t}_{\mathbf{K}^{\bullet}, E_{\nu}} \times_{\Spec E_{\nu}}
  \Spec E^{nr}_{\nu}  \ar[r] & 
  \Sh_{\mathbf{K}^{\bullet}}(G^{\bullet}, h^{\bullet}_D)_{E_{\nu}}(\pi_{\mathfrak{p}_0})
  \times_{\Spec E_{\nu}} \Spec E^{nr}_{\nu}. 
   }
    \end{displaymath}
In the same way as in Definition \ref{BZKpPkt1e} we obtain a model
$\widetilde{\Sh}_{\mathbf{K}^{\bullet}}(G^{\bullet}, h^{\bullet}_D)(\pi_{\mathfrak{p}_0})$ over
$O_{E_{\nu}}$. The diagram (\ref{tildeSh_D1e}) continues to hold for arbitrary
$\mathbf{K}^{\bullet}$ if we substitute
$\widetilde{\Sh}_{\mathbf{K}^{\bullet}}(G^{\bullet}, h^{\bullet}_D)(\pi_{\mathfrak{p}_0})$ for
$\widetilde{\Sh}_{\mathbf{K}^{\bullet}}(G^{\bullet}, h^{\bullet}_D)$. 
We note that the last two schemes are canonically identified if
$\mathbf{K}^{\bullet}$ is of the type
$\mathbf{M}_{\mathfrak{p}_0} = O_{F_{\mathfrak{p}_0}}^{\times}$. 

One could regard
$\Sh_{\mathbf{K}^{\bullet}}(G^{\bullet}, h^{\bullet}_D)_{E_{\nu}}(\pi_{\mathfrak{p}_0})$ as the
twist of $\Sh_{\mathbf{K}^{\bullet}}(G^{\bullet}, h^{\bullet}_D)_{E_{\nu}}$ by the
character of $\Gal(E_{\nu}^{ab}/E_{\nu})$ associated to the Lubin-Tate group
defined by $\pi_{\mathfrak{p}_0}$. 

\end{remark} 

Our next aim is to compare the functors ${\mathcal A}_{\mathbf K}$  and ${\mathcal A}_{\mathbf K^\bullet}$. For this we need  the following variant of a theorem of Chevalley \cite{Che}. 
\begin{proposition}\label{Chevalley1p}
  Let $F$ be a totally real number field. We set $[F: \mathbb{Q}] = d = 2^h d'$
  such that $d'$ is odd. Let $M \geq 2$ be a natural number and let $\ell$ be a prime number such
  that
    \begin{displaymath}
   \ell \equiv 2  \mod d', \qquad \ell \equiv 3 \mod 4.
  \end{displaymath}
For a natural number $N$, let $U_{N\ell}$ be the principal congruence subgroup of $(O_F \otimes \hat{\mathbb{Z}})^{\times}$,
  $$U_{N\ell} =\{u \equiv 1 \mod N \ell (O_F \otimes \hat{\mathbb{Z}}) \}.
  $$
For each natural number $m$ there is a power $N$ of $M$ with the following
  property:  for each element $f \in F^{\times}$ which is totally positive and such
  that $f \in U_{N\ell} \cdot \mathbb{A}_f^{\times}$, there is a unit $g \in O_F^{\times}$
  such that
  \begin{displaymath}
f = g^m q, \quad \text{for some} \; q \in \mathbb{Q}^{\times}, \; q >0. 
    \end{displaymath}
  \end{proposition}
\begin{proof} Set $U=U_{N\ell}$, where $N$ will be determined in the proof. 
  We write $f = u \alpha$ with $u \in U$, $\alpha \in \mathbb{A}_f^{\times}$.
  We find $q$ such that $\alpha = q \beta$ and
  $\beta \in \hat{\mathbb{Z}}^{\times}$. Therefore we may assume
  that $f \in O_F^{\times}$ and hence $q = 1$. We obtain
  \begin{displaymath}
f = u \alpha, \quad u \in U, \; \alpha \in \hat{\mathbb{Z}}^{\times}.
    \end{displaymath}
  We note that $\Nm_{F/\mathbb{Q}} u \in U$. We find
  \begin{displaymath}
    f^d (\Nm_{F/\mathbb{Q}} f)^{-1} = u^d (\Nm_{F/\mathbb{Q}} u)^{-1}
    \alpha^d (\Nm_{F/\mathbb{Q}} \alpha)^{-1} = u^d (\Nm_{F/\mathbb{Q}} u)^{-1}
    \in U. 
    \end{displaymath}
  Since $f$ is a totally positive unit $\Nm_{F/\mathbb{Q}} f = 1$ and therefore
  $f^d \in U$. By Chevalley \cite{Che} there exists for a suitable $N$ a unit
  $g \in O_F^{\times}$ such that $f^d = g^{md}$. Replacing $m$ by a multiple, we may 
  assume that $m$ is even and that
    \begin{displaymath}
g^m \equiv 1 \; \mod \ell (O_F \otimes \hat{\mathbb{Z}}). 
  \end{displaymath}

  We consider the $d$-th root of unity 
  \begin{displaymath}
f/g^m = \zeta. 
  \end{displaymath}
  Since $f \equiv \alpha \mod \ell (O_F \otimes \hat{\mathbb{Z}})$
    we obtain
    \begin{displaymath}
\zeta \equiv \alpha  \mod \ell (O_F \otimes \hat{\mathbb{Z}}). 
    \end{displaymath}
The right hand side is in $\mathbb{Z}/ \ell \mathbb{Z} \subset O_F/ \ell O_F$.
This shows $\zeta^{\ell -1} \equiv 1$. Here and below, this is meant
$\!\!\mod \ell (O_F \otimes \hat{\mathbb{Z}})$. On the other hand,
  we have
  \begin{displaymath}
\zeta^{2^h d'} \equiv 1.
  \end{displaymath}
  Since $\ell -1 \equiv 1 \mod d'$, we obtain that
  $\zeta^{2^h} \equiv 1$. If $h = 0$ we conclude from Serre's lemma that
  $\zeta = 1$. Let $h > 0$. By our assumption $(\ell -1)/2$ is odd. Therefore
  the greatest common divisor of $\ell - 1$ and $2^{h}$ is $2$. We conclude
  that $\zeta^2 \equiv 1$ and by Serre's 
  lemma that $\zeta^2 = 1$. We obtain
  \begin{displaymath}
f/g^m = \pm 1.
    \end{displaymath}
  Since $m$ is even, the left hand side is totally positive by  assumption.
  This gives finally $f = g^m$. 
\end{proof}

Let $\mathbf{K}^{\bullet} \subset G^{\bullet}(\mathbb{A}_f)$ be an open and compact subgroup.
 We set
$\mathbf{K} = \mathbf{K}^{\bullet} \cap G(\mathbb{A}_f)$. For an open  compact subgroup
$U \subset (F \otimes \mathbb{A}_f)^{\times}$, 
we define
\begin{displaymath}
  \mathbf{K}^{\bullet}_U = \{ g \in \mathbf{K}^{\bullet} \; | \;
  \mu(g) \in U \hat{\mathbb{Z}}^{\times} \}. 
\end{displaymath}
Then 
\begin{equation}\label{BZ17e}
 \mathbf{K}= \mathbf{K}_U^{\bullet} \cap G(\mathbb{A}_f). 
\end{equation}

\begin{proposition}\label{BZ4p}
  We fix $M$ and $\ell$ as in Proposition \ref{Chevalley1p}. 
  Let $\mathbf{K}^{\bullet} \subset G^{\bullet}(\mathbb{A}_f)$ be an open  compact
  subgroup. Then there exists a power $N$ of $M$ such that for the principal congruence 
  subgroup $U =U_{N\ell}\subset (F \otimes \mathbb{A}_f)^{\times}$ of Proposition
  \ref{Chevalley1p}, the natural map  of functors
  \begin{equation}\label{BZ8e}
\mathcal{A}_{\mathbf{K}} \rightarrow \mathcal{A}^{\bullet}_{\mathbf{K}_U^{\bullet}}
    \end{equation}
  is a monomorphism. 
\end{proposition}
To show this, it is enough to check injectivity for points with values in
$S=\Spec R$, where $R$ is a noetherian $E_{\nu}${\rm -Alg}ebra and $\Spec R$ is
connected. We begin with two lemmas.  Since the meaning of the class in the notation $(A, \bar{\lambda}, \bar{\eta})$ depends on whether this is an object of $\mathcal{A}_{\mathbf{K}}$ or of $\mathcal{A}^{\bullet}_{\mathbf{K}^{\bullet}}$, we use the notation $(A, \tilde{\lambda}, \tilde{\eta})$ in the latter case. 
\begin{lemma}\label{BZ4l}
Let $(A, \bar{\lambda}, \bar{\eta}) \in \mathcal{A}_{\mathbf{K}}(R)$
with image 
$(A', \tilde{\lambda}', \tilde{\eta}') \in \mathcal{A}^{\bullet}_{\mathbf{K}_U^{\bullet}}(R)$.
Then there is a polarization $\lambda' \in \tilde{\lambda}'$ and level structure
$\eta' \in \tilde{\eta}'$ such that the point
$(A, \bar{\lambda}, \bar{\eta})$ may be represented in the form
$(A', \bar{\lambda}', \bar{\eta}')$.
\end{lemma}
\begin{proof}
  We start with arbitrary polarizations $\lambda' \in \tilde{\lambda'}$ and
  $\lambda \in \bar{\lambda}$ and arbitrary level structures $\eta' \in \tilde{\eta'}$,
  $\eta \in \bar{\eta}$. Since we have the same point in
  $\mathcal{A}^{\bullet}_{\mathbf{K}_U^{\bullet}}(R)$, there is an isogeny
  $\alpha: A' \rightarrow A$ such that
  $\alpha^{*} (\lambda) =  f \lambda'$. Since we have chosen polarizations,
  $f \in F^{\times}$ must be totally positive. Moreover, $\alpha$ must respect
  $\eta$ and $\eta'$ up to a factor in $\mathbf{K}_U^{\bullet}$,
  \begin{displaymath}
\alpha \circ \eta' \dot{c} = \eta, \quad \dot{c} \in \mathbf{K}_U^{\bullet}. 
    \end{displaymath}
  We claim that $(A', \overline{f \lambda'}, \overline{\eta' \dot{c}})$ is
  a point of $\mathcal{A}_{\mathbf{K}}(R)$. We have to check that the isomorphism 
  \begin{displaymath}
\eta_1 \dot{c}: V \otimes \mathbb{A}_f \isoarrow \hat{V} (A')  
  \end{displaymath}
  respects the form $\psi$ and the Riemann form $E^{f \lambda'}$, up to a factor
  $a \in \mathbb{A}^{\times}_f(1)$. But  for
  $x, y \in V \otimes \mathbb{A}_f$ we have
   \begin{equation*}
    \begin{aligned}
    E^{f \lambda'}(\eta' (\dot{c} x), \eta' (\dot{c} y)) &=
    E^{\alpha^{*} (\lambda)}(\alpha^{-1}\circ \eta (x), \alpha^{-1}\circ \eta (y)) =
    E^{\lambda}(\eta (x), \eta (y))\\ 
    &= a \psi(x, y)     
      \end{aligned}
  \end{equation*}
  for some $a \in \mathbb{A}^{\times}_f(1)$, because
  $(A, \bar{\lambda}, \bar{\eta})$ is a point of
  $\mathcal{A}_{\mathbf{K}}(R)$.

  It is obvious that
  \begin{displaymath}
    \alpha: (A', \overline{f \lambda'}, \overline{\eta' \dot{c}})
    \rightarrow (A, \bar{\lambda}, \bar{\eta}) 
  \end{displaymath}
  is an isomorphism and therefore both sides give the same point of
  $\mathcal{A}_{\mathbf{K}}(R)$. 
\end{proof}
\begin{lemma}\label{BZ42}
  Let $(A_1, \bar{\lambda}_1, \bar{\eta}_1)$ and
  $(A_2, \bar{\lambda}_2, \bar{\eta}_2)$ be two points of
  $\mathcal{A}_{\mathbf{K}}(R)$ whose images in
  $\mathcal{A}^{\bullet}_{\mathbf{K}_U^{\bullet}}(R)$ by (\ref{BZ8e}) are the
  same. Then there exists a totally positive $f \in F^{\times}$ and an element
  $\dot{c} \in \mathbf{K}_U^{\bullet}$, such that
  \begin{displaymath}
f (\dot{c}' \dot{c}) \in \mathbb{A}_f^{\times}, 
  \end{displaymath}
  and such that $(A_2, \bar{\lambda}_2, \bar{\eta}_2)$ is isomorphic
  to $(A_1, f \bar{\lambda}_1, \bar{\eta}_1 \dot{c})$. 
  \end{lemma}
\begin{proof}
  We choose  arbitrary polarizations $\lambda_1 \in \bar{\lambda}_1$ and
  $\lambda_2 \in \bar{\lambda}_2$ and arbitrary level structures $\eta_1 \in \bar{\eta}_1$ and
  $\eta_2 \in \bar{\eta}_2$. We remark that for each
  $\dot{c} \in \mathbf{K}_U^{\bullet}$ the class $\eta_1 \dot{c} \mathbf{K}$
  is invariant under the action of $\pi_1(\bar{s}, S)$ because
  $\mathbf{K} \subset \mathbf{K}_U^{\bullet}$ is a normal subgroup. 

  By the Lemma \ref{BZ4l} we may assume that
  $(A_2, \bar{\lambda}_2, \bar{\eta}_2) = (A_1, \overline{f \lambda_1}, \overline{\eta_1 \dot{c}})$.
  We have factors $a_1, a_2 \in \mathbb{A}_f(1)$ such that for all
  $x, y \in V \otimes \mathbb{A}_f$ 
  \begin{equation}\label{BZ18e} 
    \begin{aligned}
    a_2 \psi(x,y) &= E^{f \lambda_1}(\eta_1(\dot{c} x), \eta_1(\dot{c} y)) =
    E^{\lambda_1}(\eta_1(\dot{c} x), \eta_1(\dot{c} fy))= \\
    &=a_1 \psi(\dot{c} x, \dot{c} fy) = a_1 \psi (x, \dot{c}' \dot{c} fy). 
      \end{aligned}
  \end{equation}
  The assertion follows. 
  \end{proof}
\begin{proof}[Proof of Proposition \ref{BZ4p}]
  We may assume that $S=\Spec R$ is connected. We consider a point
  $(A, \bar{\lambda}, \bar{\eta}) \in \mathcal{A}_{\mathbf{K}}(R)$. Any
  other point with the same image by (\ref{BZ8e}) is of the form 
  \begin{equation}\label{BZ9e}
    (A, f \bar{\lambda}, \bar{\eta} \dot{c}), \quad \text{such that} \;
    f \in F^{\times}, \; \dot{c} \in \mathbf{K}^{\bullet}_U , \;
    f \dot{c}' \dot{c} = a \in \mathbb{A}^{\times}_f. 
  \end{equation}
  Moreover $f$ is totally positive. Replacing $f$ by $f q$ for some
  $q \in \mathbb{Q}^{\times}$, $q >0$, does not change the point (\ref{BZ9e}).
  Therefore we may assume that $a \in \hat{\mathbb{Z}}^{\times}$ and that
  $f$ is a unit.

  By Proposition \ref{Chevalley1p}, for each natural number $m$ we find
  $U=U_{N\ell}$ in such a way that $f = g_m^{2m}$ for some $g_m \in O_F^{\times}$.
  Since $\mathbf{K}^{\bullet} \cap (F \otimes \mathbb{A}_f)^{\times}$ is open
  in $(F \otimes \mathbb{A}_f)^{\times}$ 
  we may choose $m$ such that $g_m^m \in \mathbf{K}^{\bullet}$.
  We set $g = g_m^m$. Since $f = g^2$, the multiplication isomorphism
  by $g$ is an isomorphism 
  \begin{displaymath}
    g: (A, f \bar{\lambda}, \bar{\eta} \dot{c}) \isoarrow
    (A, \bar{\lambda}, \bar{\eta} g \dot{c}).
    \end{displaymath}
  We obtain
  \begin{displaymath}
    (g \dot{c})' \cdot (g \dot{c}) = g^2 \dot{c}' \dot{c} =
    f \dot{c}' \dot{c} = a \in \mathbb{A}_f^{\times}, 
    \end{displaymath}
  and therefore
  $g \dot{c} \in G(\mathbb{A}_f) \cap \mathbf{K}^{\bullet} = \mathbf{K}$.
  We see that $(A, f \bar{\lambda}, \bar{\eta} \dot{c})$ and
  $(A, \bar{\lambda}, \bar{\eta})$ define the same point of
  $\mathcal{A}_{\mathbf{K}}$. 
\end{proof}

We know that for $\mathbf{K} \subset G(\mathbb{A}_f)$ small enough the
functor $\mathcal{A}_{\mathbf{K}, E_{\nu}}$ is representable by the scheme
$\Sh(G,h)_{\mathbf{K}, E_{\nu}}$. In general the latter is a coarse moduli scheme.
\begin{proposition}\label{BZ7p} 
  Let $\mathbf{K}^{\bullet} \subset G^{\bullet}(\mathbb{A}_f)$ be an open and
  compact subgroup. We set
  $\mathbf{K} = G(\mathbb{A}_f) \cap \mathbf{K}^{\bullet}$. We
  assume that there is an $O_K$-lattice $\Gamma \subset V$ and 
  an integer $m \geq 3$ such that for each $u \in \mathbf{K}^{\bullet}$ we have
  $u \Gamma\otimes \hat{\mathbb{Z}}\subset \Gamma \otimes \hat{\mathbb{Z}}$
  and such that $u$ acts trivially on
  $\Gamma/m \Gamma$, so that $\mathcal{A}_{\mathbf{K}, E_{\nu}}$ is representable (this is the analogue of condition \eqref{BZneat1e} for $\mathbf{K}^{\bullet}$ instaed of $\mathbf{K}$).

  Let $U$ be as in Proposition \ref{BZ4p}. Then the \'etale sheafification of
  the presheaf $\mathcal{A}^{\bullet}_{\mathbf{K}^{\bullet}_{U}, E_{\nu}}$ on the big
  \'etale site is represented by
  $\Sh(G^{\bullet}, h)_{\mathbf{K}^{\bullet}_{U}, E_{\nu}}$. 
  \end{proposition}
\begin{proof}
  We begin with some general remarks on \cite[Prop. 1.15]{D-TS} in our case.
  The morphism of schemes (not of finite type) 
  $\Sh(G, h) \rightarrow \Sh (G^{\bullet}, h)$ is an open and closed immersion.
  More precisely, for any open compact subgroup
  $\mathbf{K}_1 \subset G(\mathbb{A}_f)$, there is an open compact subgroup 
  $\mathbf{K}^{\bullet}_1 \subset G^{\bullet}(\mathbb{A}_f)$ such that
  $\mathbf{K}^{\bullet}_1 \cap G(\mathbb{A}_f) = \mathbf{K}_1$ and such that
  $\Sh(G, h)_{\mathbf{K}_1} \subset \Sh(G^{\bullet}, h)_{\mathbf{K}^{\bullet}_1}$ is
  an open and closed   immersion. Indeed, it is a closed immersion by
  \cite[Prop. 1.15]{D-TS} and
  it is open because the local rings of these varieties are normal and have
  both the same constant dimension. If $Z$ is a connected component of
  $\Sh(G, h)_{\mathbb{C}}$, then its 
  image $Z^{\bullet}$ in $ \Sh(G^{\bullet}, h)_{\mathbb{C}}$ is a connected component.
  For arbitrary open compact subgroups $\mathbf{K} \subset G(\mathbb{A}_f)$
  resp. $\mathbf{K}^{\bullet} \subset G(\mathbb{A}_f)$ the image
  $Z_{\mathbf{K}}$ of $Z$ in $ \Sh(G, h)_{\mathbf{K}, \mathbb{C}}$, resp. the image
  $Z^{\bullet}_{\mathbf{K}^{\bullet}}$ of $Z^{\bullet}$ in 
  $\Sh(G^{\bullet}, h)_{\mathbf{K}^{\bullet}, \mathbb{C}}$, is a connected component.
  For $\mathbf{K}_1$ and $\mathbf{K}_1^{\bullet}$ as above, the map
$Z_{\mathbf{K}_1} \rightarrow Z^{\bullet}_{\mathbf{K}^{\bullet}_1}$ is an isomorphism. For
$g \in G^{\bullet}(\mathbb{A}_f)$, the multiplication by $g$ induces a map
 \begin{displaymath}
   g: \Sh(G^{\bullet}, h)_{g\mathbf{K}^{\bullet}_1g^{-1}} \rightarrow
   \Sh(G^{\bullet}, h)_{\mathbf{K}^{\bullet}_1}. 
   \end{displaymath}
Now $G^{\bullet}(\mathbb{A}_f)$ acts transitively on the connected components of 
 $\Sh(G^{\bullet}, h)_{\mathbb{C}}$, cf. \cite[Prop. 2.2]{D-TS}. Therefore the sets
 $g Z^{\bullet}_{g\mathbf{K}^{\bullet}_1 g^{-1}, \mathbb{C}}$ cover 
 $\Sh(G^{\bullet}, h)_{\mathbf{K}^{\bullet}_1, \mathbb{C}}$, as $g$ runs through all elements
 of $G^{\bullet}(\mathbb{A}_f)$. We note that
$g\mathbf{K}^{\bullet}_1g^{-1} \cap G(\mathbb{A}_f) = g\mathbf{K}_1g^{-1}$ because
$G(\mathbb{A}_f) \subset G^{\bullet}(\mathbb{A}_f)$ is a normal subgroup. 
We conclude that the images of the following composite maps  cover $\Sh(G^{\bullet},h)_{\mathbf{K}^{\bullet}_1}$, as $g$ varies in $G^{\bullet}(\mathbb{A}_f)$ ,
\begin{equation}\label{BZ30e}
  \varkappa_g: \Sh(G,h)_{g\mathbf{K}_1g^{-1}} \rightarrow
  \Sh(G^{\bullet}, h)_{g\mathbf{K}^{\bullet}_1g^{-1}} \overset{g}{\rightarrow} 
   \Sh(G^{\bullet}, h)_{\mathbf{K}^{\bullet}_1}.
  \end{equation}

Now we turn to the proof of the proposition. We show that the
morphism
  \begin{equation}\label{BZ10e} 
    \Sh(G,h)_{\mathbf{K}} \rightarrow 
    \Sh(G^{\bullet}, h)_{\mathbf{K}^{\bullet}_{U}} 
  \end{equation}
  is an open and closed immersion. Indeed by Proposition \ref{BZ4p} we know
  that (\ref{BZ10e}) induces an injection on the $\mathbb{C}$-valued points.
  By the remarks above there exists a open and compact subgroup
  $\mathbf{K}_1^{\bullet} \subset \mathbf{K}^{\bullet}_{U}$ such that
  $\mathbf{K}_1^{\bullet} \cap G(\mathbb{A}_f) = \mathbf{K}$ and such that
  $ \Sh(G,h)_{\mathbf{K}} \rightarrow \Sh(G^{\bullet}, h)_{\mathbf{K}_1^{\bullet}}$
  is an open and closed immersion. We consider the commutative diagram
  \begin{displaymath}
    \xymatrix{
      \Sh(G,h)_{\mathbf{K}} \ar[r] \ar[rd] & \Sh(G^{\bullet}, h)_{\mathbf{K}^{\bullet}_1}
      \ar[d]\\
      & \Sh(G^{\bullet}, h)_{\mathbf{K}^{\bullet}_{U}} .
 } 
  \end{displaymath}
  By our assumption on $\mathbf{K}^{\bullet}$, the vertical arrow is a finite
  \'etale morphism. Hence the same is true for the oblique arrow. Since by Proposition  \ref{BZ4p} its 
  geometric fibres contain at most one element, the claim follows.
  
   Let $Y$ be the \'etale sheafification of
  $\mathcal{A}^{\bullet}_{\mathbf{K}^{\bullet}_U, E_{\nu}}$.  We consider the preimage
  $Y^{o}$ of $\Sh(G,h)_{\mathbf{K}, E_{\nu}}$ by the natural morphism
  \begin{displaymath}
    Y \rightarrow
    \Sh(G^{\bullet}, h)_{\mathbf{K}^{\bullet}_{U}, E_{\nu}}. 
  \end{displaymath}
  Then $Y^{o} \subset Y$ is an open and closed subfunctor. We consider the
  natural morphism
  \begin{displaymath}
    \mathcal{A}_{\mathbf{K}, E_{\nu}} \rightarrow
    \mathcal{A}^{\bullet}_{\mathbf{K}^{\bullet}_U, E_{\nu}} \rightarrow  Y \rightarrow
    \Sh(G^{\bullet}, h)_{\mathbf{K}^{\bullet}_{U}, E_{\nu}}. 
  \end{displaymath}
  Since $Y^{o}$ is a fibre product we obtain a factorization
  \begin{equation}\label{et-sheaf1e} 
    \mathcal{A}_{\mathbf{K}, E_{\nu}} \rightarrow Y^{o} \rightarrow
    \Sh(G,h)_{\mathbf{K}, E_{\nu}}. 
  \end{equation} 
  We claim that both arrows are isomorphisms. Since their composite is an
  isomorphism the first arrow is a monomorphism. Therefore it suffices to
  show that the first arrow is a surjection of \'etale sheaves. 
  Since both functors $\mathcal{A}_{\mathbf{K}, E_{\nu}}$ and 
  $\mathcal{A}^{\bullet}_{\mathbf{K}^{\bullet}_U, E_{\nu}}$ commute with inductive limits,
  the stalks at a geometric point $\xi$ of $\Spec R$ of the sheafifications are
  the points of these functors with values in the strict henselization
  $R^{sh}_{\xi}$. Therefore it is enough to show that
  \begin{equation}\label{BZ11e}
\mathcal{A}_{\mathbf{K}, E_{\nu}}(R) \rightarrow Y^o(R) 
    \end{equation}
 is surjective for a
  strictly henselian local ring $R$. For an algebraically closed field
  $R$ both sides have the same coarse moduli space. Therefore the map
  is bijective in this case. In general the residue field $\kappa_R$ of $R$ is
  algebraically closed, since we are in characteristic $0$. 
   We consider a point
   $(A, \tilde{\lambda}, \tilde{\eta}) \in \mathcal{A}^{\bullet}_{\mathbf{K}_U^{\bullet}}(R) = Y(R)$  
   which is in $Y^{o}(R)$. Over $\kappa_R$ this point is in the image of
   (\ref{BZ11e}). By Lemma \ref{BZ4l}, the preimage by (\ref{BZ11e}) has the form
  $(A_{\kappa_R}, \bar{\lambda}, \bar{\eta})$ for some 
   $\lambda \in \tilde{\lambda}$ and $\eta \in \tilde{\eta}$. This is justified
   because the reduction to $\kappa_R$ defines a bijection between the class
   $\tilde{\lambda}$ on $A$ and its reduction on $A_{\kappa_R}$. The same applies
   to $\tilde{\eta}$. We must verify that
   $(A, \lambda, \eta)$ defines a point of $\mathcal{A}_{\mathbf{K}, E_{\nu}}(R)$. 
  Since there is no difference of a rigidification $\eta$ over $\kappa_R$ 
  or over $R$, this is already decided over $\kappa_R$. This proves that
  (\ref{BZ11e}) is bijective. Consequently the arrows of (\ref{et-sheaf1e})
  are isomorphism and therefore the functor $Y^o$ is representable.

  Now we deduce the representability of $Y$. Let
  $g \in G^{\bullet}(\mathbb{A}_f)$. We already noted
that $g\mathbf{K}^{\bullet}g^{-1} \cap G(\mathbb{A}_f) = g\mathbf{K}g^{-1}$. The
  multiplication by $g$ induces an isomorphism
  \begin{equation}\label{BZ12e}
    \mathcal{A}^{\bullet}_{g\mathbf{K}^{\bullet}_U g^{-1}, E_{\nu}}
    \overset{\sim}{\longrightarrow}
    \mathcal{A}^{\bullet}_{\mathbf{K}^{\bullet}_U, E_{\nu}} \rightarrow Y . 
  \end{equation}
  We have shown that $\Sh(G,h)_{g\mathbf{K}g^{-1}, E_{\nu}}$ is an open and closed
  subfunctor of the sheafification of the left hand side of (\ref{BZ12e}).
  (We note that the same $U=U_{N\ell}$ suffices for each $g\in G^{\bullet}(\mathbb{A}_f)$.)  
  Taking the composite with  (\ref{BZ12e}), we obtain an open and closed
  immersion
  \begin{equation}\label{BZ13e}
\Sh(G,h)_{g\mathbf{K}g^{-1}, E_{\nu}} \rightarrow Y.
    \end{equation}
 Its image is equal to the pullback of
 $\Sh(G,h)_{g\mathbf{K}g^{-1}, E_{\nu}} \overset{g}{\rightarrow} \Sh(G^{\bullet},h)_{\mathbf{K}^{\bullet}_U, E_{\nu}}$
 by the natural morphism
 $Y \rightarrow \Sh(G^{\bullet},h)_{\mathbf{K}^{\bullet}_U, E_{\nu}}$.
 Therefore (\ref{BZ13e}) gives, for varying   $g \in G^{\bullet}(\mathbb{A}_f)$, an open
 covering of $Y$ by representable subfunctors. 
\end{proof}  

For later use we formulate a variant of the last argument.

\begin{lemma}\label{BZCover1l} 
  Let $\mathbf{K}^{\bullet} \subset G^{\bullet}(\mathbb{A}_f)$ be an open  compact
  subgroup and let $\mathbf{K} = G(\mathbb{A}_f) \cap \mathbf{K}^{\bullet}$. 
  Assume that $U \subset (F \otimes \mathbb{A}_f)^{\times}$ is a principal congruence subgroup as constructed in the
  proof of Proposition \ref{BZ4p}. Then for all $g \in G^{\bullet}(\mathbb{A}_f)$ the canonical map
  \begin{equation}\label{BZ28e}
\Sh(G, h)_{g\mathbf{K}g^{-1}} \rightarrow \Sh(G^{\bullet}, h)_{g\mathbf{K}^{\bullet}_U g^{-1}} 
  \end{equation}
  is an open and closed immersion. The composite of this map with
$g\colon\Sh(G^{\bullet},h)_{g\mathbf{K}^{\bullet}_U g^{-1}}\rightarrow\Sh(G^{\bullet},h)_{\mathbf{K}^{\bullet}_U}$
  gives an open and closed immersion,
  \begin{displaymath}
    \varkappa_g: \Sh(G, h)_{g\mathbf{K}g^{-1}} \rightarrow
    \Sh(G^{\bullet}, h)_{\mathbf{K}^{\bullet}_U}  .
    \end{displaymath}
The maps $\varkappa_g$ for varying $g \in G^{\bullet}(\mathbb{A}^p_f)$ are an open
covering of $\Sh(G^{\bullet}, h)_{\mathbf{K}^{\bullet}_U}$. 

If the group $\mathbf{K}^{\bullet}$ satisfies the assumptions of Proposition \ref{BZ7p}, then 
the set of maps  
  \begin{displaymath}
    \{\varkappa_g: \Sh(G, h)_{g\mathbf{K}g^{-1}} \rightarrow
    \Sh(G^{\bullet}, h)_{\mathbf{K}^{\bullet}} \}_{g \in G^{\bullet}(\mathbb{A}_f^p)}  
  \end{displaymath}
  is an \'etale covering by finite \'etale maps. 
\end{lemma}
\begin{proof}
  Only the last assertion remains to be proved. 
  Let $Z$ be a connected component of $\Sh(G,h)_{\mathbb{C}}$ and let 
  $Z^{\bullet} \in \Sh(G^{\bullet}, h)$ be its image as in the proof of Proposition
  \ref{BZ7p}. 
  As in that proof, it is enough to show that the sets
  $g Z^{\bullet}_{g \mathbf{K}^{\bullet}_U g^{-1}}$ cover $\Sh(G^{\bullet}, h)_{\mathbf{K}^{\bullet}_U}$, as 
   $g$ runs through all elements of $G^\bullet(\mathbb{A}_f^p)$. 

  We consider 
  $\tilde{G}^{\bullet} = \{b \in B^{{\rm opp}} \; | \; b'b \in F^{\times} \}$ as algebraic
  group over $F$. Then $\Res_{F/\mathbb{Q}} \tilde{G}^{\bullet} = G^{\bullet}$, cf.
  (\ref{Gpunkt2e}).  We consider the homomorphisms
  \begin{equation}\label{Gpunkt3e} 
    \begin{array}{rcc}
      \mu: \tilde{G}^{\bullet} & \rightarrow & F^{\times},\\
       b\; & \mapsto & b'b
    \end{array}
    \qquad 
      \begin{array}{rcc}
      \det: \tilde{G}^{\bullet} & \rightarrow & K^{\times}\\
      b\; & \mapsto & \Nm^o_{B/K}. 
      \end{array}
    \end{equation}
  Let $\tilde{T}^{\bullet}$ be the algebraic torus over $F$ given by
  \begin{displaymath}
    \tilde{T}^{\bullet}(F) = \{(f, k) \in F^{\times} \times K^{\times} \; | \;
    f^2 = k \bar{k}  \}. 
  \end{displaymath}
  By (\ref{Gpunkt3e}) we obtain a homomorphism
  $\nu:\tilde{G}^{\bullet}\rightarrow\tilde{T}^{\bullet}$, $b \mapsto (\mu (b),\det b)$.
  Let $\tilde{H}^{\bullet}$ be the kernel of this map. One can check that
$\tilde{H}^{\bullet} \times_{\Spec F}\Spec\mathbb{C}\cong\mathrm{SL}_2(\mathbb{C})$. 
  Therefore we obtain an exact sequence
  \begin{displaymath}
    0 \rightarrow  \tilde{G}^{\bullet}_{\rm der} \rightarrow \tilde{G}^{\bullet}
    \overset{\nu}{\longrightarrow} \tilde{T}^{\bullet} \rightarrow 0 ,
  \end{displaymath}
  where the derived group is simply connected. By \cite[Thm. 2.4]{D-TS}, we obtain a
  bijection
  \begin{equation}\label{BZ29e}
    \pi_0(\Sh(G^{\bullet}, h)_{\mathbf{K}^{\bullet}_U} \overset{\sim}{\longrightarrow}
    \nu(\mathbf{K}^{\bullet}_{\infty} \times \mathbf{K}^{\bullet}_U)\backslash
    \tilde{T}^{\bullet}(\mathbb{A}_{F})/\tilde{T}^{\bullet}(F). 
  \end{equation}
  The right hand side may also be written as 
  $\nu(\mathbf{K}^{\bullet}_{\infty} \times \mathbf{K}^{\bullet}_U)\backslash
    T^{\bullet}(\mathbb{A})/T^{\bullet}(\mathbb{Q})$. 
  
Because the cyclic extension $K/F$ splits the torus $\tilde{T}^{\bullet}$,
 weak approximation holds for $\tilde{T}^{\bullet}$, cf. \cite[Thm. 6.36]{V}.
In particular $\tilde{T}^{\bullet}(F)$ is dense in
  $\tilde{T}^{\bullet}(F \otimes_{\mathbb{Q}} \BR) \tilde{T}^{\bullet}(F \otimes_{\mathbb{Q}} \mathbb{Q}_p)$.  
  This implies that
  \begin{displaymath}
\tilde{T}^{\bullet}(F) \nu(\mathbf{K}^{\bullet}_{\infty} \times \mathbf{K}^{\bullet}_U)
    \tilde{T}^{\bullet}(\mathbb{A}^p_{F,f}) = \tilde{T}^{\bullet}(\mathbb{A}_{F}).
  \end{displaymath}
  Hence
  $\tilde{T}^{\bullet}(\mathbb{A}^p_{F,f}) = T^{\bullet}(\mathbb{A}^p_f)$
  acts transitively on the right hand side of (\ref{BZ29e}). Since
  $G^{\bullet}(\mathbb{A}^p_f) \rightarrow T^{\bullet}(\mathbb{A}^p_f)$ is surjective,
  the sets $g Z^{\bullet}_{g \mathbf{K}^{\bullet}_U g^{-1}}$ for
  $g \in G^{\bullet}(\mathbb{A}^p_f)$ cover $\Sh(G^{\bullet}, h)_{\mathbf{K}^{\bullet}_U}$.
  Therefore the last assertion of the proposition follows as in the proof of
  Proposition \ref{BZ7p}. 
\end{proof}

Let $\mathbf{K}^{\bullet}_p \subset G^{\bullet}(\mathbb{Q}_p)$ be the subgroup
associated to a choice of $\Lambda_p$, $\mathbf{M}^{\bullet}$ and
$\mathbf{K}_{\mathfrak{q}_i}$, for $i = 1, \ldots, s$, cf. (\ref{BZKpPkt1e}). 
 We set
$\mathbf{M} = \mathbb{Z}_p^{\times} \cap \mathbf{M}^{\bullet}$. We denote by
$\mathbf{K}_p \subset G(\mathbb{Q}_p)$ the subgroup associated to 
the choice of $\Lambda_p$, $\mathbf{M}$ and $\mathbf{K}_{\mathfrak{q}_i}$,
cf. (\ref{BZKp1e}). We see easily that
\begin{equation}\label{BZ34e}
\mathbf{K}_p = \mathbf{K}^{\bullet}_p \cap G(\mathbb{Q}_p).  
\end{equation}
Under these hypotheses, we have an integral version of Proposition \ref{BZ4p}. It concerns $ \widetilde{\mathcal{A}}^t_{\mathbf{K}}$ instead of $\CA_{\mathbf K}$ and $ \widetilde{\mathcal{A}}^{\bullet t}_{\mathbf{K}}$ instead of $\CA^\bullet_{\mathbf K}$. 
\begin{proposition}\label{BZ6p}
  We fix $M$ and $\ell$ as in Proposition \ref{Chevalley1p}, but we assume that
  both are prime to $p$. Let $\mathbf{K}^{\bullet}=\mathbf{K}^{\bullet}_p\mathbf{K}^{\bullet,p}\subset G^\bullet(\BA_f)$, with $\mathbf{K}^{\bullet}_p$ as in \eqref{BZKpPkt1e}, where
  $\mathbf{M}_{\mathfrak{p}_0} = O_{F_{\mathfrak{p}_0}}^{\times}$.  Let
  $\mathbf{K} = \mathbf{K}^{\bullet} \cap G^{\bullet}(\mathbb{A}_f)$. 
  
  Then there exists a power $N$ of $M$ such that for the open compact
  subgroup $U \subset (F \otimes \mathbb{A}_f)^{\times}$ of Proposition
  \ref{Chevalley1p},  the natural map
  \begin{equation}\label{BZ15e}
    \widetilde{\mathcal{A}}^t_{\mathbf{K}} \rightarrow
    \tilde{\mathcal{A}}^{\bullet t}_{\mathbf{K}_U^{\bullet}}
    \end{equation}
  is a monomorphism of functors. 
\end{proposition}
By assumption we have 
$\mathbf{K}^{\bullet} = \mathbf{K}^{\bullet}_p \mathbf{K}^{\bullet,p}$ and
$\mathbf{K}_U^{\bullet} = \mathbf{K}^{\bullet}_p \mathbf{K}^{\bullet, p}_U$. The group
$\mathbf{K} = \mathbf{K}^{\bullet} \cap G(\mathbb{A}_f) = \mathbf{K}^{\bullet}_U \cap G(\mathbb{A}_f)$
has a similiar decomposition and therefore the functor
$\tilde{\mathcal{A}}^t_{\mathbf{K}}$ makes sense. Similarly to the proof of Proposition \ref{BZ4p}, we need two lemmas which are analogous to Lemmas \ref{BZ4l} and \ref{BZ42}.  
\begin{lemma}\label{BZ7l}
Let $(A, \bar{\lambda}, \bar{\eta}^p,  (\bar{\eta}_{\mathfrak{q}_j})_j, \xi_{ p})$ be a point
of $\tilde{\mathcal{A}}^t_{\mathbf{K}}(R)$, with image 
$(A', \tilde{\lambda}', \tilde{\eta'}^p,  (\bar{\eta'}_{\mathfrak{q}_j})_j, (\xi'_{ \mathfrak{p}_i})_i)$
 in $\tilde{\mathcal{A}}^{\bullet t}_{\mathbf{K}_U^{\bullet}}(R)$. Then there is a polarization $\lambda' \in \tilde{\lambda}'$ and a level structure 
$\eta'^p \in \tilde{\eta}'^p$ and an element
$\xi'_{p}(\lambda') \in \mathbb{Z}_p^{\times}$ such that for $i = 0, \ldots, s$ 
\begin{displaymath}
  \xi'_{p}(\lambda') \equiv \xi'_{ \mathfrak{p}_i} \; 
  \mod \mathbf{M}^{\bullet}_{\mathfrak{p}_i} 
  \end{displaymath}
and such that the point
$(A, \bar{\lambda}, \bar{\eta}^p, (\bar{\eta}_{\mathfrak{q}_j})_j,  \xi_{ p})$
is isomorphic to
$(A', \bar{\lambda}', \bar{\eta'}^p,  (\bar{\eta'}_{\mathfrak{q}_j})_j, \xi'_{p}(\lambda'))$.
The function $\xi'_{p}$ on $\bar{\lambda}'$ is given by
$\xi'_{p}(u \lambda') = u \xi'_{p}(\lambda')$ for $u \in U_p(\mathbb{Q})$.  
\end{lemma}
\begin{proof}
  The proof is similar to that of Lemma \ref{BZ4l}. We may assume that
  $S = \Spec R$ is connected. Then we can argue over a geometric point $\bar{s}$ of
  $S$, as explained after Definition \ref{BZAK1d}. We choose arbitrarily polarizations 
  $\lambda' \in \tilde{\lambda}'$ and $\lambda \in \bar{\lambda}$, and prime-to-$p$ level structures
  $\eta'^{p} \in \tilde{\eta'}^p$ and  $\eta^p \in \bar{\eta}^p$, and $p$-level structures 
  $\eta'_{ \mathfrak{q}_j} \in \bar{\eta'}_{ \mathfrak{q}_j}$ and 
  $\eta_{ \mathfrak{q}_j} \in \bar{\eta}_{ \mathfrak{q}_j}$. By assumption there exists an isogeny $\alpha: A' \rightarrow A$ of order
  prime to $p$ such that
   \begin{equation*}
    \begin{aligned} 
      \alpha^{*} (\lambda) = f \lambda', \quad \alpha \circ \eta^p \dot{c}^p = \eta'^p, \quad
      \alpha \circ \eta'_{ \mathfrak{q}_j} c_{\mathfrak{q}_j} = \eta_{ \mathfrak{q}_j}, \quad
      \xi'_{\mathfrak{p}}(f \lambda') \varepsilon_{\mathfrak{p}} =
      \xi_{p}(\lambda), 
      \end{aligned}
  \end{equation*}
  where $f \in U_p(F)$ is totally positive,
  $\dot{c}^p \in G^{\bullet}(\mathbb{A}_f^p)$ and
  $c_{\mathfrak{q}_j} \in \mathbf{K}_{\mathfrak{q}_j}$ and
  $\varepsilon_{\mathfrak{p}} \in \mathbf{M}_{\mathfrak{p}}^{\bullet}$.
  From this the assertion follows easily. 
  \end{proof}

\begin{lemma}\label{Leminj}
  Let $\mathbf{K}^{\bullet} \subset G^{\bullet}(\mathbb{A}_f)$ and $U$ as in
  Proposition \ref{BZ6p}. Let
  \begin{displaymath}
    (A_1, \bar{\lambda}_1, \bar{\eta}_1^{p},(\bar{\eta}_{1,\mathfrak{q}_j})_j,\xi_{1,p}),
    \; 
    (A_2, \bar{\lambda}_2, \bar{\eta}_2^{p},  (\bar{\eta}_{2,\mathfrak{q}_j})_j,\xi_{2,p})
  \end{displaymath}
  be two points of $\tilde{\mathcal{A}}^t_{\mathbf{K}}(R)$ which have the same
  image in $\tilde{\mathcal{A}}^{\bullet t}_{\mathbf{K}^{\bullet}}(R)$. 
Then there exists a totally positive $f \in O_{F}^{\times}$, an element
  $\theta \in (O_F \otimes \hat{\mathbb{Z}})^{\times}$ such that $f \theta = a \in \hat{\mathbb{Z}}^{\times}$ and 
  $\theta_p \in \prod_{i=0}^s \mathbf{M}^{\bullet}_{\mathfrak{p}_i}$, and an element $\dot{c} \in \mathbf{K}_U^{\bullet,p}$ with
  $\theta^p = \dot{c}' \dot{c}$, such that the point
$(A_2, \bar{\lambda}_2, \bar{\eta}_2^{p},  (\bar{\eta}_{2,\mathfrak{q}_j})_j,\xi_{2,p})$ 
  is isomorphic to
  \begin{displaymath}
    (A_1, \bar{\lambda}_1 f, \bar{\eta}_1^{p} \dot{c}, 
    (\bar{\eta}_{1,\mathfrak{q}_j})_j,\xi'_{1,p}). 
  \end{displaymath}
  Here the function $\xi'_{1,p}$ on $\bar{\lambda}_1 f$ is defined by
  \begin{displaymath}
 \xi'_{1,p}(\lambda_1 f) = a \xi_{1,p}(\lambda_1). 
    \end{displaymath}
\end{lemma}
\begin{proof}
  We fix a polarization $\lambda_1 \in \bar{\lambda}_1$ and
  $\eta_1 \in \bar{\eta}_1$. By Lemma \ref{BZ7l}, the point
$(A_2, \bar{\lambda}_2, \bar{\eta}_2^{p}, (\bar{\eta}_{2,\mathfrak{q}_j})_j,\xi_{2,p})$
  is isomorphic to a point of the form
  \begin{displaymath}
    (A_1, \bar{\lambda}_1 f, \bar{\eta_1^{p} \dot{c}}, 
    (\bar{\eta}_{1,\mathfrak{q}_j})_j, \xi'_{1,p}). 
    \end{displaymath}
  The value $\xi'_{1,p}(f\lambda_1) \in \mathbb{Z}_p$ satisfies the following
  congruence in $O^{\times}_{F_{\mathfrak{p}_i}}$ for each $i = 0, \ldots s$,  
  \begin{displaymath}
    \xi'_{1,p}(f \lambda_1) \equiv f \xi_{1,p}(\lambda_1) 
    \mod \mathbf{M}^{\bullet}_{\mathfrak{p}_i} .  
    \end{displaymath}
  This implies that there is an element
  $\theta_p \in \prod_{i=1}^s \mathbf{M}^{\bullet}_{\mathfrak{p}_i}$ such that
  $f \theta_p = a_p \in \mathbb{Z}_p^{\times}$. Then we obtain
  \begin{displaymath}
    \xi'_{1,p}(f \lambda_1) \equiv a_p \xi_{1,p}(\lambda_1) 
    \mod \mathbf{M}^{\bullet}_{\mathfrak{p}_i} .  
    \end{displaymath}
  Moreover an $\mathbb{A}_f^{p}$-version of (\ref{BZ18e}) shows that there
  is an element $a^p \in (\BA_f^p)^{\times}$ such that
  $f \dot{c}' \dot{c} = a^p$. We have the right to multiply $f$ by an
  element of $U_p(\mathbb{Q})$. Therefore we may assume that
  $a^p \in  \hat{\mathbb{Z}}^{p}$. The result follows by setting
  $a = a_p a^p$.    
\end{proof}
\begin{proof}[Proof of Proposition \ref{BZ6p}]
  As in the proof of Proposition \ref{BZ4p}, it is enough to show that
    \begin{displaymath}
\tilde{\mathcal{A}}^t_{\mathbf{K}}(R) \rightarrow
\tilde{\mathcal{A}}^{\bullet t}_{\mathbf{K}_U^{\bullet}}(R) 
    \end{displaymath}
    is injective if $\Spec R$ is connected.
  Assume that we are given two points as in  Lemma \ref{Leminj} which are mapped
  to the same point of $\tilde{\mathcal{A}}^{\bullet}_{\mathbf{K}_U^{\bullet}}(R)$. 
  For suitable $U$ we conclude as in the proof of Proposition \ref{BZ4p} that
  $f = g^2$, for some $g \in O_{F}^{\times} \cap \mathbf{K}^{\bullet}$.
  If in the argument of that proof we choose $m$ big enough we may assume that
  $g \in \mathbf{M}^{\bullet}$. We obtain
$a_p = f \theta_p\in\mathbf{M}^{\bullet} \cap \mathbb{Z}_p^{\times} = \mathbf{M}$
  (see before (\ref{BZ34e})). The multiplication by $g$ induces an isomorphism
  \begin{equation}\label{BZ19e} 
g: (A_1, \bar{\lambda}_1 f, \bar{\eta}_1^{p} \dot{c}, (\bar{\eta}_{1,\mathfrak{q}_j})_j,  \xi'_{1,p}) \rightarrow
    (A_1, \bar{\lambda}_1, \bar{\eta}_1^{p},(\bar{\eta}_{1,\mathfrak{q}_j})_j,\xi_{1,p}).
  \end{equation}
  Indeed, we have
  $g^{*}(\lambda_1) = \lambda_1 g^2$ and the morphism (\ref{BZ19e}) respects the data
  $\bar{\eta}_1^{p} \dot{c}$ and $\bar{\eta}_1^{p}$, comp.   the proof of Proposition
  \ref{BZ4p}. Furthermore, for $\lambda_1 \in \bar{\lambda}_1$ we obtain
  \begin{displaymath}
    g^{*}( \xi_{1,p}(\lambda_1 f)) = g^{*} (\xi_{1,p})(g^{*} (\lambda_1)) :=
    \xi_{1,p}(\lambda_1) = a_p \xi_{1,p}(\lambda_1) = \xi'_{1,p}(\lambda_1 f). 
  \end{displaymath}
  The second to last equation holds because $a_p \in \mathbf{M}$. 
  \end{proof}

We recall that we assume that $D_{\mathfrak{p}_0}$ is a quaternion division
algebra, cf. (\ref{Dsplit-in0}).  In this case, we have the following integral version of Proposition \ref{BZ7p}. 

\begin{proposition}\label{BZ8p} 
   Let $\mathbf{K}^{\bullet}=\mathbf{K}^{\bullet}_p\mathbf{K}^{\bullet,p}\subset G^\bullet(\BA_f)$, with $\mathbf{K}^{\bullet}_p$ as in \eqref{BZKpPkt1e}, where
  $\mathbf{M}_{\mathfrak{p}_0} = O_{F_{\mathfrak{p}_0}}^{\times}$.   We set
  $\mathbf{K} = G(\mathbb{A}_f) \cap \mathbf{K}^{\bullet}$. We
  assume that there is an $O_K$-lattice $\Gamma\subset V$ and an integer
  $m \geq 3$ prime to $p$ such that for each $g \in \mathbf{K}^{\bullet}$ we have
  $g \Gamma \subset \Gamma$ and such that $g$ acts trivially on
  $\Gamma/m \Gamma$. (In this case  $\tilde{\mathcal{A}}^t_{\mathbf{K}}$ is
  representable.) Let $U$ be as in Proposition \ref{BZ4p}. 

  Let $\tilde{\mathsf{A}}^t_{\mathbf{K}} $ be the $\Spec O_{E_{\nu}}$-scheme which represents the functor $\tilde{\mathcal{A}}^t_{\mathbf{K}}$ and let
  $\tilde{\mathsf{A}}^{\bullet t}_{\mathbf{K}^{\bullet}_{U}}$
  be the coarse moduli scheme of
  $\tilde{\mathcal{A}}^{\bullet t}_{\mathbf{K}^{\bullet}_{U}}$. It is a normal scheme
  which is proper over $\Spec O_{E_{\nu}}$. 

  The  canonical map 
  $\tilde{\mathsf{A}}^t_{\mathbf{K}} \rightarrow \tilde{\mathsf{A}}^{\bullet t}_{\mathbf{K}^{\bullet}_{U}}$ 
  is an open and closed immersion. The arrow 
  $\tilde{\mathcal{A}}^{\bullet t}_{\mathbf{K}^{\bullet}_{U}} \rightarrow \tilde{\mathsf{A}}^{\bullet t}_{\mathbf{K}^{\bullet}_{U}}$ 
  is the \'etale sheafification of the presheaf 
  $\tilde{\mathcal{A}}^{\bullet t}_{\mathbf{K}^{\bullet}_{U}}$ on the big \'etale site.  
\end{proposition}
\begin{proof} 
  The scheme $\tilde{\mathsf{A}}^t_{\mathbf{K}}$ is regular and the
  morphism $\tilde{\mathsf{A}}^t_{\mathbf{K}} \rightarrow \Spec O_{E_{\nu}}$ is
  generically smooth and proper.
  Its special fibre is a divisor with normal crossings.
  This follows from deformation theory because the $p$-divisible group
  $X_{\mathfrak{q}_0}$ is a special formal $O_{B_{\mathfrak{q}_0}}$-module in the sense
  of Drinfeld. The properness follows from a standard argument using that
  $B$ is a division algebra, cf. \cite[Prop. 4.1]{Dr}.

For the proof  that 
a coarse moduli scheme $\tilde{\mathsf{A}}^{\bullet t}_{\mathbf{K}^{\bullet}_{U}}$
exists we refer to \cite[1.7 Satz]{Z-sR}. Because this moduli scheme is
obtained as a quotient of a normal scheme by a finite group, the coarse
moduli scheme is normal. Since
  $\tilde{\mathsf{A}}^t_{\mathbf{K},E_{\nu}}\subset\tilde{\mathsf{A}}^t_{\mathbf{K}}$
  is an open dense subset of a scheme which is locally integral we obtain
  a bijection between connected components
    \begin{equation*}
      \pi_0 (\tilde{\mathsf{A}}^t_{\mathbf{K}, E_{\nu}})  \rightarrow 
      \pi_0 (\tilde{\mathsf{A}}^t_{\mathbf{K}}), \quad\quad
      Z  \mapsto  \bar{Z} .
    \end{equation*}
  The same is true for the connected components of
  $\tilde{\mathsf{A}}^{\bullet t}_{\mathbf{K}^{\bullet}_{U},E_{\nu}}$ and
  $\tilde{\mathsf{A}}^{\bullet}_{\mathbf{K}^{\bullet t}_{U}}$. 
  We claim that 
  \begin{equation}\label{BZ22e}
    \tilde{\mathsf{A}}^t_{\mathbf{K}} \rightarrow
    \tilde{\mathsf{A}}^{\bullet t}_{\mathbf{K}^{\bullet}_{U}} 
  \end{equation}
  is an open and closed immersion. Indeed, the morphism (\ref{BZ22e}) is proper
  because $\tilde{\mathsf{A}}_{\mathbf{K}}^t$ is proper over $\Spec O_{E_{\nu}}$. The
  general fiber over $E_{\nu}$ of this morphism coincides up to a Galois twist
  with (\ref{BZ10e}) and is therefore an open and closed immersion. Let
  $Z \subset \tilde{\mathsf{A}}^t_{\mathbf{K},E_{\nu}}$
  be a connected component which we also regard as a connected component of
  $\tilde{\mathsf{A}}^{\bullet t}_{\mathbf{K}^{\bullet}_{U}, E_{\nu}}$. We consider the
  closures $\bar{Z} \subset \tilde{\mathsf{A}}^t_{\mathbf{K}}$ and
  $\bar{Z}^{\bullet} \subset \tilde{\mathsf{A}}^{\bullet t}_{\mathbf{K}^{\bullet}_{U}}$ of $Z$.
  These are connected components and the morphism (\ref{BZ22e}) induces
  an birational proper morphism $\bar{Z} \rightarrow \bar{Z}^{\bullet}$ of normal
  schemes. If we take the values in some algebraically closed field the last
  morphism becomes injective. This follows from Proposition \ref{BZ6p} and
  the definition of a coarse moduli problem. Therefore
  $\bar{Z} \rightarrow \bar{Z}^{\bullet}$ is an isomorphism, and the claim is proved.

  Let $Y$ be the \'etale sheafification of
  $\tilde{\mathcal{A}}^{\bullet t}_{\mathbf{K}^{\bullet}_{U}}$ on the big \'etale site.
  We consider the following commutative diagram
  \begin{equation}\label{BZ27e}
   \begin{aligned}
   \xymatrix{
     \widetilde{\mathcal{A}}^t_{\mathbf{K}} \ar[r] \ar[d] &
     \tilde{\mathcal{A}}^{\bullet t}_{\mathbf{K}^{\bullet}_{U}} \ar[d]\\
     Y^{o} \ar[r] \ar[d] & Y \ar[d]\\
     \tilde{\mathsf{A}}^t_{\mathbf{K}} \ar[r] &
     \tilde{\mathsf{A}}^{\bullet t}_{\mathbf{K}^{\bullet}_{U}} .\\
   }
   \end{aligned}
  \end{equation}
  Here $Y^{o}$ is defined to be the fiber product in the lower square. We know
  that the horizontal arrows are monomorphisms and the two lower ones are
  open and closed immersions.
  We will show that $\tilde{\mathcal{A}}^t_{\mathbf{K}} \rightarrow Y^{o}$ is an
  isomorphism of sheaves. 
  
  Let $T$ be a noetherian scheme over $\Spec O_{E_{\nu}}$ and
  $\mathcal{O}_{T,t}^{sh}$  the strict henselization at a geometric point $t$
  of $T$. We have to show that the induced map of stalks
  \begin{displaymath}
(\tilde{\mathcal{A}}^t_{\mathbf{K}})_{T,t} \rightarrow (Y^o)_{T,t}   
    \end{displaymath}
  is bijective. (The stalks are the same as the stalks of the restriction
  of both sheaves to the small \'etale site $T_{et}$.) Since
  $\tilde{\mathcal{A}}^t_{\mathbf{K}}$ commutes with direct limits, we obtain
  $(\tilde{\mathcal{A}}^t_{\mathbf{K}})_{T,t} = \widetilde{\mathcal{A}}^t_{\mathbf{K}}(\mathcal{O}_{T,t}^{sh})$.
  The same is true for the presheaf
  $\tilde{\mathcal{A}}^{\bullet t}_{\mathbf{K}^{\bullet}_{U}}$. Therefore we obtain 
$(Y^o)_{T,t}\subset Y_{T,t} =\tilde{\mathcal{A}}^{\bullet t}_{\mathbf{K}^{\bullet}_{U}}(\mathcal{O}_{T,t}^{sh})$.
  This subset consists of the points on the right hand side which are mapped to
  $\tilde{\mathsf{A}}^t_{\mathbf{K}}(\mathcal{O}_{T,t}^{sh})$.
 
  Let $L$ the residue class field of $\mathcal{O}_{T,t}^{sh}$ which is separably
  closed. We firstly show that
  \begin{equation}\label{BZ25e}
\tilde{\mathcal{A}}^t_{\mathbf{K}}(L) \rightarrow Y^o(L)  
  \end{equation}
  is bijective. Equivalently we may show that
  $Y^{o}(L) \rightarrow \tilde{\mathsf{A}}^t_{\mathbf{K}}(L)$ is bijective.
  Clearly this map is surjective because
  $\tilde{\mathcal{A}}^t_{\mathbf{K}} \cong \tilde{\mathsf{A}}^t_{\mathbf{K}}$.
  Let
  $\theta_1,\theta_2\in Y^o(L)\subset\tilde{\mathcal{A}}^{\bullet t}_{\mathbf{K}^{\bullet}_{U}}(L)$
  be two elements with the same image in
  $\tilde{\mathsf{A}}^t_{\mathbf{K}}(L)$. By the properties of a coarse moduli
  scheme, the map 
$\tilde{\mathcal{A}}^{\bullet t}_{\mathbf{K}^{\bullet}_{U}}(\bar{L}) \rightarrow \tilde{\mathsf{A}}^{\bullet t}_{\mathbf{K}^{\bullet}_{U}}(\bar{L})$
  is bijective. We conclude that $\theta_{1,\bar{L}} = \theta_{2,\bar{L}}$ holds for the
  base change. Therefore we find a finite totally inseparable extension $N$
  of $L$ such that $\theta_{1,N} = \theta_{2,N}$. If the last two points
  are represented by the data 
$(A_1,\tilde{\lambda}_1,\tilde{\eta}^p_1,(\bar{\eta}_{1,\mathfrak{q}_j})_j, (\xi_{1, \mathfrak{p}_i})_i)$
  and
$(A_2,\tilde{\lambda}_2,\tilde{\eta}^p_2,(\bar{\eta}_{2,\mathfrak{q}_j})_j, (\xi_{2, \mathfrak{p}_i})_i)$, 
  we conclude by the rigidity of abelian varieties, applied to the nilimmersion
  $N \otimes_{L} N \rightarrow N$, that $\theta_1 = \theta_2$.

  Now we consider the map
  \begin{displaymath}
    \widetilde{\mathcal{A}}^t_{\mathbf{K}}(\mathcal{O}_{T,t}^{sh}) \rightarrow
    Y^o(\mathcal{O}_{T,t}^{sh}). 
  \end{displaymath}
  This map is clearly injective. We show that it is surjective. We consider
  a point
  \begin{equation}\label{BZ26e} 
    (A_1, \tilde{\lambda}_1, \tilde{\eta}^p_1,     (\bar{\eta}_{1,\mathfrak{q}_j})_j, (\xi_{1, \mathfrak{p}_i})_i) \in Y^o(\mathcal{O}_{T,t}^{sh}) \subset
    \tilde{\mathcal{A}}^{\bullet t}_{\mathbf{K}^{\bullet}_{U}}(\mathcal{O}_{T,t}^{sh}).
  \end{equation}
  Over $L$ this point is in the image of (\ref{BZ25e}). By Lemma \ref{BZ7l},
  the preimage has the form
  $(A_{1,L}, \bar{\lambda}_1, \bar{\eta}^p_1, (\bar{\eta}_{1,\mathfrak{q}_j})_j, (\xi_{2, \mathfrak{p}_i})_i)$. 
  Since $\bar{\lambda}_1 \subset \tilde{\lambda}_1$, the polarizations in
  $\bar{\lambda}_1$ lift to polarizations of $A_1$ which are principal in $p$. 
  Since there is no difference between a rigidification over
  $\mathcal{O}_{T,t}^{sh}$ and over the residue class field $L$, we see that 
  $(A_{1}, \bar{\lambda}_1, \bar{\eta}^p_1,  (\bar{\eta}_{1,\mathfrak{q}_j})_j, (\xi_{2, \mathfrak{p}_i})_i)$
  is a point of $\tilde{\mathcal{A}}^t_{\mathbf{K}}(\mathcal{O}_{T,t}^{sh})$
  which is mapped to the point (\ref{BZ26e}). We have proved that the two
  vertical arrows on the left hand side of diagram (\ref{BZ27e}) are
  isomorphisms.

  To show that
$\tilde{\mathcal{A}}^{\bullet t}_{\mathbf{K}^{\bullet}_{U}}\rightarrow\tilde{\mathsf{A}}^{\bullet t}_{\mathbf{K}^{\bullet}_{U}}$ 
  is the \'etale sheafification, we can argue as in the proof of Proposition
\ref{BZ7p} if we substitute Lemma \ref{BZCover1l} by  Lemma \ref{Lemcover} below.   
\end{proof}

The group $G^{\bullet}(\mathbb{A}_f^p)$ acts on the projective system of the
functors $\tilde{\mathcal{A}}^{\bullet t}_{\mathbf{K}^{\bullet}}$ for varying
$\mathbf{K}^{\bullet,p} \subset G^{\bullet}(\mathbb{A}_f^p)$ via the datum
$\eta^p$ of Definition \ref{BZApkt4d}. More explicitly,  each
$g \in G^{\bullet}(\mathbb{A}_f^p)$ induces by multiplication an isomorphism
\begin{displaymath}
  g: \tilde{\mathcal{A}}^{\bullet t}_{g\mathbf{K}^{\bullet}g^{-1}} \rightarrow
  \tilde{\mathcal{A}}^{\bullet t}_{\mathbf{K}^{\bullet}}, 
\end{displaymath}
which induces an isomorphism of the coarse moduli spaces.
\begin{lemma}\label{Lemcover} 
  Let $\mathbf{K}^{\bullet}$ be as in Proposition \ref{BZ8p}. Then there is an
  open subgroup $U \subset (F \otimes \mathbb{A}_f)^{\times}$ such that for each
  $g \in G^{\bullet}(\mathbb{A}_f^p)$ the natural morphism
  \begin{equation}\label{BZ31e}
    \tilde{\mathsf{A}}^t_{g\mathbf{K}g^{-1}} \rightarrow 
    \tilde{\mathsf{A}}^{\bullet t}_{g\mathbf{K}^{\bullet}_{U}g^{-1}}  
  \end{equation}
  is an open and closed immersion. If we compose the immersion with the
  morphisms
  $g:\tilde{\mathsf{A}}^{\bullet t}_{g\mathbf{K}^{\bullet}_{U}g^{-1}} \rightarrow \tilde{\mathsf{A}}^{\bullet t}_{\mathbf{K}^{\bullet}_{U}}$,
  we obtain open and closed immersions
  \begin{equation}\label{BZ32e}
    \varkappa_g: \tilde{\mathsf{A}}^t_{g\mathbf{K}g^{-1}} \rightarrow
    \tilde{\mathsf{A}}^{\bullet t}_{\mathbf{K}^{\bullet}_{U}} .
    \end{equation}
  For varying $g \in G^{\bullet}(\mathbb{A}_f^p)$ the morphisms $\varkappa_g$
  are an open covering of $\tilde{\mathsf{A}}^{\bullet t}_{\mathbf{K}^{\bullet}_{U}}$. 
\end{lemma}
\begin{proof}
  We have already seen  that (\ref{BZ31e}) is an open and closed
  immersion, cf. (\ref{BZ22e}). Therefore the same is true of $\varkappa_g$. The general fibre
  of $\tilde{\mathsf{A}}^{\bullet t}_{\mathbf{K}^{\bullet}_{U}}$ is up to a Galois twist 
  $\Sh(G^{\bullet}, h)_{\mathbf{K}^{\bullet}_U, E_{\nu}}$. By Lemma \ref{BZCover1l}
  each connected component $Z$ of $\Sh(G^{\bullet}, h)_{\mathbf{K}^{\bullet}_U, E_{\nu}}$
  is in the image of $\varkappa_g$ for some $g \in G^{\bullet}(\mathbb{A}_f^p)$.
  Since $\tilde{\mathsf{A}}^{\bullet}_{\mathbf{K}^{\bullet}_{U}}$ is locally an integral
  scheme which is flat over $O_{E_{\nu}}$, each connected component  is of the form $\bar{Z}$. It is therefore in the image of the open
  and closed immersion $\varkappa_g$. Hence (\ref{BZ32e}) is indeed a
  covering. 
  \end{proof}
In the sequel, only the Shimura varieties ${\rm Sh}(G^\bullet, h)$ and ${\rm Sh}(G^\bullet, h^\bullet_D)$ attached to the group $G^\bullet$ will play a role. These are \emph{ramified} abelian Galois twists of each other, cf. Proposition \ref{BZ3c}. The following table summarizes the Shimura varieties and their relations to moduli functors.

$$
\vcenter{
\begin{Small}
\begin{tabular}{|c|c|c|c|}
\hline
Shimura variety&Moduli problem&First occurrence&Relation  \\
\hline
${\rm Sh}(G^\bullet, h)$ &$\CA^\bullet$, resp. $\CA^{\bullet, bis}$ &Def. \ref{BZApkt3d}/\ref{BZApkt3altd}, resp. Def. \ref{BZApkt4d} & coarse moduli scheme \\

\hline
${\rm Sh}(G^\bullet, h^\bullet_D)$  &$\CA^{\bullet, t}$  &  Def. \ref{BZsApkt4d} &coarse moduli sch. of unram. twist \\
\hline
\end{tabular}
\end{Small}
}
$$

\medskip

The integral model  $\widetilde{\Sh}_{\mathbf{K}^{\bullet}}(G^{\bullet}, h^{\bullet}_{D})$ of  ${\Sh}_{\mathbf{K}^{\bullet}}(G^{\bullet}, h^{\bullet}_{D})$ over $O_{E_\nu}$ (defined if   $\mathbf{K}^{\bullet}=\mathbf{K}^{\bullet}_p\mathbf{K}^{\bullet,p}\subset G^\bullet(\BA_f)$, with $\mathbf{K}^{\bullet}_p$ as in \eqref{BZKpPkt1e}, where
  $\mathbf{M}_{\mathfrak{p}_0} = O_{F_{\mathfrak{p}_0}}^{\times}$) is defined by twisting back the integral extension $\tilde{\CA}^{\bullet, t}$ of $\CA^{\bullet, t}$.

\section{The $\mathrm{RZ}$-spaces}\label{s:RZ}

In this section, we discuss the $\mathrm{RZ}$-spaces 
needed for the $p$-adic uniformization of the Shimura varieties of the last
section.  

We first discuss the banal case of a prime ideal  $\mathfrak{p}_i$, for $i \neq 0$. 
\begin{definition}
  Let $S$ be an $O_{E_{\nu}}$-scheme. The category $\mathcal{P}_{\mathfrak{p}_i}(S)$
  is the category of all triples $(Y, \iota, \bar{\lambda})$, where
  $Y$ is a $p$-divisible group of height $8[F_{\mathfrak{p}_i} : \mathbb{Q}_p]$
  over $S$, where $\iota: O_{B_{\mathfrak{p}_i}} \rightarrow \End Y $
      is a $\mathbb{Z}_p${\rm -Alg}ebra homomorphism, and where $\bar{\lambda}$ is a
    $O_{F_{\mathfrak{p}_i}}^{\times}$-homogeneous polarization of $Y$ such that each
    $\lambda \in \bar{\lambda}$ is principal. We demand that the Rosati involution associated to
    $\lambda \in \bar{\lambda}$ is compatible with the involution
    $b\mapsto b^\star$ on $B_{\mathfrak{p}_i}$ with respect to $\iota$. 
    The decomposition
    $O_{B_{\mathfrak{p}_i}} = O_{B_{\mathfrak{q}_i}} \times O_{B_{\bar{\mathfrak{q}}_i}}$
induces a composition $Y = Y_{\mathfrak{q}_i} \times Y_{\bar{\mathfrak{q}}_i}$.
  We demand moreover that the $p$-divisible group $Y_{\mathfrak{q}_i}$
  is \'etale. 
\end{definition}

The definition implies that
$\lambda = \lambda_{\mathfrak{q}_i} \oplus \lambda_{\bar{\mathfrak{q}}_i}$ where
$\lambda_{\mathfrak{q}_i}:Y_{\mathfrak{q}_i}\rightarrow(Y_{\bar{\mathfrak{q}}_i})^{\wedge}$ and  
$\lambda_{\bar{\mathfrak{q}}_i}: Y_{\bar{\mathfrak{q}}_i} \rightarrow (Y_{\mathfrak{q}_i})^{\wedge}$ 
are isomorphisms to the dual $p$-divisible groups such that
$\lambda_{\bar{\mathfrak{q}}_i} = - (\lambda_{\mathfrak{q}_i})^{\wedge}$ and such that
for each $b_1 \in O_{B_{\mathfrak{q}_i}}$ and $b_2 \in O_{B_{\bar{\mathfrak{q}}_i}}$ the
following diagrams are commutative,
\begin{equation}\label{RZ4e}
  \begin{aligned}
   \xymatrix{
     Y_{\bar{\mathfrak{q}}_i} \ar[r]^{\iota(b_1^{\star})} \ar[d]_{\lambda_{\bar{\mathfrak{q}}_i}}
     & Y_{\bar{\mathfrak{q}}_i} \ar[d]^{\lambda_{\bar{\mathfrak{q}}_i}}\\
     (Y_{\mathfrak{q}_i})^{\wedge} \ar[r]_{\iota(b_1)^{\wedge}} &
     (Y_{\mathfrak{q}_i})^{\wedge}, \\ 
   }  \hspace{2cm} 
   \xymatrix{
     Y_{\mathfrak{q}_i} \ar[r]^{\iota(b_2^{\star})} \ar[d]_{\lambda_{\mathfrak{q}_i}}
     & Y_{\mathfrak{q}_i} \ar[d]^{\lambda_{\mathfrak{q}_i}}\\
     (Y_{\bar{\mathfrak{q}}_i})^{\wedge} \ar[r]_{\iota(b_2)^{\wedge}} &
     (Y_{\bar{\mathfrak{q}}_i})^{\wedge} .\\ 
   } 
   \end{aligned}
\end{equation}
Since one of these diagrams is the dual of the other it is enough to require the
commutativity of one of these diagrams.

We construct an object of $\mathcal{P}_{\mathfrak{p}_i}(\Spec\bar{\kappa}_{E_{\nu}})$
as follows. Let us denote the action of the Frobenius endomorphism on
$W(\bar{\kappa}_{E_{\nu}})$ by $\sigma$. Recall from (\ref{BZLambda1e}) the lattices $\Lambda_{\mathfrak{q}_i} $ and $\Lambda_{\bar{\mathfrak{q}}_i}$. We endow
$\Lambda_{\mathfrak{q}_i} \otimes_{\mathbb{Z}_p} W(\bar{\kappa}_{E_{\nu}})$ with
the structure of a Dieudonn\'e module by defining the action of the Frobenius
$F$ on this module by
\begin{displaymath}
  F(u_1 \otimes \xi_1) = p u_1 \otimes \sigma(\xi_1), \quad u_1 \in
  \Lambda_{\mathfrak{q}_i}, \; \xi_1 \in W(\bar{\kappa}_{E_{\nu}}),  
\end{displaymath}
and we endow
$\Lambda_{\bar{\mathfrak{q}}_i} \otimes_{\mathbb{Z}_p} W(\bar{\kappa}_{E_{\nu}})$
with a structure of a Dieudonn\'e module by defining the action of the
Frobenius on this module by 
\begin{displaymath}
  F(u_2 \otimes \xi_2) = u_2 \otimes \sigma(\xi_2), \quad u_2 \in
  \Lambda_{\bar{\mathfrak{q}}_i}, \; \xi_2 \in W(\bar{\kappa}_{E_{\nu}}). 
\end{displaymath}
The direct sum of these Dieudonn\'e modules defines a Dieudonn\'e module
structure on
$\Lambda_{\mathfrak{p}_i}\otimes_{\mathbb{Z}_p} W(\bar{\kappa}_{E_{\nu}})$.
We consider the perfect alternating $W(\bar{\kappa}_{E_{\nu}})$-bilinear form
\begin{equation}\label{RZ21e}  
\psi_W: \Lambda_{\mathfrak{p}_i}\otimes_{\mathbb{Z}_p} W(\bar{\kappa}_{E_{\nu}}) 
\times \Lambda_{\mathfrak{p}_i}\otimes_{\mathbb{Z}_p} W(\bar{\kappa}_{E_{\nu}})
\rightarrow W(\bar{\kappa}_{E_{\nu}}), 
\end{equation} 
cf. (\ref{BZLambda1e}).  
One checks easily that this is a bilinear form of Dieudonn\'e modules.
By covariant Dieudonn\'e theory,
$(\Lambda_{\mathfrak{p}_i}\otimes_{\mathbb{Z}_p} W(\bar{\kappa}_{E_{\nu}}), \psi_W)$
corresponds to a principally polarized $p$-divisible group
$(\Lambda_{\mathfrak{p}_i}^{pd}, \lambda_{\psi})$. We have the decomposition
$\Lambda_{\mathfrak{p}_i}^{pd} = \Lambda_{\mathfrak{q}_i}^{et} \oplus \Lambda_{\bar{\mathfrak{q}}_i}^{mult}$,
where the first factor is an \'etale $p$-divisible group and the second factor
is multiplicative. The action of $O_{B_{\mathfrak{p}_i}}$ 
on $\Lambda_{\mathfrak{p}_i}$ defines an action of $O_{B_{\mathfrak{p}_i}}$ on 
$\Lambda_{\mathfrak{p}_i}^{pd}$. We see that
$(\Lambda_{\mathfrak{p}_i}^{pd}, \bar{\lambda}_{\psi})$ is an object of the category
$\mathcal{P}_{\mathfrak{p}_i}(\bar{\kappa}_{E_{\nu}})$ and that each other object in
this category is isomorphic to it.

Recall the group $G^\bullet_{\mathfrak{p}_i}$, cf. (\ref{BZGpi1e}). An element $g \in G^\bullet_{\mathfrak{p}_i}$ induces a
quasi-isogeny of the $p$-divisible $O_{B_{\mathfrak{p}_i}}$-module 
$\Lambda_{\mathfrak{p}_i}^{pd}$ which respects the polarization $\lambda_{\psi}$
up to a factor in $F_{\mathfrak{p}_i}^{\times}$. In particular we conclude that 
\begin{displaymath}
  \mathbf{K}^{\bullet}_{\mathfrak{p}_i} \subset
  \Aut_{\mathcal{P}_{\mathfrak{p}_i}} (\Lambda_{\mathfrak{p}_i}^{pd}, \bar{\lambda}_{\psi}). 
  \end{displaymath}

We consider schemes $S$ over $\Spf O_{E_{\nu}}$ or equivalently
$O_{E_{\nu}}$-schemes $S$ where $p$ is locally nilpotent. We set
$\bar{S} = S \times_{\Spec O_{E_{\nu}}} \Spec \kappa_{E_{\nu}}$. 
\begin{definition}\label{RZ1d}
  Let $S$ be a scheme over $\Spf O_{\breve{E}_{\nu}}$ so that $\bar{S}$
  is a scheme over $\bar{\kappa}_{E_{\nu}}$. We denote by
  $\Lambda_{\mathfrak{q}_i, \bar{S}}^{et} = \Lambda_{\mathfrak{q}_i}^{et}
  \times_{\Spec \bar{\kappa}_{E_{\nu}}} \bar{S}$
  the base change. The unique lift to an \'etale $p$-divisible group over
  $S$ is denoted by $\Lambda_{\mathfrak{q}_i, S}^{et}$. This is a constant
  \'etale $p$-divisible group. 
  
  A rigidification of an object
  $(Y, \iota, \bar{\lambda}) \in \mathcal{P}_{\mathfrak{p}_i}(S)$ modulo 
  $\mathbf{K}^{\bullet}_{\mathfrak{p}_i}$ consists of a class
  $\bar{\eta}_{\mathfrak{q}_i}$ of isomorphisms of $p$-divisible
  $O_{B_{\mathfrak{q}_i}}$-modules 
  \begin{displaymath}
    \eta_{\mathfrak{q}_i}: \Lambda_{\mathfrak{q}_i, S}^{et} \isoarrow Y_{\mathfrak{q}_i}
    \quad \mod \mathbf{K}^{\bullet}_{\mathfrak{q}_i}   
  \end{displaymath}
  and a class $\bar{\xi}_{\mathfrak{p}_i}$ of maps 
  $\xi_{\mathfrak{p}_i}: \bar{\lambda} \rightarrow O_{F_{\mathfrak{p}_i}}^{\times}/\mathbf{M}_{\mathfrak{p}_i}$ such that
  $\xi_{\mathfrak{p}_i}(\lambda u) = \xi_{\mathfrak{p}_i}(\lambda) u$, for 
  $u \in O_{F_{\mathfrak{p}_i}}^{\times}$, $\lambda \in \bar{\lambda}$. 
\end{definition}
Equivalently we could replace $\bar{\eta}_{\mathfrak{q}_i}$ by a class of
$O_{B_{\mathfrak{q}_i}}$-module homomorphisms of $p$-adic \'etale sheaves
$\eta_{\mathfrak{q}_i}:\Lambda_{\mathfrak{q}_i, S}^{et}\rightarrow T_p(Y_{\mathfrak{q}_i})$ 
modulo $\mathbf{K}^{\bullet}_{\mathfrak{q}_i}$. We will use this definition
only in the case where $Y_{\mathfrak{q}_i}$ is a constant \'etale $p$-divisible
group. We denote the category of objects of $\mathcal{P}_{\mathfrak{p}_i}(S)$
with an rigidification by
$\mathcal{P}_{\mathfrak{p}_i}(S)_{\mathbf{K}^{\bullet}_{\mathfrak{p}_i}}$. 

We reformulate the definition of an rigidified object
$(Y, \iota, \bar{\lambda}, \bar{\eta}_{\mathfrak{q}_i}, \bar{\xi}_{\mathfrak{p}_i})$. 
To each $\lambda \in \bar{\lambda}$ we associate an
$O_{B_{\bar{\mathfrak{q}}_i}}$-module isomorphism of $p$-divisible groups
$\eta_{\bar{\mathfrak{q}}_i}$ by the following commutative
diagram 
\begin{equation}\label{diagreta}
\begin{aligned}
\xymatrix{
  \Lambda_{\mathfrak{q}_i,S}^{et} \ar[r]^{\eta_{\mathfrak{q}_i}} 
  \ar[d]_{\xi_{\mathfrak{p}_i}(\lambda) \lambda_{\psi}} 
     & Y_{\mathfrak{q}_i} \ar[d]^{\lambda_{\bar{\mathfrak{q}}_i}}\\
     (\Lambda_{\bar{\mathfrak{q}}_i}^{mult})_{\; S}^{\wedge}  & 
     (Y_{\bar{\mathfrak{q}}_i})^{\wedge} \ar[l]_{\eta_{\bar{\mathfrak{q}}_i}^{\wedge}} . \\ 
   }
   \end{aligned}
  \end{equation} 
Then a rigidification modulo $\mathbf{K}^{\bullet}_{\mathfrak{p}_i}$ of
$(Y, \iota, \bar{\lambda})$ is equivalently described by a class
$\bar{\eta}_{\mathfrak{p}_i}$ 
of isomorphisms in the category $\mathcal{P}_{\mathfrak{p}_i}(S)$:  
\begin{equation}\label{RZ2e}   
  \eta_{\mathfrak{p}_i}: (\Lambda_{\mathfrak{p}_i}^{pd},\bar{\lambda}_{\psi})_S
  \isoarrow
  (Y,\iota,\bar{\lambda}) \quad \mod \mathbf{K}^{\bullet}_{\mathfrak{p}_i}. 
\end{equation}
We see that, for $S$ connected, $\eta_{\mathfrak{p}_i}$ is given by its value at a geometric 
point $\omega$ of $S$, where  $(\eta_{\mathfrak{p}_i})_{\omega}$ is invariant
modulo $\mathbf{K}^{\bullet}_{\mathfrak{p}_i}$. This makes sense because, by the
diagram \eqref{diagreta} above, everything comes down to a morphism between $p$-adic \'etale
sheaves. 

We indicate how this allows to extend a Hecke operator
$g \in G^{\bullet}_{\mathfrak{p}_i} \subset G^{\bullet}(\mathbb{A}_f)$ from the
generic fiber $\tilde{\mathcal{A}}^{\bullet t}_{E_{\nu}}$ to the
whole functor $\tilde{\mathcal{A}}^{\bullet t}$. We consider a congruence
subgroup $\mathbf{K}^{\bullet} \subset G^{\bullet}(\mathbb{A}_f)$ such that
$k g \Lambda_{\mathfrak{p}_i} = g \Lambda_{\mathfrak{p}_i}$ for
$k \in \mathbf{K}^{\bullet}_{\mathfrak{p}_i}$. We consider a point   
\begin{displaymath}
(A, \iota, \bar{\lambda}, \bar{\eta}^p, (\bar{\eta}_{\mathfrak{q}_j})_j,
(\bar{\xi}_{\mathfrak{p}_j})_j) \in \tilde{\mathcal{A}}^{\bullet t}_{\mathbf{K}^{\bullet}}(S). 
  \end{displaymath}
 Note here that, by $\mathbf{M}^{\bullet}_{\mathfrak{p}_0} = O_{F_{\mathfrak{p}_0}}^{\times}$, the choice of $\xi_{\mathfrak{p}_0}$ is redundant, so we drop it from the notation. Let $Y_{\mathfrak{p}_i}$ be the $\mathfrak{p}_i$-part of the $p$-divisible
group of $A$. It inherits the structure of an rigidified object
$(Y_{\mathfrak{p}_i},\iota,\bar{\lambda},\bar{\eta}_{\mathfrak{q}_i},\bar{\xi}_{\mathfrak{p}_i})$ 
of $\mathcal{P}_{\mathfrak{p}_i}(S)$. We choose 
$\eta_{\mathfrak{q}_i} \in \bar{\eta}_{\mathfrak{q}_i}$ and write a commutative
diagram
\begin{displaymath}
\xymatrix{
  (\Lambda_{\mathfrak{p}_i,S}^{pd}, \bar{\lambda}_{\psi})_S
  \ar[r]^{\eta_{\mathfrak{p_i}}} 
       & (Y_{\mathfrak{p}_i}, \iota, \bar{\lambda})\\
  (\Lambda_{\mathfrak{p}_i,S}^{pd}, \bar{\lambda}_{\psi})_S \ar[r]^{\eta'_{\mathfrak{p_i}}}
 \ar[u]_g  & (Y'_{\mathfrak{p}_i}, \iota', \bar{\lambda}') \ar[u]_{a} .
   }
\end{displaymath}
Here the maps $a$ and $g$ are quasi-isogenies and $\eta_{\mathfrak{p}_i}$ and
$\eta'_{\mathfrak{p}_i}$ are understood as explained after (\ref{RZ2e}).
For $\lambda \in \bar{\lambda}$ we define
$\lambda' = \mu_{\mathfrak{p}_i}^{-1}(g) a^{*} (\lambda)$. This is a principal
polarization because
$g^{*} (\lambda_{\psi}) = \mu_{\mathfrak{p}_i}(g) \lambda_{\psi}$. We define
$\xi^{'}(a^{*} (\lambda)) = \mu_{\mathfrak{p}_i}(g) \xi(\lambda)$. This gives a
rigidified object
\begin{equation}\label{RZ20e}
  (Y'_{\mathfrak{p}_i},\iota', \bar{\lambda}', \bar{\eta}'_{\mathfrak{q}_i},
  \bar{\xi}'_{\mathfrak{p}_i}) .
  \end{equation}

We find a quasi-isogeny of abelian varieties
$\alpha: (A',\iota') \rightarrow (A, \iota)$ which induces on the
$\mathfrak{p}_j$-parts of the $p$-divisible groups an isomorphism
for $j \neq i$ and the map $a$ on the $\mathfrak{p}_i$-parts. Then
$(A',\iota')$ inherits the data $(\bar{\eta}^p)'$, $\bar{\eta}'_{\mathfrak{q}_j}$, 
$\bar{\xi}'_{\mathfrak{p}_j}$ for $j \neq i$ by pull back via $\alpha$. 
The data $\bar{\eta}'_{\mathfrak{q}_i}$, $\bar{\xi}'_{\mathfrak{p}_i}$ are inherited
from (\ref{RZ20e}). The $U_p(F)$-homogeneous  polarization of $A'$ consists
of all $\lambda' \in \alpha^{*} (\bar{\lambda})$ which are principal in $p$.
We define the image by the Hecke operator $g$ as
\begin{displaymath}
(A', \iota', \bar{\lambda}', (\bar{\eta}^p)', (\bar{\eta}'_{\mathfrak{q}_j})_j,
  (\bar{\xi}'_{\mathfrak{p}_j})_j) \in
  \tilde{\mathcal{A}}^{\bullet t}_{g^{-1}\mathbf{K}^{\bullet}g}(S). 
  \end{displaymath}
It follows from the discussion after the proof of Proposition \ref{BZ11p} that
this defines an extension of the Hecke operators over the generic fiber 
$\tilde{\mathcal{A}}^{\bullet t}_{E_{\nu}}$. 

We fix an object $(\mathbb{X}, \iota_{\mathbb{X}}, \bar{\lambda}_{\mathbb{X}})$ 
of the category $\mathcal{P}_{\mathfrak{p}_i}(\bar{\kappa}_{E_{\nu}})$, e.g.
$(\Lambda_{\mathfrak{p}_i}^{pd},\bar{\lambda}_{\psi})$. We choose
$\lambda_{\mathbb{X}} \in \bar{\lambda}_{\mathbb{X}}$ and call
$(\mathbb{X}, \iota_{\mathbb{X}}, \lambda_{\mathbb{X}})$ the framing object.
We define
$(\mathbb{X}, \iota_{\mathbb{X}}, \bar{\lambda}_{\mathbb{X}})_S$ in the same way as
in Definition \ref{RZ1d}. The $RZ$-space is defined as follows:
\begin{definition}\label{RZ2d} 
 We denote by  ${\rm RZ}_{\mathfrak{p}_i, \mathbf{K}^{\bullet}_{\mathfrak{p}_i}}$  the functor on the
  category of schemes $S$ over $\Spf O_{\breve{E}_{\nu}}$, where a point of
  ${\rm RZ}_{\mathfrak{p}_i, \mathbf{K}^{\bullet}_{\mathfrak{p}_i}}(S)$ is given by the
  following data up to isomorphism:
  \begin{enumerate}
  \item[({1})] an object $(Y, \iota, \bar{\lambda}) \in
    \mathcal{P}_{\mathfrak{p}_i}(S)$, 
  \item[({2})] a rigidification $(\bar{\eta}_{\mathfrak{q}_i} \text{ modulo }\mathbf{K}^{\bullet}_{\mathfrak{p}_i},
    \bar{\xi}_{\mathfrak{p}_i}  \text{ modulo }\mathbf{M}^{\bullet}_{\mathfrak{p}_i}$ ) of 
    $(Y, \iota, \bar{\lambda})$, 
  \item[({3})] a quasi-isogeny of $p$-divisible $O_{B_{\mathfrak{p}_i}}$-modules 
    $\rho: (Y, \iota) \rightarrow (\mathbb{X}, \iota_{\mathbb{X}})_S$ which
    respects
    the polarizations on both sides up to a factor in $F_{\mathfrak{p}_i}^{\times}$. 
  \end{enumerate}
\end{definition}
It follows from (\ref{RZ2e}) that we can represent a point of 
${\rm RZ}_{\mathfrak{p}_i, \mathbf{K}^{\bullet}_{\mathfrak{p}_i}}(S)$ by a class 
$\bar{\rho}$ of quasi-isogenies 
\begin{equation}\label{RZ3e}
  \rho: (\Lambda_{\mathfrak{p}_i}^{pd},\bar{\lambda}_{\psi})_S \rightarrow
  (\mathbb{X}, \iota_{\mathbb{X}}, \bar{\lambda}_{\mathbb{X}}) \text{ modulo $\mathbf{K}^{\bullet}_{\mathfrak{p}_i}$ },
  \end{equation}
 which respects the polarizations
of both sides up to a factor in $F_{\mathfrak{p}_i}^{\times}$. 

We choose an isomorphism
\begin{equation}\label{RZ1e} 
  (\mathbb{X},\iota_{\mathbb{X}},\lambda_{\mathbb{X}}) \isoarrow
  (\Lambda_{\mathfrak{p}_i}^{pd}, \lambda_{\psi})
\end{equation}
which respects the polarizations.
Then we see from (\ref{RZ3e}) that a point of
${\rm RZ}_{\mathfrak{p}_i, \mathbf{K}^{\bullet}_{\mathfrak{p}_i}}(S)$ is (locally)
represented by an element $g \in G^{\bullet}_{\mathfrak{p}_i}$. 
Therefore we obtain 
\begin{proposition}\label{RZ6p}
  The choice of an isomorphism (\ref{RZ1e}) defines an isomorphism
  \begin{displaymath}
    {\rm RZ}_{\mathfrak{p}_i, \mathbf{K}^{\bullet}_{\mathfrak{p}_i}}
    \overset{\sim}{\longrightarrow}
    G^{\bullet}_{\mathfrak{p}_i}/\mathbf{K}^{\bullet}_{\mathfrak{p}_i}, 
  \end{displaymath}
  where the right hand side denotes the constant sheaf. 
  \end{proposition}
Let $g \in G^{\bullet}_{\mathfrak{p}_i}$. If we represent a point of  
${\rm RZ}_{\mathfrak{p}_i, \mathbf{K}^{\bullet}_{\mathfrak{p}_i}}(S)$ in the form
(\ref{RZ3e}), the assignement $\rho \mapsto \rho g$ defines a functor
morphism 
\begin{equation}\label{Hecke11e}
g: {\rm RZ}_{\mathfrak{p}_i, \mathbf{K}^{\bullet}_{\mathfrak{p}_i}} \rightarrow
  {\rm RZ}_{\mathfrak{p}_i, g^{-1}\mathbf{K}^{\bullet}_{\mathfrak{p}_i}g}.  
\end{equation}
We call this a Hecke operator. Note that (\ref{Hecke11e}) is only defined
if $\mathbf{K}^{\bullet}_{\mathfrak{p}_i}$ is sufficiently small with respect to
$g$, i.e. if
$g^{-1}\mathbf{K}^{\bullet}_{\mathfrak{p}_i}g \Lambda_{\mathfrak{p}_i} \subset
\Lambda_{\mathfrak{p}_i}$. 
We could define ${\rm RZ}_{\mathfrak{p}_i, \mathbf{K}^{\bullet}_{\mathfrak{p}_i}}$ for
an arbitrary open compact subgroup
$\mathbf{K}^{\bullet}_{\mathfrak{p}_i} \subset G_{\mathfrak{p}_0}$ by making the
Hecke operators part of the definition. 

\smallskip

Now we discuss the case of the prime ideal $\mathfrak{p}_0$. 

We first define the category $\mathcal{P}_{\mathfrak{p}_0}(S)$ for a scheme $S$
over $\Spf O_{E_{\nu}}$. Because of the isomorphism
$\varphi_0: O_{F_{\mathfrak{p}_0}} \rightarrow O_{E_{\nu}}$ it makes sense to speak
of a special formal $O_{B_{\mathfrak{q}_0}}$-module in the sense of Drinfeld over $S$, cf. \cite{Dr} or \cite[\S 5.1]{KRZ}. 
We consider $p$-divisible $O_{B_{\mathfrak{p}_0}}$-modules $(Y,\iota)$. The
decomposition $O_{B_{\mathfrak{p}_0}} = O_{B_{\mathfrak{q}_0}} \times O_{B_{\bar{\mathfrak{q}}_0}}$
induces a decomposition
\begin{displaymath}
Y = Y_{\mathfrak{q}_0} \times Y_{\bar{\mathfrak{q}}_0}. 
\end{displaymath}
We consider principal polarizations $\lambda$ on $Y$ which induce on
$O_{B_{\mathfrak{p}_0}}$ the given involution $\star$. They are given by two
isomorphisms
$\lambda_{\mathfrak{q}_0}:Y_{\mathfrak{q}_0}\rightarrow(Y_{\bar{\mathfrak{q}}_0})^{\wedge}$ 
and
$\lambda_{\bar{\mathfrak{q}}_0}:Y_{\bar{\mathfrak{q}}_0}\rightarrow(Y_{\mathfrak{q}_0})^{\wedge}$ 
such that $\lambda_{\bar{\mathfrak{q}}_0} = - \lambda_{\mathfrak{q}_0}^{\; \wedge}$. 
Moreover we have  commutative diagrams like (\ref{RZ4e}) for $i=0$. 
\begin{definition}\label{RZ3d}
  An object $(Y, \iota, \bar{\lambda})$ of the category
  $\mathcal{P}_{\mathfrak{p}_0}(S)$ consists of the following data:
  \begin{enumerate}
  \item[(1)] A $p$-divisible $O_{B_{\mathfrak{p}_0}}$-module $(Y, \iota)$ over $S$
    such that $Y_{\mathfrak{q}_0}$ is a special formal $O_{B_{\mathfrak{q}_0}}$-module.
  \item[(2)] An $O_{F_{\mathfrak{p}_0}}^{\times}$-homogeneous polarization
    $\bar{\lambda}$ of $Y$, such that each $\lambda \in \bar{\lambda}$ is
    principal and such that the Rosati involution of $\lambda$ induces on
    $O_{B_{\mathfrak{p}_0}}$ the involution $\star$. 
    \end{enumerate}
  \end{definition}
We note that the functor $(Y,\iota, \bar{\lambda})$ from  
$\mathcal{P}_{\mathfrak{p}_0}(S)$ to the category of special formal
$O_{B_{\mathfrak{q}_0}}$-modules over $S$ it not faithful. But it would be
an equivalence of categories is we replace $\mathcal{P}_{\mathfrak{p}_0}(S)$
by the category of triples $(Y, \iota, \lambda)$ with a given polarization
$\lambda$ as in the Definition above. We fix an object
$(\mathbb{X}, \iota_{\mathbb{X}}, \lambda_{\mathbb{X}})$ over
$\Spec \bar{\kappa}_{E_{\nu}}$. We call this the framing object. 
We keep the notation of Definition \ref{RZ1d}.  
\begin{definition}\label{RZ4d}
  We denote by ${\rm RZ}_{\mathfrak{p}_0}$  the functor on the
  category of schemes $S$ over $\Spf O_{\breve{E}_{\nu}}$ such that a
   point of ${\rm RZ}_{\mathfrak{p}_0}(S)$ consists of the following data
  up to isomorphism: 
  \begin{enumerate}
  \item[(1)] an object 
    $(Y,\iota, \bar{\lambda}) \in \mathcal{P}_{\mathfrak{p}_0}(S)$,
  \item[(2)] a quasi-isogeny of $p$-divisible $O_{B_{\mathfrak{p}_0}}$-modules  
    \begin{displaymath}
      \rho: (Y,\iota)_{\bar{S}} \rightarrow (\mathbb{X}, \iota_{\mathbb{X}})
      \times_{\Spec \bar{\kappa}_{E_{\nu}}} \bar{S}
    \end{displaymath}
    which respects the polarizations on both sides up to a factor in
    $F_{\mathfrak{p}_0}^{\times}$. 
    \end{enumerate}
\end{definition}
Let $(G,\iota)$ be a $p$-divisible
$O_{F_{\mathfrak{p}_0}}$-module over an $O_{F_{\mathfrak{p}_0}}${\rm -Alg}ebra $R$. Let
$F_{\mathfrak{p}_0}^t \subset F_{\mathfrak{p}_0}$ be the maximal subfield
which is unramified over $\mathbb{Q}_p$ and let 
$\sigma \in \Gal(F_{\mathfrak{p}_0}^t/\mathbb{Q}_p)$ be the Frobenius. Let
$\varepsilon: O_{F_{\mathfrak{p}_0}} \rightarrow R$ be the structure morphism.
There is the natural decomposition
\begin{displaymath}
  O_{F_{\mathfrak{p}_0}} \otimes_{\mathbb{Z}_p} R \cong \prod_{i=0}^{f-1}
  O_{F_{\mathfrak{p}_0}} \otimes_{O_{F^t_{\mathfrak{p}_0}}, \sigma^i\varepsilon} R,  
\end{displaymath}
where $f = [F^t_{\mathfrak{p}_0}: \mathbb{Q}_p]$ is the index of inertia of
$F_{\mathfrak{p}_0}$. This decomposition induces  a corresponding decomposition of the
$R$-module given by the Lie algebra of $G$. We set  
\begin{displaymath}
  \height_{F_{\mathfrak{p}_0}} G := \height_{F_{\mathfrak{p}_0}} (\pi \mid G) =
         [F_{\mathfrak{p}_0} : \mathbb{Q}_p]^{-1} \height G, 
\end{displaymath}
where $\pi$ is a prime element of $F_{\mathfrak{p}_0}$. 

Let $\alpha: G_1 \rightarrow G_2$
be an isogeny which is an $O_{F_{\mathfrak{p}_0}}$-module homomorphism such that 
$\rank_R \Lie_i G_1 = \rank_R \Lie_i G_2$ for $i = 0, \ldots, f-1$. 
Then $\height \alpha$ is divisible by $[F^t:\mathbb{Q}_p]$. In this case
we define
\begin{displaymath}
  \height_{F_{\mathfrak{p}_0}} \alpha = [F_{\mathfrak{p}_0}^t:\mathbb{Q}_p]^{-1}
  \height \alpha, 
\end{displaymath}

If $\alpha: X \rightarrow X'$ is a quasi-isogeny of special formal
$O_{B_{\mathfrak{p}_0}}$-modules, the integer $\height_{F_{\mathfrak{p}_0}} \alpha$
is divisible by $2$, because there is a quadratic unramified extension
of $F_{\mathfrak{p}_0}$ which is contained in $B_{\mathfrak{p}_0}$. 

We consider a point of ${\rm RZ}_{\mathfrak{p}_0}(S)$ as in Definition
\ref{RZ4d}. The quasi-isogeny $\rho$ is the direct sum of two quasi-isogenies
\begin{equation}
  \rho_{\mathfrak{q}_0} : Y_{\mathfrak{q}_0, \bar{S}} \rightarrow
  \mathbb{X}_{\mathfrak{q}_0} \times_{\Spec \bar{\kappa}_{E_{\nu}}} \bar{S}, \quad 
  \rho_{\bar{\mathfrak{q}}_0} : Y_{\bar{\mathfrak{q}}_0, \bar{S}} \rightarrow
  \mathbb{X}_{\bar{\mathfrak{q}}_0} \times_{\Spec \bar{\kappa}_{E_{\nu}}} \bar{S}. 
  \end{equation}

For each pair of integers $(a,b)$ such that $a + b \equiv 0 \mod 2$, we
define an open and closed subfunctor of
${\rm RZ}_{\mathfrak{p}_0}$
\begin{equation}
  {\rm RZ}_{\mathfrak{p}_0}(a,b)(S) = \{(Y,\iota, \bar{\lambda}, \rho) \; | \;
  \height_{F_{\mathfrak{p}_0}} \rho_{\mathfrak{q}_0} = 2a, \;
  \height_{F_{\mathfrak{p}_0}} \rho_{\bar{\mathfrak{q}}_0} = 2b \}
\end{equation}
We remark that for a point
$(Y,\iota, \bar{\lambda}, \rho) \in {\rm RZ}_{\mathfrak{p}_0}(S)$ the sum
$(\height_{F_{\mathfrak{p}_0}} \rho_{\mathfrak{q}_0} + \height_{F_{\mathfrak{p}_0}} \rho_{\bar{\mathfrak{q}}_0})$
is always divisible by $4$. Indeed, by definition there is an
$f \in F_{\mathfrak{p}_0}^{\times}$ such that the following diagram of
quasi-isogenies is commutative,
\begin{equation}\label{RZ14e}
\xymatrix{
     Y_{\mathfrak{q}_0, \bar{S}} \ar[rr]^{\rho_{\mathfrak{q}_0}} \ar[d]_{\lambda f} & 
     & \mathbb{X}_{\mathfrak{q}_0} \times_{\Spec \bar{\kappa}_{E_{\nu}}} \bar{S}
     \ar[d]^{\lambda_{\mathbb{X}}}\\
     (Y_{\bar{\mathfrak{q}}_0, \bar{S}})^{\wedge}  
     & & \ar[ll]_{\rho_{\bar{\mathfrak{q}}_0}^{\;\wedge}}
     (\mathbb{X}_{\bar{\mathfrak{q}}_0})^{\wedge}
     \times_{\Spec \bar{\kappa}_{E_{\nu}}} \bar{S} .\\ 
   }
  \end{equation}
Since $\lambda$ and $\lambda_{\mathbb{X}}$ are isomorphisms, we obtain that
\begin{displaymath}
  \height_{F_{\mathfrak{p}_0}} \rho_{\mathfrak{q}_0} +
  \height_{F_{\mathfrak{p}_0}} \rho_{\bar{\mathfrak{q}}_0} = \height_{F_{\mathfrak{p}_0}}(f
  \mid Y_{\mathfrak{q}_0, \bar{S}}). 
\end{displaymath} 
The right hand side is divisible by $4$ because
$\height_{F_{\mathfrak{p}_0}} Y_{\mathfrak{q}_0, \bar{S}} = 4$. 

We conclude that
\begin{equation}\label{RZ12e}
  {\rm RZ}_{\mathfrak{p}_0} = \coprod_{a+b\equiv 0\,{\rm mod}\, 2}
  {\rm RZ}_{\mathfrak{p}_0}(a,b) .
  \end{equation}
We introduce Hecke operators acting on ${\rm RZ}_{\mathfrak{p}_0}$.
Let ($Y_{\mathfrak{q}_0}, \iota_{\mathfrak{q}_0})$ be a $p$-divisible
$O_{B_{\mathfrak{q}_0}}$-module.  For $u_1 \in B_{\mathfrak{q}_0}^{\times}$ we
define
\begin{displaymath}
  \iota_{\mathfrak{q}_0}^{u_1}: O_{B_{\mathfrak{q}_0}} \rightarrow \End Y_{\mathfrak{q}_0},
  \quad \iota_{\mathfrak{q}_0}^{u_1}(b) = \iota_{\mathfrak{q}_0}(u_1^{-1} b u_1). 
\end{displaymath}
We set $Y_{\mathfrak{q}_0}^{u_1} = Y_{\mathfrak{q}_0}$ and write
$(Y_{\mathfrak{q}_0}^{u_1}, \iota_{\mathfrak{q}_0}^{u_1})$. 
We use the same definition for a $p$-divisible $O_{B_{\bar{\mathfrak{q}}_0}}$-module.

Let
$u =(u_1,u_2) \in B_{\mathfrak{p}_0}^{\times} = B_{\mathfrak{q}_0}^{\times} \times B_{\bar{\mathfrak{q}}_0}^{\times}$.
   For a $p$-divisible $O_{B_{\mathfrak{p}_0}}$-module $(Y,\iota)$ we set
$\iota^{u}(b) = \iota(u^{-1}b u)$, $b \in O_{B_{\mathfrak{p}_0}}$.

\begin{lemma}\label{RZ4l} 
  Let $(Y, \iota, \bar{\lambda}) \in \mathcal{P}_{\mathfrak{p}_0}(S)$. 
  Let $u = (u_1, u_2) \in B_{\mathfrak{p}_0}^{\times}$ such that
  $u_2^{\star} u_1 \in F_{\mathfrak{p}_0}^{\times}$.
Then for each $\lambda \in \bar{\lambda}$ the Rosati involution of $\lambda$ 
  on $\End Y^u$ induces via
  \begin{displaymath}
\iota^u: O_{B_{\mathfrak{p}_0}} \rightarrow \End Y^u  
  \end{displaymath}
  the involution $\star$ on $O_{B_{\mathfrak{p}_0}}$. In particular $(Y^u, \iota^u, \bar{\lambda}) \in \mathcal{P}_{\mathfrak{p}_0}(S)$. 
  \end{lemma}
\begin{proof}
  We must verify the commutativity of the following diagram,
  \begin{displaymath}
  \xymatrix{
     Y_{\mathfrak{q}_0}^{u_1} \ar[r]^{\iota^{u_1}(b_2^{\star})} \ar[d]_{\lambda}
     & Y_{\bar{\mathfrak{q}}_0}^{u_1} \ar[d]^{\lambda}\\
     (Y_{\bar{\mathfrak{q}}_0}^{u_2})^{\wedge} \ar[r]_{\iota^{u_2}(b_2)^{\wedge}} &
     (Y_{\bar{\mathfrak{q}}_0}^{u_2})^{\wedge} .\\ 
   }
    \end{displaymath}
  Indeed,
  \begin{displaymath} 
      \lambda^{-1} \iota^{u_2}(b_2)^{\wedge} \lambda =
      \lambda^{-1} \iota(u_2^{-1}b_2u_2)^{\wedge} \lambda = 
      \iota(u_2^{\star} b_2^{\star} (u_2^{\star})^{-1}) 
      = \iota(u_1^{-1} b_2^{\star} u_1) = \iota^{u_1}(b_2^{\star}). 
    \end{displaymath}
  The second equation holds because the Rosati involution of $\lambda$ induces
  via $\iota$ the involution $\star$ on $O_{B_{\mathfrak{p}_0}}$. The third equation
  holds because $u_2^{\star} u_1 \in F_{\mathfrak{p}_0}$ implies that
  \begin{displaymath}
    u_1^{-1} b_2^{\star} u_1 = u_2^{\star} u_1 u_1^{-1} b_2^{\star} u_1
    (u_2^{\star} u_1)^{-1} = u_2^{\star} b_2^{\star} (u_2^{\star})^{-1}. 
    \end{displaymath}
\end{proof}

Let  $(Y, \iota, \bar{\lambda}) \in \mathcal{P}_{\mathfrak{p}_0}(S)$. Then
\begin{equation}\label{RZ5e}  
\iota(u): (Y^u, \iota^u) \rightarrow (Y, \iota)  
\end{equation}
is a quasi-isogeny of $p$-divisible $O_{B_{\mathfrak{p}_0}}$-modules. 
\begin{lemma}\label{RZ5l}
  Let $\lambda \in \bar{\lambda}$. We assume that for $u = (u_1, u_2)$
  we have $u_2^{\star}u_1 \in F_{\mathfrak{p}_0}^{\times}$.
Then the quasi-isogeny (\ref{RZ5e}) respects the polarization $\lambda$ on
  both sides of (\ref{RZ5e}) up to a factor in $F_{\mathfrak{p}_0}^{\times}$.
  \end{lemma}
\begin{proof}
  We must show that there exists $f \in O_{F_{\mathfrak{p}_0}}^{\times}$ such that the
  following diagram in commutative.
  \begin{displaymath}
  \xymatrix{
     Y_{\mathfrak{q}_0} \ar[r]^{\lambda \iota(f)\;} 
     & (Y_{\bar{\mathfrak{q}}_0})^{\wedge} \ar[d]^{\iota(u_2)^{\wedge}}\\
     Y_{\mathfrak{q}_0}^{u_1} \ar[r]_{\lambda \quad} \ar[u]^{\iota(u_1)} &
     (Y_{\bar{\mathfrak{q}}_0}^{u_2})^{\wedge} .\\ 
   }
    \end{displaymath}
  The commutativity is equivalent to the first of the following equations, 
  \begin{displaymath}
    \lambda = \iota(u_2)^{\wedge} \lambda \iota(f) \iota(u_1) = 
    \lambda \lambda^{-1} \iota(u_2)^{\wedge} \lambda \iota(f) \iota(u_1) =
    \lambda \iota(u_2^{\star}) \iota(f) \iota(u_1) =
    \lambda \iota(f u_2^{\star}u_1) .
  \end{displaymath}
  Therefore we obtain a commutative diagram if we choose $f u_2^{\star}u_1 =1$. 
\end{proof}

We define the group
\begin{displaymath}
  \mathcal{H}_{\mathfrak{p}_0} = \{u \in B_{\mathfrak{p}_0} \; |
  u^{\star} u \in F^{\times}_{\mathfrak{p}_0} \; \} \subset B^{\times}_{\mathfrak{p}_0}
\end{displaymath}
Note that for $u = (u_1, u_2)$ as above, the conditions
$u^{\star} u \in F^{\times}_{\mathfrak{p}_0}$, resp.
$u_2^{\star} u_1 \in F^{\times}_{\mathfrak{p}_0}$, resp. 
$u_1 u_2^{\star} \in F^{\times}_{\mathfrak{p}_0}$ are equivalent.

The group $\mathcal{H}_{\mathfrak{p}_0}$ acts from the left on the functor
${\rm RZ}_{\mathfrak{p}_0}$. Let
$(Y,\iota,\bar{\lambda}, \rho) \in {\rm RZ}_{\mathfrak{p}_0}(S)$.
For $u \in \mathcal{H}_{\mathfrak{p}_0}$ we define the Hecke correspondence 
\begin{equation}\label{HOat0}
  \mathfrak{h}(u): {\rm RZ}_{\mathfrak{p}_0} \rightarrow
  {\rm RZ}_{\mathfrak{p}_0}, \quad  
  \mathfrak{h}(u)((Y,\iota,\bar{\lambda}, \rho) )=
  (Y^u,\iota^{u},\bar{\lambda}, \rho \iota(u)). 
\end{equation} 
This definition makes sense because of  Lemmas \ref{RZ4l} and
\ref{RZ5l}. We note that for $v \in \mathcal{H}_{\mathfrak{p}_0}$ we obtain
$(Y^{u})^{v} = Y^{vu}$, $\iota(u) \iota^{u}(v) = \iota(vu)$.
The map $\mathfrak{h}(u)$ is the identity on ${\rm RZ}_{\mathfrak{p}_0}$
if $u \in O_{B_{\mathfrak{p}_0}}^{\times}$ because 
$\iota(u): (Y^u,\iota^u,\bar{\lambda}) \rightarrow (Y,\iota,\bar{\lambda})$
is then an isomorphism.

If $u = (u_1, u_2) \in \mathcal{H}_{\mathfrak{p}_0}$, then the Hecke operator
induces maps
\begin{equation}\label{RZ13e}
  \mathfrak{h}((u_1,u_2)): {\rm RZ}_{\mathfrak{p}_0}(a, b) \rightarrow
  {\rm RZ}_{\mathfrak{p}_0}(a + \ord_{B_{\mathfrak{q}_0}} u_1, \; 
  b + \ord_{B_{\bar{\mathfrak{q}}_0}} u_2). 
  \end{equation}
We conclude that $\mathcal{H}_{\mathfrak{p}_0}$ acts transitively on the
set of subspaces $\{ {\rm RZ}_{\mathfrak{p}_0}(a,b) \}$ in the decomposition
(\ref{RZ12e}). Indeed, if $c + d$ is an even sum of integers we can find
$u_1 \in B_{\mathfrak{q}_0}^{\times}$ and $f \in F_{\mathfrak{p}_0}^{\times}$ such that
$\ord_{B_{\mathfrak{q}_0}} u_1 = c$ and $\ord_{B_{\mathfrak{q}_0}} f = c + d$. We define
$u_2$ by the equation $u_2^{\star} u_1 = f$. Then the right hand side of
(\ref{RZ13e}) becomes ${\rm RZ}_{\mathfrak{p}_0}(a +c, b + d)$.  

If we use the right action of $B_{\mathfrak{p}_0}$ on $V_{\mathfrak{p}_0}$, 
we can write 
\begin{displaymath}
  G^\bullet_{\mathfrak{p}_0} = \{g \in B_{\mathfrak{p}_0}^{{\rm opp}} \mid
   g g' \in F^{\times}_{\mathfrak{p}_0}\; \} .
  \end{displaymath}
The anti-isomorphism
\begin{equation}\label{Hecke12e}
  \begin{array}{ccc} 
    B_{\mathfrak{q}_0} \times B_{\bar{\mathfrak{q}}_0} & \rightarrow & 
    B^{{\rm opp}}_{\mathfrak{q}_0} \times B^{{\rm opp}}_{\bar{\mathfrak{q}}_0}\\
    (b_1,b_2) & \mapsto & (b_1, (b_2^{\star})')  
    \end{array}
\end{equation}
defines an anti-isomorphism
$\mathcal{H}_{\mathfrak{p}_0} \rightarrow G^{\bullet}_{\mathfrak{p}_0}$.
Therefore $G^{\bullet}_{\mathfrak{p}_0}$ acts from the right on
${\rm RZ}_{\mathfrak{p}_0}$. We write this action
\begin{equation}\label{Heckep0e}  
  (Y,\iota,\bar{\lambda}, \rho) \mapsto (Y,\iota,\bar{\lambda}, \rho)_{\mid g},
  \quad g \in G^{\bullet}_{\mathfrak{p}_0}. 
  \end{equation}
From the properties of the action of $\mathcal{H}_{\mathfrak{p}_0}$ above,
we conclude that
$\mathbf{K}^{\bullet}_{\mathfrak{p}_0} \subset G^{\bullet}_{\mathfrak{p}_0}$ acts trivially
on ${\rm RZ}_{\mathfrak{p}_0}$. Therefore the group 
$G^{\bullet}_{\mathfrak{p}_0} / \mathbf{K}^{\bullet}_{\mathfrak{p}_0}$ acts
on ${\rm RZ}_{\mathfrak{p}_0}$. This group isomorphic to $\mathbb{Z}^2$ and
acts simply transitively on the set of subspaces
$\{ {\rm RZ}_{\mathfrak{p}_0}(a,b) \}$ in the decomposition (\ref{RZ12e}).

We denote by $\hat{\Omega}_{F_{\mathfrak{p}_0}}^2$ the Drinfeld upper half plane
over $\Spf O_{F_{\mathfrak{p}_0}}$. This is a regular formal scheme of dimension $2$,
comp. \cite{Dr}, \cite{RZ}, \cite[\S 5.1]{KRZ}. The formal scheme
$\hat{\Omega}_{F_{\mathfrak{p}_0}}^2\times_{\Spf O_{F_{\mathfrak{p}_0}},\varphi_0}\Spf O_{\breve{E}_\nu}$ 
represents  the Drinfeld functor $\mathcal{M}_{\mathrm{Dr}}(0)$ whose points with values in
a scheme $S$ over $\Spf O_{\breve{E}_\nu}$ are given by pairs
$(Y_{\mathfrak{q}_0},\rho_{\mathfrak{q}_0})$, where $Y_{\mathfrak{q}_0}$ is a special
formal $O_{B_{\mathfrak{q}_0}}$-module over $S$ and $\rho$  a quasi-isogeny of
special formal $O_{B_{\mathfrak{q}_0}}$-modules of height zero,
\begin{displaymath}
\rho_{\mathfrak{q}_0} : Y_{\mathfrak{q}_0, \bar{S}} \rightarrow
  \mathbb{X}_{\mathfrak{q}_0} \times_{\Spec \bar{\kappa}_{E_{\nu}}} \bar{S}. 
\end{displaymath}
We denote by $\tilde{\mathcal{M}}_{\mathrm{Dr}}$ the functor whose points are
given by pairs $(Y_{\mathfrak{q}_0},\rho_{\mathfrak{q}_0})$ where $\rho_{\mathfrak{q}_0}$
is allowed to have arbitrary height. Note that
$\mathrm{height}_{O_{F_{\mathfrak{p}_0}}} \rho_{\mathfrak{q}_0} = 2a$ is automatically
even. We obtain the decomposition
\begin{displaymath}
  \tilde{\mathcal{M}}_{\mathrm{Dr}} = \coprod_{a \in \mathbb{Z}}
  \mathcal{M}_{\mathrm{Dr}}(a), 
\end{displaymath}
cf. \cite{KRZ} \S 5.1. Let $\mathbf{J}_{\mathfrak{q}_0}$ be the group of all
quasi-isogenies 
$\delta \in \End^{o}_{B_{\mathfrak{q}_0}} \mathbb{X}_{\mathfrak{q}_0}$. Then
$\mathbf{J}_{\mathfrak{q}_0} \cong \GL_{2}(F_{\mathfrak{p}_0})$, cf. \cite{Dr}. 
This group acts from the left on $\tilde{\mathcal{M}}_{\mathrm{Dr}}$ by
changing $\rho_{\mathfrak{q}_0}$ to $\delta \rho_{\mathfrak{q}_0}$.

Let $\Pi \in O_{B_{\mathfrak{q}_0}}$ be a prime element. The Hecke operator
 $\mathfrak{h}(\Pi)$  in the sense of \cite{KRZ} (5.1.16)  
acts on $\tilde{\mathcal{M}}_{\mathrm{Dr}}$ as 
\begin{equation}\label{Hecke14e}
  \mathfrak{h}(\Pi): \mathcal{M}_{\mathrm{Dr}}(a) \overset{\sim}{\longrightarrow}
  \mathcal{M}_{\mathrm{Dr}}(a+1), \quad   
  (Y_{\mathfrak{q}_0},\rho_{\mathfrak{q}_0}) \mapsto 
  (Y_{\mathfrak{q}_0}^{\Pi}, \iota_{\mathbb{X}_{\mathfrak{q}_0}}(\Pi) \circ
  \rho^{\Pi}_{\mathfrak{q}_0}).   
\end{equation}
We define an action of $\mathbf{J}_{\mathfrak{q}_0}$ on
$\mathcal{M}_{\mathrm{Dr}}(0)$. For $\delta \in \mathbf{J}_{\mathfrak{q}_0}$, we set

\begin{equation}\label{PGL-O1e}
  \mathrm{pr}(\delta) (Y_{\mathfrak{q}_0},\rho_{\mathfrak{q}_0}) =
  \mathfrak{h}(\Pi)^{-\ord_{\mathfrak{p}_0} \det \delta} \circ 
  \delta(Y_{\mathfrak{q}_0},\rho_{\mathfrak{q}_0}). 
\end{equation}
Because the Hecke operators commute with the action of 
$\mathbf{J}_{\mathfrak{q}_0}$ this is an action of the group
$\mathbf{J}_{\mathfrak{q}_0}$. One can easily see that this action of
$\mathcal{M}_{\mathrm{Dr}}(0)$ factors through
$\mathbf{J}_{\mathfrak{q}_0} = \GL_{2}(F_{\mathfrak{p}_0}) \rightarrow
\PGL_{2}(F_{\mathfrak{p}_0})$.
Using (\ref{Hecke14e}) as identifications we obtain an isomorphism
\begin{equation}\label{PGL-O2e}
  \tilde{\mathcal{M}}_{\mathrm{Dr}} \cong \mathcal{M}_{\mathrm{Dr}}(0) \times
  \mathbb{Z} \cong (\hat{\Omega}_{F_{\mathfrak{p}_0}}^2
  \times_{\Spf O_{F_{\mathfrak{p}_0}},\varphi_0} \Spf O_{\breve{E}_\nu}) \times \mathbb{Z}. 
\end{equation}
\begin{proposition}
  The isomorphism (\ref{PGL-O2e}) does not depend on the choice of the prime
  element $\Pi$. 
  The action of $\mathbf{J}_{\mathfrak{q}_0}$ on the left hand side induces on the 
  right hand side
  \begin{displaymath}
    \delta (\omega, m) = (\mathrm{pr}(\delta) \omega, \; 
    \ord_{\mathfrak{p}_0} \det \delta+m), \quad \delta\in \mathbf{J}_{\mathfrak{q}_0},
    \; \omega \in \mathcal{M}_{\mathrm{Dr}}(0), \; m \in \mathbb{Z},  
    \end{displaymath}
  cf. (\ref{PGL-O1e}).

  In the next section we will write
  $\delta \omega := \mathrm{pr}(\delta) \omega$. 
  \end{proposition}
\begin{proof}
This is clear. 
  \end{proof}

\begin{lemma}\label{RZ7l} 
  There is a canonical isomorphism of functors over $\Spf O_{\breve{E}_\nu}$ 
  \begin{displaymath}
    {\rm RZ}_{\mathfrak{p}_0}(0,0) \isoarrow \hat{\Omega}_{F_{\mathfrak{p}_0}}^2
    \times_{\Spf O_{F_{\mathfrak{p}_0}},\varphi_0}\Spf O_{\breve{E}_\nu}. 
    \end{displaymath}
\end{lemma}
\begin{proof}
  We begin with a general remark which is useful later on. We consider the
  isomorphism of rings 
  \begin{equation}\label{RZ9e} 
    \begin{array}{ccc}  
    B_{\mathfrak{p}_0} = B_{\mathfrak{q}_0} \times B_{\bar{\mathfrak{q}}_0} & 
    \overset{\sim}{\longrightarrow} & B_{\mathfrak{q}_0} \times
    B_{\mathfrak{q}_0}^{\mathrm{opp}}. \\
    \quad (b,c) & \mapsto & (b, c^{\star}) 
      \end{array}
  \end{equation}
  The involution $\star$ on $B_{\mathfrak{p}_0}$ induces on the right hand side
  the involution $(b_1, b_2) \mapsto (b_2, b_1)$. The maximal orders defined on each
  side (cf.  (\ref{Dsplit-in0})) are mapped isomorphically to each other.
  Consider a point $(Y, \iota, \bar{\lambda}) \in \mathcal{P}_{\mathfrak{p}_0}(S)$. We
  choose $\lambda \in \bar{\lambda}$. It defines an isomorphism
  $\lambda: Y_{\mathfrak{q}_0} \rightarrow (Y_{\bar{\mathfrak{q}}_0})^{\wedge}$.    
  It becomes an isomorphism of $O_{B_{\mathfrak{q}_0}}$-modules  if we consider
  $(Y_{\bar{\mathfrak{q}}_0})^{\wedge}$ as an $O_{B_{\mathfrak{q}_0}}$-module via the
  isomorphism (\ref{RZ9e}), cf. (\ref{RZ4e}).
  By the choice of $\lambda$ we may identify
  $Y_{\bar{\mathfrak{q}}_0}$ with the $p$-divisible
  $O_{B_{\mathfrak{q}_0}}^{\mathrm{opp}}$-module $(Y_{\mathfrak{q}_0})^{\wedge}$. The
  $O_{F_{\mathfrak{p}_0}}^{\times}$-homogeneous polarization on
  $Y \cong Y_{\mathfrak{q}_0} \oplus (Y_{\mathfrak{q}_0})^{\wedge}$ becomes the
  polarization induced by
  $\id: Y_{\mathfrak{q}_0} \rightarrow ((Y_{\mathfrak{q}_0})^{\wedge})^{\wedge}$.
  We call this the canonical polarization.
  Let us denote by $B^{o}_{\mathfrak{p}_0}$ the right hand side of (\ref{RZ9e})
  with its involution. Then we may describe an object of
  $\mathcal{P}_{\mathfrak{p}_0}(S)$ as a triple 
  $(Y = Y_{\mathfrak{q}_0} \oplus Y_{\mathfrak{q}_0}, \iota, \bar{\lambda})$,
  where $Y_{\mathfrak{q}_0}$ is a special formal $O_{B_{\mathfrak{q}_0}}$-module, where 
  $\iota: O_{B^o_{\mathfrak{q}_0}} \rightarrow \End Y$ is the natural action
  and where $\bar{\lambda}$ is the $O_{F_{\mathfrak{p}_0}}^{\times}$-homogeneous class
  of the canonical polarization. 

  Now we remark that for a point $(Y, \iota, \bar{\lambda}, \rho)$ of 
  ${\rm RZ}_{\mathfrak{p}_0}(0,0) (S)$ there is a unique
  $\lambda \in \bar{\lambda}$ which makes the diagram (\ref{RZ14e}) commute
  with $f = 1$. Indeed, one notes that in this diagram it follows that 
  $f \in O_{F_{\mathfrak{p}_0}}^{\times}$ if $\rho_{\mathfrak{q}_0}$ and
  $\rho_{\bar{\mathfrak{q}}_0}$ are quasi-isogenies of height zero.

  Therefore the point is uniquely determined by
  $(Y_{\mathfrak{q}_0}, \iota_{\mathfrak{q}_0}, \rho_{\mathfrak{q}_0})$. This proves
  the Lemma. 
  \end{proof}

We introduce the group $\mathbf{I}_{\mathfrak{p}_0}$ of all quasi-isogenies 
$\gamma\colon (\mathbb{X},\iota_{\mathbb{X}})\rightarrow (\mathbb{X},\iota_{\mathbb{X}})$
which respect the polarization
$\lambda_{\mathbb{X}}$ up to a factor in $F_{\mathfrak{p}_0}^{\times}$. If we denote
by $\gamma \mapsto \gamma'$ the involution on $\End^{o} \mathbb{X}$ induced by
$\lambda_{\mathbb{X}}$,  we obtain
\begin{equation}\label{RZ17e}
  \mathbf{I}_{\mathfrak{p}_0} = \{\gamma \in \End^{o}_{B_{\mathfrak{p}_0}} 
\mathbb{X} \; | \; \gamma' \gamma \in F_{\mathfrak{p}_0}^{\times} \} .
\end{equation}
The group $\mathbf{I}_{\mathfrak{p}_0}$ acts from the left on the functor
${\rm RZ}_{\mathfrak{p}_0}$,
\begin{displaymath}
  (Y,\iota, \bar{\lambda},\rho) \mapsto (Y,\iota, \bar{\lambda}, \gamma\rho),
  \qquad \gamma \in \mathbf{I}_{\mathfrak{p}_0}. 
  \end{displaymath}
This action commutes with the action of $G^{\bullet}_{\mathfrak{p}_0}$.

We make this action more explicit by using the description of the category
$\mathcal{P}_{\mathfrak{p}_0}(S)$ given in the proof of Lemma \ref{RZ7l}. 
Recall the group $\mathbf{J}_{\mathfrak{q}_0}$ of all quasi-isogenies
$\delta \in \End^{o}_{B_{\mathfrak{q}_0}} \mathbb{X}_{\mathfrak{q}_0}$. Then an 
element $\delta_2 \in \mathbf{J}^{\mathrm{opp}}_{\mathfrak{q}_0}$ acts on
$\mathbb{X}^{\wedge}_{\mathfrak{q}_0}$ by
$\iota_{\mathbb{X}_{\mathfrak{q}_0}}(\delta_2)^{\wedge}$. 
We conclude that 
\begin{displaymath}
  \mathbf{I}_{\mathfrak{p}_0} = \{(\delta_1, \delta_2) \in
  \mathbf{J}_{\mathfrak{q}_0} \times \mathbf{J}_{\mathfrak{q}_0}^{{\rm opp}} \; | \;
  \delta_1 \delta_2 \in F_{\mathfrak{p}_0}^{\times} \} .
  \end{displaymath}

If we replace in (\ref{RZ13e}) $(u_1, u_2)$ by any other
$(v_1,v_2) \in \mathcal{H}_{\mathfrak{p}_0}$ such that
$\ord_{B_{\mathfrak{q}_0}} u_i = \ord_{B_{\mathfrak{q}_0}} v_i$ for $i=1,2$, we obtain the
same morphism. Using this morphism as an identification of both sides of
(\ref{RZ13e}), we obtain an isomorphism
\begin{equation}\label{RZ18e}
  \mathrm{RZ}_{\mathfrak{p}_0} = \mathrm{RZ}_{\mathfrak{p}_0}(0,0) \times \Lambda,   
  \end{equation}
where $\Lambda = \{(a,b) \in \mathbb{Z}^2 \; | a+b \equiv 0 \mod 2 \; \}$.
Combining this with Lemma \ref{RZ7l}, we obtain

\begin{proposition}\label{RZ7p} 
  There is an isomorphism of functors
  \begin{displaymath}
    (\hat{\Omega}_{F_{\mathfrak{p}_0}}^2 \times_{\Spf O_{F_{\mathfrak{p}_0}},\varphi_0}
    \Spf O_{\breve{E}_\nu}) \times
    G^{\bullet}_{\mathfrak{p}_0} / \mathbf{K}^{\bullet}_{\mathfrak{p}_0} 
     \overset{\sim}{\longrightarrow}
{\rm RZ}_{\mathfrak{p}_0} 
  \end{displaymath}
  which is equivariant with respect to the actions of
  $G^{\bullet}_{\mathfrak{p}_0} / \mathbf{K}^{\bullet}_{\mathfrak{p}_0}$
  on both sides. 

  The right hand side of (\ref{RZ18e}) can be written as 
  $\mathcal{M}_{\mathrm{Dr}}(0) \times \Lambda$. An element
  $(\delta_1, \delta_2) \in \mathbf{I}_{\mathfrak{p}_0}$ then sends 
  a point $(\omega, (m_1,m_2))$ to 
$(\mathrm{pr}(\delta_1) \omega,\; (m_1 +\ord_{\mathfrak{p}_0} \det \delta_1, \; m_2 + \ord_{\mathfrak{p}_0} \det \delta_2 ))$. 

  As noted above we will write in the next section
 $\delta \omega := \mathrm{pr}(\delta) \omega$. 
\end{proposition}
\begin{proof}
Only the last assertion needs a proof. 
We consider a point $(x, (m_1,m_2))$ from the right hand
side of (\ref{RZ18e}), where $x$ corresponds to
$\omega = (Y_{\mathfrak{q}_0} \iota_{\mathfrak{q}_0}, \rho_{\mathfrak{q}_0}) \in \mathcal{M}_{\mathrm{Dr}}(0)$.
The image of a point $(x, (m_1,m_2))$ under the action of
$(\delta_1,\delta_2)$ is computed by looking at $Y_{\mathfrak{q}_0}$ only.
By the description of $\mathrm{pr}(\delta_1)$, this shows the result.
\end{proof}

\section{The $p$-adic uniformization of Shimura curves}\label{s:uniform}
In this section, we prove the $p$-adic uniformization of the integral model $\tilde{\rm Sh}_{\mathbf{K}^\bullet}(G^\bullet, h^\bullet_D)$, cf. Definition \ref{tildeSh_D1d}. Here $\mathbf{K}^{\bullet}=\mathbf{K}^{\bullet}_p\mathbf{K}^{\bullet,p}$, with 
$\mathbf{K}^{\bullet}_p \subset G^{\bullet}(\mathbb{Q}_p)$ defined by 
(\ref{BZKpPkt1e}), where
$\mathbf{M}^{\bullet}_{\mathfrak{p}_0} = O_{F_{\mathfrak{p}_0}}^{\times}$. From this,  Cherednik uniformization, i.e., Theorem \ref{MainIntro} will follow. We stress that the Shimura varieties ${\rm Sh}(G, h)$ and ${\rm Sh}(G^\bullet, h)$ from the previous sections will not reappear. 

We consider the functor
\begin{displaymath}
  {\rm RZ}_{p,\mathbf{K}^{\bullet}_p} = {\rm RZ}_{\mathfrak{p}_0} \times \prod_{i=1}^s
  {\rm RZ}_{\mathfrak{p}_i,\mathbf{K}^{\bullet}_{\mathfrak{p}_i}}, 
\end{displaymath}
cf. Definitions \ref{RZ4d} and \ref{RZ2d}.  
Each of these factors is defined by a choice of a framing object
which we denote by $(\mathbb{X}_i, \iota_{\mathbb{X}_i}, \lambda_{\mathbb{X}_i})$, for $i=0,\ldots,s$.
We choose the framing objects as follows. We fix a point
\begin{equation}\label{uniform2e} 
  (A_{o}, \iota_{o}, \bar{\lambda}_{o},
  \bar{\eta}^p_{o}, (\bar{\eta}_{\mathfrak{q}_j,o})_j, 
 ( \xi_{\mathfrak{p}_j, o})_j) \in
  \tilde{\mathcal{A}}^{\bullet t}_{\mathbf{K}^{\bullet}}(\bar{\kappa}_{E_{\nu}}). 
  \end{equation}
The last two data of the point are for $j = 1, \ldots, s$. Indeed, by $\mathbf{M}^{\bullet}_{\mathfrak{p}_0} = O_{F_{\mathfrak{p}_0}}^{\times}$, the choice of $\xi_{\mathfrak{p}_0}$ is redundant. We choose
an element $\eta_{o}^p \in \bar{\eta}_{o}^p$. 
We denote by $\mathbb{X}$ the $p$-divisible group of $A_{o}$. We set
$\mathbb{X}_i = \mathbb{X}_{\mathfrak{p}_i}$, with its action
$\iota_{\mathbb{X}_i}$ from $\iota_{o}$ and a polarization
$\lambda_{\mathbb{X}_i}$ from some element of $\bar{\lambda}_{o}$.

We denote by $\tilde{\mathcal{A}}^{\bullet t}_{\mathbf{K}^{\bullet} / \Spf O_{E_{\nu}}}$, 
resp. $\tilde{\mathcal{A}}^{\bullet t}_{\mathbf{K}^{\bullet} / \Spf O_{\breve{E}_{\nu}}}$ the
restriction of the functor $\tilde{\mathcal{A}}^{\bullet t}_{\mathbf{K}^{\bullet}}$ to
the category of schemes over $\Spf O_{E_{\nu}}$, resp. $\Spf O_{\breve{E}_{\nu}}$. 
We define the uniformization morphism of functors on the category of schemes
$S$ over $\Spf O_{\breve{E}_{\nu}}$ (the definition depends on the choice of
the tuple \eqref{uniform2e}),
\begin{equation}\label{unimorph1e} 
\tilde\Theta^\bullet :  {\rm RZ}_{p,\mathbf{K}^{\bullet}_p} \times
  G^{\bullet}(\mathbb{A}_f^p)/(\mathbf{K}^{\bullet})^p \rightarrow  
 \tilde{\mathcal{A}}^{\bullet t}_{\mathbf{K}^{\bullet} {/ \Spf O_{\breve{E}_{\nu}}}}. 
  \end{equation}
For the definition we recall that a point with values in $S$ of the functor
on the left hand side consists the following data
\begin{enumerate}
\item[({1})] a point $(Y_0, \iota_0, \bar{\lambda}_0, \rho_0)$ of 
  ${\rm RZ}_{\mathfrak{p}_0}(S)$, cf. Definition \ref{RZ4d}, 
\item[({2})] a point
$(Y_j,\iota_j, \bar{\lambda}_j, \bar{\eta}_{\mathfrak{q}_j},\bar{\xi}_{\mathfrak{p}_j}, \rho_j)$
  of ${\rm RZ}_{\mathfrak{p}_j, \mathbf{K}^{\bullet}_{\mathfrak{p}_j}}(S)$ for
  $j = 1, \ldots s$, cf. Definition \ref{RZ2d}, 
\item[({3})] an element $g \in G^{\bullet}(\mathbb{A}_f^p)$. 
\end{enumerate}
Here 
\begin{equation}\label{RZ7e}
\rho_i: (Y_i,\iota_i)_{\bar{S}} \rightarrow (\mathbb{X}_i, \iota_{\mathbb{X}_i})
      \times_{\Spec \bar{\kappa}_{E_{\nu}}} \bar{S} 
\end{equation}
is a quasi-isogeny for $i = 0, \ldots, s$ which respects the polarizations on
both sides up to a factor in $F_{\mathfrak{p}_i}^{\times}$. 
We define as follows an abelian scheme $(\bar{A}, \bar{\iota}_A)$ over $\bar{S}$ and
an isogeny
\begin{equation}\label{RZ8e}
  (\bar{A}, \iota_{\bar{A}}) \rightarrow (A_{o}, \iota_{o})
  \times_{\Spec \bar{\kappa}_{E_{\nu}}} \bar{S}  .
\end{equation}
 Let us denote by $\bar{Y}$ the $p$-divisible group of $\bar{A}$.
Then $(\bar{Y}_{\mathfrak{p}_i}, \iota_i)$ is identified with
$(Y_{i, \bar{S}}, \iota_{i,\bar{S}})$ and the quasi-isogeny (\ref{RZ8e}) induces
on the $p$-divisible groups the quasi-isogenies (\ref{RZ7e}). We choose 
$\lambda_{o} \in \bar{\lambda}_{o}$ and consider the inverse
image $\theta$ of $\lambda_{o}$ on $\bar{A}$ by (\ref{RZ8e}). Let
$\bar{\theta} = F^{\times} \theta$ be the $F$-homogeneous polarization it generates.
By the definition of the $RZ$-spaces, the polarization induced by $\theta$ on
the $p$-divisible group $Y_{i, \bar{S}}$ differs from $\lambda_{i, \bar{S}}$ by
a factor in $F_{\mathfrak{p}_i}^{\times}$. We then define $\bar{\lambda}_{\bar{A}}$ to be the
$U_p(F)$-homogeneous polarization on $\bar{A}$ which consists of all elements
of $\bar{\theta}$ which on the $p$-divisible groups $\bar{Y}_i$ differ
from $\lambda_{i, \bar{S}}$ by a factor in $O_{F_{\mathfrak{p}_i}}^{\times}$. 

Since a lifting
of the $p$-divisible group of $\bar{A}$ with these extra structures is given
by the data $({ 1})$ and $({ 2})$  above, we obtain by the Serre-Tate theorem a lifting $A$  of $\bar A$ over $S$ with extra
structures $\iota_A$ and $\bar{\lambda}_A$. We obtain a point
\begin{displaymath}
  (A, \iota_A, \bar{\lambda}_A, \bar{\eta}^p_A, (\bar{\eta}_{\mathfrak{q}_j,A})_j,
  (\bar{\xi}_{\mathfrak{p}_j})_j) ,
\end{displaymath}
where the last three items are defined as follows. We take the inverse image
of $\eta^p_{o}$ by (\ref{RZ8e}) to obtain $\eta^p_{\bar{A}}$. Since we
have \'etale sheaves, this gives $\eta^p_A$ and then its class $\bar{\eta}^p_A$.
The last two items are deduced directly from the item $(2)$ above.
Indeed $T_{\mathfrak{q}_j}(A) = T_p(Y_{j, \mathfrak{q}_j})$ and a rigidification in the
sense of Definition \ref{RZ1d} is equivalent to
\begin{displaymath}
  \eta_{\mathfrak{q}_j}: \Lambda_{\mathfrak{q}_j} \isoarrow T_p(Y_{j, \mathfrak{q}_j}) 
  \mod \mathbf{K}^{\bullet}_{\mathfrak{q}_j},  
\end{displaymath}
and a function 
\begin{displaymath}
  \xi_{\mathfrak{p}_j}: \bar{\lambda}_A \rightarrow
  O_{F_{\mathfrak{p}_j}}^{\times} \mod \mathbf{M}^{\bullet}_{\mathfrak{p}_j}. 
\end{displaymath}
This last function is induced by the injection
$\bar{\lambda}_A\rightarrow\bar{\lambda}_j$ from $({ 2})$. 

We define the image  under $\tilde\Theta^\bullet$ in (\ref{unimorph1e}) of the point given by
the data $({1})$, $({ 2})$, $({ 3})$ to be
\begin{equation}\label{unimorph4e}
  (A, \iota_A, \bar{\lambda}_A, \eta^p_A g, (\bar{\eta}_{\mathfrak{q}_j,A})_j,
  (\bar{\xi}_{\mathfrak{p}_j})_j) .
\end{equation}
For this we have used our choice of $\eta_{o}^p$. 

\begin{proposition}
  Recall from (\ref{Heckep0e}) the  Hecke operator action of $G^{\bullet}_{\mathfrak{p}_0}$ on
  ${\rm RZ}_{\mathfrak{p}_0}$   and from (\ref{Hecke11e}) the  Hecke operator action of
  $G^{\bullet}_{\mathfrak{p}_i}$ on 
  ${\rm RZ}_{\mathfrak{p}_i, \mathbf{K}^{\bullet}_{\mathfrak{p}_i}}$. Together we obtain an action of $G^\bullet(\mathbb{A}_f)$ by
  Hecke operators on the left hand side of (\ref{unimorph1e}). 
  
  There is an extension of the Hecke operators $G^{\bullet}(\mathbb{A}_f)$
  from the tower $\mathcal{A}^{\bullet t}_{\mathbf{K}^{\bullet}}$ to the tower 
  $\tilde{\mathcal{A}}^{\bullet t}_{\mathbf{K}^{\bullet}}$ such that the uniformization
  morphism $\tilde\Theta^\bullet$ is compatible with the actions of Hecke operators
  on both sides. 
\end{proposition}
\begin{proof}
  This is trivial for the action of $G^{\bullet}(\mathbb{A}_f^{p})$.
  The proof for elements in $G^\bullet(\BQ_p)$ is based on the description of the Hecke operators after the 
  proof of Proposition \ref{BZ11p}, comp. Remark \ref{remHecke}. For $j=1,\ldots, s$, the local component at $\mathfrak{p}_j$  of a Hecke correspondence is described  after (\ref{RZ2e}). 
  The local component at $\mathfrak{p}_0$  of a Hecke correspondence is described  by \eqref{HOat0}. We write here the argument only for the action of 
  $g \in G^{\bullet}_{\mathfrak{p}_0} \subset G(\mathbb{A}_f)$. We consider a point
  on the left hand side of (\ref{unimorph1e}) defined by the data
  (1), (2), (3). We may assume that the element of
  (3) is $1$.   We take the image by $\tilde\Theta^\bullet$,  
  \begin{equation}\label{Hecke13e}
    (A, \iota, \bar{\lambda}, \bar{\eta}^p, (\bar{\eta}_{\mathfrak{q}_j})_j,
    (\bar{\xi}_{\mathfrak{p}_j})_j) \in
    \tilde{\mathcal{A}}^{\bullet t}_{\mathbf{K}^{\bullet}}(S),  
  \end{equation}
  cf. (\ref{unimorph4e}). Let $(Y_0, \iota_0, \bar{\lambda}_0, \rho_0)$
  be the datum ({1}). Then $Y_{\mathfrak{p}_0}$ is the $\mathfrak{p}_0$-part
  of the $p$-divisible group of $A$ with the induced action $\iota_0$ 
  and polarization $\bar{\lambda}_0$. Let
  $u= (u_1, u_2)\in B_{\mathfrak{q}_0}\times B_{\bar{\mathfrak{q}}_0}=B_{\mathfrak{p}_0}$ 
  be the element which corresponds to $g$ by the anti-isomorphism
  (\ref{Hecke12e}). The Hecke operator $g$ on the left hand side of \eqref{unimorph1e} is
  given by $(Y^u_0, \iota^u_0, \bar{\lambda}^u_0, \rho_0 \circ\iota_0(u))$. Note that for the underlying  polarized $p$-divisble groups we have 
  $(Y^u_0, \lambda^u_0) = (Y_0,\lambda_0)$. We consider the quasi-isogeny
  \begin{equation}\label{Hecke15e}
    \iota_0(u) : (Y^u_0, \iota^u_0, \bar{\lambda}^u_0) \rightarrow
    (Y_0, \iota_0, \bar{\lambda}_0).
  \end{equation}
  We note that
  $\iota_0(u)^{*}(\lambda) =\lambda^u\circ\iota_0(u_2^{\star} u_1) =\lambda^u\circ\mu(g)$.  
  Therefore, applying $g$ on the left hand side of (\ref{unimorph1e}) leads
  on the right hand side to the following point. There is a quasi-isogeny of
  abelian varieties 
  \begin{displaymath}
\alpha: (A', \iota') \rightarrow (A, \iota)  
  \end{displaymath}
  which induces on the $\mathfrak{p}_0$-parts of the $p$-divisible groups the map 
  (\ref{Hecke15e}) and  is an isomorphism on the $\mathfrak{p}_j$-parts
  for $j > 0$. Looking at the $p$-divisible groups, we see that, for the given 
  $p$-principal polarization $\lambda$ on $A$,  the polarization
  $\mu_{\mathfrak{p}_0}(g)^{-1} \alpha^{*} (\lambda)$ on the $\mathfrak{p}_0$-part
  of the $p$-divisible group is principal and for $j > 0$ the
  $\mathfrak{p}_j$-parts of $\alpha^{*} (\lambda)$ is principal. We define
  $\bar{\lambda}'$ on $A'$ as the class of all polarizations in
  $F^{\times} \alpha^{*} (\lambda)$ which are $p$-principal. We define
  all other data $(\bar{\eta}')^p$, $\bar{\eta}'_{\mathfrak{q}_j}$, $\bar{\xi}_j$
  for $j > 0$ by pull back via $\alpha$. Then $A'$ with the extra structure
  just introduced is a point of
  $\tilde{\mathcal{A}}^{\bullet t}_{\mathbf{K}^{\bullet}}(S)$. To see that this point
  represents for an $E_{\nu}$-scheme $S$ the Hecke operator, we need the
  functions $\xi_{\mathfrak{p}_0}$ and $\xi'_{\mathfrak{p}_0}$ with values in
  $F^{\times}_{\mathfrak{p}_0}/ O_{F_{\mathfrak{p}_0}}^{\times}$. So far it was not necessary
  to mention them because they have value $1$ for a polarization which is
  principal in $p$. Therefore we have
  \begin{displaymath}
    1 = \xi_{\mathfrak{p}_0}(\lambda) =
    \xi'_{\mathfrak{p}_0}(\mu_{\mathfrak{p}_0}(g)^{-1} \alpha^{*} (\lambda)). 
    \end{displaymath}
  But then (\ref{BZGdot5e}) shows that $A'$ gives the Hecke operator of 
  $g \in G^{\bullet}_{\mathfrak{p}_0}$.
   \end{proof}

By \cite[6.29]{RZ} the following maps are isomorphisms,
\begin{equation}\label{uniform1e}
  \begin{array}{ccc}  
\End^{o}_B(A_{o}) \otimes_{\mathbb{Q}}  \mathbb{A}_f^p & 
  \overset{\sim}{\longrightarrow} & 
  \End^{o}_{B \otimes \mathbb{A}_f^{p}} \mathrm{V}^p(A_{o})\\[2mm] 
  \End^{o}_B(A_{o}) \otimes_{\mathbb{Q}}  \mathbb{Q}_p & 
  \overset{\sim}{\longrightarrow} & 
  \End^{o}_{B \otimes \mathbb{Q}_p} \mathbb{X} .\\ 
  \end{array}
\end{equation}
Here, as above, $\mathbb{X} = \prod_{i=0}^{s} \mathbb{X}_i$ is the $p$-divisible group
of $A_{o}$. We obtain
\begin{displaymath}
\End^{o}_{B \otimes \mathbb{Q}_p} \mathbb{X} \cong \prod_{i=0}^{s}
\End^o_{B_{\mathfrak{p}_i}} \mathbb{X}_i
  \end{displaymath}
We denote by $\mathbf{I}$ the algebraic group of all $B$-linear quasi-isogenies
$A_{o} \rightarrow A_{o}$ which respect the polarization
$\lambda_{o}$ up to a constant in $F^{\times}$. Let
$\gamma \mapsto \gamma'$ be the Rosati involution on 
$\End_B^o A_{o}$ induced by $\lambda_{o}$. We can write
\begin{equation}\label{groupI1e}
  \mathbf{I}(\mathbb{Q}) = \{\alpha \in \End_B^o A_{o} \; | \;
  \alpha' \alpha \in F^{\times} \}. 
\end{equation}
In the framing object (\ref{uniform2e}) we choose
$\eta^p_{o} \in \bar{\eta}^p_{o}$. This isomorphism
$\eta^p_{o}:\mathrm{V}(A_{o})\isoarrow V\otimes\mathbb{A}_f^p$
induces by (\ref{uniform1e}) an isomorphism
\begin{equation}
  \mathbf{I}(\mathbb{A}_f^p) \overset{\sim}{\longrightarrow}
\{\gamma \in \End^{o}_{B\otimes \mathbb{A}_f^p} \mathrm{V}^p(A_{o})  
\; | \; \gamma' \gamma \in (F\otimes\BA_f)^\times \}
  \overset{\sim}{\longrightarrow}
    G^{\bullet}(\mathbb{A}_f^p). 
  \end{equation}
We also denote by $\gamma \mapsto \gamma'$ the involution on
$\End^{o}_{B \otimes \mathbb{Q}_p} \mathbb{X}$ induced by the polarization
$\lambda_{\mathbb{X}}$. We define
\begin{displaymath}
  \mathbf{I}_{\mathfrak{p}_0} = \{\gamma \in \End^{o}_{B_{\mathfrak{p}_0}}
\mathbb{X}_0 \; | \; \gamma' \gamma \in F_{\mathfrak{p_0}}^{\times} \},  
\end{displaymath}
cf. (\ref{RZ17e}). If $j> 0$  we can take
$(\Lambda_{\mathfrak{p}_j}^{pd}, \lambda_{\psi})$ for the framing object
$\mathbb{X}_j$, cf. (\ref{RZ1e}). Using
the definition (\ref{RZ21e}) of $\lambda_{\psi}$ we obtain an isomorphism
\begin{equation}\label{groupI2e}
  \begin{array}{ll}  
    \mathbf{I}_{\mathfrak{p}_j} & \cong \{\gamma\in
  \End_{B_{\mathfrak{p}_j}} V_{\mathfrak{p}_i}
\; | \; \psi(\gamma v_1, \gamma v_2) = \psi (f v_1, v_2), \;
\text{for some} \; f \in F_{\mathfrak{p}_j}^{\times}  \} \\ 
& = G^{\bullet}_{\mathfrak{p}_j} . 
    \end{array}
\end{equation}

By (\ref{uniform1e}) we obtain
\begin{displaymath}
\mathbf{I}(\mathbb{Q}_p) = \prod_{i=0}^{s} \mathbf{I}_{\mathfrak{p}_i}. 
  \end{displaymath}
The group $\mathbf{I}_{\mathfrak{p}_i}$ acts from the left on the functor
${\rm RZ}_{\mathfrak{p}_i}$, for $i = 0, \ldots, s$ by
\begin{displaymath}
 \alpha_{\mathfrak{p}_i} \in \mathbf{I}_{\mathfrak{p}_i}:\quad (Y_{i}, \iota_{i}, \bar{\lambda}_{i}, \bar{\eta}_{\mathfrak{q}_i},
  \bar{\xi}_{\mathfrak{p}_i} , \rho_{i}) \mapsto (Y_{i}, \iota_{i},
  \bar{\lambda}_{i} \bar{\eta}_{\mathfrak{q}_i},
  \bar{\xi}_{\mathfrak{p}_i} , \alpha_{\mathfrak{p}_i} \rho_{i})  .
  \end{displaymath}
 This makes sense
because
$\alpha_{\mathfrak{p}_i}: (\mathbb{X}_i,\iota_i)\rightarrow (\mathbb{X}_i,\iota_i)$
is a quasi-isogeny which respects the polarization $\lambda_{\mathbb{X}_i}$ up to
constant. Note that for $i = 0$ the data
$\bar{\eta}_{\mathfrak{q}_i}, \bar{\xi}_{\mathfrak{p}_i}$ are absent. 

The group $\mathbf{I}(\mathbb{Q})$ acts on the left hand side of
(\ref{unimorph1e}). If $\alpha \in \mathbf{I}(\mathbb{Q})$, with components
$\alpha_{\mathfrak{p}_i} \in \mathbf{I}_{\mathfrak{p}_i}$ and
$\alpha^p \in \mathbf{I}(\mathbb{A}^p_f) \cong G^{\bullet}(\mathbb{A}^p_f)$, then 
a point of the left hand side given by the data {1}, {2}, {3}
is mapped to the data
\begin{equation}\label{unimorph2e}
\big((Y_{i}, \iota_{i}, \bar{\lambda}_{i} \bar{\eta}_{\mathfrak{q}_i},
  \bar{\xi}_{\mathfrak{p}_i} , \alpha_{\mathfrak{p}_i} \rho_{i}), \,\, \alpha^p g\big). 
  \end{equation}
With respect to this action the morphism (\ref{unimorph1e}) is equivariant.
To see this, we consider the morphism derived from (\ref{RZ8e}) 
\begin{equation}
(\bar{A}, \iota_{\bar{A}}) \rightarrow (A_{o}, \iota_{o})
  \times_{\Spec \bar{\kappa}_{E_{\nu}}} \bar{S} \overset{\alpha}{\longrightarrow}
  (A_{o}, \iota_{o}) \times_{\Spec \bar{\kappa}_{E_{\nu}}} \bar{S}  
  \end{equation}
This composite may be used to compute the image of the data (\ref{unimorph2e}) 
by (\ref{unimorph1e}). One can easily see that this gives the same point in 
$\tilde{\mathcal{A}}^{\bullet t}_{\mathbf{K}^{\bullet}}(S)$ as for the original data. 
\begin{proposition}
  Let $\mathbf{K}^{\bullet}=\mathbf{K}^{\bullet}_p\mathbf{K}^{\bullet,p}\subset G^\bullet(\BA_f)$, with $\mathbf{K}^{\bullet}_p$ as in \eqref{BZKpPkt1e}, where
  $\mathbf{M}_{\mathfrak{p}_0} = O_{F_{\mathfrak{p}_0}}^{\times}$.  
  Let $\tilde{{\mathsf A}}^{\bullet t}_{\mathbf{K}^{\bullet}}$ be the coarse moduli
  scheme of $\tilde{\mathcal{A}}^{\bullet t}_{\mathbf{K}^{\bullet}}$.  
  The morphism (\ref{unimorph1e}) induces an isomorphism of formal schemes 
  \begin{equation}\label{unimorph5e}  
    \Theta^\bullet: \mathbf{I}(\mathbb{Q}) \backslash ({\rm RZ}_{p,\mathbf{K}^{\bullet}_p}
    \times G^{\bullet}(\mathbb{A}_f^p)/(\mathbf{K}^{\bullet})^p) \rightarrow   
 \tilde{{\mathsf A}}^{\bullet t}_{\mathbf{K}^{\bullet}{/ \Spf O_{\breve{E}_{\nu}}}}. 
  \end{equation}
  The morphism is compatible with the Weil descent data on both sides as
  spelled out in the proof. 
\end{proposition}
\begin{proof}
  This is a variant of the general uniformization theorem \cite[6.30]{RZ}. By  Proposition \ref{BZ8p}, if
  $\mathbf{K}^{\bullet}$ is small enough,  the morphism
  \begin{displaymath}
    \tilde{\mathcal{A}}^{\bullet t}_{\mathbf{K}^{\bullet}} \rightarrow
    \tilde{{\mathsf A}}^{\bullet t}_{\mathbf{K}^{\bullet}} 
    \end{displaymath}
  is the etale sheafification. Therefore we can use deformation theory as
  in \cite[6.23]{RZ} to show that $\Theta^\bullet$ is \'etale. For this one needs that
  the action of $\mathbf{I}(\BQ)$ is fixpoint free, if $\mathbf{K}^{\bullet, p}$ is
  small enough. We refer to the argument in loc.cit. for details. 
  In addition to the arguments given in \cite{RZ},  one needs that $\Theta^\bullet$ is
  surjective on the $\bar{\kappa}_{E_{\nu}}$-valued points. This follows from 
  the Hasse principle for $G^{\bullet}$ as explained in \cite[Prop. 7.1.11, Prop. 7.3.2]{KRZ}. We omit the details.
  If we drop the smallness assumption on  $\mathbf{K}^{\bullet, p}$, it follows from Proposition \ref{BZ8p}
  that $\Theta^\bullet$ is an isomorphism for the normal subgroup
  $\mathbf{K}^{\bullet}_U \subset \mathbf{K}^{\bullet}$. Dividing by the
  action of $\mathbf{K}^{\bullet}$ we obtain that $\Theta^\bullet$ is an isomorphism.

Both sides of (\ref{unimorph5e}) are endowed with a Weil descent datum
relative to the extension $O_{\breve{E}_{\nu}} / O_{E_{\nu}}$. Because the right
hand side is obtained by a base change
$\Spf O_{\breve{E}_{\nu}} \rightarrow \Spf O_{E_{\nu}}$, we have there the
natural Weil descent datum. On the left hand side the Weil descent
datum is induced by a Weil descent datum of ${\rm RZ}_{p,\mathbf{K}^{\bullet}_p}$. It is
defined on each ${\rm RZ}_{\mathfrak{p}_i}$  as follows. Let
$\tau \in \Gal(\breve{E}_{\nu}/E_{\nu})$ be the Frobenius. Let $R$ be
$O_{\breve{E}_{\nu}}${\rm -Alg}ebra such that $p$ is nilpotent in $R$. We give
the descent datum as a functorial map 
\begin{equation}\label{uniform3e}
  \omega_{\mathfrak{p}_i}(R): {\rm RZ}_{\mathfrak{p}_i}(R) \rightarrow
  {\rm RZ}_{\mathfrak{p}_i}(R_{[\tau]}).  
\end{equation}
For this we write $\bar{R} = R \otimes_{O_{\breve{E}_{\nu}}} \bar{\kappa}_{E_{\nu}}$ and
we denote by $\varepsilon: \bar{\kappa}_{E_{\nu}} \rightarrow \bar{R}$ the
structure morphism. Then (\ref{uniform3e}) maps a point  
$(Y_i,\iota_i, \bar{\lambda}_i, \bar{\eta}_{\mathfrak{q}_i},\bar{\xi}_{\mathfrak{p}_i}, \rho_i)$
from the right hand side of (\ref{uniform3e}) to the point given by the
same data except the $\rho_i$ is replaced by the composite 
\begin{displaymath}
Y_i \times_{\Spec R} \Spec \bar{R} \overset{\rho_i}{\longrightarrow}
\varepsilon_{*} \mathbb{X}_i
\overset{\varepsilon_{*} F_{\mathbb{X}, \tau}}{\longrightarrow}
\varepsilon_{*} \tau_{*} \mathbb{X}_i. 
\end{displaymath}
We denote here by $F_{\mathbb{X}, \tau}$ the Frobenius morphism of the
$p$-divisible group relative to $\kappa_{E_{\nu}}$.
The compatibility of the Weil descent data is explained in the proof
of \cite[Lem. 7.3.1]{KRZ}. 
\end{proof}
For the following one should keep in mind that in (\ref{uniform3e}) we have
used $\tau$ to describe the Weil descent datum and not $\tau^{-1}$ as e.g.
in Proposition \ref{Sh_D1p}. 
Recall from Definition \ref{tildeSh_D1d} that the  scheme $\widetilde{\Sh}_{\mathbf{K}^{\bullet}}(G^{\bullet}, h^{\bullet}_{D})$
over $O_{E_{\nu}}$ is a Galois twist of
$\tilde{{\mathsf A}}^{\bullet t}_{\mathbf{K}^{\bullet}}$ according to the following
diagram 
\begin{equation}\label{uniform5e}
\begin{aligned}
\xymatrix{
  \tilde{\mathsf A}^{\bullet t}_{\mathbf{K}^{\bullet}} \times_{\Spec O_{E_{\nu}}}
  \Spec O_{E^{nr}_{\nu}}
  \ar[d]_{\dot{z}_{\mid \xi} \times \tau_c} \ar[r] & 
 \widetilde{\Sh}_{\mathbf{K}^{\bullet}}(G^{\bullet}, h^{\bullet}_{D})
  \times_{\Spec O_{E_{\nu}}} \Spec O_{E^{nr}_{\nu}} \ar[d]^{\id \times \tau_c}\\
  \tilde{\mathsf A}^{\bullet t}_{\mathbf{K}^{\bullet}} \times_{\Spec O_{E_{\nu}}}
  \Spec O_{E^{nr}_{\nu}} \ar[r] &
  \widetilde{\Sh}_{\mathbf{K}^{\bullet}}(G^{\bullet}, h^{\bullet}_{D}) \times_{\Spec O_{E_{\nu}}}
  \Spec O_{E^{nr}_{\nu}} ,
     }
     \end{aligned}
\end{equation}
where the horizontal arrows are  isomomorphisms. 
Indeed, the morphism (\ref{BZGdot10e}) becomes an isomorphism if we replace on the
left hand side $\mathcal{A}$ by $\mathsf A$. Therefore it follows
from Proposition \ref{Sh_D1p} that
\begin{equation}
  \widetilde{\Sh}_{\mathbf{K}^{\bullet}}(G^{\bullet}, h^{\bullet}_{D}) \times_{\Spec O_{E_{\nu}}} 
  \Spec E_{\nu} \cong \Sh_{\mathbf{K}^{\bullet}}(G^{\bullet}, h^{\bullet}_{D})_{E_{\nu}}.
    \end{equation}

\begin{theorem}\label{4epeg1t}
    Let $\mathbf{K}^{\bullet}=\mathbf{K}^{\bullet}_p\mathbf{K}^{\bullet,p}\subset G^\bullet(\BA_f)$, with $\mathbf{K}^{\bullet}_p$ as in \eqref{BZKpPkt1e}, where
  $\mathbf{M}_{\mathfrak{p}_0} = O_{F_{\mathfrak{p}_0}}^{\times}$. 
  Let $\Pi \in D_{\mathfrak{p}_0}$ be a prime element of this division
  algebra. It acts on
  $V_{\mathfrak{p}_0} = D^{{\rm opp}}_{\mathfrak{p}_0} \otimes_{F_{\mathfrak{p}_0}} K_{\mathfrak{p}_0}$ 
  by multiplication with $\Pi \otimes 1$ from the right. This defines an
  element of $G^{\bullet}_{\mathfrak{p}_0}$ which we denote simply by $\Pi$ and
  we use this notation also for its image by
  $G^{\bullet}_{\mathfrak{p}_0} \subset G^{\bullet}(\mathbb{A}_f)$. The action of
  the Hecke operator is denoted by $\mid_{\Pi}$.  

  Let $\widetilde{\Sh}_{\mathbf{K}^{\bullet}}(G^{\bullet}, h^{\bullet}_{D})$ be the integral model over 
  $O_{E_{\nu}}$ of Definition \ref{tildeSh_D1d}
  of the Shimura variety $\Sh_{\mathbf{K}^{\bullet}}(G^{\bullet}, h^{\bullet}_{D})$. We denote by
  $\widetilde{\Sh}_{\mathbf{K}^{\bullet}}(G^{\bullet}, h^{\bullet}_{D})^\wedge_{/ \Spf O_{E_{\nu}}}$ the
  $p$-adic completion and we set  
  \[ 
  \widetilde{\Sh}_{\mathbf{K}^{\bullet}}(G^{\bullet}, h^{\bullet}_{D})^\wedge_{/ \Spf O_{\breve{E}_{\nu}}} =  \widetilde{\Sh}_{\mathbf{K}^{\bullet}}(G^{\bullet}, h^{\bullet}_{D})^\wedge_{/ \Spf O_{E_{\nu}}} \times_{\Spf O_{E_{\nu}}} \Spf O_{\breve{E}_{\nu}}.
   \]
Then there is an isomorphism of formal schemes
  \begin{equation}\label{unimorph6e} 
    \mathbf{I}(\mathbb{Q}) \backslash (\hat{\Omega}^2_{E_{\nu}} \times
    G^{\bullet}(\mathbb{A}_f)/\mathbf{K}^{\bullet})
    \times_{\Spf O_{E_{\nu}}} \Spf O_{\breve{E}_{\nu}} \overset{\sim}{\longrightarrow}
    \widetilde{\Sh}_{\mathbf{K}^{\bullet}}(G^{\bullet}, h^{\bullet}_{D})^\wedge_{/ \Spf O_{\breve{E}_{\nu}}} 
  \end{equation}
  For varying $\mathbf{K}^{\bullet}$ this morphism is compatible with the action
  of $G^{\bullet}(\mathbb{A}_f)$ by Hecke operators on both sides. 
  
  Let $\tau \in \Gal(\breve{E}_{\nu}/ E_{\nu})$ the Frobenius and 
  $\tau_c = \Spf\tau^{-1}: \Spf O_{\breve{E}_{\nu}}\rightarrow\Spf O_{\breve{E}_{\nu}}$.
  The canonical Weil descent datum on the right hand side of (\ref{unimorph6e})
  is given on the left hand side by the commutative diagram 
  \begin{displaymath}
\xymatrix{
  \mathbf{I}(\mathbb{Q}) \backslash (\hat{\Omega}^2_{E_{\nu}} \times
  G^{\bullet}(\mathbb{A}_f)/\mathbf{K}^{\bullet})
    \times_{\Spf O_{E_{\nu}}} \Spf O_{\breve{E}_{\nu}}
  \ar[d]_{  \mid_{\Pi^{-1}}\times\id  \times \tau_c} \ar[r] & 
 \widetilde{\Sh}_{\mathbf{K}^{\bullet}}(G^{\bullet}, h^{\bullet}_{D})^\wedge_{/ \Spf O_{\breve{E}_{\nu}}}
  \ar[d]^{\id \times \tau_c}\\
  \mathbf{I}(\mathbb{Q}) \backslash (\hat{\Omega}^2_{E_{\nu}} \times
  G^{\bullet}(\mathbb{A}_f)/\mathbf{K}^{\bullet})
    \times_{\Spf O_{E_{\nu}}} \Spf O_{\breve{E}_{\nu}}
  \ar[r] & 
  \widetilde{\Sh}_{\mathbf{K}^{\bullet}}(G^{\bullet}, h^{\bullet}_{D})^\wedge_{/ \Spf O_{\breve{E}_{\nu}}}
     }
    \end{displaymath}
\end{theorem}
\begin{proof}
  From the morphism $\Theta^\bullet$ in (\ref{unimorph5e}) and the definition
  and the horizontal line of the diagram (\ref{uniform5e}) 
we obtain an isomorphism of formal schemes over $\Spf O_{\breve{E}_{\nu}}$ 
  \begin{equation}\label{4epeg1e} 
\mathbf{I}(\mathbb{Q}) \backslash ({\rm RZ}_{p,\mathbf{K}^{\bullet}_p}
\times G^{\bullet}(\mathbb{A}_f^p)/(\mathbf{K}^{\bullet})^p)
\overset{\sim}{\longrightarrow}
\widetilde{\Sh}_{\mathbf{K}^{\bullet}}(G^{\bullet}, h^{\bullet}_{D})^\wedge_{/ \Spf O_{\breve{E}_{\nu}}}. 
  \end{equation} 
We obtain the isomorphism (\ref{unimorph6e}) if we rewrite the left hand side
using the Propositions \ref{RZ6p} and \ref{RZ7p}. 
We consider an $R$ and $\varepsilon$
  as in the definition (\ref{uniform3e}) of $\omega_{\mathfrak{p}_i}$.  
  Let us denote the functor on the left hand side of (\ref{4epeg1e})
  by $\mathcal{F}$. It is endowed with its natural Weil descent datum
  $\mathcal{F}(R) \rightarrow \mathcal{F}(R_{[\tau]})$ given by the 
  $\omega_{\mathfrak{p}_i}$, $i = 0, \ldots, s$. By \eqref{uniform5e}, the 
  morphism (\ref{4epeg1e}) becomes compatible with the Weil descent 
  data if we multiply the natural Weil descent datum on $\mathcal{F}$  by the
  operator 
  $\dot{z}_{\mid \xi}^{-1}$. The exponent $-1$ appears because we use here
  $\tau$ instead of $\tau^{-1}$ as in the statement of the theorem. By the
  explanation after (\ref{xi-action1e}) 
  this means that we have to replace $\omega_{\mathfrak{p}_o}$ by
  $\omega_{\mathfrak{p}_0}(1, \pi_{\mathfrak{p}_0}p^{-f_{\nu}})$, resp.
  $\omega_{\mathfrak{p}_j}$ by $\omega_{\mathfrak{p}_j}(1, p^{-f_{\nu}})$, where
  $(1, \pi_{\mathfrak{p}_0}p^{-f_{\nu}}) \in G^{\bullet}_{\mathfrak{p}_0}$ and
  $(1, p^{-f_{\nu}}) \in G^{\bullet}_{\mathfrak{p}_j}$, cf. Proposition \ref{Sh_D1p}.
  
  We first check what this modified Weil descent datum does on
  ${\rm RZ}_{\mathfrak{p}_0}$. Let 
  $(Y,\iota, \bar{\lambda}, \rho) \in {\rm RZ}_{\mathfrak{p}_0}(R)$. The
  action of the Hecke operator $(1, \pi_{\mathfrak{p}_0}p^{-f_{\nu}})$ is the
  same as $\mathfrak{h}((1, \pi_{\mathfrak{p}_0}p^{-f_{\nu}}))$, where we regard
  $u:= (1,\pi_{\mathfrak{p}_0}p^{-f_{\nu}})$ as an element of
  $\mathcal{H}_{\mathfrak{p}_0}$, cf. \eqref{HOat0}. We note that
  $\iota^u = \iota$ because $u$ lies in
  the center. Therefore the action of the Hecke operator $\mathfrak{h}(u)$ is
  \begin{displaymath}
(Y,\iota, \bar{\lambda}, \rho) \mapsto (Y,\iota, \bar{\lambda}, \rho\circ\iota(u)). 
    \end{displaymath}
  The image of $(Y,\iota, \bar{\lambda}, \rho)$ by the map

  \begin{displaymath}
    \omega_{\mathfrak{p}_0} \mathfrak{h}(u): {\rm RZ}_{\mathfrak{p}_0}(R) \rightarrow
    {\rm RZ}_{\mathfrak{p}_0}(R[\tau]) 
    \end{displaymath}
  is $(Y,\iota, \bar{\lambda}, \rho')$, where $\rho'$ is given by
    \begin{displaymath}
\xymatrix{
  (Y_{\mathfrak{q}_0})_{\bar{R}} \times (Y_{\bar{\mathfrak{q}}_0})_{\bar{R}}
  \ar[r]^{\rho_{\mathfrak{q}_0} \times \rho_{\bar{\mathfrak{q}}_0}} & 
  \varepsilon_{*} \mathbb{X}_{\mathfrak{q}_0} \times 
  \varepsilon_{*} \mathbb{X}_{\bar{\mathfrak{q}}_0}  
  \ar[rrr]^{\varepsilon_{*} F_{\mathbb{X}_{\mathfrak{q}_0},\tau} \times
    \varepsilon_{*} \pi_{\mathfrak{p}_0}
    p^{-f_{\nu}} F_{\mathbb{X}_{\bar{\mathfrak{q}}_0},\tau}}
  & & & \qquad 
   \varepsilon_{*} \tau_{*} \mathbb{X}_{\mathfrak{q}_0} \times
   \varepsilon_{*} \tau_{*} \mathbb{X}_{\bar{\mathfrak{q}}_0} .
}
    \end{displaymath}
We note that the Weil descent datum commutes with all Hecke operators.
It is straightforward to compute the following heights,
\begin{displaymath}
  \height_{F_{\mathfrak{p}_0}}  F_{\mathbb{X}_{\mathfrak{q}_0},\tau} = 2, \quad
  \height_{F_{\mathfrak{p}_0}} 
  \pi_{\mathfrak{p}_0}  p^{-f_{\nu}} F_{\mathbb{X}_{\bar{\mathfrak{q}}_0},\tau} = 2. 
  \end{displaymath}
Therefore $\omega_{\mathfrak{p}_0} \mathfrak{h}(u)$ is of degree $(1,1)$, 
\begin{displaymath}
\omega_{\mathfrak{p}_0} \mathfrak{h}(u): {\rm RZ}_{\mathfrak{p}_0}(a, b) \rightarrow
  {\rm RZ}_{\mathfrak{p}_0}(a + 1, b +1) .
  \end{displaymath}
The Hecke operator $\mid_{\Pi}$ has also degree $(1,1)$. We write 
$\omega_{\mathfrak{p}_0} \mathfrak{h}(u) = \mid_{\Pi} (\mid_{\Pi})^{-1} \omega_{\mathfrak{p}_0} \mathfrak{h}(u)$. 
The Weil descent datum $(\mid_{\Pi})^{-1} \omega_{\mathfrak{p}_0} \mathfrak{h}(u)$
is of degree $(0,0)$ and defines therefore a Weil descent datum on
${\rm RZ}_{\mathfrak{p}_0}(0,0)$. By the isomorphism of Lemma \ref{RZ7l}, it
induces a Weil descent datum on
\begin{equation}\label{uniform10e} 
\hat{\Omega}_{F_{\mathfrak{p}_0}}^2\times_{\Spf O_{F_{\mathfrak{p}_0}},\varphi_0}\Spf O_{\breve{E}_\nu}.
  \end{equation}

But this isomorphism is just the projection to the $\mathfrak{q}_0$-part.
Therefore the induced Weil descent datum maps an $R$-valued point
$(Y_{\mathfrak{q}_0},\iota_{\mathfrak{q}_0}, \rho_{\mathfrak{q}_0})$ to the point
$(Y_{\mathfrak{q}_0}^{\Pi^{-1}},\iota_{\mathfrak{q}_0}^{\Pi^{-1}}, \rho'_{\mathfrak{q}_0})$ 
where $\rho'_{\mathfrak{q}_0}$ is the following composite 
\begin{displaymath}
  (Y_{\mathfrak{q}_0}^{\Pi^{-1}}) \overset{\rho_{\mathfrak{q}_0}}{\longrightarrow}
  \varepsilon_{*} \mathbb{X}_{\mathfrak{q}_0}^{\Pi^{-1}} 
  \overset{\varepsilon_* F_{\mathbb{X}_{\mathfrak{q}_0},\tau}}{\longrightarrow}
  \varepsilon_{*} \tau_{*} \mathbb{X}_{\mathfrak{q}_0}^{\Pi^{-1}} 
  \overset{\iota(\Pi^{-1})}{\longrightarrow}
  \varepsilon_{*} \tau_{*} \mathbb{X}_{\mathfrak{q}_0}. 
\end{displaymath}
But by the proof of \cite[Prop. 5.1.7]{KRZ}, this is exactly the natural
Weil descent datum on the base change 
$\hat{\Omega}_{F_{\mathfrak{p}_0}}^2\times_{\Spf O_{F_{\mathfrak{p}_0}},\varphi_0}\Spf O_{\breve{E}_\nu}$.
Therefore the Weil descent datum on this factor is as claimed.

Now we consider the Weil descent data $\omega_{\mathfrak{p}_j} (1,p^{-f_{\nu}})$ on
${\rm RZ}_{\mathfrak{p}_j, \mathbf{K}^{\bullet}_{\mathfrak{p}_j}}$.
We have to show that under the isomorphism of Proposition \ref{RZ6p}  
\begin{equation}\label{uniform8e}
    {\rm RZ}_{\mathfrak{p}_j, \mathbf{K}^{\bullet}_{\mathfrak{p}_j}}
    \overset{\sim}{\longrightarrow}
    G^{\bullet}_{\mathfrak{p}_j}/\mathbf{K}^{\bullet}_{\mathfrak{p}_j}
    \times_{\Spf O_{E_{\nu}},\varphi_0}\Spf O_{\breve{E}_\nu} ,
\end{equation}
the Weil descent datum $\omega_{\mathfrak{p}_j}(1, p^{-f_{\nu}})$ induces on the
right hand side the natural descent datum.
The formal group $(\mathbb{X}_j, \iota_j)$ which was used to define
(\ref{uniform8e}) is defined over
the field $\kappa_{E_{\nu}}$, as we see from (\ref{RZ1e}). Therefore
${\rm RZ}_{\mathfrak{p}_i, \mathbf{K}^{\bullet}_{\mathfrak{p}_i}}$ is defined over
$\kappa_{E_{\nu}}$ and Proposition \ref{RZ6p} holds over that field.  
We obtain a natural isomorphism $\mathbb{X}_i \cong \tau_{*} \mathbb{X}_i$.
With this identification, the morphism
\begin{displaymath}
\xymatrix{
   \mathbb{X}_{\mathfrak{q}_i} \times \mathbb{X}_{\bar{\mathfrak{q}}_i}   
  \ar[rrr]^{ F_{\mathbb{X}_{\mathfrak{q}_i},\tau} \times
        p^{-f_{\nu}} F_{\mathbb{X}_{\bar{\mathfrak{q}}_i},\tau} \qquad}
  & & & \quad 
  \tau_{*} \mathbb{X}_{\mathfrak{q}_i} \times
  \tau_{*} \mathbb{X}_{\bar{\mathfrak{q}}_i}  \cong 
  \mathbb{X}_{\mathfrak{q}_i} \times \mathbb{X}_{\bar{\mathfrak{q}}_i} 
}
\end{displaymath}
becomes the identity. We see that $\omega_{\mathfrak{p}_i}(1, p^{-f_{\nu}})$ induces
on the right hand side of (\ref{uniform8e}) the natural descent datum. 
This proves the commutativity of the last diagram in the theorem. 
\end{proof}
We now turn to the $p$-adic uniformization of the quaternionic Shimura curve. We first recall the following well-known fact. 
\begin{lemma}
  Let $K/F$ be separable quadratic extension of fields. Let $L$ be a
  quaternion algebra with center $K$ and let $l \mapsto l'$ be an involution
  of the second kind of $L$. Then there exist a quaternion algebra $C$ with
  center   $F$ and an isomorphism of $K$-algebras
  \begin{displaymath}
L \cong C \otimes_F K,  
  \end{displaymath}
  such that the involution $l \mapsto l'$ induces on the right hand side the map $c \otimes k \mapsto c^{\iota} \otimes \sigma(k)$, where
  $c^{\iota}$ is the main involution and $\sigma$ is the nontrivial element in
  $\Gal(K/F)$. 
\end{lemma}
\begin{proof}
  We consider the main involution $l \mapsto l^{\iota}$ of $L$ over $K$. It is
  characterized by $l + l^{\iota} = \Trace^{o}_{L/K}$. Since the reduced trace
  is respected by an isomorphism of $F$-algebras, one verifies easily that  
  \begin{displaymath}
(l')^{\iota} = (l^{\iota})'. 
  \end{displaymath}
  Therefore $\rho(l) := (l^{\iota})'$ is a $\sigma$-linear isomorphism
  $L \rightarrow L$ such that $\rho^2 = \id_{L}$.
  The invariants define $C$ by Galois descent. 
\end{proof}
As an example we consider
$(A_{o}, \iota_{o}, \bar{\lambda}_{o})$, cf. \eqref{uniform2e}. The ring $L = \End^{o}_B A_{o}$ is by
(\ref{uniform1e}) a quaternion algebra with center $K$. Let $l \mapsto l'$
be the Rosati involution induced by $\lambda_{o}$. We define
the Cherednik twist of $D$, 
\begin{equation}\label{Cheredniktwist} 
\check{D} = \{l \in \End^{o}_B A_{o}\; | \; l' = l^{\iota}\}. 
  \end{equation}
This is a quaternion algebra over $F$. Since $l l^{\iota} \in F$, we obtain
by (\ref{groupI1e}) that
\begin{equation}\label{breveD1e}
\check{D}^{\times} \subset \mathbf{I}(\mathbb{Q}). 
  \end{equation} 
Because the Rosati involution is
positive, the main involution is positive on $\check{D}$. It follows that
at each infinite place of $F$, the localization of $\check{D}$ is isomorphic
to the Hamilton quaternions. By (\ref{uniform1e}) we find
\begin{displaymath}
\End^{o}_B A_{o} \otimes \mathbb{A}^p_f \cong
B^{{\rm opp}} \otimes \mathbb{A}^p_f \cong D \otimes_{F} K. 
  \end{displaymath}
Since the Riemann form of $\lambda_{o}$ induces via
$\eta^p_{o}$ on $V \otimes \mathbb{A}_f^p$ the form $\psi$
(up to a constant), we see that the induced involution on
$B^{{\rm opp}} \cong D \otimes_{F} K$ is by (\ref{BZpsi3e}) the involution
$d \otimes k \mapsto d^{\iota} \otimes \bar{k}$. We obtain the isomorphism
\begin{equation}\label{breveD2e} 
\check{D} \otimes \mathbb{A}^p_f \cong D \otimes \mathbb{A}^p_f. 
\end{equation}
The data $\eta_{\mathfrak{q}_j, o}$ for $j = 1, \ldots, s$ in
 Definition \ref{BZsApkt4d} provide isomorphisms
\begin{displaymath}
  (\Lambda_{\mathfrak{p}_j}^{pd},\bar{\lambda}_{\psi}) \isoarrow
  (\mathbb{X}_j, \iota_{\mathbb{X}_j}, \bar{\lambda}_{\mathbb{X}_j}). 
\end{displaymath}
From this we obtain
\begin{displaymath}
\End^{o}_{B} A_{o} \otimes_{F} F_{\mathfrak{p}_j} =
\End_{B_{\mathfrak{p}_j}} V_{\mathfrak{p}_j} =
B_{\mathfrak{q}_j}^{{\rm opp}} \times B_{\bar{\mathfrak{q}}_j}^{{\rm opp}}.   
\end{displaymath}
The Rosati involution of $\lambda_{o}$ induces on the right hand side the map 
$(b_1, b_2) \mapsto (b'_2, b'_1)$. We obtain
\begin{equation}\label{breveD3e}
  \check{D}_{\mathfrak{p}_j} =
  \{(b_1,b_2)\in B_{\mathfrak{q}_j}^{{\rm opp}} \times B_{\bar{\mathfrak{q}}_j}^{{\rm opp}} \; | \;
  (b'_2, b'_1) = (b_1^{\iota}, b_2^{\iota})\} \cong B_{\mathfrak{q}_j}^{{\rm opp}} =
  D_{\mathfrak{p}_j} ,
\end{equation}
where the last isomorphism is given by the projection. Comparing with 
(\ref{BZGpi2e})  gives the embedding
\begin{equation}\label{uniform18e}
  \check{D}^{\times}_{\mathfrak{p}_j} \subset G^{\bullet}_{\mathfrak{p}_j}, \quad
   \; j = 1, \ldots, s,  
\end{equation}
which coincides with the inclusion 
$\check{D}^{\times}_{\mathfrak{p}_j}\subset\mathbf{I}_{\mathfrak{p}_j} = G^{\bullet}_{\mathfrak{p}_j}$ via (\ref{groupI2e}). 

We obtain in the same way 
\begin{displaymath}
\End^{o}_{B} A_{o} \otimes_{F} F_{\mathfrak{p}_0} \cong 
\End_{B_{\mathfrak{p}_0}} \mathbb{X}_0 \cong
\End_{B_{\mathfrak{q}_0}} \mathbb{X}_{\mathfrak{q}_0} \times 
\End_{B_{\bar{\mathfrak{q}}_0}} \mathbb{X}_{\bar{\mathfrak{q}}_0} ,
\end{displaymath}
and therefore 
\begin{equation}\label{uniform20e}
  \check{D}_{\mathfrak{p}_0} = \{(\gamma_1, \gamma_2) \in
  \End_{B_{\mathfrak{q}_0}} \mathbb{X}_{\mathfrak{q}_0} \times 
\End_{B_{\bar{\mathfrak{q}}_0}} \mathbb{X}_{\bar{\mathfrak{q}}_0}\mid
\gamma'_2 = \gamma_1^{\iota} \; \}. 
  \end{equation}
The projection to the first factor defines an isomorphism
\begin{equation}\label{uniform30e} 
  \check{D}_{\mathfrak{p}_0} \cong \End_{B_{\mathfrak{q}_0}} \mathbb{X}_{\mathfrak{q}_0}
\cong {\rm M}_2(F_{\mathfrak{p}_0}). 
\end{equation}
Altogether we find that the quaternion algebras $D$ and $\check{D}$ over
$F$ have the same invariants for all places except for $\mathfrak{p}_0$
and the infinite place $\chi_0 ={\varphi_0}_{\mid F}$. For the last two places
they have {\rm opp}osite invariants. 

We denote by $H$ the multiplicative group of $D$ considered as an algebraic
group over $\mathbb{Q}$. The Shimura curve ${\Sh}(H, h_{D})$ is 
defined over $E(H, h_D) = \chi_0(F)$. We have $E(H, h_D) \subset E$ 
and $E(H, h_D) \rightarrow E_{\nu}$ is a $p$-adic place of $E(H, h_D)$, cf.  (\ref{BZE1e}), resp.   (\ref{BZ4e}).

Let
$\mathbf{K}_{\mathfrak{p}_0} = O_{D_{\mathfrak{p}_0}}^{\times}\subset D_{\mathfrak{p}_0}^{\times}$
be the unique maximal compact subgroup. We choose for $i = 1,\ldots, s$ arbitrary open compact
subgroups $\mathbf{K}_{\mathfrak{p}_i} \subset D^\times_{\mathfrak{p}_i}$ and set $\mathbf{K}_{p}=\prod_{i=0}^s\mathbf{K}_{\mathfrak{p}_i}$. Finally, we choose an arbitrary open compact subgroup 
 $\mathbf{K}^{p}\subset (D\otimes \mathbb{A}_f^p)^{\times}$ and set
\begin{equation}\label{uniform25e} 
  \mathbf{K} = \mathbf{K}_{p} \mathbf{K}^p \subset
  H(\mathbb{A}_f). 
  \end{equation}
(Note that, since the group $G$ plays no role anymore, we can recycle the notation used in \eqref{BZ7e}.) The natural embedding $D \rightarrow D \otimes_F K = B^{{\rm opp}}$ induces by
(\ref{Gpunkt2e}) an embedding $H \subset G^{\bullet}$. In terms of the exact
sequence of Lemma \ref{BZ1l} this embedding is obtained from
$D^{\times}\rightarrow D^{\times}\times K^{\times}\rightarrow G^{\bullet}(\mathbb{Q})$.

We find isomorphisms
\begin{equation}\label{uniform21e}
  \begin{aligned} 
    D_{\mathfrak{p}_i}^{\times} \cong (B^{{\rm opp}}_{\mathfrak{q}_i})^{\times} &=
    \Aut_{B_{\mathfrak{q}_i}} V_{\mathfrak{q}_i} \\ 
    D^{\times}_{\mathfrak{p}_i} \cong (B^{{\rm opp}}_{\bar{\mathfrak{q}}_i})^{\times} &= 
    \Aut_{B_{\bar{\mathfrak{q}}_i}} V_{\bar{\mathfrak{q}}_i} .
  \end{aligned}
  \end{equation}
We define $\mathbf{K}_p^{\bullet} \subset G^{\bullet}(\mathbb{Q}_p)$ in the form
(\ref{BZKpPkt1e}) using the isomorphism above such that
\begin{equation}\label{uniform27e}
  \mathbf{K}^{\bullet}_{\mathfrak{q}_i} = \mathbf{K}_{\mathfrak{p}_i}, \quad
  \mathbf{M}^{\bullet}_{\mathfrak{p}_i} = O_{F_{\mathfrak{p}_i}}^{\times}, \quad
 i= 0, \cdots, s. 
  \end{equation}
Then $\mathbf{K}_p = H(\mathbb{Q}_p) \cap \mathbf{K}^{\bullet}_p$.  
By Proposition \ref{Chevalley2p} below, there is an open compact
subgroup $\mathbf{K}^{\bullet, p} \subset G^{\bullet}(\mathbb{A}_f^p)$ such that
for $\mathbf{K}^{\bullet} = \mathbf{K}^{\bullet}_p \mathbf{K}^{\bullet, p}$ we have 
\begin{displaymath}
\mathbf{K} = H(\mathbb{A}_f) \cap \mathbf{K}^{\bullet} ,
  \end{displaymath}
and the natural morphism
\begin{displaymath}
  \Sh_{\mathbf{K}}(H, h_{D}) \times_{\Spec E(H, h_D)} \Spec E
  \rightarrow \Sh_{\mathbf{K}^{\bullet}}(G^{\bullet}, h^{\bullet}_{D})  
  \end{displaymath}
is an open and closed immersion.  
It follows for example by Theorem \ref{4epeg1t} that
$\widetilde{\Sh}_{\mathbf{K}^{\bullet}}(G^{\bullet}, h^{\bullet}_{D})$ is a flat
and integral scheme over $O_{E_{\nu}}$. Therefore the inclusion of the generic fiber 
\begin{equation}\label{uniform22e} 
  \Sh_{\mathbf{K}^{\bullet}}(G^{\bullet}, h^{\bullet}_{D}) \times_{\Spec E} \Spec E_{\nu}
  \subset \widetilde{\Sh}_{\mathbf{K}^{\bullet}}(G^{\bullet}, h^{\bullet}_{D})
\end{equation}
induces a bijection between the sets of connected components of these schemes. 
The connected components of the right hand side are the closures of the
connected components of the left hand side.
\begin{definition}\label{uniform3d}
  We define the scheme $\widetilde{\Sh}_{\mathbf{K}}(H, h_{D})$ over
  $\Spec O_{E_{\nu}}$ as the Zariski closure of the
  open and closed subscheme
  $\Sh_{\mathbf{K}}(H, h_{D}) \times_{\Spec E(G, h_D)} \Spec E_{\nu}$ of the
 left hand side of (\ref{uniform22e}) in 
$\widetilde{\Sh}_{\mathbf{K}^{\bullet}}(G^{\bullet}, h^{\bullet}_{D})$.  
Hence $\widetilde{\Sh}_{\mathbf{K}}(H, h_{D})$ is a union of connected
components of $\widetilde{\Sh}_{\mathbf{K}^{\bullet}}(G^{\bullet}, h^{\bullet}_{D})$. 
\end{definition}
We consider the diagonal embedding 
\begin{displaymath}
  D_{\mathfrak{p}_0}^{\times} \subset G^{\bullet}_{\mathfrak{p}_0} \subset
  B^{{\rm opp},\times}_{\mathfrak{q}_0} \times B^{{\rm opp}, \times}_{\bar{\mathfrak{q}}_0}, 
  \end{displaymath}
cf. (\ref{BZGpi2e}), (\ref{uniform21e}). It defines an open and closed
embedding 
    \begin{equation*}
  \begin{aligned}
    (\hat{\Omega}_{F_{\mathfrak{p}_0}}^2 \times_{\Spf O_{F_{\mathfrak{p}_0}},\varphi_0}
    \Spf O_{\breve{E}_\nu}) \times D^{\times}_{\mathfrak{p}_0}/\mathbf{K}_{\mathfrak{p}_0}
    \subset  
    (\hat{\Omega}_{F_{\mathfrak{p}_0}}^2 \times_{\Spf O_{F_{\mathfrak{p}_0}},\varphi_0}
    \Spf O_{\breve{E}_\nu}) \times
    G^{\bullet}_{\mathfrak{p}_0} / \mathbf{K}^{\bullet}_{\mathfrak{p}_0} =
    {\rm RZ}_{\mathfrak{p}_0} .
    \end{aligned}
  \end{equation*}
In terms of the decomposition (\ref{RZ12e}), the left hand side is
\begin{displaymath}
\coprod_{a \in \mathbb{Z}} {\rm RZ}_{\mathfrak{p}_0}(a,a). 
\end{displaymath}
This is invariant by the action of
$\check{D}^{\times}_{\mathfrak{p}_0} \subset \mathbf{I}_{\mathfrak{p}_0}$
on ${\rm RZ}_{\mathfrak{p}_0}$. Therefore $\check{D}^{\times}_{\mathfrak{p}_0}$
acts on
\begin{equation}\label{uniform23e}
(\hat{\Omega}_{F_{\mathfrak{p}_0}}^2 \times_{\Spf O_{F_{\mathfrak{p}_0}},\varphi_0}
    \Spf O_{\breve{E}_\nu}) \times D^{\times}_{\mathfrak{p}_0}/\mathbf{K}_{\mathfrak{p}_0} .
\end{equation}
We now combine the actions of $\check{D}^\times$ in its prime-to-$p$ component (\ref{breveD2e}), its $\mathfrak{p}_0$-component (\ref{uniform23e}) and its $\mathfrak{p}_j$-components (\ref{breveD3e}).
\begin{proposition}\label{uniform4l}
  Let $\mathbf{K}_p \subset (D \otimes \mathbb{Q}_p)^{\times}$ as in
  (\ref{uniform25e}).
  By (\ref{uniform27e}) this defines a subgroup
  $\mathbf{K}^{\bullet}_p \subset G^{\bullet}(\mathbb{Q}_p)$. Let
  $\mathbf{K}^{\bullet,p} \subset G^{\bullet}(\mathbb{A}_f^p)$ be a sufficiently
  small open compact subgroup. We set
  $\mathbf{K}^{p} = \mathbf{K}^{\bullet, p} \cap (D\otimes\mathbb{A}_f^p)^{\times}$.
  Let $\mathbf{K} = \mathbf{K}_{p} \mathbf{K}^p$ and
  $\mathbf{K}^{\bullet} = \mathbf{K}^{\bullet}_{p} \mathbf{K}^{\bullet,p}$ as above. 
  Then the morphism
  \begin{displaymath}
    \begin{array}{l} 
    \check{D}^{\times} \backslash ((\hat{\Omega}_{F_{\mathfrak{p}_0}}^2
    \times_{\Spf O_{F_{\mathfrak{p}_0}},\varphi_0} \Spf O_{\breve{E}_\nu})\times
    D^{\times}(\mathbb{A}_f)/\mathbf{K}) \longrightarrow \\[2mm]  \hspace{5cm}   
    \mathbf{I}(\mathbb{Q}) \backslash ((\hat{\Omega}^2_{E_{\nu}}
    \times_{\Spf O_{E_{\nu}}} \Spf O_{\breve{E}_{\nu}}) \times
    G^{\bullet}(\mathbb{A}_f)/\mathbf{K}^{\bullet}) 
      \end{array} 
  \end{displaymath}
  is an open and closed immersion.
\end{proposition}
\begin{proof}
  The obvious sequence of algebraic groups over $\mathbb{Q}$ which is
  induced by (\ref{breveD1e}) 
  \begin{equation}\label{BZ1l-var} 
  0 \rightarrow F^{\times} \rightarrow \check{D}^{\times} \times K^{\times}
    \overset{\kappa}{\rightarrow}
    \mathbf{I} \rightarrow 0 
  \end{equation}
  is exact. By Hilbert 90 the sequence remains exact if we take the
  $\mathbb{Q}$-valued points
  $0\rightarrow F^{\times}\rightarrow\check{D}^{\times}\times K^{\times}\overset{\kappa}{\rightarrow} \mathbf{I}(\mathbb{Q}) \rightarrow 0$. 
  To check the exactness, we can make a base change from $\mathbb{Q}$
to an algebraically closed field where there is no difference to Lemma 
\ref{BZ1l}. We will regard (\ref{BZ1l-var}) also as a sequence of algebraic
groups over $F$. Note that
$\mathbf{I}$ is the Weil restriction $\Res_{F/\mathbb{Q}} \tilde{\mathbf{I}}$
where $\tilde{\mathbf{I}}$ is an algebraic group over $F$ which is defined
in terms of the $F$-algebra $\End_B^o A_{o}$, cf. (\ref{groupI1e}). Recall that
$\tilde{\mathbf{I}}(F) = \mathbf{I}(\mathbb{Q})$. 
For $i = 0, \ldots s$ we write the group $\tilde{\mathbf{I}}(F_{\mathfrak{p}_i})$
as follows
\begin{equation}\label{Le66-1e} 
  \begin{array}{ccccc} 
\check{D}_{\mathfrak{p}_i}^{\times} \times \check{D}_{\mathfrak{p}_i}^{\times} & \supset 
&  \check{D}_{\mathfrak{p}_i}^{\times} \times^{F_{\mathfrak{p}_i}} (K_{\mathfrak{q}_i} \times
  K_{\bar{\mathfrak{q}}_i}) & \overset{\sim}{\longrightarrow} & 
  \check{D}_{\mathfrak{p}_i}^{\times} \times F_{\mathfrak{p}_i}^{\times} = 
\tilde{\mathbf{I}}(F_{\mathfrak{p}_i})  \\
  (df_1, d f_2) & \hookleftarrow & (d, (f_1,f_2)) &\longmapsto &
  (df_1, f_1^{-1}f_2)\\
\end{array} 
\end{equation}
The canonical embedding
$\check{D}_{\mathfrak{p}_i}^{\times} \subset \tilde{\mathbf{I}}(F_{\mathfrak{p}_i})$
becomes
$\check{D}_{\mathfrak{p}_i}^{\times} \rightarrow \check{D}_{\mathfrak{p}_i}^{\times} \times F_{\mathfrak{p}_i}^{\times}$, 
$d \mapsto (d,1)$. 

For $j \geq 1$, 
$\mathbf{K}_{\mathfrak{p}_j}\subset D^{\times}_{\mathfrak{p}_j} =\check{D}^{\times}_{\mathfrak{p}_j}$
is the arbitrary given compact open subgroup and
\begin{equation}\label{KPJ}
  \mathbf{K}_{\mathfrak{p}_j}^{\bullet} = \mathbf{K}_{\mathfrak{p}_j} \times
  O_{F_{\mathfrak{p}_j}}^{\times} \subset
  D_{\mathfrak{p}_j}^{\times} \times F_{\mathfrak{p}_j}^{\times} = 
  \check{D}_{\mathfrak{p}_j}^{\times} \times F_{\mathfrak{p}_j}^{\times} =
  \tilde{\mathbf{I}}(F_{\mathfrak{p}_j}).  
\end{equation}
We see again that
$\mathbf{K}_{\mathfrak{p}_j}^{\bullet} \cap \check{D}_{\mathfrak{p}_j}^{\times} = \mathbf{K}_{\mathfrak{p}_j}$.
We introduce the groups
$C_{F, \mathfrak{p}_j} = \mathbf{K}_{\mathfrak{p}_j} \cap F_{\mathfrak{p}_j}^{\times}$ 
and $C_{K, \mathfrak{p}_j} = C_{F, \mathfrak{p}_j} \times O_{\mathfrak{p}_j}^{\times} \subset K_{\mathfrak{p}_j}^{\times}$  
in the sense of the right hand side of \eqref{KPJ}. Then we obtain
\begin{equation}\label{Le66-2e}
  \mathbf{K}_{\mathfrak{p}_j}^{\bullet} = \mathbf{K}_{\mathfrak{p}_j}
  \times^{C_{F, \mathfrak{p}_j}} C_{K, \mathfrak{p}_j} \subset
  \check{D}^{\times}_{\mathfrak{p}_j} \times^{F_{\mathfrak{p}_j}^{\times}}
  K_{\mathfrak{p}_j}^{\times}. 
  \end{equation}

For the proof we may assume that $\mathbf{K}^{\bullet,p}$ is of the following
type. Let $C_F^p = \mathbf{K}^p \cap (\mathbb{A}^p_{F,f})^{\times}$. We choose
an arbitrary open compact subgroup
$C_K^{p} \subset (K \otimes_F \mathbb{A}^p_{F,f})^{\times}$ such that
$C_K^{p} \cap (\mathbb{A}^p_{F,f})^{\times} = C_F^p$. We define
$\mathbf{K}^{\bullet,p}_C$ as the image of $\mathbf{K}^p \times C^p_K$
by the homomorphism
\begin{equation}\label{Le66-3e}
  (\check{D} \otimes_F \mathbb{A}^p_{F,f})^{\times} \times
  (K \otimes_F \mathbb{A}^p_{F,f})^{\times} \rightarrow
  \tilde{\mathbf{I}}(\mathbb{A}^p_{F,f}).  
  \end{equation}
We set
\begin{displaymath}
  \mathbf{K}^{\mathfrak{p}_0} = (\prod_{j=1}^s \mathbf{K}_{\mathfrak{p}_j})
  \mathbf{K}^p 
  \subset (\check{D} \otimes_F \mathbb{A}^{\mathfrak{p}_0}_{F,f})^{\times}, \quad
  \mathbf{K}^{\bullet, \mathfrak{p}_0}_C =
  (\prod_{j=1}^s \mathbf{K}^{\bullet}_{\mathfrak{p}_j}) \mathbf{K}_C^{\bullet, p}
  \subset \tilde{\mathbf{I}}(\mathbb{A}^{\mathfrak{p}_0}_{F,f}). 
  \end{displaymath}
Moreover we set
$\mathcal{Z} = \hat{\Omega}_{F_{\mathfrak{p}_0}}^2 \times_{\Spf O_{F_{\mathfrak{p}_0}},\varphi_0} \Spf O_{\breve{E}_\nu}$ 
and $\Lambda = \{(a,b) \in \mathbb{Z}^2 \mid a+b \equiv 0 \mod 2 \}$. 
Then we may write the morphism of the Proposition as follows,
\begin{equation}\label{DIPfeil1e} 
  \check{D}^{\times} \backslash (\mathcal{Z} \times \mathbb{Z}) \times
  (\check{D} \otimes_F \mathbb{A}^{\mathfrak{p}_0}_{F,f})^{\times}/\mathbf{K}^{\mathfrak{p}_0}
  \longrightarrow
  \tilde{\mathbf{I}}(F) \backslash (\mathcal{Z} \times \Lambda) \times
  \tilde{\mathbf{I}}(\mathbb{A}^{\mathfrak{p}_0}_{F,f})
  /\mathbf{K}_C^{\bullet, \mathfrak{p}_0}. 
\end{equation}
The group $\check{D}^{\times}$ acts on $\mathcal{Z}$ via
$\check{D}^{\times} \rightarrow \check{D}^{\times}_{\mathfrak{p}_0} \rightarrow \PGL_2(F_{\mathfrak{p}_0})$, 
as we have explained at the end of section 5. We will denote the image in
$\PGL_2(F_{\mathfrak{p}_0})$ of an element $g$ by the last map by $\bar{g}$.
An element $g \in \check{D}^{\times}_{\mathfrak{p}_0}$ acts on
$\mathcal{Z} \times \mathbb{Z}$ by
\begin{displaymath}
  g[\omega, m] = [\bar{g} \omega, m + \ord_{\mathfrak{p}_0} \det g], \quad
  [\omega, m] \in \mathcal{Z} \times \mathbb{Z}.  
\end{displaymath}
An element
$(g,f) \in \check{D}_{\mathfrak{p}_0}^{\times} \times F_{\mathfrak{p}_0}^{\times} = \tilde{\mathbf{I}}(F_{\mathfrak{p}_0})$ acts on $\mathcal{Z} \times \Lambda$ by
\begin{displaymath}
  (g,f)[\omega, (a,b)] = [\bar{g} \omega, (a + \ord_{\mathfrak{p}_0} \det g,
    b + \ord_{\mathfrak{p}_0} \det g + 2 \ord_{\mathfrak{p}_0} f], \quad
  [\omega, (a,b)] \in \mathcal{Z} \times \Lambda.  
\end{displaymath} 
The morphism (\ref{DIPfeil1e}) induces the diagonal
$\mathbb{Z} \rightarrow \Lambda \subset \mathbb{Z}^2$.

We fix an element
$m \times g \in \mathbb{Z} \times (\check{D} \otimes_F \mathbb{A}^{\mathfrak{p}_0}_{F,f})^{\times}$.
The image of $\mathcal{Z} \times m \times g$ in the left hand side of
(\ref{DIPfeil1e}) is of the form 
\begin{displaymath}
\bar{\Gamma}_g \backslash\mathcal{Z}, 
\end{displaymath}
where $\bar{\Gamma}_g$ is the image of the group
$\{d \in \check{D} \cap g\mathbf{K}^{\mathfrak{p}_0}g^{-1} \; | \; \ord_{\mathfrak{p}_0} \det d = 0 \}$
in $\PGL_2(F_{\mathfrak{p}_0})$.  
Here $\det d \in F^{\times}$ denotes the reduced norm of $d$.

Let $J$ be the projective group of inner automorphisms of the $F$-algebra
$\check{D}$ considered as an algebraic group over $F$.
We denote by $\bar{\mathbf{K}}^{\mathfrak{p}_0}$ the image of
$\mathbf{K}^{\mathfrak{p}_0}$ in $J(\mathbb{A}_{F,f}^{\mathfrak{p}_0})$. This is an open
and compact subgroup. To see this one notes that for each place $w$ of $F$
the map $(\check{D} \otimes_F F_w)^{\times} \rightarrow J(F_w)$ is open because
$\check{D}^{\times} \rightarrow J$ is a smooth morphism of algebraic varieties.  
Because $J$ is compact at each archimedian places of $F$ the subgroup
$J(F) \subset J(\mathbb{A}_{F,f})$ is discrete. It follows that
$\Gamma'_g := J(F) \cap \bar{g} \bar{\mathbf{K}}^{\mathfrak{p}_0} \bar{g}^{-1} \subset J(F_{\mathfrak{p}_0})$
is a discrete subgroup. If $\mathbf{K}^p$ is sufficiently small, $\Gamma'_g$ acts  
without fixed points on the Bruhat-Tits building of $\PGL_2(F_{\mathfrak{p}_0})$.
Then each point of $\mathcal{Z}(\bar{\kappa}(\mathfrak{p}_0))$ has a Zariski 
neighbourhood $U$ such that $\gamma U \cap U = \emptyset$ for
$\gamma \in \Gamma'_g$ and $\gamma \neq 1$. This also holds for
$\bar{\Gamma}_g \subset \Gamma'_g$.
By our considerations, we can write the left hand side of (\ref{DIPfeil1e}) as 
\begin{displaymath}
\coprod_{i} \bar{\Gamma}_{g_i} \backslash \mathcal{Z} ,
\end{displaymath}
for a suitable choice of elements
$g_i \in (\check{D} \otimes_F \mathbb{A}^{\mathfrak{p}_0}_{F,f})^{\times}$. In the
same way we can write the right hand side of (\ref{DIPfeil1e}) as 
\begin{displaymath}
\coprod_{j} \bar{\Gamma}_{h_j} \backslash \mathcal{Z},  
  \end{displaymath}
with a suitable choice of elements 
$h_j \in \tilde{\mathbf{I}}(\mathbb{A}^{\mathfrak{p}_0}_{F,f})$.
To show the Proposition, it is sufficient to show that (\ref{DIPfeil1e}) 
induces an injection on the $\bar{\kappa}(\mathfrak{p}_0)$-valued points. 
Indeed, the restriction of (\ref{DIPfeil1e}) to
$\bar{\Gamma}_{g_i} \backslash \mathcal{Z}$ induces a morphism
\begin{equation}\label{Le66-7e} 
  \bar{\Gamma}_{g_i} \backslash \mathcal{Z} \rightarrow
  \bar{\Gamma}_{h_j} \backslash \mathcal{Z}    
  \end{equation}
for a suitable $j$, which is injective on 
$\bar{\kappa}(\mathfrak{p}_0)$-valued points.  Up to isomorphism we obtain the
same map if we replace $h_j$ by the image 
\begin{displaymath}
  g_i \times 1 \in (\check{D} \otimes_F \mathbb{A}^{\mathfrak{p}_0}_{F,f})^{\times}
  \times (K \otimes_F \mathbb{A}^{\mathfrak{p}_0}_{F,f})^{\times} 
  \end{displaymath}
in $\tilde{\mathbf{I}}(\mathbb{A}^{\mathfrak{p}_0}_{F,f})$. By the injectivity of
(\ref{Le66-7e}) and from the fact that  the actions of the $\Gamma$-groups
on both sides of (\ref{Le66-7e}) are fixed point free on the
sets of $\bar{\kappa}(\mathfrak{p}_0)$-valued points, one obtains that the groups
$\bar{\Gamma}_{g_i}$ and $\bar{\Gamma}_{h_j}$ coincide
on both sides of (\ref{Le66-7e}) and that this morphism is an isomorphism. 
From this the Proposition easily follows. 

It remains to prove the injectivity. We consider two elements
$[\omega_i, h_i] \in \mathcal{Z} \times \mathbb{Z} \times (\check{D} \otimes_F \mathbb{A}^{\mathfrak{p}_0}_{F,f})^{\times}$, 
$i = 1,2$, with $\omega_i \in \mathcal{Z}(\bar{\kappa}(\mathfrak{p}_0))$ and
$h_i \in \mathbb{Z} \times (\check{D} \otimes_F \mathbb{A}^{\mathfrak{p}_0}_{F,f})^{\times}$,  
which represent the same element on the right hand side of (\ref{DIPfeil1e}).
We will show that they also represent the same element on the left hand
side. 

By assumption there exists $g^{\bullet} \in \tilde{\mathbf{I}}(F)$ and
$k^{\bullet} \in \mathbf{K}^{\bullet, \mathfrak{p}_0}_C$ such that
\begin{displaymath}
  [\omega_1, h_1] = g^{\bullet} [\omega_2, h_2] k^{\bullet}. 
\end{displaymath}
We write
$g^{\bullet} = g \lambda \in \check{D}^{\times} \times^{F^{\times}} K^{\times}$ with
$g \in \check{D}^{\times}$ and $\lambda \in K^{\times}$. We set
$C_K^{\mathfrak{p}_0} = (\prod_{j=1}^s C_{K,\mathfrak{p}_j}) C_{K}^{p}$. By
(\ref{Le66-2e}) and (\ref{Le66-3e}) we can write
$k^{\bullet} = k c$ with $k \in \mathbf{K}^{\mathfrak{p}_0}$ and
$c \in C_K^{\mathfrak{p}_0}$. Replacing $[\omega_2, h_2]$ by $g[\omega_2, h_2]k$
we may assume that
\begin{equation}\label{Le66-8e} 
  [\omega_1, h_1] = \lambda [\omega_2, h_2] c. 
\end{equation}
This implies $\omega_1 = \omega_2$ and
\begin{equation}\label{Le66-4e}
h_2^{-1}h_1 = \lambda c.
  \end{equation}
This equation takes place in
$\Lambda \times \tilde{\mathbf{I}}(\mathbb{A}^{\mathfrak{p}_0}_{F,f})$.
The $\Lambda$-part of (\ref{Le66-4e}) is equivalent with
\begin{equation}\label{Le66-5e} 
\ord_{\mathfrak{q}_0} \lambda = \ord_{\bar{\mathfrak{q}}_0} \lambda.  
  \end{equation}
Next we consider the $\tilde{\mathbf{I}}(\mathbb{A}^{\mathfrak{p}_0}_{F,f})$-part
of (\ref{Le66-4e}). We obtain
\begin{equation}\label{Le66-6e}
  \lambda c = h_2^{-1} h_1 \in
  (\check{D} \otimes_F \mathbb{A}^{\mathfrak{p}_0}_{F,f})^{\times} \cap
  (K \otimes_{F} \mathbb{A}^{\mathfrak{p}_0}_{F,f})^{\times} =
  (\mathbb{A}^{\mathfrak{p}_0}_{F,f})^{\times}. 
\end{equation}
We consider the torus $S = K^{\times}/F^{\times}$ over $F$. This torus is
compact at all infinite places of $F$. Therefore
$S(F) \subset S(\mathbb{A}_{F,f})$ is discrete and the group of units of $S$
is finite. The equation (\ref{Le66-6e}) tells us that $\lambda$ is a unit
in $S(F_w)$ for all finite places $w \neq \mathfrak{p}_0$ of $F$, because
$\lambda$ is in the image $\bar{C}_K^{\mathfrak{p}_0}$
of the compact open subgroup $C_K^{\mathfrak{p}_0}$ 
by the morphism
$(K \otimes_{F} \mathbb{A}^{\mathfrak{p}_0}_{F,f})^{\times} \rightarrow S(\mathbb{A}^{\mathfrak{p}_0}_{F,f})$.
On the other the equation (\ref{Le66-5e}) tells us that $\lambda$ is
a unit in $S(F_{\mathfrak{p}_0})$ because by this equation there exists an
element $\alpha \in F_{\mathfrak{p}_0}^{\times}$ such that $\alpha \lambda$
is a unit in $K_{\mathfrak{p}_0}^{\times}$. Therefore the image of $\lambda$ in
$S(F)$ is a unit. If we choose $C_K^{p}$ sufficiently small,
the Theorem of Chevalley implies that the image is $1$. We conclude that
$\lambda \in F^{\times}$. Going back to (\ref{Le66-6e}) we obtain that
\begin{displaymath}
  c \in C_K^{\mathfrak{p}_0} \cap (\mathbb{A}^{\mathfrak{p}_0}_{F,f})^{\times} \subset
  \mathbf{K}^{\mathfrak{p}_0}. 
\end{displaymath}
This shows that the right hand side of (\ref{Le66-8e}) represents the same 
element on the left hand side of (\ref{DIPfeil1e})  as $[\omega_2,h_2]$. 
\end{proof}
We can now prove our main result, the Cherednik uniformization of quaternionic Shimura curves.
\begin{theorem}\label{4epeg2t} 
  Let $\mathbf{K} \subset D^{\times}(\mathbb{A}_f)$ be of the form $\mathbf{K}=\mathbf{K}_p\mathbf{K}^p$, where $\mathbf{K}_p$ is chosen as in
  (\ref{uniform27e}). Let $\widetilde{\Sh}_{\mathbf{K}}(H, h_{D})$ be the model
  over $\Spec O_{E_{\nu}}$ of the Shimura curve associated to $D$, cf. Definition \ref{uniform3d}. Then there
  is a uniformization isomorphism of formal schemes
  \begin{equation}\label{unimorph7e}
  \Theta\colon  \check{D}^{\times} \backslash ((\hat{\Omega}_{F_{\mathfrak{p}_0}}^2 
    \times_{\Spf O_{F_{\mathfrak{p}_0}},\varphi_0} \Spf O_{\breve{E}_\nu})\times
    D^{\times}(\mathbb{A}_f)/\mathbf{K}) 
\overset{\sim}{\longrightarrow}
    \widetilde{\Sh}_{\mathbf{K}}(H, h_{D})^\wedge_{\, / \Spf O_{\breve{E}_{\nu}}}
  \end{equation}
  For varying $\mathbf{K}$ this uniformization isomorphism is compatible
  with the action of the Hecke operators in $D^{\times}(\mathbb{A}_f)$ on both sides.

  Let $\Pi \in D_{\mathfrak{p}_0}$ be a prime element in this division algebra
  over $F_{\mathfrak{p}_0}$. We denote also by $\Pi$ the image by the canonical
  embedding $D_{\mathfrak{p}_0} \subset (D \otimes \mathbb{A}_f)^{\times}$.
  Let $\tau \in \Gal(\breve{E}_{\nu}/ E_{\nu})$ be the Frobenius automorphism and 
  $\tau_c = \Spf\tau^{-1}\colon \Spf O_{\breve{E}_{\nu}}\rightarrow\Spf O_{\breve{E}_{\nu}}$. 
  The natural Weil descent datum with respect to  
  $O_{\breve{E}_\nu}/O_{E_{\nu}}$ on the right hand side of (\ref{unimorph7e})
  induces on the
  left hand side the Weil descent datum given by the following commutative diagram
  \begin{displaymath}
\xymatrix{
    \check{D}^{\times} \backslash ((\hat{\Omega}_{F_{\mathfrak{p}_0}}^2 
    \times_{\Spf O_{F_{\mathfrak{p}_0}},\varphi_0} \Spf O_{\breve{E}_\nu})\times
    D^{\times}(\mathbb{A}_f)/\mathbf{K})
  \ar[d]_{ \id \times \mid_{\Pi^{-1}} \times \tau_c} \ar[r] & 
 \widetilde{\Sh}_{\mathbf{K}}(H, h_{D})^\wedge_{\, / \Spf O_{\breve{E}_{\nu}}}
   \ar[d]^{\id \times \tau_c}\\
   \check{D}^{\times} \backslash ((\hat{\Omega}_{F_{\mathfrak{p}_0}}^2 
    \times_{\Spf O_{F_{\mathfrak{p}_0}},\varphi_0} \Spf O_{\breve{E}_\nu})\times
    D^{\times}(\mathbb{A}_f)/\mathbf{K})
  \ar[r] & 
 \widetilde{\Sh}_{\mathbf{K}}(H, h_{D})^\wedge_{\, / \Spf O_{\breve{E}_{\nu}}}
     }
    \end{displaymath}
\end{theorem} 
The following Corollary provides us with an intrinsic characterization of
the integral model $\widetilde{\Sh}_{\mathbf{K}}(H, h_{D})$  of
${\Sh}_{\mathbf{K}}(H, h_{D})$. 
\begin{corollary}\label{remstab}
  If $\mathbf{K}^p$ is sufficiently small, the integral model
  $\widetilde{\Sh}_{\mathbf{K}}(H, h_{D})$ is a stable relative curve over
  $\Spec O_{E_\nu}$, in the sense of \cite{DM}. In addition, it has semi-stable
  reduction, i.e.,  it is regular and the special fiber is a reduced divisor
  with normal crossings.
\end{corollary}
\begin{proof} By the Theorem the formal scheme
  $\widetilde{\Sh}_{\mathbf{K}}(H, h_{D})_{/ \Spf O_{\breve{E}_{\nu}}}$ a union of
  connected components which are isomorphism to
  $\bar{\Gamma}\backslash (\hat{\Omega}_{F_{\mathfrak{p}_0}}^2 \times_{\Spf O_{F_{\mathfrak{p}_0}},\varphi_0} \Spf O_{\breve{E}_\nu})$
  where $\bar{\Gamma} \subset \PGL_2(F_{\mathfrak{p}_0})$ is a disrete subroup.
  It is known \cite{Mum} that $\mathbf{K}^p$ can be chosen such that
  $\bar{\Gamma}$ acts without fixed point and such that 
  \begin{equation}\label{remstab1e} 
    \hat{\Omega}_{F_{\mathfrak{p}_0}}^2
    \times_{\Spf O_{F_{\mathfrak{p}_0}},\varphi_0} \Spf O_{\breve{E}_\nu} \rightarrow
    \bar{\Gamma}\backslash (\hat{\Omega}_{F_{\mathfrak{p}_0}}^2
    \times_{\Spf O_{F_{\mathfrak{p}_0}},\varphi_0} \Spf O_{\breve{E}_\nu})
    \end{equation} 
  is a local isomorphisms for the Zariski topology. We denote by
  $\Omega_{\bar{\kappa}_\nu}$ the special fibre over $\Spf O_{\breve{E}_{\nu}}$ of
  the left hand side.
  All components of $\bar{\Gamma} \backslash \Omega_{\bar{\kappa}_\nu}$ are
  rational curves. We show that each of these components $\bar{C}$ is
  met by other components in at least $3$ different points. This proves
  that the right hand side of (\ref{remstab1e}) is a stable curve.

  $\bar{C}$ is the image of a component $C \subset \Omega_{\bar{\kappa}_\nu}$
  Let $E_1, \ldots, E_t$ all different components of $\Omega_{\bar{\kappa}_\nu}$
  which meet $C$ properly. Each $E_i$ meets $C$ in a single point $z_i$.
  The points $z_1, \ldots, z_t$ are all different. We know that $t \geq 3$.
  We denote by $\bar{E}_1, \ldots, \bar{E}_t$ the images in
  $\bar{\Gamma} \backslash \Omega_{\bar{\kappa}_\nu}$. We note the that
  $\bar{E}_i \neq \bar{C}$ for $i=1, \ldots t$. Indeed, let $\bar{z}_i$ be
  the image of $z_i$ in $\bar{\Gamma} \backslash \Omega_{\bar{\kappa}_\nu}$. The
  inequality follows because a neighbourhood of $z_i$ is isomorphically
  mapped to a neighbourhood of $\bar{z}_i$.
  We will show that the points $\bar{z}_1, \ldots, \bar{z}_t$ are all different.
  If not we find an element $\gamma \in \bar{\Gamma}$ such that for example
  $\gamma z_1 = z_2$. Then all three components $C$, $\gamma C$, $\gamma E_1$
  contain the point $z_2$. Therefore two of these components must be equal.
  By what we said above $C = \gamma C$ follows. But this implies that $\gamma$
  has a fixpoint on $C$ which is excluded by assumption. We conclude that
  $\gamma z_1 = z_2$ is impossible. This proves that the points
  $\bar{z}_1, \ldots, \bar{z}_t$ are different. 
\end{proof}  
  
\begin{proof}[Proof of Theorem \ref{4epeg2t}]
  We consider only open compact subgroups
  $\mathbf{K} \subset H(\mathbb{A}_f) = D^{\times}(\mathbb{A}_f)$ of
  the type as in the statement of the theorem. For the proof it will suffice to consider those
  $\mathbf{K}$ where $\mathbf{K}^p$ is small enough. We choose a chain of
  open compact subgroups of this type 
  \begin{equation}
    \mathbf{K}_1 \supset \mathbf{K}_2 \supset \ldots \supset \mathbf{K}_t
    \supset \ldots ,
  \end{equation}
  which is cofinal to all subgroups of this type.

  We consider open and compact subgroups
  $\mathbf{K}^{\bullet} \subset G^{\bullet}(\mathbb{A}_f)$ as in Proposition
  \ref{BZ8p}  with the following properties:
  \begin{enumerate}
  \item[(a)] $\mathbf{K}^{\bullet} \cap H(\mathbb{A}_f) = \mathbf{K}_t$ for some
    $t \in \mathbb{N}$.
  \item[(b)] The groups $\mathbf{K}^{\bullet}_{\mathfrak{p}_i}$ and
    $\mathbf{K}_{t, \mathfrak{p}_i}$ are related as in (\ref{uniform27e}) for $i=0,\ldots,s$.
  \item[(c)] The natural morphism
    \begin{displaymath}
  \Sh_{\mathbf{K}_t}(H, h_{D}) \times_{\Spec E(H, h_D)} \Spec E
  \rightarrow \Sh_{\mathbf{K}^{\bullet}}(G^{\bullet}, h^{\bullet}_{D})  
    \end{displaymath}
    is an open and closed immersion. Here $\mathbf{K}^{\bullet}$ is chosen such that $(a)$ is satisfied. 
  \end{enumerate}
  We find a chain of open and compact subgroups $\mathbf{K}^{\bullet}$ with the
  properties $(abc)$
  \begin{equation}
    \mathbf{K}^{\bullet}_1 \supset \mathbf{K}^{\bullet}_2 \supset \ldots \supset
    \mathbf{K}^{\bullet}_s \supset \ldots ,
  \end{equation}
  which has the following properties. For each $\mathbf{K}_t$ there is
  a group $\mathbf{K}^{\bullet}_{s}$ which satifies $(abc)$ with respect to
  $\mathbf{K}_t$. Moreover, for an arbitrary $\mathbf{K}^{\bullet}$ satisfying
  $(abc)$, there is a group $\mathbf{K}^{\bullet}_s$ such that
  $\mathbf{K}^{\bullet}_s \subset \mathbf{K}^{\bullet}$ and such that
  $\mathbf{K}^{\bullet}_s \cap H(\mathbb{A}_f) = \mathbf{K}_{t'}$ for some
  $t' > t$. We set
  \begin{equation}\label{uniform28e} 
    \Sh^{pro}(H, h_{D})_{\breve{E}_{\nu}} =\varprojlim\nolimits_{\mathbf{K}_t}
       \Sh_{\mathbf{K}_t}(H, h_{D})_{\breve{E}_{\nu}}. 
    \end{equation} 
  We remark that the connected components of
  $\Sh_{\mathbf{K}_t}(H, h_{D})_{\breve{E}_{\nu}}$ are geometrically connected. This
  follows from \cite[(2.7.1) and (3.9.1)]{D-TS}  because
  $\mathbf{K}_{\mathfrak{p}_0} \in D^\times_{\mathfrak{p}_0}$ is maximal.
  We choose a connected component $Z$ of the left hand side of
  (\ref{uniform28e}). This induces a connected component 
  $Z_{\mathbf{K}_t}$ of $\Sh_{\mathbf{K}_t}(H, h_{D})_{\breve{E}_{\nu}}$ for each
  $t$. The closure $\tilde{Z}_{\mathbf{K}_t}$ of
  $Z_{\mathbf{K}_t}$ in $\widetilde{\Sh}_{\mathbf{K}_t}(H, h_{D})$
  is a connected component there. Since the last schemes are proper over
  $\Spec O_{\breve{E}_{\nu}}$, the natural restriction morphisms
  $\tilde{Z}_{\mathbf{K}_{t+1}} \rightarrow \tilde{Z}_{\mathbf{K}_t}$ are surjective.
  We choose points
  $z_{\mathbf{K}_t} \in \tilde{Z}_{\mathbf{K}_t}(\bar{\kappa}_{E_{\nu}})$ such that
  $z_{\mathbf{K}_{t+1}}$ is mapped to $z_{\mathbf{K}_t}$ for all $t$. Let $\mathbf{K}^{\bullet}_{s}$
  be a subgroup such that
  $\mathbf{K}^{\bullet}_{s}$ induces $  \mathbf{K}_{t}$ as in $(abc)$. Then by the open
  and closed immersion of Definition \ref{uniform3d}, $\tilde{Z}_{\mathbf{K}_t}$
  is also a connected component of
  $\widetilde{\Sh}_{\mathbf{K}^{\bullet}_s}(G^{\bullet}, h^{\bullet}_{D})$.
    We consider $z_{\mathbf{K}_t}$ as a point of
  $\tilde{\mathcal{A}}^{\bullet t}_{\mathbf{K}_s^{\bullet}}(\bar{\kappa}_{E_{\nu}})$. We
  denote this point by $z^{\bullet}_{\mathbf{K}^{\bullet}_s}$. It is represented by the
  isomorphism class of a tuple
  \begin{equation}
    z^{\bullet}_{\mathbf{K}^{\bullet}_s} =
    (A(\mathbf{K}^{\bullet}_{s}),\iota(\mathbf{K}^{\bullet}_{s}),
    \bar{\lambda}(\mathbf{K}^{\bullet}_{s}),
    \bar{\eta}^p(\mathbf{K}^{\bullet}_{s}), ( \bar{\eta}_{\mathfrak{q}_j}(\mathbf{K}^{\bullet}_{s}))_j). 
  \end{equation}
  We note that no datum $(\xi_{\mathfrak{p}_i})_i$ appears because of our choice
  (\ref{uniform27e}). By construction $z^{\bullet}_{\mathbf{K}^{\bullet}_{s+1}}$ is
  mapped to $z^{\bullet}_{\mathbf{K}^{\bullet}_s}$ for all $s$. The triples
  \begin{displaymath}
(A(\mathbf{K}^{\bullet}_{s}),\iota(\mathbf{K}^{\bullet}_{s}),
    \bar{\lambda}(\mathbf{K}^{\bullet}_{s})) 
  \end{displaymath}
  are all isomorphic. Therefore we may choose them independent of $s$. The
  classes $\bar{\eta}_{\mathfrak{q}_j}(\mathbf{K}^{\bullet}_{{s+1}})$ and
  $\bar{\eta}^p(\mathbf{K}^{\bullet}_{{s+1}})$ generate classes modulo
  $\mathbf{K}^{\bullet}_{s}$. We denote these classes by
  $\bar{\eta}_{\mathfrak{q}_j}(\mathbf{K}^{\bullet}_{{s+1}})_{\mid s}$, resp.
  $\bar{\eta}^p(\mathbf{K}^{\bullet}_{{s+1}})_{\mid s}$. Since
  $z^{\bullet}_{\mathbf{K}^{\bullet}_{s+1}}$ is mapped to $z^{\bullet}_{\mathbf{K}^{\bullet}_s}$,
  we obtain an isomorphism of tuples
  \begin{displaymath}
    (A,\iota,\bar{\lambda}, \bar{\eta}^p(\mathbf{K}^{\bullet}_{{s+1}})_{\mid s}, 
    (\bar{\eta}_{\mathfrak{q}_j}(\mathbf{K}^{\bullet}_{{s+1}})_{\mid s})_j
    ) \overset{\sim}{\rightarrow}
    (A,\iota,\bar{\lambda},     \bar{\eta}^p(\mathbf{K}^{\bullet}_{s}), (\bar{\eta}_{\mathfrak{q}_i}(\mathbf{K}^{\bullet}_{s})
)_j). 
    \end{displaymath}
  By this isomorphism the data
  $\bar{\eta}^p(\mathbf{K}^{\bullet}_{{s+1}})$ and  $\bar{\eta}_{\mathfrak{q}_j}(\mathbf{K}^{\bullet}_{{s+1}})$ on the left hand side induce on the
  right hand side data
  $\bar{\eta}^p(\mathbf{K}^{\bullet}_{{s+1}})'$ and  $\bar{\eta}_{\mathfrak{q}_j}(\mathbf{K}^{\bullet}_{{s+1}})'$ such that
  \begin{displaymath}
    \bar{\eta}^p(\mathbf{K}^{\bullet}_{{s+1}})'_{\mid s} =
    \bar{\eta}^p(\mathbf{K}^{\bullet}_{s}),\quad \bar{\eta}_{\mathfrak{q}_j}(\mathbf{K}^{\bullet}_{{s+1}})'_{\mid s} =
    \bar{\eta}_{\mathfrak{q}_i}(\mathbf{K}^{\bullet}_{s}) .
    \end{displaymath}
 Therefore we may assume that 
   \begin{equation} 
       \bar{\eta}^p(\mathbf{K}^{\bullet}_{{s+1}})_{\mid s} =
    \bar{\eta}^p(\mathbf{K}^{\bullet}_{s}), \quad  \bar{\eta}_{\mathfrak{q}_j}(\mathbf{K}^{\bullet}_{{s+1}})_{\mid s} =
    \bar{\eta}_{\mathfrak{q}_j}(\mathbf{K}^{\bullet}_{s}). 
 \end{equation}
Now $\bar{\eta}^p(\mathbf{K}^{\bullet}_{s}) \subset \mathrm{Isom}_{B \otimes \mathbb{A}_f^p}(V \otimes \mathbb{A}^p_f, \mathrm{V}^p(A))$
is a compact subset. Therefore the intersection
$\cap_s \bar{\eta}^p(\mathbf{K}^{\bullet}_{s})$ is not empty. We choose an
element $\eta^p$ in this intersection. It generates the class
$\bar{\eta}^p(\mathbf{K}^{\bullet}_{s})$ for each $s$. Similiarly, we find
for each $j = 1, \ldots, s$ an  isomorphism
$\eta_{\mathfrak{q}_i}: \Lambda_{\mathfrak{q}_j} \isoarrow T_{\mathfrak{q}_i}$
which induces all classes $\bar{\eta}_{\mathfrak{q}_j}(\mathbf{K}^{\bullet}_{s})$.

The tuple 
\begin{equation}
(A, \iota, \bar{\lambda}, \eta^p, ( \eta_{\mathfrak{q}_j})_j) 
  \end{equation}
makes it possible to define the uniformization morphism of the theorem.
For this we consider the morphism (\ref{unimorph1e}) defined by 
substituting $(A, \iota, \bar{\lambda}, \eta^p, ( \eta_{\mathfrak{q}_j})_j) $ for the choice of \eqref{uniform2e}
used there. Let $\mathbb{X} = \prod_{i=0}^s \mathbb{X}_{\mathfrak{p}_i}$ be the
$p$-divisible group of $A$. In the Definition \ref{RZ4d} we take
$\mathbb{X}_0=\mathbb{X}_{\mathfrak{p}_0}$ as
the framing object. If we take for $\rho$ the identity, we obtain a point
of ${\rm RZ}_{\mathfrak{p}_0}(0,0) \subset {\rm RZ}_{\mathfrak{p}_0}$ and by
Lemma \ref{RZ7l} a point 
\begin{displaymath}
  \tilde{z} \in \hat{\Omega}^2_{F_{\mathfrak{p}_0}}(\bar{\kappa}_{E_{\nu}}) =
  \hat{\Omega}^2_{E_{\nu}}(\bar{\kappa}_{E_{\nu}}).  
  \end{displaymath}
From $\mathbb{X}_j=\mathbb{X}_{\mathfrak{p}_j}$, with the $O_{F_{\mathfrak{p}_i}}^{\times}$-homogeneous polarization
induced from $\bar{\lambda}$, the rigidification $\bar{\eta}_{\mathfrak{q}_j}$, 
and the datum $\rho = \id_{\mathbb{X}_i}$, we obtain a point
$\tilde{z}_j(\mathbf{K}_{s,\mathfrak{p}_j}^{\bullet})$ of
${\rm RZ}_{\mathfrak{p}_j, \mathbf{K}_{s,\mathfrak{p}_j}^{\bullet}}$. If we use for 
(\ref{RZ1e}) the isomorphism given by $\eta_{\mathfrak{q}_j}$, the point 
$\tilde{z}_j(\mathbf{K}_{s,\mathfrak{p}_j}^{\bullet})$ corresponds to
$1 \in G^{\bullet}_{\mathfrak{p}_j}/\mathbf{K}^{\bullet}_{s,\mathfrak{p}_j}$ under the
isomorphism of Proposition \ref{RZ6p}. By construction of the uniformization
morphism \eqref{unimorph6e}, the point
\begin{displaymath}
\tilde{z} \times 1 \in  (\hat{\Omega}^2_{E_{\nu}} \times
    G^{\bullet}(\mathbb{A}_f)/\mathbf{K}^{\bullet}_s)(\bar{\kappa}_{E_{\nu}}) 
  \end{displaymath}
is mapped to the point
\begin{displaymath}
  z^{\bullet}_{\mathbf{K}^{\bullet}_s} \in \tilde{Z}_{\mathbf{K}_t}(\bar{\kappa}_{E_{\nu}})
  \subset
  \widetilde{\Sh}_{\mathbf{K}^{\bullet}_s}(G^{\bullet}, h^{\bullet}_{D})(\bar{\kappa}_{E_{\nu}}). 
\end{displaymath}
This implies that
$(\hat{\Omega}^2_{E_{\nu}} \times_{\Spf O_{E_{\nu}}} \Spf O_{\breve{E}_{\nu}})\times 1$ 
is mapped by (\ref{unimorph6e})  to the formal completion of the connected component
$\tilde{Z}_{\mathbf{K}_t}$ of
$\widetilde{\Sh}_{\mathbf{K}^{\bullet}_s}(G^{\bullet}, h^{\bullet}_{D})_{O_{\breve{E}_{\nu}}}$. Now
we restrict (\ref{unimorph6e}) to
\begin{equation}\label{uniform29e}
\check{D}^{\times} \backslash ((\hat{\Omega}_{F_{\mathfrak{p}_0}}^2 
\times_{\Spf O_{F_{\mathfrak{p}_0}},\varphi_0} \Spf O_{\breve{E}_\nu})\times
D^{\times}(\mathbb{A}_f)/\mathbf{K}_t) \rightarrow
\widetilde{\Sh}_{\mathbf{K}^{\bullet}_s}(G^{\bullet}, h^{\bullet}_{D})^\wedge_{/ \Spf O_{\breve{E}_{\nu}}},  
  \end{equation}
cf. Lemma \ref{uniform4l}. The image of the connected component
$(\hat{\Omega}_{F_{\mathfrak{p}_0}}^2 \times_{\Spf O_{F_{\mathfrak{p}_0}},\varphi_0} \Spf O_{\breve{E}_\nu}) \times 1$
is mapped to a connected component of the open and closed formal subscheme 
$\widetilde{\Sh}_{\mathbf{K}}(H, h_{D})^\wedge_{/ \Spf O_{\breve{E}_{\nu}}}$.
But since the Hecke operators $D^{\times}(\mathbb{A}_f)$ act transitively on
the connected components of the last formal scheme and the morphism
(\ref{unimorph6e}) is compatible with Hecke operators, we conclude that 
(\ref{uniform29e}) is a surjective map onto
$\widetilde{\Sh}_{\mathbf{K}}(H, h_{D})^\wedge_{/ \Spf O_{\breve{E}_{\nu}}}$. Since by
Theorem \ref{4epeg1t} and Lemma \ref{uniform4l} the morphism is an open
and closed immersion we conclude that  (\ref{unimorph7e}) is an isomorphism for
$\mathbf{K} = \mathbf{K}_t$. 

Now the tuple $(A, \iota, \bar{\lambda}, \eta^p, ( \eta_{\mathfrak{q}_j})_j) $
defines the uniformization morphism for an arbitrary $\mathbf{K}$. By choosing
$\mathbf{K}_t \subset \mathbf{K}$, we see that (\ref{unimorph7e}) is surjective
and therefore an isomorphism by  Lemma \ref{uniform4l}. 
The compatibility with the Weil descent data is a consequence of
Theorem \ref{4epeg1t}. This completes the proof.
\end{proof}

\section{Conventions about Galois descent}\label{s:desc}

Let $L/E$ be a Galois extension (possibly infinite) with Galois group
$G = \Gal(L/E)$. For $\sigma \in G$ we set 
\begin{equation}\label{tau_c1e}
\sigma_c = \Spec \sigma^{-1}: \Spec L \rightarrow \Spec L.  
\end{equation}
If $\tau \in G$ we find $(\sigma \circ \tau)_{c} = \sigma_{c} \circ \tau_{c}$.  
Let $\pi: X \rightarrow \Spec L$ be a scheme over $L$. We recall that a
descent datum on $X$ relative to $L/E$ is a collection of morphisms
$\varphi_{\sigma}: X \rightarrow X$ for $\sigma \in G$, making the
following diagram  commutative
\begin{equation*}
   \xymatrix{
     X \ar[r]^{\varphi_{\sigma}} \ar[d]^{\pi} & X \ar[d]_{\pi}\\
     \Spec L \ar[r]^{\sigma_c} & \Spec L , 
   }
\end{equation*}
such that  $\varphi_{\sigma} \circ \varphi_{\tau} = \varphi_{\sigma \tau}$ for all
$\sigma, \tau \in G$. In other words, a descent datum is a left action of $G$
on $X$ by semi-linear automorphisms.
A descent datum $(X, \varphi_{\sigma})$ defines a left action of $G$ on
$X(L) = \Hom_{\Spec L}(\Spec L, X)$, 
\begin{displaymath}
  \begin{array}{ccc}
    G \times X(L) & \longrightarrow & X(L).\\
    (\sigma , \alpha) & \longmapsto & \varphi_{\sigma} \circ \alpha \circ
    \Spec \sigma\\
    \end{array}
  \end{displaymath}
We denote the right hand side by $\sigma \times_{\varphi} \alpha$. This is indeed
a point of $X(L)$:
\begin{displaymath}
\pi \circ \varphi_{\sigma} \circ \alpha \circ \Spec \sigma = \sigma_c \circ \pi 
\circ \alpha \circ \Spec \sigma = \sigma_c \circ \Spec \sigma = \id_{\Spec L}. 
\end{displaymath}

Let $u: G \rightarrow \Aut_{L}((X, \varphi))$ be an action of $G$ on this
descent datum. This means that for each $\sigma \in G$ an $L$-morphism
$u_{\sigma}: X \rightarrow X$ is given such that for each $\sigma, \tau \in G$
\begin{displaymath}
  u_{\sigma} \circ u_{\tau} = u_{\sigma \tau}, \quad u_{\sigma} \circ \varphi_{\tau} =
  \varphi_{\tau} \circ u_{\sigma}.
\end{displaymath}
Then $\psi_{\sigma} := u_{\sigma} \circ \varphi_{\sigma}$ is another descent datum
on $X$. It defines another action $\sigma \times_{\psi} \alpha$ of $G$ on $X(L)$.
From the definition we obtain
\begin{equation}\label{GalAb1e}
\sigma \times_{\psi} \alpha = u_{\sigma} \circ (\sigma \times_{\phi} \alpha). 
\end{equation}

If $X_0$ is a scheme over $\Spec E$, there is the canonical
descent datum on $X = X_0 \times_{\Spec E} \Spec L$, 
\begin{displaymath}
  \kappa_{\sigma} = \id_{X_0} \times \sigma_c: X_0 \times_{\Spec E} \Spec L
  \rightarrow X_0 \times_{\Spec E} \Spec L.  
\end{displaymath}
The action of $G$ induced by $\kappa_{\sigma}$ on $X(L)$ coincides with the
action on $X_0(L) = \Hom_{\Spec E}(\Spec L, X_0)$ via $L$, taking into account
the identification $X(L) = X_0(L)$. We denote this action by
$\sigma \alpha_0 = \sigma \times_{\kappa} \alpha_0$,
($\sigma \in G$, $\alpha_0 \in X_0(L)$). 

A homomorphism $a: G \rightarrow \Aut_{E} X_0$ defines an action on the
canonical descent datum $u: G \rightarrow \Aut_{L}((X, \kappa))$ via 
$u_{\sigma} = a_{\sigma} \times \id_{\Spec L}$. We obtain the new descent datum
\begin{displaymath}
  \psi_{\sigma} = u_{\sigma} \circ \kappa_{\sigma} = a_{\sigma} \times \sigma_c:
  X_0 \times_{\Spec E} \Spec L \rightarrow X_0 \times_{\Spec E} \Spec L.  
  \end{displaymath}
The action of $G$ on $X(L) = X_0(L)$ defined by this descent datum is 
\begin{displaymath}
  \sigma \times_{\psi} \alpha_0 = a_{\sigma} \circ \sigma \alpha_0, \quad
  \alpha_0 \in X_0(L). 
  \end{displaymath}

Assume that $X_0 = \coprod_{X_0(E)} \Spec E$ is the constant scheme. Then the
action of $G$ on $X_0(L)$ is trivial. Let $a: G \rightarrow \Aut X_0(E)$ be
an action. It induces an action on $X_0$ which we denote by the same letter
$a$. It defines on $X = X_0 \times_{\Spec E} \Spec L$ the descent datum
$a_{\sigma} \times \sigma_c$. The action on $X(L) = X_0(L)$ induced by this
descent datum is the operation on $X_0(L)$ by $a$ acting on $X_0$. (Note that
$X_0(L) = X_0(E))$.

A descent datum $(X, \varphi_{\sigma})$ is effective if there is scheme
$X_0$ over $E$ such that there is an isomorphism
$X_0 \times_{\Spec E} \Spec L \rightarrow X$ 
which respects the descent data $\kappa_{\sigma}$ resp. $\varphi_{\sigma}$. 
Assume that $L/K$ is a finite field extension and that $X$ is quasiprojective
over $\Spec L$. Then any decent datum
$\varphi_{\sigma}$ is effective (cf. SGA1 Exp. VIII).

\section{Conventions about Shimura varieties}\label{s:shimvar} 

Let $(G,h)$ a Shimura datum. We denote by $X$ the $G(\mathbb{R})$-conjugacy
class of $h$. We consider the operation of $G(\mathbb{R})$ on $X$ from
the left
\begin{displaymath} 
G(\mathbb{R}) \times X \rightarrow X, \quad (g,x) \mapsto gxg^{-1}.  
\end{displaymath}
Let $\mathbf{K}_{\infty} \subset G(\mathbb{R})$ be the stabilizer of $h$.
Then we have a $G(\mathbb{R})$-equivariant map
\begin{displaymath}
  G(\mathbb{R})/\mathbf{K}_{\infty} \overset{\sim}{\longrightarrow} X, \quad
  g \mapsto ghg^{-1}. 
  \end{displaymath}
Let $\mathbf{K} \subset G(\mathbb{A}_f)$ be an open and compact subgroup.
Then we define the complex Shimura variety
\begin{equation}
  \begin{array}{ll} 
  \mathrm{Sh}_{\mathbf{K}}(G,h)_{\mathbb{C}} & = G(\mathbb{Q}) \backslash 
  (X \times (G(\mathbb{A}_f)/\mathbf{K}))\\[2mm] & = G(\mathbb{Q}) \backslash 
  ((G(\mathbb{R})/\mathbf{K}_{\infty}) \times (G(\mathbb{A}_f)/\mathbf{K})).\\ 
    \end{array}
\end{equation}
The group $G(\mathbb{Q})$ acts from the left via the homomorphisms
$G(\mathbb{Q}) \rightarrow G(\mathbb{R})$ and
$G(\mathbb{Q}) \rightarrow G(\mathbb{A}_f)$. 

The group $G(\mathbb{A}_f)$ acts from the right on the tower
$\{ \mathrm{Sh}_{\mathbf{K}} \}_{\mathbf{K}}$ for varying $\mathbf{K}$. For
$a \in G(\mathbb{A}_f)$ this action is given by 
\begin{equation}
  \begin{array}{lccc} 
    \mid_{a}: & \mathrm{Sh}_{\mathbf{K}} & \rightarrow &
    \mathrm{Sh}_{a^{-1} \mathbf{K} a}, \\
    & (x, u) & \mapsto & (x, ua)\\ 
    \end{array} 
  \end{equation} 
where $x \in X$ and $u \in G(\mathbb{A}_f)$. We call this a Hecke operator.
  
\begin{remark}
In \cite{D-TS},  the action of $G(\mathbb{R})$ from the right on $X$ is considered,
\begin{displaymath}
h \times g := g^{-1} h g, \quad h \in X, \; g \in G(\mathbb{R}).  
  \end{displaymath}
The complex Shimura variety is defined as
\begin{displaymath}
  \begin{array}{ll}
    \Sh^{D}_{\mathbf{K}}(G, h)_{\mathbb{C}} & =
    (X\times(\mathbf{K}\backslash G(\mathbb{A}_f))/G(\mathbb{Q})\\ 
  & = ((\mathbf{K}_{\infty} \backslash G(\mathbb{R})) \times
  (\mathbf{K}\backslash G(\mathbb{A}_f))/G(\mathbb{Q})
    \end{array}
  \end{displaymath}
Let $X_{-}$ be the conjugacy class of $h^{-1}$. Then there is a  natural
isomorphism
\begin{equation}\label{Shh1e}
  \Sh_{\mathbf{K}}(G, h)_{\mathbb{C}} \overset{\sim}{\longrightarrow}
  \Sh_{\mathbf{K}}^{D}(G, h^{-1})_{\mathbb{C}} ,
  \end{equation}
 given by 
\begin{displaymath}
  \begin{array}{ccc} 
    G(\mathbb{Q}) \backslash (X \times (G(\mathbb{A}_f)/\mathbf{K})) &
    \overset{\sim}{\longrightarrow} & 
    (X \times(\mathbf{K}\backslash G(\mathbb{A}_f))/G(\mathbb{Q})\\
    (x, g) & \longmapsto & (x^{-1}, g^{-1})  
    \end{array}
\end{displaymath}
\end{remark}
Let $H$ be a torus over $\mathbb{Q}$ and let
\begin{displaymath}
h: \mathbb{S} \rightarrow H_{\mathbb{R}} 
\end{displaymath}
be a morphism of algebraic groups over $\mathbb{R}$. It induces a morphism
of algebraic groups over $\mathbb{C}$
\begin{displaymath}
\mu: \mathbb{G}_{m,\mathbb{C}} \rightarrow H_{\mathbb{C}}. 
  \end{displaymath}
The field of definition $E$ of $\mu$ is the reflex field of $(H, h)$. We consider 
the composite 
\begin{displaymath}
  \mathfrak{r}: \mathrm{Res}_{E/\mathbb{Q}} \mathbb{G}_{m,E} 
  \overset{\mathrm{Res}\, \mu}{\longrightarrow} \mathrm{Res}_{E/\mathbb{Q}} H_E 
  \overset{\Nm_{E/\mathbb{Q}}}{\longrightarrow} H. 
\end{displaymath}
The homomorphism
\begin{equation}\label{Sh-rec3e}
  r(H,h) = \mathfrak{r}^{-1}: \mathrm{Res}_{E/\mathbb{Q}} \mathbb{G}_{m,E}
  \rightarrow H 
  \end{equation}
is called the reciprocity law of $(H,h)$, cf.  \cite[(3.9.1)]{D-TS}.
Let $\mathbf{K} \subset H(\mathbb{A}_f)$ be an open
compact subgroup. There is an open and compact subgroup
$C \subset (E \otimes \mathbb{A}_f)^{\times}$ such that
$r(H,h)(\mathbb{A}_f)(C) \subset \mathbf{K}$. Therefore $r(H,h)$
induces a map
\begin{equation}\label{Sh-rec1e}
  E^{\times} \backslash (E \otimes \mathbb{A}_f)^{\times}/C \rightarrow
  H(\mathbb{Q}) \backslash H(\mathbb{A}_f)/\mathbf{K}
  \end{equation}
By class field theory
\begin{displaymath}
  E^{\times} \backslash (E \otimes \mathbb{A}_f)^{\times}/C =
  E^{\times} \backslash (E \otimes \mathbb{A})^{\times}/C (E \otimes\mathbb{R})^{\times}
\end{displaymath}
corresponds to a finite abelian extension $L$ of $E$. We consider the
homomorphism
\begin{displaymath}
\Gal(E^{ab}/E) \rightarrow \Gal(L/E) =
  E^{\times} \backslash (E \otimes \mathbb{A}_f)^{\times}/C.
  \end{displaymath}
If we compose this with (\ref{Sh-rec1e}) we obtain the class field version of the reciprocity map,
\begin{equation}\label{Sh-rec2e}
  r^{\rm cft}(H,h): \Gal(\bar{E}/E) \rightarrow \Gal(E^{ab}/E) \rightarrow
  H(\mathbb{Q}) \backslash H(\mathbb{A}_f)/\mathbf{K}. 
\end{equation}
This Galois action on $H(\mathbb{Q}) \backslash H(\mathbb{A}_f)/\mathbf{K}$
defines a finite \'etale scheme over $E$ which we denote by
$\mathrm{Sh}_{\mathbf{K}}(H,h)$.
This is called the canonical model of $\mathrm{Sh}_{\mathbf{K}}(H,h)$ over $E$.
By definition 
\begin{displaymath}
  \mathrm{Sh}_{\mathbf{K}}(H,h)_{\mathbb{C}} =
  H(\mathbb{Q}) \backslash H(\mathbb{A}_f)/\mathbf{K}. 
\end{displaymath}
In other words we can say that the $E$-scheme $\mathrm{Sh}_{\mathbf{K}}(H,h)$
is obtained from the constant scheme
$H(\mathbb{Q}) \backslash H(\mathbb{A}_f)/\mathbf{K}$ over $E$ by the descent
datum
\begin{equation}\label{Sh-rec4e}
  \begin{array}{lr}
  r^{\rm cft}(H,h)(\sigma) \times \sigma_c: &
  (H(\mathbb{Q})\backslash
  H(\mathbb{A}_f)/\mathbf{K})\times_{\Spec E}\Spec\bar{E}
  \longrightarrow \qquad \quad \\[2mm]  
  & \qquad (H(\mathbb{Q})\backslash
  H(\mathbb{A}_f)/\mathbf{K})\times_{\Spec E}\Spec\bar{E}, 
    \end{array}
  \end{equation}
for $\sigma \in \Gal(\bar{E}/E)$ and $\sigma_c := \Spec \sigma^{-1}$. We can also
express the last statement by a commutative diagram. There is an isomorphism
of schemes over $\bar{E}$ 
\begin{equation}\label{r-descent1e}
  H(\mathbb{Q})\backslash H(\mathbb{A}_f)/\mathbf{K})\times_{\Spec E}\Spec\bar{E}
  \rightarrow
  \mathrm{Sh}_{\mathbf{K}}(H,h)\times_{\Spec E}\Spec\bar{E}
\end{equation}
such that for each 
$\sigma \in \Gal(\bar{E}/E)$  the following diagram is commutative,
\begin{displaymath}
\xymatrix{
  (H(\mathbb{Q})\backslash H(\mathbb{A}_f)/\mathbf{K})
  \times_{\Spec E}\Spec\bar{E} \ar[d]_{r^{\rm cft}(H,h)(\sigma) \times \sigma_c} \ar[r] 
  & \mathrm{Sh}_{\mathbf{K}}(H,h) \times_{\Spec E}\Spec\bar{E}
  \ar[d]^{id \times \sigma_c}\\
  (H(\mathbb{Q}) \backslash H(\mathbb{A}_f)/\mathbf{K})
  \times_{\Spec E}\Spec\bar{E} \ar[r] & 
  \mathrm{Sh}_{\mathbf{K}}(H,h) \times_{\Spec E}\Spec\bar{E} .
   }
  \end{displaymath}
 For varying $\mathbf{K}$ the isomorphism
is compatible with the action of the Hecke operators $H(\mathbb{A}_f)$.
In the considerations above we can replace the field of definition $E$ of
$\mu$ by any finite extension $E'$, $E \subset E' \subset \mathbb{C}$.
The definition (\ref{Sh-rec3e}) gives
\begin{displaymath}
  r_{E'}(H,h) = \mathfrak{r}^{-1}: \mathrm{Res}_{E'/\mathbb{Q}} \mathbb{G}_{m,E'}
  \rightarrow H.
\end{displaymath}
This reciprocity law gives the action of $\Gal(\bar{E}/E')$ on
$H(\mathbb{Q}) \backslash H(\mathbb{A}_f)/\mathbf{K}$ 
obtained by restriction from (\ref{Sh-rec2e}). 

\begin{remark}
The map
\begin{displaymath}
  \begin{array}{ccc} 
    H(\mathbb{Q}) \backslash H(\mathbb{A}_f)/\mathbf{K}  &
    \overset{\sim}{\longrightarrow}  & 
     \mathbf{K}\backslash H(\mathbb{A}_f)/H(\mathbb{Q})\\
    h & \longmapsto & h^{-1}   
    \end{array}
\end{displaymath}
is equivariant with respect to the action of $E^{\times}(\mathbb{A}_f)$ on the
left hand side by $r(H,h)$ and the action of $E^{\times}(\mathbb{A}_f)$ on the
right hand side by $r(H, h^{-1})$.
Therefore we obtain an isomorphism of canonical models over $E$, 
\begin{displaymath}
\Sh_{\mathbf{K}}(H, h) \overset{\sim}{\longrightarrow}
  \Sh_{\mathbf{K}}^{D}(H, h^{-1}). 
\end{displaymath}
\end{remark}
We go back to a general Shimura datum $(G,h)$. We denote by $E(G,h)$ the
Shimura field. We call a model $\Sh(G,h)$ over $E$ of $\Sh(G,h)_{\mathbb{C}}$
canonical if for each maximal torus $H \subset G$ and each
$h': \mathbb{S} \rightarrow H_{\mathbb{R}}$ which is conjugate to $h$ in
$G(\mathbb{R})$, the induced morphism
$\Sh(H,h')_{\mathbb{C}} \rightarrow \Sh(G,h)_{\mathbb{C}}$ is defined over the
compositum of $E(H,h')$ and $E(G,h)$.
With this definition (\ref{Shh1e}) induces an isomorphism of canonical models,
cf. \cite[3.13]{D-TS}.

\begin{equation}\label{Shh2e}
\Sh_{\mathbf{K}}(G, h) \overset{\sim}{\longrightarrow}
  \Sh_{\mathbf{K}}^{D}(G, h^{-1}). 
\end{equation}

We now consider the situation of \cite[4.9--4.11]{D-TS}.  
Let $L$ be a semisimple algebra over $\mathbb{Q}$ with a positive involution
$*: L \rightarrow L$. Let $F \subset L$ be a subfield in the center of $L$
which is invariant by the involution $*$. 
Let $V$ be a faithful $L$-module which is finite-dimensional over $\mathbb{Q}$. Let $\psi: V \times V \rightarrow \mathbb{Q}$
be an alternating $\mathbb{Q}$-bilinear form such that
\begin{displaymath}
\psi(\ell x, y) = \psi(x, \ell^{*} y), \quad \ell \in L, \; x,y \in V. 
\end{displaymath}
We consider the algebraic group $G$ over $\mathbb{Q}$ given by
\begin{displaymath}
  G(\mathbb{Q}) = \{g \in \GL_{L} (V) \; | \; \psi(gx, gy) = \psi(\mu(g) x,y),
  \; \mu(g) \in F^{\times}\}. 
  \end{displaymath}
There is up to conjugation by an element of $G(\mathbb{R})$ a unique complex
stucture $J: V \otimes \mathbb{R} \rightarrow V \otimes \mathbb{R}$,
$J^{2} = - \id$ which commutes with the action of $L$ and such that
\begin{displaymath}
\psi(Jx,y) 
  \end{displaymath}
is a symmetric and positive definite. Let
$h: \mathbb{S} \rightarrow G_{\mathbb{R}}$ such that  
$h(z)$ acts on $V \otimes \mathbb{R}$ by multiplication by $z$ with respect
to the complex structure just introduced. Then $(G,h)$ is a Shimura datum.
We set
\begin{displaymath}
  \mathbf{t}(\ell) = {\rm Tr}_{\mathbb{C}} (\ell \mid V \otimes \mathbb{R}), \quad
  \ell \in L. 
  \end{displaymath}
The numbers $\mathbf{t}(\ell)$ generate over $\mathbb{Q}$ the Shimura field
$E(G,h)$. The canonical model $\Sh_{\mathbf{K}}(G,h)$ is the coarse moduli
scheme of the following functor $\mathcal{M}(L,V,\psi)$ on the category
of $E$-schemes $S$.
\begin{definition}\label{Shh1d}
  A point of $\mathcal{M}(L,V,\psi)(S)$ is given by the following data: 
  \begin{enumerate}
  \item[(a)] an abelian scheme $A$ over $S$ up to isogeny with an action
    $\iota: L \rightarrow \End^{o} A$, 
\item[(b)] an $F^{\times}$-homogeneous polarization $\bar{\lambda}$ of $A$,
\item[(c)] a class $\bar{\eta}$  modulo $\mathbf{K}$ of $L \otimes \mathbb{A}_f$-module
  isomorphisms  
\begin{displaymath}
\eta: V \otimes \mathbb{A}_f \isoarrow \hat{V}(A). 
\end{displaymath}
such that for each $\lambda \in \bar{\lambda}$ there is locally for the
  Zariski topology a constant 
  $\xi(\lambda) \in (F \otimes \mathbb{A}_f)^{\times}(1)$
  with
  \begin{displaymath}
\psi(\xi(\lambda) v_1, v_2) = E^{\lambda}(\eta(v_1), \eta(v_2)). 
  \end{displaymath} 
\item[(d)] The $L$-module $H_1(A, \mathbb{Q})$ with its Riemann form
  defined by $\lambda$ is isomorphic to $(V, \psi)$, up to a factor
  in $F^{\times}$. 
  \end{enumerate}

  We require that the following condition holds
  \begin{displaymath}
\Trace (\iota(\ell) \mid \Lie A) = \mathbf{t}(\ell), \quad \ell \in L. 
  \end{displaymath}
\end{definition}

We reformulate \cite[5.11]{D-TS}  with our conventions.
\begin{proposition}\label{zentralerTwist1p}
  Let $(G,h)$ be a Shimura datum. Let $Z \subset G$ be the connected center
  of $G$. Let $\delta: \mathbb{S} \rightarrow Z_{\mathbb{R}}$ be a
  homomorphism. Let $E \subset \mathbb{C}$ be a finite extension of
  $\mathbb{Q}$ which contains the Shimura fields $E(G,h)$ and $E(Z,\delta)$. 
  Let $\mathbf{K} \subset G(\mathbb{A}_f)$ be an open and compact subgroup.
  We denote by $\mathrm{Sh}_{\mathbf{K}}(G,h)$ and
  $\mathrm{Sh}_{\mathbf{K}}(G,h\delta)$ the quasi-canonical models over $E$. Let
  \begin{displaymath}
    r_E^{\rm cft}(Z,\delta): \Gal(\bar{E}/E) \rightarrow
    Z(\mathbb{Q}) \backslash Z(\mathbb{A}_f)/(\mathbf{K} \cap Z(\mathbb{A}_f))  
  \end{displaymath}
  be the reciprocity law. There is an isomorphism of schemes over $\bar{E}$
  \begin{equation}\label{deltaTwist1e}
    \mathrm{Sh}_{\mathbf{K}}(G,h) \times_{\Spec E}\Spec\bar{E} \rightarrow
    \mathrm{Sh}_{\mathbf{K}}(G,h\delta) \times_{\Spec E}\Spec\bar{E}
    \end{equation}
  such that for each $\sigma \in \Gal(\bar{E}/E)$ the following diagram
  is commutative,
  \begin{displaymath}
\xymatrix{
\mathrm{Sh}_{\mathbf{K}}(G,h)
  \times_{\Spec E}\Spec\bar{E} \ar[d]_{r^{\rm cft}_E(Z,\delta)(\sigma) \times \sigma_c} \ar[r] 
  & \mathrm{Sh}_{\mathbf{K}}(G,h\delta) \times_{\Spec E}\Spec\bar{E}
  \ar[d]^{id \times \sigma_c}\\
  \mathrm{Sh}_{\mathbf{K}}(G,h) \times_{\Spec E}\Spec\bar{E} \ar[r] & 
  \mathrm{Sh}_{\mathbf{K}}(G,h\delta) \times_{\Spec E}\Spec\bar{E}  . 
   }
  \end{displaymath}
  For varying $\mathbf{K}$ the morphism (\ref{deltaTwist1e}) is compatible
  with the Hecke operators induced by elements $g \in G(\mathbb{A}_f)$. 
  \end{proposition}
\begin{proof}
  By the definition of a canonical model one can reduce the question to the
  case when $G$ is an algebraic torus (cf. \cite[5.11]{D-TS}). Then the proposition
  is a consequence of (\ref{Sh-rec4e}). 
  \end{proof}

We formulate a ''local'' version of the last proposition, keeping the notations
there. We fix a diagram as in (\ref{BZ2e})
\begin{displaymath}
\mathbb{C} \leftarrow \bar{\mathbb{Q}} \rightarrow \bar{\mathbb{Q}}_p.
  \end{displaymath}
It determines a $p$-adic place $\nu$ of $E$. Let
$\mu(\delta): \mathbb{G}_{m,\mathbb{C}} \rightarrow Z_{\mathbb{C}}$ the
homomorphism associated to $\delta$ as usual. It is defined over $E_{\nu}$,
\begin{displaymath}
  \mu_{\nu}: \mathbb{G}_{m,E_{\nu}} \rightarrow Z_{E_{\nu}}.
\end{displaymath}
We consider the homomorphism
\begin{displaymath}
  \mathfrak{r}_{\nu}: \mathbb{G}_{m,E_{\nu}} \overset{\mu_{\nu}}{\longrightarrow}
  Z_{E_{\nu}} \overset{\Nm_{E_{\nu}/\mathbb{Q}_p}}{\longrightarrow} Z_{\mathbb{Q}_p}.
\end{displaymath}
We define $r_{\nu}(Z,\delta)$ as the composite
\begin{equation}\label{local-rec1e}
  r_{\nu}(Z,\delta): E_{\nu}^{\times}
  \overset{\mathfrak{r}_{\nu}^{-1}}{\longrightarrow} Z(\mathbb{Q}_p)
  \rightarrow Z(\mathbb{Q})\backslash Z(\mathbb{A}_f)/
  (\mathbf{K} \cap Z(\mathbb{A}_f)),   
\end{equation}
where the last arrow is induced by the inclusion
$Z(\mathbb{Q}_p) \subset Z(\mathbb{A}_f)$. By local class field theory, this
induces a homomorphism
\begin{displaymath}
  r_{\nu}^{\rm cft}(Z,\delta): \Gal(\bar{E}_{\nu}/E_{\nu}) \rightarrow
  Z(\mathbb{Q})\backslash Z(\mathbb{A}_f)/ (\mathbf{K} \cap Z(\mathbb{A}_f)). 
  \end{displaymath}
\begin{corollary}\label{zentralerTwist1c}
  We denote by $\mathrm{Sh}_{\mathbf{K}}(G,h)_{E_{\nu}}$ and
  $\mathrm{Sh}_{\mathbf{K}}(G,h\delta)_{E_{\nu}}$ the schemes over $E_{\nu}$
  obtained by base change from the canonical models. There is an isomorphism of schemes over $\bar{E}_{\nu}$ 
  \begin{displaymath}
    \mathrm{Sh}_{\mathbf{K}}(G,h)_{E_{\nu}} \times_{\Spec E_{\nu}} \Spec \bar{E}_{\nu} 
    \rightarrow
  \mathrm{Sh}_{\mathbf{K}}(G,h\delta)_{E_{\nu}}\times_{\Spec E_{\nu}}\Spec\bar{E}_{\nu}
    \end{displaymath}
  such that for any $\sigma \in \Gal(\bar{E}_{\nu}/E_{\nu})$ the following
  diagram is commutative
    \begin{displaymath}
\xymatrix{
\mathrm{Sh}_{\mathbf{K}}(G,h)_{E_{\nu}} \times_{\Spec E_{\nu}} \Spec \bar{E}_{\nu}
   \ar[d]_{r^{\rm cft}_{\nu}(Z,\delta)(\sigma) \times \sigma_c} \ar[r] 
   & \mathrm{Sh}_{\mathbf{K}}(G,h\delta)_{E_{\nu}} 
   \times_{\Spec E_{\nu}} \Spec \bar{E}_{\nu} \ar[d]^{id \times \sigma_c}\\
   \mathrm{Sh}_{\mathbf{K}}(G,h)_{E_{\nu}} \times_{\Spec E_{\nu}} \Spec \bar{E}_{\nu}
   \ar[r] & 
  \mathrm{Sh}_{\mathbf{K}}(G,h\delta)_{E_{\nu}}\times_{\Spec E_{\nu}}\Spec \bar{E}_{\nu}.
   }
  \end{displaymath}
  \end{corollary}
\begin{proof}
This follows from the compatibilities of local and global class field theory. 
\end{proof}

Our final topic is  the following variant of \cite[Prop. 1.15]{D-TS}. A similar variant appears in  Kisin \cite[Lem. 2.1.2]{KisinJAMS}.  

\begin{proposition}\label{Chevalley2p}
  Let $S$ be a finite set of prime numbers. 
  Let $M \subset G$ be closed immersion of reductive subgroups over
  $\mathbb{Q}$. We assume that $M$ is the kernel of a homomorphism
  $G \rightarrow T$ to a torus over $\mathbb{Q}$.  Let
  $h: \mathbb{S} \rightarrow M_{\mathbb{R}} \rightarrow G_{\mathbb{R}}$
  be a homomorphism of algebraic groups such that $(M,h)$ and $(G,h)$
  are Shimura data. 
  Let $\mathbf{I} = \mathbf{I}_S \mathbf{I}^S$ be an open and compact subgroup
  $M(\mathbb{A}_f)$, where $\mathbf{I}_S \subset \prod_{p \in S} M(\mathbb{Q}_p)$
  and $\mathbf{I}^S \subset M(\mathbb{A}_f^S)$. Let    
  $\mathbf{K}_S \subset \prod_{p \in S} G(\mathbb{Q}_p)$ be a compact  open 
   subgroup such that
  \begin{displaymath}
\mathbf{K}_S \cap (\prod\nolimits_{p \in S} M(\mathbb{Q}_p)) = \mathbf{I}_S. 
    \end{displaymath}
Then there exists an open compact subgroup
  $\mathbf{K}^S \subset G(\mathbb{A}_f^S)$ which contains $\mathbf{I}^S$ 
  such that the induced morphism of schemes over $\mathbb{C}$ 
  \begin{equation}\label{Einbettung1e}  
    \Sh_{\mathbf{I}}(M,h)_{\mathbb{C}} \rightarrow
    \Sh_{\mathbf{K}_S \mathbf{K}^S}(G,h)_{\mathbb{C}} 
    \end{equation}
is an open and closed immersion. 
  \end{proposition}
\begin{proof}
  Let $Z_G$ be the center of $G$ and let $G^{\mathrm{der}}$ be the derived group.
  The map $Z_G \times G^{\mathrm{der}} \rightarrow G$ is an isogeny. Since
  $G^{\mathrm{der}}$ is mapped to $\{ 1 \} \in T$ we see that $G$ and $Z_G$
  have the same image in $T$. We obtain that the homomorphism
  $Z_G \times M \rightarrow G$ is surjective. Therefore $Z_M \subset Z_G$.
  We obtain a morphism $M^{\mathrm{ad}} \rightarrow  G^{\mathrm{ad}}$ which is an isomorphism.
  Let $X_G$ be the set of conjugates of $h$ by elements of $G(\mathbb{R})$
  and define $X_M$ in the same way. Since the adjoint groups are the same, the induced map 
  $X_M \rightarrow X_G$ is an isomorphism onto a union of connected components
  of $X_G$. This implies that the Shimura varieties
  $\Sh_{\mathbf{I}}(M,h)_{\mathbb{C}}$ and
  $\Sh_{\mathbf{K}_S\mathbf{K}^S}(G,h)_{\mathbb{C}}$ have the same dimension.
  By \cite{KisinJAMS} we may choose $\mathbf{K}^{S}$ in such a way that
  (\ref{Einbettung1e}) is a closed immersion. But since both varieties
  are normal of the same dimension, the induced morphisms on the local
  rings must be isomorphisms. Therefore (\ref{Einbettung1e}) is also open. 
  \end{proof}
The main ingredient of the proof in \cite{KisinJAMS} is the following theorem.
Because it is needed for other purposes in this paper, we state it here. 

\begin{proposition} (Theorem of Chevalley) \; 
  Let $T$ be an algebraic torus over $\mathbb{Q}$. Let
  $\mathcal{E} \subset T(\mathbb{Q})$ be a finitely generated subgroup.
  Let $S$ be a  finite set of rational primes. We denote by
  $\mathbb{A}_{f}^{S}$ the restricted product over all $\mathbb{Q}_{\ell}$
  where $\ell$ runs over all prime numbers $\ell \notin S$. 
  Let $m$ be an integer.
 Then there exists a compact  open subgroup
  $C \subset T(\mathbb{A}_{f}^{S})$ such that
  $C \cap \mathcal{E} \subset \mathcal{E}^m$. 
\end{proposition}

In the case where $L$ is a number field and
$T = \Res_{L/\mathbb{Q}} \mathbb{G}_{m,L}$, this is the first theorem in \cite{Che}.
The general case is easily reduced to this.


\begin{thebibliography}{ABC3} 
 \bibitem[B]{B} J.-F. Boutot, \textit{Uniformisation $p$-adique des vari\'et\'es de Shimura},  S\'eminaire Bourbaki, Vol. 1996/97. Ast\'erisque {\bf 245} (1997), Exp. No. 831, 307--322. 
 
 \bibitem[BC]{BC} J.-F. Boutot, H. Carayol, \textit{ Uniformisation $p$-adique des courbes de Shimura: les th\'eor\`emes de Cherednik et de Drinfeld}, in: Courbes modulaires et courbes de Shimura (Orsay, 1987/1988). Ast\'erisque {\bf 196-197} (1991),  45--158.
 
  \bibitem[BZ]{BZ} J.-F. Boutot, T. Zink, \textit{ The $p$-adic uniformization of Shimura curves}, preprint 95--107, Univ. Bielefeld (1995).
   
  \bibitem[Car]{C} H. Carayol, \textit{Sur la mauvaise r\'eduction des courbes de Shimura}, Compositio math. {\bf 59} (1986), 151-230.
  
  \bibitem[CF]{CF} Cassels-Fr\"ohlich, {Algebraic Number Theory}, 2nd ed. London Math. Soc., 2010.
\bibitem[Ch]{Ch} I. V. Cherednik, \textit{ Uniformization of algebraic curves by discrete arithmetic subgroups of $\PGL_2(k_w)$ with compact quotient spaces},  (Russian) Mat. Sb. (N.S.) {\bf 100}(142) (1976), no. 1, 59--88, 165. 
\bibitem[Che]{Che} C. Chevalley, \textit{Deux th\'eor\`emes d'arithm\'etique}, J. Math. Soc. Japan {\bf 3} (1951), 36--44.
\bibitem[De]{D-TS}P. Deligne, \textit{ Travaux de Shimura}, S\'em. Bourbaki 1970/71, expos\'e 389, Springer Lecture Notes 244 (1971).

\bibitem[DM]{DM} P. Deligne, D. Mumford, \textit{The irreducibility of the space of curves of given genus}. Inst. Hautes Etudes Sci. Publ. Math. {\bf 36} (1969), 75--109. 

\bibitem[DG]{DG} M.Demazure, P.Gabriel, \textit{Groupes Alg\'ebriques},
  Paris, Amsterdam 1970.
\bibitem[Dr]{Dr} V. G. Drinfeld,
  \textit{ Coverings of p-adic symmetric domains},  (Russian) Funkcional.
  Anal. i Prilozen. {\bf 10} (1976), no. 2, 29--40. 

\bibitem[K]{KisinJAMS}
M.~Kisin, \textit{Integral models for {S}himura varieties of abelian type}.  J.
  Amer. Math. Soc. {\bf 23} (2010), no.~4, 967--1012.
\bibitem[KR1]{KRalt} S. Kudla, M. Rapoport, \textit{An alternative description of the Drinfeld $p$-adic half-plane},  Annales de l'Institut Fourier {\bf 64}, no. 3 (2014), 1203--1228. 

\bibitem[KR2]{KRnew} S. Kudla, M. Rapoport, \textit{New cases of $p$-adic uniformization}, Ast\'erisque {\bf 370} (2015), 207--241.

\bibitem[KRZ]{KRZ} S. Kudla, M. Rapoport, Th. Zink,  \textit{On the $p$-adic uniformization of unitary Shimura curves}, arXiv:2007.05211 
 
\bibitem[Mum]{Mum} D. Mumford, \textit{An analytic construction of degenerating curves over complete local rings,} Compositio Math. {\bf 24} (1972), 129--174.

\bibitem[RZ]{RZ} M. Rapoport, Th. Zink,  Period spaces for $p$-divisible groups. Annals  of Mathematics Studies, {\bf 141}, Princeton University Press,  Princeton, 1996.

\bibitem[V]{V} V.E.Voskresenskij, \textit{Algebraic Tori} (in Russian), Izdat. ``Nauka'', Moscow, 1977.
  
\bibitem[Z1]{Z-sR} Th. Zink,
  \textit{\"Uber die schlechte Reduktion einiger Shimuramannigfaltigkeiten,}
  Compositio Math. {\bf 45} (1982), no. 1, 15--107. 

  \end{thebibliography}
\end{document}